\documentclass[11pt,a4paper]{article}

\usepackage[utf8]{inputenc}
\usepackage[english]{babel}
\usepackage{amsmath,amsfonts,amssymb,amsthm,mathrsfs}
\usepackage{lmodern}

\usepackage{tikz-cd}

\usepackage{mathtools}
\mathtoolsset{showonlyrefs,showmanualtags}	

\usepackage{geometry}

\usepackage{bbold} 

\usepackage{tabularx}
\usepackage{xltabular}
\usepackage{wrapfig}

\usepackage[shortlabels]{enumitem}

\usepackage{ytableau}

\usepackage[bookmarks=true]{hyperref}
\usepackage{xcolor}
\hypersetup{
    colorlinks,
    linkcolor={red!50!black},
    citecolor={blue!50!black},
    urlcolor={blue!80!black},
}

\graphicspath{{figures/}}

\theoremstyle{plain}
\newtheorem{theorem}{Theorem}[section]
\newtheorem{corollary}[theorem]{Corollary}
\newtheorem{lemma}[theorem]{Lemma}
\newtheorem{proposition}[theorem]{Proposition}
\newtheorem*{lemma*}{Lemma}
\newtheorem*{theorem*}{Theorem}
\newtheorem*{corollary*}{Corollary}

\theoremstyle{definition}
\newtheorem{definition}[theorem]{Definition}

\newtheorem*{example*}{Example}
\newtheorem*{notation*}{Notation}
\newtheorem*{definition*}{Definition}

\theoremstyle{remark}
\newtheorem{remark}[theorem]{Remark}

\numberwithin{equation}{section}

\newcommand{\R}{\mathbb{R}} 
\newcommand{\RP}{\mathbb{RP}} 
\newcommand{\N}{\mathbb{N}} 
\newcommand{\SU}{\mathrm{SU}}
\newcommand{\SL}{\mathrm{SL}}

\newcommand{\A}{\mathbb{A}}
\newcommand{\B}{\mathbb{B}}
\newcommand{\C}{\mathbb{C}}
\newcommand{\D}{\mathbb{D}}
\newcommand{\E}{\mathbb{E}}
\newcommand{\F}{\mathbb{F}}
\newcommand{\G}{\mathbb{G}}
\renewcommand{\H}{\mathbb{H}}
\newcommand{\I}{\mathbb{I}}
\renewcommand{\L}{\mathbb{L}}

\renewcommand{\P}{\mathbb{P}}
\newcommand{\Q}{\mathbb{Q}}
\renewcommand{\S}{\mathbb{S}}

\DeclareMathOperator{\rank}{\mathrm{rank}}
\newcommand{\col}{\mathrm{col}}
\newcommand{\size}{\mathrm{size}}

\newcommand{\distr}{\mathscr{D}}
\newcommand{\ver}{\mathscr{V}}
\newcommand{\hor}{\mathscr{D}} 
\newcommand{\Cut}{\mathrm{Cut}}

\newcommand{\DOM}{{\rm{DOM}}}

\DeclareMathOperator{\End}{\mathrm{End}}
\DeclareMathOperator{\supp}{\mathrm{supp}}
\newcommand{\ve}{\varepsilon}
\newcommand{\CD}{\mathsf{CD}}

\newcommand{\Geo}{\mathrm{Geo}}
\newcommand{\MCP}{\mathsf{MCP}}
\newcommand{\QCD}{\mathsf{QCD}}
\newcommand{\OptGeo}{{\rm OptGeo}}
\newcommand{\Opt}{{\rm Opt}}
\newcommand{\Ent}{\mathrm{Ent}}
\newcommand{\U}{\mathrm{U}}
\newcommand{\Dom}{\mathrm{Dom}}

\newcommand{\cK}{\mathcal{K}}
\newcommand{\cD}{\mathcal{D}}

\newcommand{\sfs}{\mathsf{s}}

\newcommand{\rv}{\mathrm{v}}
\newcommand{\rs}{\mathrm{s}}

\newcommand{\Cpl}{\mathrm{Cpl}}

\newcommand{\loc}{\mathrm{loc}}
\newcommand{\mui}{\mu_\infty}
\newcommand{\Lip}{\mathrm{Lip}}

\newcommand{\Prob}{\mathcal{P}}
\newcommand{\sM}{\mathscr{M}}

\DeclareMathOperator{\spn}{\mathrm{span}}
\DeclareMathOperator{\tr}{\mathrm{tr}}
\DeclareMathOperator{\diam}{\mathrm{diam}}
\DeclareMathOperator*{\esssup}{ess\,sup}
\DeclareMathOperator*{\essinf}{ess\,inf}
\newcommand{\Ric}{\mathrm{Ric}}
\newcommand{\eps}{\varepsilon}
\newcommand{\di}{\mathrm{d}}
\newcommand{\mm}{\mathfrak m}
\newcommand{\ee}{{\rm e}}
\newcommand{\sfd}{\mathsf d}
\newcommand{\sfD}{\mathsf D}
\newcommand{\sfG}{\mathsf G}
\newcommand{\NN}{\dim_{\mathrm{geo}}}
\newcommand{\NH}{\dim_{\mathrm{Haus}}}

\DeclareMathSymbol{\shortminus}{\mathbin}{AMSa}{"39}

\title{Unified synthetic Ricci curvature lower bounds for Riemannian and sub-Riemannian structures}

\author{
Davide Barilari\thanks{Barilari: Dipartimento di Matematica ``Tullio Levi-Civita'', Universit\`a di
Padova, Via Trieste 63, Padova, Italy. E-mail: \href{mailto:barilari@math.unipd.it}{barilari@math.unipd.it}} 
\and
Andrea Mondino\thanks{Mondino: Mathematical Institute, University of Oxford (UK). E-mail:  \href{mailto:andrea.mondino@maths.oxford.ac.uk}{mondino@maths.oxford.ac.uk}}
\and
Luca Rizzi\thanks{Rizzi: SISSA, via Bonomea 265, 34136 Trieste, Italy and Univ. Grenoble Alpes, CNRS, Institut Fourier, F-38000 Grenoble, France. 
E-mail: \href{mailto:luca.rizzi@sissa.it}{luca.rizzi@sissa.it}}
}

\begin{document}

\maketitle

\begin{abstract}
Recent advances in the theory of metric measures spaces on the one hand, and of sub-Riemannian ones on the other hand, suggest the possibility of a ``great unification'' of Riemannian  and sub-Riemannian geometries in a comprehensive framework of synthetic Ricci curvature lower bounds, as put forth in \cite[Sec.\@ 9]{CVNB}. With the aim of achieving such a unification program, in this paper we initiate the study of gauge metric measure spaces.
\end{abstract}

\tableofcontents

\section{Introduction}

The subject of this work is the synthetic treatment of curvature lower bounds. To illustrate our contribution, let us start by recalling the celebrated theory of Alexandrov spaces. These are metric spaces $(X,\sfd)$ where curvature bounds are defined via comparison of geodesic triangles in $X$  with corresponding ones in model surfaces of constant curvature. The key observation is that Toponogov triangle comparison theorem gives a purely metric characterization of lower bounds for the sectional curvature on smooth Riemannian manifolds, which yields a synthetic definition for metric spaces. 

A synthetic characterization of \emph{Ricci} curvature lower bounds needs an additional structure, namely  a  reference measure $\mm$: following the seminal works by Gromov   \cite{Gr}, by Fukaya on spectral convergence \cite{Fukaya}, by Cheeger-Colding  on Ricci limit spaces \cite{CheegerColdingJDG1}, it emerged that a natural framework for Ricci curvature lower bounds  is  the one of metric measure spaces.

A key fact is that, on a smooth Riemannian manifold, Ricci curvature lower bounds can be characterized in terms of suitable convexity inequalities of entropy functionals in the Wasserstein space \cite{CEMS,vRSt}, after \cite{McC97, OttoVillani}.  
This line of research culminated in the seminal work by Lott-Sturm-Villani in the 2000's \cite{lottvillani:metric,sturm:I,sturm:II}, introducing the synthetic $\CD(K,N)$ conditions for metric measure spaces. The latter conditions depend on two parameters: $K\in \R$ playing the role of a lower bound on the Ricci curvature, and $N\geq 1$ playing the role of an upper bound on the dimension. Concretely, $K$ and $N$ enter into the theory via model  \emph{distortion coefficients} $\beta_{K,N}$, encoding the effect of the Ricci curvature on the distortion of the measure along geodesics. These coefficients have the form:
\begin{equation}\label{eq:defbetaKNK>0}
 (\beta_{K,N})_t(\theta) :=
 t  \, \frac{\sin\left(t\theta \sqrt{K/(N-1)}\right)^{N-1}}{\sin\left(\theta \sqrt{K/(N-1)}\right)^{N-1}}, \qquad  \forall\,t\in [0,1],
\end{equation}
for  $0<K\theta^2<(N-1)\pi^2$ and with standard interpretation otherwise, see Section \ref{sec:howtorecover}.

A fundamental aspect of the $\CD(K,N)$ condition is that it enjoys compactness and stability properties under pmGH convergence. We recall that the class of $\CD(K,N)$ spaces includes Riemannian manifolds with Ricci curvature bounded from below, as well as Finsler ones \cite{Ohta2009}. 

However, there is an important class of metric measure spaces which does not fit into this framework: sub-Riemannian structures. These are smooth manifolds endowed with a length metric obtained by considering only curves that are tangent to a bracket-generating vector distribution (see Appendix \ref{a:SR} for a self-contained account).

The simplest examples are the Heisenberg groups. In \cite{NJ09}, Juillet proved that the $\CD(K,N)$ condition fails for all values of $K\in \R$ and $N\geq 1$ on the Heisenberg groups, and indeed this is the case also all sub-Riemannian manifolds \cite{AmbStef, NJ21, MagnaRossi22, RS-Failure} that are not Riemannian.  Moreover, in \cite{NJ09}, it was also proved that a weaker synthetic condition, known as measure contraction property $\MCP(K,N)$, holds in Heisenberg groups for suitable values of $K$ and $N$. This property has been proved to hold in more general Carnot groups \cite{BR-realanalMCP, RiffordCarnot}.

Recently, in their seminal work  \cite{BKS}, Balogh-Krist\'aly-Sipos  showed that, despite the failure of the classical $\CD(K,N)$ condition,  weaker entropy inequalities hold  in the Heisenberg groups $\H^d$, with the following distortion coefficients:
\begin{equation}\label{eq:betaHeisIntro}
\beta^{\H^d}_t(\theta) :=t \, \dfrac{ \sin\left(\frac{t\theta}{2}\right)^{2d-1}\left[\sin\left(\frac{t\theta}{2}\right)-\frac{t\theta}{2} \cos\left(\frac{t\theta}{2}\right)\right]}{ \sin\left(\frac{\theta}{2}\right)^{2d-1}\left[\sin\left(\frac{\theta}{2}\right)-\frac{\theta}{2} \cos\left(\frac{\theta}{2}\right)\right]},  \qquad \forall\,t\in [0,1],
 \end{equation}
for $0<\theta<2\pi$, and standard interpretation otherwise. See also \cite{BKS2} for similar inequalities for general co-rank $1$ Carnot groups.

Two remarkable differences between the Riemannian  distortion coefficients \eqref{eq:defbetaKNK>0}  and the Heisenberg ones \eqref{eq:betaHeisIntro} are:
\begin{itemize}
\item as expected, the functional expressions  \eqref{eq:defbetaKNK>0} and \eqref{eq:betaHeisIntro} do not match;
\item more strikingly, when used to encode measure distortion, while in the Riemannian coefficients \eqref{eq:defbetaKNK>0}, the parameter $\theta$ takes as value the \emph{length}, in the Heisenberg  ones  \eqref{eq:betaHeisIntro}, $\theta$ takes as value the \emph{curvature}  of geodesics (as curves in $\R^{2d+1}$).
\end{itemize}

A different approach to Ricci curvature lower bounds in sub-Riemannian manifolds was put forward by Baudoin-Garofalo \cite{BG17}. Inspired by the Bakry-\'Emery semi-group techniques, they generalized curvature-dimension conditions  for sub-Riemannian structures with symmetries, introducing a suitable generalization of Bochner's identity and $\Gamma$-calculus. See the lecture notes \cite{BaudoinBook} and references therein for an account of this theory. We also mention the recent work by Stefani \cite{Stefani-Heat} which, in the setting of Carnot groups, establishes a first link between the Lagrangian approach to Ricci curvature lower bounds (dealing with convexity-type properties of entropy along Wasserstein geodesics) and the Eulerian one (focused instead on the properties of the heat flow).

Motivated by the aforementioned contributions, in his 2017 Bourbaki Seminar, Villani envisioned the possibility of a ``great unification'' of Riemannian  and sub-Riemannian geometries in a comprehensive theory of synthetic Ricci curvature lower bounds. See \cite[Sec.\@ 9, conclusions et perspectives]{CVNB}, and \cite[Rmk.\@ 14.23]{Vil}. Developing such a theory is the ambition of the present paper.

\subsection*{Gauge metric measure spaces}
Let  $(X,\sfd,\mm)$ be a metric measure space, m.m.s.\@ for short. We add a supplementary structure to the metric measure one, namely a non-negative Borel function $\sfG: X\times X\to [0,+\infty]$, that we call   \emph{gauge function}.
The quadruple  $(X,\sfd,\mm,\sfG)$ will be called \emph{gauge metric measure space}. 

The following analogy explains the role of the gauge function. A Riemannian manifold has Ricci curvature bounded from below by $K \in \R$ if for any pair of points $x,y$ and any geodesic $\gamma$ between $x$ and $y$ it holds
\begin{equation}
\Ric(\dot\gamma,\dot\gamma) \geq K \|\dot\gamma\|^2 = K  \sfd(\gamma_0,\gamma_1)^2.
\end{equation}
The distance function $\sfd$ is used in the right hand side as a \emph{gauge} to measure the \emph{extent} of the lower Ricci curvature bound, quantified by the constant $K$. The idea is to replace the distance $\sfd$ with a general gauge function $\sfG$.  

Gauge functions will be a key object in our extension of the synthetic theory of Ricci curvature bounds to the sub-Riemannian setting, where it is now well-understood that the effect of geometry on transport inequalities may not depend uniquely on the distance, but rather on other intrinsic sub-Riemannian quantities  (see \cite{BKS}, \cite[Sec.\@ 8.1]{BRInv}).

\subsection*{Distortion coefficients}
 Motivated by the comparison theory in sub-Riemannian geometry  \cite{AG1, AG2, ZeLi,  ZeLi2,  HughenPhD, AAPL-Ricci, LeeLiZel-Sasakian, BR-comparison, BRMathAnn}  (from the Hamiltonian viewpoint) and the discussion above, we now introduce general distortion coefficients.

Let $\sfs:[0,+\infty)\to \R$ be a continuous function and $N\in [1,+\infty)$ such that
\begin{equation}\label{eq:fftcN-Intro}
\sfs(\theta)=c \, \theta^{N}+o(\theta^{N}) \qquad \text{ as $\theta\to 0$},
\end{equation}
for some $c >0$. The parameter $N$ will be a sharp upper bound for a new notion of dimension, which is in general different from the Hausdorff one, see Section \ref{Sec:Dim}. 
Denote:
\begin{equation}\label{eq:defcD-Intro}
\cD:=\inf\{\theta>0\mid \sfs(\theta)=0\}.
\end{equation}  
It is clear that $\cD>0$. The latter will be a sharp upper bound on the gauge function, see Section \ref{sec:diam}.
 Define the \emph{distortion coefficient} $\beta_{(\cdot)}(\cdot):[0,1]\times [0,+\infty]\to [0,+\infty]$ as
\begin{align}\label{eq:defbeta-Intro}
 (t,\theta)\in [0,1]\times [0,+\infty]\mapsto \beta_{t}(\theta):=
 \begin{cases}
 t^{N} &\theta =0, \\
\dfrac{\sfs(t\theta )}{\sfs(\theta)}& 0<\theta <\cD,\\
 \displaystyle  \liminf_{\phi \to \cD^-} \dfrac{\sfs(t\phi )}{\sfs(\phi)} & \theta \geq \cD.\end{cases}
\end{align}
Notice that, when properly understood,  both the Riemannian  \eqref{eq:defbetaKNK>0} and  the Heisenberg  \eqref{eq:betaHeisIntro}  distortion coefficients  are obtained as in \eqref{eq:defbeta-Intro}, for a suitable $\sfs$.

\medskip
In applications to sub-Riemannian geometry, the function  $\sfs$ in \eqref{eq:fftcN-Intro} is chosen in a class of models characterized as solutions to suitable ODEs. We illustrate this characterization in Section \ref{sec:comparison}, and in particular we refer to Proposition \ref{p:basic} and Remark \ref{r:link}, which provides a bridge between the synthetic viewpoint and sub-Riemannian geometry. Here we develop the theory in full generality, without further constraints on the function $\sfs$.

\subsection*{Entropy functionals}
We consider the metric space  $(\mathcal{P}_{2}(X), W_2)$ of probability measures with finite second moment endowed with the  Kantorovich-Rubinstein-Wasserstein  quadratic transportation distance, see Section \ref{Sec:OptTransp} for the definitions.
A $W_2$-geodesic $(\mu_t)_{t\in [0,1]}$ can be equivalently represented by a probability measure  $\nu$ on the space of geodesics $\Geo(X)$, with $\mu_t=(\ee_t)_{\sharp} \nu$, where $\ee_t:\Geo(X)\to X$ is the evaluation map at time $t$. Such a $\nu$ is called \emph{optimal dynamical plan} from $\mu_0$ to $\mu_1$, and the set of such measures is denoted by $\OptGeo(\mu_0, \mu_1)$.  Let also $\Prob_{ac}(X,\mm)$ be the space of probability measures that are absolutely continuous with respect to $\mm$.

Recall that, for $\mu \in \mathcal{P}_{2}(X)$, its relative (Boltzmann-Shannon) entropy  is defined by
\begin{equation}\label{def:Ent-Intro}
\Ent(\mu|\mm): = \int_{X} \rho \log \rho \, \mm,\qquad  \text{ if $\mu = \rho \, \mm\in \Prob_{2}(X)\cap \Prob_{ac}(X,\mm)$},
\end{equation}
in case  $\rho\log \rho\in L^{1}(X,\mm)$,  otherwise we set $\Ent(\mu|\mm) := +\infty$.

Let $\Dom(\Ent(\cdot|\mm))$ be the finiteness domain of the entropy and 
\begin{align}
\Prob_{bs}(X,\sfd,\mm)&:=  \{\mu\in \Prob(X,\sfd)\mid \text{ $\supp \, \mu$ is bounded and }  \supp \, \mu\subseteq \supp \, \mm \},\\
\Prob_{bs}^*(X,\sfd,\mm)&:= \Dom(\Ent(\cdot|\mm)) \cap \Prob_{bs}(X,\sfd,\mm).
\end{align}
The subspaces $\Prob_{bs}(X,\sfd,\mm)$ and $\Prob_{bs}^*(X,\sfd,\mm)$ will play a key role throughout the paper.

In order to formulate ``dimensional''  Ricci curvature lower bounds, it is convenient  to introduce also  the following dimensional entropy (cf.\@ \cite{EKS}):
\begin{equation}\label{eq:defUn-Intro}
\U_{n}(\mu|\mm) : = \exp\left(-\frac{\Ent(\mu|\mm)}{n} \right), \qquad n\in [1,+\infty),
\end{equation}
with the understanding that $\U_{n}(\mu|\mm):=0$ if $\mu\notin \Dom(\Ent(\cdot|\mm))$.

\subsection*{$\CD(\beta,n)$ spaces and $\MCP(\beta)$ spaces}

We introduce synthetic Ricci curvature lower bounds on gauge m.m.s.: the Curvature-Dimension condition $\CD(\beta,n)$ and the Measure Contraction Property $\MCP(\beta)$.

\begin{definition*}[Definition \ref{def:CDMCPbetan}]
Let $n\in [1,+\infty)$, and $\beta$ as in \eqref{eq:defbeta-Intro}.  A gauge  metric measure space $(X,\sfd,\mm, \sfG)$ satisfies:
\begin{itemize}
\item $\CD(\beta, n)$ if for all $\mu_0\in \Prob_{bs}(X,\sfd,\mm)$, $\mu_1\in \Prob_{bs}^{*}(X,\sfd,\mm)$ with $\supp \mu_0\cap \supp \mu_1 =\emptyset$,  there exists a $W_2$-geodesic $(\mu_t)_{t\in[0,1]} \subset \Prob_2(X,\sfd)$ connecting them, induced by $\nu\in \OptGeo(\mu_0,\mu_1)$, such that it holds
\begin{multline}\label{eq:defCDbetan-Intro}
\U_{n}(\mu_{t}|\mm)\geq  \exp\left( \frac{1}{n}\int_{\Geo(X)} \log \beta_{1-t}\big(\sfG(\gamma_1,\gamma_0))\, \nu(\di\gamma) \right) \U_{n}(\mu_{0}|\mm)  \\
+\exp\left( \frac{1}{n}\int_{\Geo(X)} \log \beta_{t}\big(\sfG(\gamma_0,\gamma_1)\big) \, \nu(\di\gamma) \right)   \U_{n}(\mu_{1}|\mm), \qquad \forall\, t\in (0,1),
\end{multline}
with the convention that $\infty \cdot 0 = 0$. 
\item $\MCP(\beta)$ if for any $\bar{x} \in \supp\mm$ and $\mu_1\in \Prob_{bs}^{*}(X,\sfd,\mm)$   with $\bar{x}\notin \supp \mu_1$ there exists a $W_2$-geodesic $(\mu_t)_{t\in[0,1]} \subset \Prob_2(X,\sfd)$ from $\mu_0=\delta_{\bar{x}}$ to $\mu_1$ such that
\begin{equation}\label{eq:defMCPbeta-Intro}
\U_{n}(\mu_{t}|\mm)\geq \exp\left( \frac{1}{n}\int_{X}  \log \beta_{t}\big(\sfG(\bar{x},x)\big)  \, \mu_1(\di x) \right) \U_{n}(\mu_{1}|\mm) , \qquad \forall\, t\in (0,1),
\end{equation}
for some (and then every) $n\geq 1$.
\end{itemize}
\end{definition*}

\noindent
We anticipate here a few remarks.
\begin{enumerate}
\item The $\MCP(\beta)$ condition does not depend on the value of $n$. See Remark \ref{rem:eqMCPbeta}.
\item Let us stress that non-absolutely continuous $\mu_0$ are allowed, which by construction gives the implication $\CD(\beta, n) \Rightarrow  \MCP(\beta)$. See Remarks \ref{rem:CDimpliesMCP} and \ref{rem:mu0Pbs}.
\item Conversely, assuming that $(X,\sfd,\mm)$ supports interpolation inequalities for densities  with dimensional parameter $n$ (in the sense of Definition \ref{def:mcptocd}), we show that $\MCP(\beta) \Rightarrow  \CD(\beta,n)$. See Theorem \ref{thm:mcptocd}.
\item The $\CD / \MCP$ conditions  imply that  $\Dom(\Ent(\cdot | \mm))\subset \Prob_{2}(X,\sfd)$  (with metric $W_{2}$) and $\supp \mm$ (with metric $\sfd$) are length spaces.  See Remark \ref{rem:LengthSpace}.
\end{enumerate}

\subsection*{Compatibility with classical synthetic theories} 
The $\CD(\beta,n)$ and $\MCP(\beta)$ conditions satisfy the following compatibility properties:
\begin{itemize}
\item \emph{Compatibility  with Lott-Sturm-Villani's $\CD$}: by choosing as distortion coefficients the Riemannian ones, i.e.\@ $\beta=\beta_{K,N}$ as in \eqref{eq:defbetaKNK>0},  and as gauge function the distance, i.e.\@ $\sfG=\sfd$,  we show that, for essentially non-branching m.m.s., the Lott-Sturm-Villani's $\CD(K,N)$ conditions are equivalent to the corresponding $\CD(\beta_{K,N}, N)$. See Section \ref{sec:howtorecover}.
\item \emph{Compatibility with Ohta-Sturm's $\MCP$}: as above, we show that the Ohta-Sturm   $\MCP(K,N)$  are equivalent to the corresponding $\MCP(\beta_{K,N})$. See Section \ref{sec:howtorecover}.
\item  \emph{Compatibility  with Balogh-Krist\'aly-Sipos}:  by choosing as distortion coefficients the Heisenberg ones, i.e.\@ $\beta=\beta^{\H^d}$ as in \eqref{eq:betaHeisIntro},  and as gauge function $\sfG(x,y)=\theta^{x,y}$ where $\theta^{x,y}$ is the curvature of the geodesic from $x$ to $y$ (as a curve in $\R^{2d+1}$), we show that the Heisenberg group $\H^d$ (endowed with the Carnot-Carath\'eodory distance, and the $2d+1$ Lebesgue measure) satisfies the $\CD(\beta^{\H^d}, 2d+1)$ condition. See Section \ref{sec:howtorecover2}.
\item \emph{Compatibility with E.\,Milman's $\mathsf{CGTD}$}:  the Curvature-Geodesic-Topological-Di\-men\-sion conditions $\mathsf{CGTD}(K,N,n)$ introduced by E.\@ Milman in \cite[Sec.\@ 7]{MilmanSR}, for $K\in \R$, $n\geq 1$ and $N\geq n$ is equivalent to the $\CD(\beta_{K,N},n)$ condition, for gauge function $\sfG=\sfd$. See Remark \ref{rem:Milman}.
\end{itemize}

\subsection*{Geometric consequences} 
One of the novel features of the $\CD(\beta,n)$ and $\MCP(\beta)$ conditions above is the interplay of the gauge function $\sfG$ (measuring the extent of the Ricci curvature lower bounds) and the distance $\sfd$ (governing the geodesics and thus optimal transport). This leads to a decoupling between metric vs distortion aspects in the classical geometric consequences of the curvature-dimension conditions. In Section \ref{sec:geometricconsequences} we will obtain the following results:

\begin{itemize}
\item \emph{Generalized Brunn-Minkowski inequality}: given two sets $A_{0}, A_{1}$, we estimate from below the volume of the set $A_t$ of $t$-intermediate points of geodesics from $A_{0}$ to $A_{1}$,  with a distortion which is quantified \emph{purely in terms of the gauge function $\sfG$ and the general distortion coefficients $\beta$}. See Section \ref{sec:Brunn-Mink}. 
\item \emph{$\MCP \Rightarrow \CD$} for spaces supporting interpolation inequalities. See Section \ref{s:MCPtoCD}.
\item \emph{Gauge-diameter estimates}: the parameter $\cD$ in \eqref{eq:defcD-Intro} yields an upper bound on the essential supremum of the gauge function,  called gauge diameter. A remarkable difference with respect to classical Bonnet-Myers Theorem is that the gauge diameter can be bounded also for non-compact metric measure spaces (e.g.\@ the Heisenberg group, where the estimate we obtain is sharp) and, conversely, unbounded for compact ones. See Section \ref{sec:diam}.
\item \emph{Local doubling inequalities for metric balls}: in contrast with the classical case when ``locality'' is measured in terms of the distance, here it is expressed in terms of the gauge function which also determines the doubling constant. See Section \ref{Sec:Doubling}.
\item \emph{Geodesic dimension estimates}:  the parameter $N$ occurring in \eqref{eq:fftcN-Intro} yields an upper bound for the Hausdorff dimension and for a  notion of dimension for m.m.s.\@ recently introduced in \cite{ABR-curvature, R-MCP}, the so-called geodesic dimension. See Section \ref{Sec:Dim}. The latter bound is sharp while the former is not, see Remark \ref{rmk:sharpnessgeodim}.  
\item \emph{Generalized Bishop-Gromov inequalities}: we obtain volume estimates on the (truncated)  sub-level sets of the gauge function. Such sets are reminiscent of the ``butterfly-shaped'' sets appearing in \cite[Fig.\@ 2]{NJ09}. For the validity of the Bishop-Gromov inequality, we need an additional compatibility condition between the distance and the gauge function, which we name \emph{meek property}. See Section \ref{s:GBGI}.
\end{itemize}

\subsection*{Stability and compactness}
 Key features of the Lott-Sturm-Villani's theory are the stability and compactness properties with respect to the measured-Gromov-Hausdorff convergence (mGH for short).

For  a sequence of gauge m.m.s.\@ it is natural to consider the mGH convergence of the metric measure structures, coupled with  an additional notion of convergence for the gauge functions.
A naive generalisation of  the mGH convergence to gauge functions,  see \eqref{eq:iiforgauge}, would fail for  natural sequences  (e.g.\@ convergence to the tangent cone for sub-Riemannian structures, cf. Section \ref{sec:ex:tangent}). This is due to the low regularity of the gauge functions which in general are not even continuous, in sharp contrast with the Lott-Sturm-Villani's theory where $\sfG=\sfd$ is clearly $1$-Lipschitz. 

For this reason, in Section \ref{sec:stabilityandcompactness}, we introduce a significantly weaker condition which provides a good balance yielding both \emph{stability} and \emph{compactness} properties for the $\CD(\beta,n)$ condition and the $\MCP(\beta)$.  We ask for a sort of  $L^1_{\loc}$ convergence of gauge functions in varying spaces, and a regularity condition on the limit gauge  function outside of the diagonal. Compared to the classical Lott-Sturm-Villani's theory, the low regularity of the gauge functions and their weaker convergence introduce new challenges in the proof. 

\subsection*{The vector-valued case}
 In Section \ref{sec:vectorial} we extend the above theory to vector-valued gauge functions, namely 
\begin{equation} 
\sfG: X\times X\to \RP^{m}_{+},
\end{equation}
where  $\RP^{m}_{+}$ is a projective semi-space. This generalization is necessary in order to obtain sharp estimates for important classes of sub-Riemannian examples, see Section \ref{sec:ex:vectorial}.

\subsection*{Natural gauge functions} 
In the Heisenberg group a natural gauge function is the quantity $\theta$ mentioned earlier in the introduction ($\theta^{x,y}$ is the curvature of the geodesic from $x$ to $y$ as a curve of $\R^{2d+1}$). A key observation is that $\theta$ can be written in terms of the Carnot-Carath\'eodory distance and the norm of its gradient, computed with respect to a Riemannian extension. 

This observation suggests a synthetic way to define natural gauge functions for any metric measure space $(X,\sfd,\mm)$ equipped with a reference metric $\sfd_R$ (the subscript stands for \emph{Reference}, or \emph{Riemannian}). The construction is well-suited to sub-Riemannian geometry where, in several cases, one has natural Riemannian extensions serving as reference.  

Natural gauge functions will be constructed out of two main building blocks: the distance function and the \emph{$\sfD$ function} defined as follows 
\begin{equation}\label{eq:defD-Intro}
\sfD :X\times X \to [0,+\infty], \qquad \sfD(x,y) :=  \limsup_{z,w\to x} \frac{|\sfd^2(z,y)-\sfd^2(w,y)|} {2 \sfd_R(z,w)}  ,
\end{equation}
with the convention that the limit is $0$ if $x$ is an isolated point.

Let us show how $\theta$ is obtained via the above procedure, in the Heisenberg group $(\H^d,\sfd_{cc},\mathscr{L}^{2d+1})$. As extension $\sfd_R$, we choose the left-invariant Riemannian metric obtained declaring the Reeb field $\partial_z$ to be of unit length and orthogonal to the Heisenberg distribution.  For $(x,y)$ out of the cut locus, we prove in Proposition \ref{prop:naturalobject} that
\begin{equation}
\theta^{x,y}=\sqrt{\sfD^2(x,y)-\sfd_{cc}^2(x,y)}.
\end{equation}
We note that $\sfD \geq \sfd$, see Proposition \ref{prop:propertiesD}\ref{i:propertiesD1}. Natural gauge functions are studied in  Section \ref{sec:nonsmooth}, in particular:
\begin{itemize}
\item in Section \ref{sec:D} we define the $\sfD$ function and we establish its basic properties;
\item in Section \ref{sec:ngf} we define \emph{natural gauge functions} on metric measure spaces equipped with a reference metric, of which $\sfD$ and $\sfd$ are particular cases;
\item in Section \ref{sec:naturalinSR} we establish further properties of natural gauge functions, namely:
\begin{itemize}
\item the meek condition used for the generalized Bishop-Gromov Theorem; see Section \ref{subsec:meek};
\item the regularity and  boundedness properties required for the stability and compactness results; see Sections \ref{sec:regularityGExamples} and  \ref{sec:boundedness} respectively.
\end{itemize}
We prove that all these properties follow from natural hypotheses in sub\--\-Rie\-man\-nian geometry, including: step $\leq$ 2, minimizing Sard property, $*$-minimizing Sard property, ideal, absence of non-trivial Goh geodesics, real-analyticity. For a glimpse of all the implications see Figure \ref{fig:implications};
\item in Section \ref{sec:Doutofcut} we show how the natural gauge functions are related to the sub-Riemannian distance out of the cut locus.
\end{itemize}

\subsection*{Compatibility with the Hamiltonian theory of curvature}

In Section \ref{sec:comparison} we establish the compatibility of the synthetic $\CD(\beta,n)$ condition with the Hamiltonian theory of Ricci curvature lower bounds for sub-Riemannian manifolds. The latter has its roots in the pioneering works by Agrachev and Gamkrelidze on the geometry of curves in Lagrange Grassmannians \cite{AG1,AG2}, and was subsequently developed by Zelenko and Li \cite{ZeLi} with the introduction of the so-called \emph{canonical frame}. This technical tool, generalizing in a very broad sense the notion of parallel transport, has been pivotal in the development of the comparison theory for sub-Riemannian structures from the Hamiltonian point of view, started in \cite{ZeLi2,AAPL-Ricci,LeeLiZel-Sasakian}. Sub-Riemannian Ricci curvatures were finally introduced in \cite{BR-comparison} in full generality, as \emph{partial traces} of the canonical curvatures. The corresponding comparison theory for distortion coefficients was finally set in \cite{BRMathAnn}.

We illustrate such a theory in Section \ref{sec:comparison}, in a form adapted to our purposes. For the sake of self-consistency, in Appendix \ref{a:canonicalframe}, we include an account of the Agrachev-Zelenko-Li theory, that is used extensively in Sections \ref{sec:comparison}-\ref{sec:fat}.

We depict here the main ideas. Let $(M,\sfd,\mm)$ be a sub-Riemannian metric measure space, i.e.\@ $M$ is a smooth manifold, $\sfd$ is a sub-Riemannian distance, and $\mm$ is a smooth measure, see Definition \ref{def:srmms}. The paradigm of sub-Riemannian comparison is that, for the generic geodesic $\gamma$, one can decompose the tangent space along $\gamma$ in a family of subspaces, depending on the Lie bracket structure. These are labeled by the boxes $\alpha$ of a Young diagram. For each subspace we define a \emph{canonical Ricci curvature} $\mathfrak{Ric}^\alpha_\gamma$ as the corresponding partial trace of the canonical curvature. We remark that in the Riemannian case there is only one such a box $\alpha$, and
\begin{equation}
\mathfrak{Ric}^\alpha_{\gamma} = \Ric(\dot\gamma,\dot\gamma),
\end{equation}
recovering the classical Ricci tensor along that geodesic.

For a general sub-Riemannian metric measure space $(M,\sfd,\mm)$, consider the ``true'' distortion coefficient of the m.m.s.:
\begin{equation}
\beta_t^{(M,\sfd,\mm)}(x,y):=\limsup_{r\to 0^+}\frac{\mm(Z_t(\{x\},B_r(y)))}{\mm(B_r(y))}, \qquad \forall\, x,y\in M,\quad t\in [0,1],
\end{equation}
see Definition \ref{def:truedistcoeff}. For fixed $(x,y)$ out of the cut locus, there is a unique geodesic $\gamma$ joining $x$ with $y$, and when \emph{every} canonical Ricci curvature along $\gamma$ is bounded from below, $\mathfrak{Ric}^\alpha_\gamma\geq \kappa_{\alpha}$, for some $\kappa_\alpha \in \R$, and all $\alpha$, then the following comparison holds:
\begin{equation}\label{eq:thissense-intro}
\beta_t^{(M,\sfd,\mm)}(x,y)\geq \beta_t^{\mathrm{mod}}, \qquad \forall\, t\in [0,1].
\end{equation}
Here, $\beta^{\mathrm{mod}}$ is a \emph{model} distortion coefficient. We illustrate its properties.
\begin{itemize}
\item $\beta^{\mathrm{mod}}$ is associated with a variational problem on $\R^n$, of Linear-Quadratic type, coming from optimal control theory. This is a change of paradigm with respect to classical Riemannian comparison theory, where model distortion coefficients are the ones of Riemannian space forms, while the models in the present theory are not even metric spaces. See Section \ref{sec:LQ}.
\item $\beta^{\mathrm{mod}}$ can be computed explicitly by solving a Hamiltonian ODE,  which is identified by the values of the Ricci lower bounds $\kappa_\alpha$, and the structure of the Young diagram associated with the geodesic $\gamma$. See Section \ref{sec:constcurvmodels}.
\item $\beta^{\mathrm{mod}}$ has the form \eqref{eq:defbeta-Intro}, in other words it is a model distortion coefficient in the sense of the synthetic theory. This is not by chance, and in fact Proposition \ref{p:basic}, in particular item \ref{i:basic-6}, is the bridge between the curvature-dimension theory of Section \ref{sec:CDbeta} and the comparison theory of Section \ref{sec:comparison}.
\item When the Young diagram of $\gamma$ is of Riemannian type (i.e.\@ it consists of only one box), $\beta^{\mathrm{mod}}$ recovers the familiar Riemannian coefficients. See Section \ref{ss:Riemanniancase}.
\item The general class of $\beta^{\mathrm{mod}}$ contains \emph{all} distortion coefficients observed so far in sub-Riemannian geometry: Heisenberg groups \cite{BKS}, corank $1$ Carnot groups \cite{BKS2}, Sasakian \cite{AAPL-Ricci,LeeLiZel-Sasakian} and $3$-Sasakian manifolds \cite{RS-3Sas}. See Sections \ref{ss:Sasakiancase} and \ref{ss:twocolumnscase}.
\end{itemize}

The comparison result in the sense of \eqref{eq:thissense-intro} is Theorem \ref{thm:maincomparison}, describing the distortion along a single geodesic. If the Ricci curvature bounds hold uniformly, in the precise technical sense of Definition \ref{def:srbddbelow}, then one obtains Theorem \ref{thm:RiccbddbelowareCD}.

\begin{theorem*}[Ideal sub-Riemannian structures with Ricci bounded below are $\CD$]
Let $(M,\sfd,\mm)$ be an ideal sub-Riemannian metric measure space, with $n=\dim M$, equipped with a finite gauge function $\sfG : M \times M \to \R^m_+$, $m\in \N$, with Ricci curvatures bounded from below in the sense of Definition \ref{def:srbddbelow}, and let $\beta$ be the corresponding distortion coefficient. Then $(M,\sfd,\mm,\sfG)$ satisfies the $\CD(\beta,n)$ condition.
\end{theorem*}

This result establishes the compatibility of the $\CD(\beta,n)$ theory with the one of sub-Riemannian Ricci bounds.

\subsection*{Compact fat structures satisfy the curvature-dimension condition}

Clearly, any compact Riemannian manifold has Ricci curvature bounded below, and thus it verifies the $\CD(\beta,n)$ condition for suitable $\beta$ (of course, in this case, $\beta = \beta_{\kappa,n}$ where $\kappa$ is a lower bound for the Ricci curvature and $n$ is the topological dimension). In Section \ref{sec:fat} we study the sub-Riemannian counterpart of such a statement, for the case of fat distributions (also called strong bracket generating, see \cite{montgomerybook,Strichartz}). We refer to \eqref{eq:fat-definition} for the definition.

In order to apply Theorem \ref{thm:RiccbddbelowareCD}, one must prove uniform lower bounds on the Ricci curvatures of fat structures. In contrast with the Riemannian case, this requires a considerable effort and delicate estimates, which are the object of Section \ref{sec:fat}. There, we prove that, when equipped with the natural gauge function $\sfD$ of \eqref{eq:defD-Intro}, compact fat sub-Riemannian metric measure spaces satisfy the $\CD(\beta,n)$ condition, for an explicit $\beta$. See Theorem \ref{thm:fatcptareCD} for a precise statement, which we anticipate here in a simplified form.

\begin{theorem*}[Compact fat structures are $\CD$]
Let $(M,\sfd,\mm)$ be a compact, $n$-dimensional, fat sub-Riemannian metric measure space. Then for the natural gauge function $\sfD: M\times M \to [0,+\infty)$ of \eqref{eq:defD-Intro}, there exists an explicit distortion coefficient $\beta$ such that $(M,\sfd,\mm,\sfD)$ satisfies the $\CD(\beta, n)$ condition.
\end{theorem*}

As a direct consequence of this result, we obtain the following, see Corollary \ref{cor:fatcptareMCP}.

\begin{corollary*}[Classical $\MCP$ for compact fat structures]
Let $(M,\sfd,\mm)$ be a compact $n$-dimensional sub-Riemannian metric measure space, with fat distribution of rank $k<n$. Then there exists $N'\geq 3n-2k$ such that $(M,\sfd,\mm)$ satisfies the classical $\MCP(0,N')$.
\end{corollary*}

The above result removes the real-analytic assumption of the analogous statement by Badreddine and Rifford in \cite[Thm.\@ 1.3]{BR-realanalMCP}, in the case of fat distributions, obtaining it with a completely different strategy.

\subsection*{Examples and applications}

Finally, in Section \ref{sec:applications} we present examples and applications of the synthetic theory to sub-Riemannian spaces.
\begin{itemize}
\item In Section \ref{sec:heisenbergpuro} we illustrate in detail our constructions in the case of the three-dimensional Heisenberg group $\H^1$, including the computation of the model distortion coefficients via the comparison theory of Section \ref{sec:comparison}.
\item In Section \ref{sec:Grushin} we show how the Grushin plane enters within the same $\CD(\beta,n)$ class of the Heisenberg group. This is natural, considering the fact that the Grushin plane is a quotient of the Heisenberg group (even though it is not a homogeneous space in the classical sense). This can be seen as a generalization of the by-now known fact that the Grushin plane satisfies the classical $\MCP(0,5)$.
\item In Section \ref{sec:ex:canvar} we study a one-parameter family of Riemannian metric measure spaces, namely the canonical variations of the first Heisenberg group $(\H^{1},\sfd_\varepsilon,\mathscr{L}^{3})$. We discuss in particular the sub-Riemannian limit ($\varepsilon \to 0$) and the adiabatic limit ($\varepsilon \to +\infty$), both occurring within a single $\CD(\beta,n)$ class.
\item In Section \ref{sec:ex:tangent} we study the convergence to the tangent cone of sub-Riemannian structures with natural gauge functions, and how the $\CD(\beta,n)$ condition behaves under the blow-up process.
\item In Section \ref{sec:ex:vectorial} we illustrate the vector-valued theory of Section \ref{sec:vectorial}, in the setting of the three-dimensional left-invariant structures on Lie groups.
\end{itemize}

\paragraph{Acknowledgments.} Davide Barilari acknowledges support by the STARS Consolidator Grant 2021 ``NewSRG'' of the University of Padova. Andrea Mondino and Luca Rizzi have received funding from the European Research Council (ERC) under the European Union's Horizon 2020 research and innovation programme (grant agreements No. 802689 ``CURVATURE'' and No. 945655 ``GEOSUB'').  For the purpose of Open Access, the authors have applied a CC BY public copyright licence to any Author Accepted Manuscript (AAM) version arising from this submission.	
\newpage
\section{Table of notations}
\begin{xltabular}{\textwidth}{p{0.22\textwidth} X}
\underline{Metric notation} & \\
& \\
$(X,\sfd,\mm)$\dotfill &  metric measure space\\
$\bar{\mm}$ \dotfill & completion of $\mm$ \\
$\Geo(X)$ \dotfill & set of length-minimizing, constant-speed curves $\gamma:[0,1]\to X$ of a metric space $(X,\sfd)$ \\
$B_r(x)$ \dotfill & open metric ball of radius $r$ and center $x$ \\
$\NH$ \dotfill & Hausdorff dimension \\
$\NN$ \dotfill & geodesic dimension, Definition \ref{def:geodim} \\
$W_2(\mu,\nu)$\dotfill &$2$-Wasserstein distance, Section \ref{Sec:OptTransp} \\
$\ee_t$\dotfill & evaluation maps $\gamma\mapsto\gamma_t$ at time $t$, Section \ref{Sec:OptTransp}\\
$\mathrm{Opt}(\mu,\nu)$ \dotfill & optimal couplings from $\mu$ to $\nu$, Section \ref{Sec:OptTransp}  \\
$\mathrm{OptGeo}(\mu,\nu)$\dotfill & optimal dynamical plans from $\mu$ to $\nu$, Section \ref{Sec:OptTransp} \\
$\Ent(\mu|\mm)$ \dotfill & Boltzmann-Shannon entropy of $\mu$ w.r.t.\@ $\mm$, Section \ref{sec:entropies} \\
$\Dom(\Ent(\cdot|\mm))$ \dotfill & finiteness domain of $\Ent(\cdot|\mm)$, Section \ref{sec:entropies}\\
$\U_{n}(\mu|\mm)$ \dotfill & $n$-dimensional entropy of $\mu$ w.r.t.\@ $\mm$, Section \ref{sec:entropies} \\
$\Prob_{bs}(X,\sfd,\mm)$ \dotfill &  probability measures with bounded support in $\supp\mm$, Section \ref{sec:entropies}\\
$\Prob_{bs}^*(X,\sfd,\mm)$\dotfill &  $\Dom(\Ent(\cdot|\mm)) \cap \Prob_{bs}(X,\sfd,\mm)$, Section \ref{sec:entropies}\\
$o(\cdot)$ \dotfill & function $o(\cdot):[0,+\infty)\to \R$ satisfying $\displaystyle \limsup_{\theta\downarrow 0} \frac{|o(\theta)|}{\theta}=0$ \\
$\sfG$ \dotfill & gauge function, Eq.\@ \eqref{eq:defG} scalar case, Eq.\@ \eqref{eq:defGVec} vector case\\
$\sfs$ \dotfill & defining function for distortion coefficients, Eq.\@ \eqref{eq:fftcN} scalar case, Eq.\@ \eqref{eq:fftcNVec} vector case \\
$N$ \dotfill & order of $\sfs$ at 0, Eq.\@ \eqref{eq:fftcN} scalar case, Eq.\@ \eqref{eq:fftcNVec} vector case \\
$\cD$ \dotfill &  first zero of $\sfs$ in the scalar case, Eq.\@ \eqref{eq:defcD} \\
$\RP^m_+$ \dotfill & real projective semi-space, Section \ref{sec:vectorial} \\
$\DOM$ \dotfill & positivity domain of $\sfs$ in the vector case, Section \ref{sec:vectorial} \\
$\cD_\theta$ \dotfill & boundary of $\DOM$ along $\theta$, in the vector case, Section \ref{sec:vectorial} \\
$\beta$ \dotfill & distortion coefficient,  scalar case Eq.\@ \eqref{eq:defbeta}, vector case Eq.\@ \eqref{eq:defbetaVec} \\
$\beta^{\tau}_{K,N}$, $\beta^\sigma_{K,N}$ \dotfill & distortion coefficients for the Lott-Sturm-Villani's theory, Section \ref{sec:howtorecover} \\
$\beta^{\H^d}$ \dotfill & distortion coefficient for $\mathbb{H}^d$, Section \ref{sec:howtorecover2} \\
$\beta^{(X,\sfd,\mm)}$ \dotfill & distortion coefficient of a m.m.s., Definition \ref{def:truedistcoeff} \\
$\MCP(\beta)$\dotfill & measure contraction property with distortion coefficient $\beta$, Definition \ref{def:CDMCPbetan} \\
$\CD(\beta,n)$\dotfill & curvature-dimension condition with distortion coefficient $\beta$ and dimensional parameter $n$, Definition \ref{def:CDMCPbetan} \\
$A_t, \, Z_t$\dotfill & set of $t$-intermediate points, Eq.\@ \eqref{eq:defAt} \\
$\beta(A_0, A_1)$\dotfill & distortion coefficient between the sets $A_0, A_1$, Eq.\@ \eqref{eq:betainf} \\
$\diam$ \dotfill & metric diameter \\
$\diam_\sfG$ \dotfill & $\sfG$-diameter, Definition \ref{def:Gdiam} scalar case, Definition \ref{def:GdiamVec} vector case \\
$\rv_{\sfG}$, $\rs_{\sfG}$ \dotfill & measures of ``gauge balls'' and ``spheres'', Definition \ref{def:gaugesets} scalar case, Definition \ref{def:gaugesetsVec} vector case \\
$\Lip^\rho_a$ \dotfill & asymptotic Lipschitz number w.r.t.\@ a metric $\rho$, Definition \ref{def:asLipnumb} \\
$\sfD$ \dotfill & $\sfD$ function,  Definition \ref{def:D} \\
& \\
\underline{Sub-Riemannian notation} & \\
& \\
$\langle\cdot,\cdot\rangle$ \dotfill & pairing of covectors with vectors \\
$\distr$ \dotfill & distribution (possibly rank-varying), Appendix \ref{a:SR}\\
$\End$ \dotfill & end-point map, Appendix \ref{a:SR} \\
$\ver$ \dotfill & complement of a constant rank distribution, Section \ref{sec:nonsmooth}--\ref{sec:fat} \\
$\Cut$ \dotfill & sub-Riemannian cut locus, Definition \ref{def:cut}\\
$\beta^{A,B,Q}$ \dotfill & distortion coefficient of a general LQ problem, Definition \ref{d:LQdist} \\
$\beta^\kappa$ or $\beta^{\bar\kappa}$ \dotfill & distortion coefficient of a model LQ problem, $\kappa \in \R^\ell$ Proposition \ref{p:basic}\ref{i:basic-4}, or $\bar\kappa: \R^m_+\to \R^\ell$ Proposition \ref{p:basic}\ref{i:basic-6}. See Remark \ref{rmk:identify} \\
$t_\kappa$ \dotfill & first conjugate time for model LQ problems, Proposition \ref{p:basic} \\
$\DOM_{\bar{\kappa}}$ \dotfill & positivity domain for model functions in the vector case, Proposition \ref{p:basic} \\
$V^0$ \dotfill & annihilator of a vector space/bundle $V$, Section \ref{sec:fat} \\
$\Lambda_{\neq 0}$ \dotfill & for $\Lambda$, the set of $\lambda\in\Lambda \subseteq T^*M$ s.t.\@ $H(\lambda)\neq 0$, Section \ref{sec:preliminary} \\
$\vec{H}$ \dotfill & Hamiltonian vector field, Appendix \ref{a:canonicalframe}\\
$Y$ \dotfill & reduced Young diagram of a Jacobi curve, Appendix \ref{a:canonicalframe}\\
$\Upsilon$ \dotfill & set of levels of a reduced Young diagram, Appendix \ref{a:canonicalframe}\\
$\size(\cdot)$ \dotfill & size of a superbox of a reduced Young diagram, Appendix \ref{a:canonicalframe} \\
$\{E_a,F_b\}_{a,b\in Y}$ \dotfill & canonical frame, Appendix \ref{a:canonicalframe} \\
$\mathfrak{R}$ \dotfill & canonical curvature, Appendix \ref{a:canonicalframe} \\
$\mathfrak{Ric}$ \dotfill & canonical Ricci curvature, Appendix \ref{a:canonicalframe} \\
$\rho_{\mm}$ \dotfill & geodesic volume derivative, Appendix \ref{a:canonicalframe}\\
$\boxtimes,\boxplus,\boxdot,\boxminus$ \dotfill & superbox of a Young diagram (notation used in Section \ref{sec:fat}) \\
\end{xltabular}				
\section{Synthetic Ricci curvature lower bounds for gauge spaces}\label{sec:CDbeta} 

\subsection{Preliminaries and notation}

\subsubsection{Convergence of measures}
Throughout all the paper, $(X,\sfd)$ is a complete and separable metric space. Denote with $\sM_{\loc}(X)$ the set of Borel measures with values in $[0,+\infty]$ which are finite on every bounded subset and with  $\sM(X)$ the collection of  all finite Borel measures.  Notice that every measure in $\sM_{\loc}(X)$ is $\sigma$-finite, by the exhaustion $X=\cup_{j\in \N} B_{j}(x)$, for some $x\in X$. 

We endow  $\sM_{\loc}(X)$ with the (weak) topology induced by the duality with the space $C_{bs}(X)$ of bounded and continuous functions on $X$ with bounded support: a sequence $(\mu_{j})_{j\in\N}\subset \sM_{\loc}(X)$ \emph{converges weakly}\footnote{Such weak convergence of measures is also known in the literature as ``convergence in vague topology''.} to $\mu_{\infty}\in \sM_{\loc}(X)$ if
\begin{equation}\label{eq:DefWeakConv}
\lim_{j\to \infty} \int_{X} f\,  \mu_{j} = \int_{X} f\, \mu_{\infty}, \qquad \forall\, f\in C_{bs}(X). 
\end{equation}

Let  $\Prob(X)$ denote the set of Borel probability measures and $\Prob_{bs}(X,\sfd)$ the set of probability measures with bounded support.

When $(\mu_{j})_{j\in \N\cup \{\infty\}}\subset \Prob(X)$, the weak convergence \eqref{eq:DefWeakConv} is equivalent to the \emph{narrow} convergence, i.e.\@ convergence in $\Prob(X)$ in duality with the space $C_{b}(X)$ of  bounded and continuous functions on $X$:
\begin{equation*}
\lim_{j\to \infty} \int_{X} f\,  \mu_{j} = \int_{X} f\, \mu_{\infty}, \qquad \forall\, f\in C_{b}(X). 
\end{equation*}
Let us mention that \cite{Vil} adopts the convention ``weak convergence''  for what we denote as ``narrow convergence'' in $\Prob(X)$; however, since our convention of  ``weak'' and ``narrow'' give equivalent notions of convergence on $\Prob(X)$, there is no ambiguity.

Relative narrow compactness in $\Prob(X)$ can be characterized by Prokhorov's Theorem. In order to state it, recall that a subset $\cK \subset \Prob(X)$ is said to be \emph{tight} if, for every $\ve>0$, there exists a compact subset $K_{\ve}\subset X$ such that
\begin{equation}
\mu(X\setminus K_{\ve})\leq \ve, \qquad \forall\, \mu\in \cK.
\end{equation}
 
\begin{theorem}[Prokhorov]\label{thm:Prok}
 Let $(X,\sfd)$ be complete and separable, and let $\cK\subset \Prob(X)$. Then $\cK$ is pre-compact in the narrow topology if and only if $\cK$ is tight.
 \end{theorem}
 
\subsubsection{Optimal transport and \texorpdfstring{$W_{2}$}{W2}-distance in metric spaces}\label{Sec:OptTransp}
Recall that if $f:X\to Y$ is a Borel map between the separable metric spaces $(X,\sfd_{X})$, $(Y,\sfd_{Y})$, then any Borel (resp. probability) measure $\mu$ can be \emph{pushed forward} to a Borel (resp. probability) measure $f_{\sharp} \mu$  defined as $(f_{\sharp} \mu)(B):=\mu(f^{-1}(B))$ for every Borel subset $B\subset Y$. Call $P_{i}:X\times X\to X$, $i=1,2$, the projection on the $i^{th}$ factor. Given $\mu_{i}\in \Prob(X)$, $i=1,2$, denote
\[
\Cpl(\mu_{1}, \mu_{2}):=\{\pi\in \Prob(X\times X) \mid (P_{i})_{\sharp} \pi=\mu_{i}, \, i=1,2 \}
\]
the set of \emph{admissible couplings} (also called ``transference plans'') from $\mu_{1}$ to $\mu_{2}$.  Let $\Prob_{2}(X,\sfd)$ the subspace of probability measures with finite second moment, i.e. 
\[
\Prob_{2}(X,\sfd):=\left\{\mu\in \Prob(X) \mid \int_{X} \sfd(x,\bar x)^{2}\, \mu(\di x)<\infty \text{ for some (and thus for all) } \bar{x}\in X \right\}.
\]
 We endow the space $\Prob_{2}(X)$ with the quadratic (Kantorovich-Rubinstein-Wasserstein) transportation distance $W_{2}$ defined as
\[
W_{2}(\mu,\nu)^{2}:= \inf_{\pi\in \Cpl(\mu,\nu)} \int_{X\times X} \sfd^{2}(x,y)\,  \pi( \di x \di y). 
\]
A coupling $\pi\in \Cpl(\mu,\nu)$ achieving the infimum in the right hand side is called \emph{optimal coupling} from $\mu$ to $\nu$, and the set of such optimal couplings is denoted by $\Opt(\mu,\nu)$. Since $\Cpl(\mu,\nu)$ is non-empty and compact in the weak topology, it is easily checked that  $\Opt(\mu,\nu)\neq \emptyset$.

It is well-known that $(\Prob_{2}(X,\sfd), W_{2})$ is a complete and separable metric space (see e.g.\@ \cite{Vil}). 
In order to discuss the relation between the narrow and $W_{2}$ convergence, recall that a subset $\cK\subset \Prob_{2}(X)$ is \emph{$2$-uniformly integrable} provided that
\[
\lim_{R\to \infty }\sup_{\mu\in \cK} \int_{X\setminus B_{R}(\bar{x})} \sfd^{2}(x,\bar{x}) \, \mu(\di x)=0 , \qquad \text{for some (and thus for every) } \bar{x}\in X.
\]
For a proof of the following result see for instance  \cite[Thm. 6.8]{Vil}.
\begin{proposition}[Characterization of $W_{2}$ convergence]\label{prop:W2convNarrow}
Let $(X,\sfd)$ be a complete and separable metric space and $(\mu_{j})_{j\in \N\cup\{\infty\}}\subset \Prob_{2}(X,\sfd)$.  Then the following are equivalent:
\begin{itemize}
\item $(\mu_{j})_{j\in \N}$ is $2$-uniformly integrable and converges narrowly to $\mui$;
\item $W_{2}(\mu_{j}, \mui)\to 0$ as $j\to \infty$.
\end{itemize}
\end{proposition}

We next recall some basics about the geodesic structure of  $(\Prob_{2}(X,\sfd), W_{2})$, cf.\@ \cite{Vil}. A \emph{geodesic} is a curve $\gamma:[0,1]\to X$ satisfying 
\begin{equation}
\sfd(\gamma_{s}, \gamma_{t})= |s-t| \,  \sfd(\gamma_{0}, \gamma_{1}), \qquad \forall\, s,t\in [0,1].
\end{equation}
In particular, geodesics are length-minimizing and parametrized with constant speed on the interval $[0,1]$.

The space of all geodesics on $(X,\sfd)$ is denoted by $\Geo(X)$, which is endowed with the complete and separable distance 
\begin{equation}
\sfd_{\Geo(X)}(\gamma, \eta):=\sup_{t\in [0,1]}\sfd(\gamma_{t}, \eta_{t}).
\end{equation}
Recall that $(X,\sfd)$ is said to be a \emph{geodesic space} if every two points in $X$ can be joined by a geodesic;  $(X,\sfd)$ is a geodesic space if and only if $(\Prob_{2}(X,\sfd), W_{2})$ is so.  

A useful procedure is to represent a geodesic in $(\Prob_{2}(X,\sfd), W_{2})$ by a single probability measure defined on $\Geo(X)$. More precisely: for every $W_{2}$-geodesic $(\mu_{t})_{t\in[0,1]}\subset \Prob_{2}(X,\sfd)$, there exists  $\nu\in \Prob(\Geo(X))$ such that $\mu_{t}=(\ee_{t})_{\sharp}(\nu)$ for all $t\in [0,1]$, where 
\begin{equation}
\ee_{t}:\Geo(X)\to X, \qquad \ee_{t}(\gamma):=\gamma_{t},
\end{equation}
is the evaluation map. Such a measure $\nu$ is called an \emph{optimal dynamical plan} from $\mu_{0}$ to $\mu_{1}$, the set of which is denoted by $\OptGeo(\mu_{0}, \mu_{1})$.
Given $\nu\in \Prob(\Geo(X))$, it holds that $\nu\in \OptGeo(\mu_{0},\mu_{1})$ if and only if $(\ee_{0}, \ee_{1})_{\sharp}\nu\in \Opt(\mu_{0}, \mu_{1})$.

A set $\Gamma \subset \Geo(X)$ is  \emph{a set of non-branching geodesics} if and only if for any $\gamma^{1},\gamma^{2} \in \Gamma$, it holds:
\begin{equation}
\exists \;  \bar t\in (0,1) \text{ such that } \; \forall\, t \in [0, \bar t\,] \quad  \gamma_{t}^{1} = \gamma_{t}^{2}   
\quad 
\Longrightarrow 
\quad 
\gamma^{1}_{s} = \gamma^{2}_{s}, \quad \forall\, s \in [0,1].
\end{equation}
It is clear that if $(X,\sfd)$ is a smooth Riemannian manifold, then any subset $\Gamma \subset \Geo(X)$ is a set of non branching geodesics. This is true, more generally, if $(X,\sfd)$ is an ideal sub-Riemannian manifold, i.e.\@ admitting no non-trivial abnormal geodesics: in fact, a sub-Riemannian geodesic can branch only if it contains an abnormal segment, cf.\@ \cite{MR-branching}.
\medskip

We endow the metric space $(X,\sfd)$ with a  measure $\mm\in \sM_{\loc}(X)$, i.e.\@ non-negative and finite on bounded sets.  The triple $(X,\sfd,\mm)$ is called \emph{metric measure space} (m.m.s.\@ for short). We denote with 
 $\Prob_{ac}(X,\mm)\subset \Prob(X)$ the subspace of probability measures  which are absolutely continuous  with respect to the reference measure $\mm$. 

A metric measure space $(X,\sfd, \mm)$ is \emph{essentially non-branching} if and only if for any $\mu_{0},\mu_{1} \in \Prob_{2}(X,\sfd)\cap \Prob_{ac}(X,\mm)$, any  $\nu\in \OptGeo(\mu_{0},\mu_{1})$ is concentrated on a set $\Gamma\subset \Geo(X)$ of non-branching geodesics.

\subsubsection{Relative entropies}\label{sec:entropies}

For $\mu \in \mathcal{P}_{2}(X)$, we define its relative (Boltzmann-Shannon) entropy  by
\begin{equation}\label{def:Ent}
\Ent(\mu|\mm): = \int_{X} \rho \log(\rho) \, \mm,\qquad  \text{ if $\mu = \rho \, \mm\in \Prob_{2}(X)\cap \Prob_{ac}(X,\mm)$},
\end{equation}
in case  $\rho\log(\rho)\in L^{1}(X,\mm)$,  otherwise we set $\Ent(\mu|\mm) := +\infty$.

Since $(0,+\infty)\ni x\mapsto x\log x$ is convex, Jensen's inequality gives
\begin{equation}\label{eq:jensenEnt}
\Ent(\mu|\mm)\geq -\log \mm(\supp \mu), \qquad  \forall\, \mu\in  \mathcal{P}(X) \text{ with } \mm(\supp \mu)<\infty.
\end{equation}

We set  $\Dom(\Ent(\cdot|\mm)):=\{\mu\in \mathcal{P}_{2}(X)\mid \Ent(\mu|\mm)\in \R\}$ to be the finiteness domain of the entropy and 
\begin{align*}
\Prob_{bs}(X,\sfd,\mm)&:=  \{\mu\in \Prob_{bs}(X,\sfd)\mid \supp \, \mu\subseteq \supp \, \mm\},\\
\Prob_{bs}^*(X,\sfd,\mm)&:= \Dom(\Ent(\cdot|\mm)) \cap \Prob_{bs}(X,\sfd,\mm).
\end{align*}

In order to formulate ``dimensional''  Ricci curvature lower bounds, it is convenient  to introduce also  the following dimensional entropy (cf.\@ \cite{EKS}):
\begin{equation}\label{eq:defUn}
\U_{n}(\mu|\mm) : = \exp\left(-\frac{\Ent(\mu|\mm)}{n}\right), \qquad n\in [1,+\infty),
\end{equation}
with the understanding that $\U_{n}(\mu|\mm):=0$ if $\mu\notin \Dom(\Ent(\cdot|\mm))$.

\subsection{Gauge functions on metric spaces}

In order to obtain a unified framework
\begin{itemize}
\item embracing both sub-Riemannian structures and Lott-Sturm-Villani's $\CD$ m.m.s., 
\item yielding sharp geometric and functional inequalities, 
\end{itemize}
we add an additional structure to a metric measure space $(X,\sfd,\mm)$, that is a non-negative Borel function, called \emph{gauge function}:
\begin{equation}\label{eq:defG}
\sfG: X\times X\to [0,+\infty].
\end{equation}

The following analogy explains the role of the gauge function. A Riemannian manifold $(M,g)$ has Ricci curvature bounded from below if there exists $K \in \R$ such that, for all $x,y\in M$ and any geodesic $\gamma$ between $x$ and $y$ it holds
\begin{equation}
\Ric(\dot\gamma,\dot\gamma) \geq K \|\dot\gamma\|^2 = K \,  \sfd(\gamma_0,\gamma_1)^2.
\end{equation}
The distance function $\sfd$ is used in the right hand side as a \emph{gauge} to measure the extent of the lower Ricci curvature bound, quantified by the constant $K$. The idea is to replace the distance $\sfd$ with a general gauge function $\sfG$ for curvature bounds.

Gauge functions will be key in our extension of the synthetic theory of curvature bounds to the sub-Riemannian setting, where it is well-known that the effect of the curvature on transport inequalities is not expressed via the distance but rather via other intrinsically sub-Riemannian functions (a phenomenon observed in \cite{BKS}, \cite[Sec.\@ 8.1]{BRInv}).

\subsection{Curvature-dimension conditions for gauge spaces}

Let $\sfs:[0,+\infty)\to \R$ be a continuous function, with $N\in [1,+\infty)$ such that
\begin{equation}\label{eq:fftcN}
\sfs(\theta)=c \, \theta^{N}+o(\theta^{N}) \qquad \text{ for some } c >0.
\end{equation}
The parameter $N$ will be the sharp upper bound for a new notion of dimension, which is in general different from the Hausdorff one, see Section \ref{Sec:Dim}. 
Denote:
\begin{equation}\label{eq:defcD}
\cD:=\inf\{\theta>0\mid \sfs(\theta)=0\}.
\end{equation}  
From the assumptions on $\sfs$ it is clear that $\cD>0$.
The parameter $\cD$ will give a sharp upper bound on the gauge function, see Section \ref{sec:diam}.
 Define the \emph{distortion coefficient} $\beta_{(\cdot)}(\cdot):[0,1]\times [0,+\infty]\to [0,+\infty]$ as
\begin{align}\label{eq:defbeta}
 (t,\theta)\in [0,1]\times [0,+\infty]\mapsto \beta_{t}(\theta):=
 \begin{cases}
 t^{N} &\theta =0, \\
\dfrac{\sfs(t\theta )}{\sfs(\theta)}& 0<\theta <\cD,\\
 \displaystyle  \liminf_{\phi \to \cD^-} \dfrac{\sfs(t\phi )}{\sfs(\phi)} & \theta \geq \cD.\end{cases}
\end{align}

\begin{remark}\label{rem:sODE}
In applications to sub-Riemannian geometry,  the function  $\sfs$ is chosen in a class of models which are  characterized as solutions to suitable ODEs, see Proposition \ref{p:basic} and Remark \ref{r:link}. However, here we develop the theory in full generality. 
\end{remark}

We collect some elementary properties following from the definition.
\begin{proposition}\label{prop:propertiesbeta}
Any distortion coefficient satisfies the following:
\begin{enumerate}[label=(\roman*)]
\item \label{i:propertiesbeta-1}  $\beta_0(\theta)=0$ and $\beta_1(\theta) = 1$ for all $\theta \in [0,+\infty]$;
\item \label{i:propertiesbeta0}  if $\cD<+\infty$, then $\beta_t(\theta) = +\infty$ for all $\theta \geq \cD$ and $t\in (0,1)$;
\item \label{i:propertiesbeta1} $\beta_{t}(\theta)=0$ for some $\theta < +\infty$ if and only if $t=0$;
\item \label{i:propertiesbeta2} for every $t\in [0,1]$,  $\theta_{j}, \theta\in  [0,+\infty)$, if $\theta_{j}\to \theta$ then $\beta_{t}(\theta_{j})\to \beta_{t}(\theta)$ in $[0,+\infty]$;
\item \label{i:propertiesbeta3} $\beta$ is  continuous and finite when restricted to $[0,1]\times [0,\cD)$;
\item \label{i:propertiesbeta4} for all fixed $t\in (0,1]$, the distortion coefficient $\beta_t(\cdot)$ is  bounded from below away from zero on any bounded set of $[0,+\infty)$; 
\item \label{i:propertiesbeta5} assume that $\cD<+\infty$, and that $\sfs \in C^\omega([0,\cD])$. Then there exists $N'\geq N$ such that $\beta_t(\theta) \geq t^{N'}$ for all $t\in [0,1]$ and $\theta \in[0,+\infty]$.
\end{enumerate}
\end{proposition}

\begin{definition}\label{def:CDMCPbetan}
Let $n\in [1,+\infty)$, and $\beta$ as in \eqref{eq:defbeta}. We say that a  metric measure space $(X,\sfd,\mm)$ with gauge $\sfG$ satisfies:
\begin{itemize}
\item  $\CD(\beta, n)$ if for all $\mu_0\in \Prob_{bs}(X,\sfd,\mm)$, $\mu_1\in \Prob_{bs}^{*}(X,\sfd,\mm)$ with $\supp \mu_0\cap \supp \mu_1 =\emptyset$,  there exists a $W_2$-geodesic $(\mu_t)_{t\in[0,1]} \subset \Prob_2(X,\sfd)$ connecting them, induced by $\nu\in \OptGeo(\mu_0,\mu_1)$, such that it holds
\begin{multline}\label{eq:defCDbetan}
\U_{n}(\mu_{t}|\mm)\geq  \exp\left( \frac{1}{n}\int_{\Geo(X)} \log \beta_{1-t}\big(\sfG(\gamma_1,\gamma_0))\, \nu(\di\gamma) \right) \U_{n}(\mu_{0}|\mm)  \\
+\exp\left( \frac{1}{n}\int_{\Geo(X)} \log \beta_{t}\big(\sfG(\gamma_0,\gamma_1)\big) \, \nu(\di\gamma) \right)   \U_{n}(\mu_{1}|\mm), \qquad \forall\, t\in (0,1),
\end{multline}
with the convention that $\infty \cdot 0 = 0$. 
We say that $\CD(\beta, n)$ is satisfied \emph{in the strong sense} if \eqref{eq:defCDbetan} is satisfied for all $W_2$-geodesics connecting $\mu_0$ and $\mu_1$.
\item  $\MCP(\beta)$ if for any $\bar{x} \in \supp\mm$ and $\mu_1\in \Prob_{bs}^{*}(X,\sfd,\mm)$   with $\bar{x}\notin \supp \mu_1$ there exists a $W_2$-geodesic $(\mu_t)_{t\in[0,1]} \subset \Prob_2(X,\sfd)$ from $\mu_0=\delta_{\bar{x}}$ to $\mu_1$ such that
\begin{equation}\label{eq:defMCPbeta}
\U_{n}(\mu_{t}|\mm)\geq \exp\left( \frac{1}{n}\int_{X}  \log \beta_{t}\big(\sfG(\bar{x},x)\big)  \, \mu_1(\di x) \right) \U_{n}(\mu_{1}|\mm) , \qquad \forall\, t\in (0,1),
\end{equation}
for some (and then every) $n\geq 1$.
\end{itemize}
\end{definition}
\begin{remark}\label{rem:eqMCPbeta}
Notice that \eqref{eq:defMCPbeta} is equivalent to
\begin{equation}\label{eq:defMCPbeta2}
\Ent(\mu_t|\mm) \leq \Ent(\mu_1|\mm) - \int_X\log\beta_t\big(\sfG(\bar{x},x)\big)  \, \mu_1(\di x) , \qquad \forall\, t\in (0,1),
\end{equation}
so that \eqref{eq:defMCPbeta} does not depend on the value of $n$.
\end{remark}

\begin{remark}[$\CD$ implies $\MCP$]\label{rem:CDimpliesMCP}
It is clear that, with our definitions, $\CD(\beta,n)$ for some $n \in [1,+\infty)$ implies $\MCP(\beta)$, since one can choose in \eqref{eq:defCDbetan} $\mu_0$ equal to a Dirac mass.
\end{remark}

\begin{remark}[About the absolute continuity of $\mu_0$]\label{rem:mu0Pbs}
Our definition of $\CD(\beta,n)$ is not symmetric, in the sense that $\mu_0 \in \Prob_{bs}(X,\sfd,\mm)$ while $\mu_1 \in \Prob_{bs}^*(X,\sfd,\mm)$, and in particular $\mu_1\ll \mm$. When $\sfG$ is not continuous, the flexibility  to take a non-absolutely continuous $\mu_0$ turns out to be technically convenient; for instance it makes neat the implication $\CD(\beta,n)\Rightarrow \MCP(\beta)$, implication which would not be clear otherwise.

On the other hand, if the gauge function $\sfG:X\times X \to [0,+\infty)$ is \emph{continuous}, via a standard approximation argument it is possible to see that the $\CD(\beta,n)$ condition of Definition \ref{def:CDMCPbetan} is equivalent to the following, symmetric, requirement: 
 for all  $\mu_0, \mu_1\in \Prob_{bs}^{*}(X,\sfd,\mm)$  with $\supp \mu_0\cap \supp \mu_1 =\emptyset$,  there exists a $W_2$-geodesic $(\mu_t)_{t\in[0,1]} \subset \Prob_2(X,\sfd)$ connecting them, induced by $\nu\in \OptGeo(\mu_0,\mu_1)$, such that \eqref{eq:defCDbetan} holds.
 The latter formulation is closer in spirit to   Lott-Sturm-Villani's  curvature-dimension conditions for metric measure spaces; note that in this case $\sfG=\sfd$ is of course continuous.
\end{remark}

\begin{remark}[Geodesic property for $\Dom(\Ent(\cdot | \mm))\subset \Prob_{2}(X,\sfd)$ and length property for  $\supp \mm$]\label{rem:LengthSpace}
The fact that $(X,\sfd,\mm)$ satisfies $\CD(\beta, n)$ implies directly from the definition that $\Dom(\Ent(\cdot | \mm))\subset \Prob_{2}(X,\sfd)$ is a geodesic space (with metric $W_{2}$). Moreover if $(X,\sfd,\mm)$ satisfies $\CD(\beta, n)$ or $\MCP(\beta)$, then $\supp \mm$ is a length space (with metric $\sfd$); the proof follows verbatim  \cite[Remark 4.6, (iii)]{sturm:I} and thus is omitted.   We stress that these facts do not use the inequalities \eqref{eq:defCDbetan} or \eqref{eq:defMCPbeta}, but only the existence of $W_2$-geodesics between suitable pairs of measures. Moreover, as a consequence of \cite[Thm.\@ 1.1]{CavaMondCCM} it follows that: if $(X,\sfd, \mm)$ is an essentially non-branching $\MCP(K,N)$ space, then  $\Dom(\Ent(\cdot | \mm))\subset \Prob_{2}(X,\sfd)$ is a geodesic space (with metric $W_{2}$). Along the same lines, one can prove that $\Dom(\Ent(\cdot | \mm))\subset \Prob_{2}(X,\sfd)$ is a geodesic  space under the assumption that  $(X,\sfd, \mm)$ is an essentially non-branching $\MCP(\beta)$ space.
\end{remark}

Fix a metric measure space $(X,\sfd,\mm)$ with gauge function $\sfG$. The next proposition contains interpolation and monotonicity properties of the $\CD$ classes.

\begin{proposition}\label{prop:hierarchy}
Let $\alpha,\beta,\beta_0,\beta_1$ distortion coefficients as in \eqref{eq:defbeta}, let $n,n_0,n_1 \geq 1$, and $\theta \in [0,1]$. Then the following properties hold
\begin{enumerate}[label = (\roman*)]
\item \label{i:hierarchy1} $\mathrm{CD}(\beta_{0},n_{0})\cap \mathrm{CD}(\beta_{1},n_{1})$ with at least one of the two being satisfied in the strong sense $\Longrightarrow$ $\mathrm{CD}(\beta_{\theta},n_{\theta})$ where we have set
\[
\beta_{\theta}:=\beta_{0}^{1-\theta}\beta_{1}^{\theta},\qquad n_{\theta}:=(1-\theta)n_{0}+\theta n_{1};
\]
\item \label{i:hierarchy2} if $m\geq 0$. Then $\mathrm{CD}(\beta,n)$ $\Longrightarrow$ $\mathrm{CD}(t^{m} \beta,n+m)$;
\item \label{i:hierarchy3} if $\beta\geq \alpha$, then $\mathrm{CD}(\beta,n)$ $\Longrightarrow$ $\mathrm{CD}(\alpha,n)$.
\end{enumerate}
\end{proposition}
\begin{proof}
\textbf{Proof of \ref{i:hierarchy1}.} Observe that $\mathrm \U_{n}= (\U_m)^{\frac{m}{n}}$ for all $n,m\geq 1$. Hence
\begin{equation}
\U_{n_{\theta}}=\U_{n_{\theta}}^{(1-\theta)} \U_{n_{\theta}}^{\theta}=\mathrm U_{n_{0}}^{(1-\theta)\frac{n_{0}}{n_{\theta}}}\mathrm U_{n_{1}}^{\theta\frac{n_{1}}{n_{\theta}}}.
\end{equation}
Since at least one of the two $\CD(\beta_i,n_i)$ conditions is satisfied in the strong sense, we may assume that for any given $\mu_0\in\Prob_{bs}(X,\sfd,\mm)$, $\mu_1\in \Prob_{bs}^*(X,\sfd,\mm)$,  $\supp \mu_0\cap \supp \mu_1 =\emptyset$ there exists a $W_2$-geodesic $(\mu_t)_{t\in [0,1]}$ such that
\begin{multline}
\U_{n_i}(\mu_t|\mm) \geq \exp\left(\frac{1}{n_i} \int_{\Geo(X)} \log\beta_{i,1-t}\big(\sfG(\gamma_1,\gamma_0)\big)\,\nu(\di\gamma) \right) \U_{n_i}(\mu_0|\mm) \\
+ \exp\left(\frac{1}{n_i} \int_{\Geo(X)} \log\beta_{i,t}\big(\sfG(\gamma_0,\gamma_1)\big)\,\nu(\di\gamma) \right) \U_{n_i}(\mu_1|\mm),\qquad \forall\, t\in(0,1),\; i=0,1.
\end{multline}
For $i=0,1$ and $t\in (0,1)$, set 
\begin{align*}
a_{i,t}&:= \exp\left(\frac{1}{n_i} \int_{\Geo(X)} \log\beta_{i,1-t}\big(\sfG(\gamma_1,\gamma_0)\big)\,\nu(\di\gamma) \right),\\
b_{i,t}&:= \exp\left(\frac{1}{n_i} \int_{\Geo(X)} \log\beta_{i,t}\big(\sfG(\gamma_0,\gamma_1)\big)\,\nu(\di\gamma) \right).
\end{align*}
Omitting $\mm$ from the notation, we have
\begin{align}
\U_{n_\theta}(\mu_t) & = \U_{n_0}(\mu_t)^{(1-\theta)\tfrac{n_0}{n_\theta}}\U_{n_1}(\mu_t)^{\theta\tfrac{n_1}{n_\theta}} \\
& \geq \left[  a_{0,t} \U_{n_0}(\mu_0) + b_{0,t} \U_{n_0}(\mu_1)\right]^{(1-\theta)\tfrac{n_0}{n_\theta}}\left[  a_{1,t} \U_{n_1}(\mu_0) + b_{1,t} \U_{n_1}(\mu_1)\right]^{\theta\tfrac{n_1}{n_\theta}}.
\end{align}
Now we use the H\"older inequality: for $\frac{1}{p}+\frac{1}{q}=1$ we have for positive $x_j,y_j$
\begin{equation}
\left(\sum_{j=1}^{m} x_{j}\right)^{1/p}\left(\sum_{j=1}^{m} y_{j}\right)^{1/q}\geq \sum_{j=1}^{m} x_{j}^{1/p}y_{j}^{1/q}.
\end{equation}
Hence we obtain
\begin{equation}
\U_{n_\theta}(\mu_t) \geq a_{0,t}^{(1-\theta)\tfrac{n_0}{n_\theta}} a_{1,t}^{\theta\tfrac{n_1}{n_\theta}} \U_{n_\theta}(\mu_0)+b_{0,t}^{(1-\theta)\tfrac{n_0}{n_\theta}} b_{1,t}^{\theta\tfrac{n_1}{n_\theta}} \U_{n_\theta}(\mu_1).
\end{equation}
This corresponds to the desired inequality for the condition $\CD(\beta_\theta,n_\theta)$. 

\textbf{Proof of \ref{i:hierarchy2}.} We observe first that the map $\R^2\ni (x,y)\mapsto \log(e^x+e^y)$ is convex, and thus for $n\geq 1$, $m\geq 0$, all $t\in (0,1)$ and $x,y\in \R$ it holds
\begin{align}
\frac{n}{n+m}\log(e^x+e^y) & = \frac{n}{n+m}\log(e^x+e^y) + \frac{m}{n+m}\log\left(e^{\log(1-t)}+e^{\log t}\right) \\ 
& \geq \log\left(e^{\tfrac{n}{n+m}x + \tfrac{m}{m+n}\log(1-t)} + e^{\tfrac{n}{n+m}y + \tfrac{m}{n+m}\log t} \right).\label{eq:exeyconvex}
\end{align}
Fix a $W_2$-geodesic $(\mu_t)_{t\in [0,1]}$ for which the $\CD(\beta,n)$ inequality holds true. Observe that:
\begin{align}
\log\U_{n+m}(\mu_t) & = \frac{n}{n+m} \log\U_{n}(\mu_t), \qquad \forall\, t \in [0,1].
\end{align}
Using the $\CD(\beta,n)$ inequality in the previous identity and applying \eqref{eq:exeyconvex}, we immediately obtain the validity of the $\CD(t^{m}\beta,n+m)$ condition.

\textbf{Proof of \ref{i:hierarchy3}.} Obvious.
\end{proof}

\begin{proposition}\label{prop:mcptoclassicalmcp} 
Let $(X,\sfd, \mm)$ be essentially non-branching. Assume that  $(X,\sfd, \mm)$ with gauge $\sfG=\sfd$ satisfies the $\MCP(\beta)$, for $\beta$ defined as in \eqref{eq:defbeta} in terms of a function $\sfs$ as in \eqref{eq:fftcN}, with order $N$ at zero. If $\cD$ is finite and $\sfs\in C^\omega([0,\cD])$, then there exists $N'\geq N$ such that $(X,\sfd,\mm)$ satisfies the $\MCP(t^{N'})$ with gauge $\sfG=\sfd$ and thus $(X,\sfd,\mm)$  satisfies the classical measure contraction property $\MCP(0,N')$.
\end{proposition}
\begin{proof}
By Proposition \ref{prop:propertiesbeta}\ref{i:propertiesbeta5}, there exists $N'\geq N$ such that $\beta_t \geq t^{N'}$ for all $t\in [0,1]$. Hence $\MCP(t^{N'})$ holds and we conclude by Theorem \ref{thm:CDbetaVsLSV}\ref{CDbetaVsLSV-iii}.
\end{proof}

\subsubsection{Recovering the Lott-Sturm-Villani's theory}\label{sec:howtorecover} 

In this section we explain how our general theory of curvature bounds for gauge spaces contains the classical curvature-dimension theory for m.m.s.\@ \`a la Lott-Sturm-Villani. In this section we consider $K\in \R$ and $N\in [1,+\infty)$. Set
\begin{equation}
[0,+\infty]\ni \theta \mapsto \sfs_{K,N}(\theta) = \begin{cases}   \sin\left(\theta \sqrt{K/N}\right)^{N} & K >0, \\
\theta^N & K =0, \\
   \sinh\left(\theta \sqrt{-K/N}\right)^{N} & K <0.
 \end{cases}
\end{equation}
Notice that $\sfs_{K,N}(\theta)$ has order $N$ as $\theta \to 0$ for all $K \in \R$. Moreover,
\begin{equation}
\cD_{K,N}=\inf\{\theta>0\mid \sfs_{K,N}(\theta)=0\}=\begin{cases}
\frac{\pi}{\sqrt{K/N}} & K >0, \\
+\infty & K \leq 0.
\end{cases}
\end{equation}
The corresponding distortion coefficient, following the recipe of \eqref{eq:defbeta}, is then
\begin{equation}
[0,1]\times [0,+\infty]\ni (t,\theta) \mapsto (\beta_{K,N}^{\sigma})_t(\theta) =
 \begin{cases}
+ \infty & K \theta^2 \geq N \pi^2, \\
 \frac{\sin\left(t\theta \sqrt{K/N}\right)^{N}}{\sin\left(\theta \sqrt{K/N}\right)^{N}} & 0<K\theta^2<N\pi^2, \\
t^N & K \theta^2 =0, \\
 \frac{\sinh\left(t\theta \sqrt{|K|/N}\right)^{N}}{\sinh\left(\theta \sqrt{|K|/N}\right)^{N}} & K\theta^2<0.
\end{cases}
\end{equation}
The notation is motivated by the fact that
\begin{equation}\label{eq:betaVssigma}
(\beta_{K,N}^{\sigma})_t(\theta) = \left(\sigma_{K,N}^{(t)}(\theta) \right)^N, \qquad \forall\, t\in [0,1],\quad \forall\, \theta \in [0,+\infty),
\end{equation}
where $\sigma_{K,N}^{(\cdot)}(\cdot)$ are the functions considered in \cite{BS10, EKS} (note that \cite{EKS} used the notation $\sigma_{K/N}^{(\cdot)}(\cdot)$). Consider also 
\begin{equation}
[0,1]\times [0,+\infty]\ni (t,\theta) \mapsto (\beta_{K,N}^{\tau})_t(\theta) =
 \begin{cases}
+ \infty & K\theta^2 \geq (N-1)\pi^2, \\
t \frac{\sin\left(t\theta \sqrt{K/(N-1)}\right)^{N-1}}{\sin\left(\theta \sqrt{K/(N-1)}\right)^{N-1}} & 0<K\theta^2<(N-1)\pi^2, \\
t^N & K\theta^2 =0, \\
t \frac{\sinh\left(t\theta \sqrt{|K|/(N-1)}\right)^{N-1}}{\sinh\left(\theta \sqrt{|K|/(N-1)}\right)^{N-1}} & K\theta^2<0.
\end{cases}
\end{equation}
Again, the notation is motivated by the fact that
\begin{equation}\label{eq:barbetatau}
(\beta_{K,N}^{\tau})_t(\theta) = \big(\tau_{K,N}^{(t)}(\theta)\big)^N, \qquad \forall\, t\in [0,1],\quad \forall\, \theta \in [0,+\infty),
\end{equation}
where $ \tau_{K,N}^{(t)}(\theta):= t^{1/N} \sigma^{(t)}_{K, N-1}(\theta)^{1-1/N}$ are the functions considered in \cite{sturm:II}.

\begin{theorem}\label{thm:CDbetaVsLSV}
Let $(X,\sfd,\mm)$ be an essentially non-branching m.m.s.\@ equipped with the gauge function $\sfG=\sfd$. The following hold:
\begin{enumerate}[label=(\roman*)]
\item \label{CDbetaVsLSV-i-ii} $(X,\sfd,\mm)$ satisfies the entropic curvature-dimension condition $\CD^{e}(K,N)$ of Er\-bar-Kuwada-Sturm \cite{EKS} if and only if $(X,\sfd,\mm, \sfG)$ satisfies $\CD(\beta_{K,N}^{\sigma}, N)$ in the sense of Definition \ref{def:CDMCPbetan};
\item \label{CDbetaVsLSV-iii} $(X,\sfd,\mm)$ satisfies the measure contraction property $\MCP(K,N)$ in the formulation of Ohta \cite{OhtaMCP} if and only if $(X,\sfd,\mm, \sfG)$ satisfies $\MCP(\beta_{K,N}^{\tau})$ in the sense of Definition \ref{def:CDMCPbetan}.
\end{enumerate}
\end{theorem}

\begin{proof}
\textbf{Proof of \ref{CDbetaVsLSV-i-ii},  $\CD(\beta_{K,N}^{\sigma}, N) \, \Rightarrow\, \CD^{e}(K,N)$.}
The map $\theta\mapsto \log\left[ (\beta_{K,N}^{\sigma})_t(\sqrt{\theta})\right]$ is convex, as one can check by a direct computation.  Using Jensen's inequality, it follows that for any $\nu\in\OptGeo(\mu_0,\mu_1)$
\begin{align}
\frac{1}{N}\int_{\Geo(X)} \log (\beta_{K,N}^{\sigma})_t(\sfd(\gamma_{1}, \gamma_{0})) \nu(\di\gamma) & \geq \frac{1}{N} \log (\beta_{K,N}^{\sigma})_t   \left(\sqrt{ \int_{\Geo(X)} \sfd(\gamma_{1}, \gamma_{0})^{2} \nu(\di\gamma)}\right)\\
& = \frac{1}{N}   \log \left[ (\beta_{K,N}^{\sigma})_t (W_2(\mu_0,\mu_1))\right]\\
& =  \log  \sigma_{K/N}^{(t)}(W_2(\mu_0,\mu_1))  \, . \label{eq:logbetasigma}
\end{align}
Thus, if $(X,\sfd,\mm)$ is a metric measure space satisfying the $\CD(\beta_{K,N}^{\sigma}, N)$ condition for the Gauge function $\sfG=\sfd$, then the entropic  $\CD^{e}(K,N)$ convexity inequality of \cite{EKS} is satisfied for measures $\mu_0,\mu_1\in \Prob_{bs}^{*}(X,\sfd,\mm)$ with $\supp \mu_{0}\cap \supp \mu_{1}=\emptyset$. A straightforward approximation argument, yields the validity of the  $\CD^{e}(K,N)$ convexity inequality for measures $\mu_0,\mu_1\in \Prob_{bs}^{*}(X,\sfd,\mm)$ with $\mm(\supp \mu_{0}\cap \supp \mu_{1})=0$. At this point, one can follow verbatim the proof of  \cite[Thm.\@ 3.12, (iii) $\Rightarrow$ (ii)]{EKS} (noticing that there the  $\CD^{e}(K,N)$ convexity inequality is used only for measures whose supports have $\mm$-negligible intersection)  and infer that for any $\mu_0,\mu_{1} \in \Prob_{bs}^{*}(X,\sfd,\mm)$ (without any condition on the intersection of the supports) there exists  $\nu\in \OptGeo(\mu_{0}, \mu_{1})$ such that $(\ee_{t})_{\sharp} \nu\ll \mm$ for all $t\in (0,1)$ and 
\begin{equation}\label{eq:CDeKNgeo}
\rho_{t}(\gamma_{t})^{-\frac 1 N}\geq \sigma_{K/N}^{(1-t)} (\sfd(\gamma_{0}, \gamma_{1})) \rho_{0}(\gamma_{0})^{-\frac  1 N}+ \sigma_{K/N}^{(t)} (\sfd(\gamma_{0}, \gamma_{1})) \rho_{1}(\gamma_{1})^{-\frac  1 N} ,
\end{equation}
for $\nu$-a.e.\@ $\gamma\in \Geo(X)$, where $\rho_{t}$ is the density of $(\ee_{t})_{\sharp}\nu$ w.r.t.\@ $\mm$. Using now  \cite[Thm.\@ 3.12, (ii) $\Rightarrow$ (iii)]{EKS} we infer that  $(X,\sfd,\mm)$ satisfies the entropic curvature-dimension condition $\CD^{e}(K,N)$.

\textbf{Proof of \ref{CDbetaVsLSV-i-ii}, $\CD^{e}(K,N) \, \Rightarrow\, \CD (\beta_{K,N}^{\sigma}, N)$.}
Conversely, if $(X,\sfd,\mm)$ is essentially non-branching and it satisfies $\CD^{e}(K,N)$, then from \cite[Thm.\@ 3.12, (iii) $\Rightarrow$ (ii)]{EKS} we know that for any $\mu_0,\mu_{1} \in \Prob_{bs}^{*}(X,\sfd,\mm)$ there exists  $\nu\in \OptGeo(\mu_{0}, \mu_{1})$ such that $(\ee_{t})_{\sharp} \nu\ll \mm$ for all $t\in (0,1)$ and  \eqref{eq:CDeKNgeo} hold for $\nu$-a.e.\@ $\gamma\in \Geo(X)$, where $\rho_{t}$ is the density of $(\ee_{t})_{\sharp}\nu$ w.r.t. $\mm$. With the same argument one can prove that \eqref{eq:CDeKNgeo} holds more generally when $\mu_0\in \mathcal{P}_{bs}(X,\sfd,\mm)$, up to removing the first term in the right hand side (we will assume this convention throughout the proof).
We rewrite \eqref{eq:CDeKNgeo} as
\begin{multline}
-\frac{1}{N}\log \rho_{t}(\gamma_{t})  \geq \log\left[ \exp\left(-\frac{1}{N} \log \rho_{0}(\gamma_{0}) + \log a_{1-t}(\gamma_{0}, \gamma_{1}) \right) \right. 
\\
\left. + \exp\left(-\frac{1}{N} \log \rho_{1}(\gamma_{1}) + \log a_{t}(\gamma_{0}, \gamma_{1})\right)  \right],
   \label{eq:CDeKNgeo2}
\end{multline}
where we set $a_{t}(\gamma_{0}, \gamma_{1})= \sigma_{K/N}^{(t)} (\sfd(\gamma_{0}, \gamma_{1}))$. Recalling that the map $(x,y)\mapsto \log(e^{x}+e^{y})$ is convex, integrating \eqref{eq:CDeKNgeo2} w.r.t.\@ $\nu$ and using Jensen's inequality, we obtain
\begin{multline}\label{eq:CDeKNInt}
-\frac{1}{N}\Ent(\mu_{t}|\mm) \geq \log\left[\exp\left(-\frac{1}{N} \Ent(\mu_{0}|\mm) +\int_{\Geo(X)} \log a_{1-t}(\gamma_{0}, \gamma_{1}) \nu(\di\gamma) \right) \right.
\\
\left. + \exp\left(-\frac{1}{N} \Ent(\mu_{1}|\mm) +\int_{\Geo(X)} \log a_{t}(\gamma_{0}, \gamma_{1})  \nu(\di\gamma)  \right) \right].
\end{multline}
Taking the exponential of \eqref{eq:CDeKNInt} and recalling \eqref{eq:betaVssigma}, we obtain that $(X,\sfd,\mm)$ endowed with the gauge function $\sfG=\sfd$ satisfies the $\CD(\beta_{K,N}^{\sigma}, N)$ condition \eqref{eq:defCDbetan}.

\textbf{Proof of \ref{CDbetaVsLSV-iii}, $\MCP(\beta_{K,N}^{\tau})\Rightarrow \MCP(K,N)$.} Using \eqref{eq:barbetatau} and arguing along the lines of the proof of \cite[Thm.\@ 3.12, (iii) $\Rightarrow$ (ii)]{EKS}, one uses the essential non-branching assumption to localize the $\MCP(\beta_{K,N}^{\tau})$ integral estimate to a pointwise estimate along geodesics. More precisely, one can show that for every $\mu_1\in \Prob_{bs}^{*}(X,\sfd,\mm)$ and every $x_{0}\in \supp \mm \setminus \supp \mu_1$  there exists  $\nu\in \OptGeo(\delta_{x_{0}}, \mu_{1})$ such that $(\ee_{t})_{\sharp} \nu\ll \mm$ for all $t\in (0,1)$ and
\begin{equation}\label{eq:MCPKNgeo}
\rho_{t}(\gamma_{t})\leq \tau_{K, N}^{(t)} (\sfd(\gamma_{0}, \gamma_{1}))^{-N} \rho_{1}(\gamma_{1}), \qquad \nu\text{-a.e. } \gamma \in \Geo(X),
\end{equation}
where $\rho_{t}$ denotes the density of $(\ee_{t})_{\sharp}\nu$ w.r.t.\@ $\mm$. Let $A$ be a bounded Borel set with positive measure and let $x_{0} \in \supp\mm \setminus  \bar{A}$, where $\bar{A}$ is the closure of $A$. Letting $\mu_1:=\mm(A)^{-1} \, \mm\llcorner_{A}$ and applying \eqref{eq:MCPKNgeo},  one gets
\begin{equation}
\frac{1}{\rho_{t}(\gamma_{t})} \geq  \tau_{K, N}^{(t)} (\sfd(\gamma_{0}, \gamma_{1}))^{N}\; \mm(A)  , \qquad \nu\text{-a.e. } \gamma \in \Geo(X).
\end{equation}
Let now $E\subset X$ be an arbitrary Borel set. Integrating the last inequality over $\ee_t^{-1}(E)$ and recalling that $(\ee_{t})_{\sharp}\nu=\rho_t \, \mm$,  yields
\begin{align*}
\mm(E)&= \int_{\ee_t^{-1}(E)} \frac{1}{\rho_{t}(\gamma_{t})} \nu(\di \gamma)  \\
&\geq   \int_{\ee_t^{-1}(E)} \tau_{K, N}^{(t)} (\sfd(\gamma_{0}, \gamma_{1}))^{N}\; \mm(A)  \, \nu(\di \gamma).
\end{align*}
One thus obtains Ohta's \cite{OhtaMCP} formulation of $\MCP(K,N)$
\begin{equation}\label{eq:OhtaMCP}
\mm \geq  (\ee_{t})_{\sharp} \left(\tau_{K,N}^{(t)}\big(\sfd(\gamma_{0},\gamma_{1})\big)^{N} \mm(A) \, \nu \right),
\end{equation}
 with the caveat  $x_{0}\in \supp \mm \setminus \bar{A}$ and $A$ bounded. The general case follows by a standard approximation argument, that we sketch below.
 
Observing that the distance is a \emph{meek} gauge function in the sense of Definition \ref{def:meek}, the generalized Bishop-Gromov Theorem \ref{thm:BG1} holds. We infer that $\mm(\{x\})=0$ for every $x\in X$ and that $(\supp\mm, \sfd)$ is a proper metric space (see Corollary \ref{cor:totallyboundedclassic}).
 
 Let $A$ be a bounded Borel set with positive measure and let $x_{0} \in  \bar{A}\, \cap\, \supp\mm$.  For $\ve>0$ small enough we have that $\mm(A\setminus B_{\ve}(x_{0}))>0$. Set 
\begin{equation}
\mu_{1,\ve}:= (\mm(A\setminus B_{\ve}(x_{0})))^{-1} \, \mm \llcorner_{A \setminus B_{\ve}(x_{0})}.
\end{equation} 
From the discussion above, we know that there exists  $\nu_\ve\in \OptGeo(\delta_{x_{0}}, \mu_{1,\ve})$ such that
\begin{equation}\label{eq:OhtaMCP2}
\mm   \geq  (\ee_{t})_{\sharp} \left(\tau_{K,N}^{(t)}\big(\sfd(\gamma_{0},\gamma_{1})\big)^{N} \mm\big(A \setminus B_{\ve}(x_{0})\big) \nu_\ve \right).\end{equation}
By letting  $\ve \downarrow 0$, by stability of optimal transport \cite[Thm.\@ 28.9]{Vil} and using that $(\supp\mm, \sfd)$ is a proper metric space, one obtains that there exists  $\nu\in \OptGeo(\delta_{x_{0}}, \mu_{1})$ such that \eqref{eq:OhtaMCP} holds.  For a general Borel set $A$ of finite positive measure: intersect $A$ with the metric ball $B_R(x_0)$, obtain \eqref{eq:OhtaMCP} and take the limit as $R\uparrow +\infty$ to conclude.

\textbf{Proof of \ref{CDbetaVsLSV-iii}, $ \MCP(K,N) \Rightarrow \MCP(\beta_{K,N}^{\tau})$.} Conversely,  it is standard  to check that \eqref{eq:OhtaMCP} implies \eqref{eq:MCPKNgeo}
 (see for instance \cite[App.\@ A]{CavaMondCCM}). Arguing along the lines of  $\CD^{e}(K,N) \, \Rightarrow\, \CD (\beta_{K,N}^{\sigma}, N)$, we obtain that $(X,\sfd,\mm)$ endowed with the gauge function $\sfG=\sfd$ satisfies the $\MCP(\beta_{K,N}^{\tau})$ condition \eqref{eq:defMCPbeta}.
\end{proof}

\begin{remark}[Connection with other synthetic formulations]\label{rem:CDbetaVsLSV}
Let $K\in \R, N\in [1,+\infty)$. In case $(X,\sfd,\mm)$ is essentially non-branching:
\begin{itemize}
\item It was proved in \cite[Thm.\@ 3.12]{EKS} that the reduced $\CD^{*}(K,N)$ condition (see \cite{BS10}) is equivalent to the entropic $\CD^{e}(K,N)$. 
\item The $\CD(K,N)$ condition (as in \cite{sturm:II}) is equivalent to $\CD^{*}(K,N)$. This is a deep result establishing the local-to-global property for $\CD(K,N)$; it was proved in \cite{CavaMil} in case $\mm(X)<\infty$.
Their approach is based on the \emph{needle decomposition}  of \cite{Klartag} and the \emph{localization technique} developed in \cite{CavaMond}.

\item Ohta's \cite{OhtaMCP} and Sturm's \cite{sturm:II} formulations of $\MCP(K,N)$ are equivalent, see for instance \cite[App.\@ A]{CavaMondCCM}.
\item Our theory is modeled on the convexity inequalities for the entropic curvature-dimension condition à la Erbar-Kuwada-Sturm \cite{EKS}, that involves the functional $\U_N$. One could have developed the theory studying instead the Renyi entropy functional, as in \cite{sturm:II,BS10}, or general functionals in the McCann class  $\mathcal{DC}(N)$ as in \cite{lottvillani:metric,LVJFA}, with analogous outcomes. For a discussion of connections between the approach in  \cite{lottvillani:metric,LVJFA} and the $\CD^{*}(K,N)$ condition see for instance  \cite[Sec.\@ 9]{AMSmemo}.
\end{itemize}
\end{remark}

\begin{remark}[Connection with  E.\,Milman's theory]\label{rem:Milman}
In \cite{MilmanSR}, E.\,Milman introduced a quasi-convex relaxation of the Lott–Sturm–Villani's $\CD(K,N)$, called the Quasi Cur\-va\-ture-Di\-men\-sion condition $\QCD(Q,K,N)$, see \cite[Def.\@ 2.3]{MilmanSR}. $Q\geq 1$ is an auxiliary parameter so that for $Q=1$ the $\QCD(1,K,N)$ is equivalent to the $\CD(K,N)$, and it is strictly weaker than the latter for $Q>1$. With the same arguments used in the proof of Theorem \ref{thm:CDbetaVsLSV}, for essentially non-branching m.m.s., the $\QCD(Q,K,N)$ for $N=n$ corresponds to the $\CD(\beta,n)$ inequality \eqref{eq:defMCPbeta} with gauge function $\sfG=\sfd$, and
\begin{equation}\label{eq:betaQCD}
\beta_t(\theta) = \frac{1}{Q}(\beta^{\tau}_{K,n})_t(\theta), \qquad \forall\, t\in [0,1],\quad \forall\, \theta \in [0,+\infty),
\end{equation}
where $(\beta^{\tau}_{K,n})_t$ is the same coefficient \eqref{eq:barbetatau} of the Lott-Sturm-Villani's theory. Formally the coefficient \eqref{eq:betaQCD} is not of the form \eqref{eq:defbeta} due to the extra factor $1/Q$. However, all results of Section \ref{sec:geometricconsequences} hold with the same proofs for this more general class of coefficients, up to keeping track of this extra factor in the statements.

On the other hand the finer Curvature-Geodesic-Topological-Dimension condition $\mathsf{CGTD}(K,N,n)$ introduced by E.\, Milman in \cite[Sec.\@ 7]{MilmanSR}, for $K\in \R$, $n\geq 1$ and $N\geq n$ corresponds to the $\CD(\beta_{K,N}^\tau,n)$ for gauge function $\sfG=\sfd$. 

Notice that for $Q=1$ and $N=n$, $\mathsf{CGTD}(K,N,n)$ and $\QCD(Q,K,N)$ are equivalent, and in turn equivalent to the classical $\CD(K,N)$.
\end{remark}

\subsubsection{Recovering the Balogh-Kristály-Sipos' theory}\label{sec:howtorecover2}

In this section we show how the general theory introduced above also contains the one developed in \cite{BKS} for Heisenberg groups.

Let $d\geq 1$. The Heisenberg group $\mathbb{H}^{d}$ is the non-commutative group structure on $\R^{2d+1}$ given, in coordinates $(x_{1},\ldots,x_{d},y_{1},\ldots,y_{d},z)$, by the law
\begin{equation}
(x,y,z) \star (x',y',z') = \left(x+x',y+y', z+z'+\sum_{i=1}^{d}\frac{1}{2}(x_{i}y_{i}
'-y_{i}x_{i}')\right).
\end{equation}
Consider the $2d$ left-invariant vector fields 
\begin{equation}
X_{i} = \frac{\partial}{\partial x_{i}} -\frac{y_{i}}{2}\frac{\partial}{\partial z}, \qquad Y_{i} =\frac{\partial}{\partial y_{i}} + \frac{x_{i}}{2}\frac{\partial}{\partial z},\qquad i=1,\ldots,d,
\end{equation}
and the left-invariant metric making $\{X_{i},Y_{i}\}_{i=1}^d$ a global orthonormal frame. We equip $\mathbb{H}^{d}$ with the associated sub-Riemannian (Carnot-Carath\'eodory) distance, denoted by $\sfd_{cc}$ (see Appendix \ref{a:SR}), and the Lebesgue measure $\mathscr{L}^{2d+1}$, which is a Haar measure.

We set
\begin{equation}
\sfs(\theta) =\theta \sin\left(\frac{\theta}{2}\right)^{2d-1}\left[\sin\left(\frac{\theta}{2}\right)-\frac{\theta}{2} \cos\left(\frac{\theta}{2}\right)\right], \qquad \theta\in [0,+\infty).
\end{equation} 
Notice that $\sfs(\theta)$ has order $N=2d+3$ as $\theta \to 0$. Moreover, it holds
\begin{equation}
\cD=\inf\{\theta>0\mid \sfs_{d}(\theta)=0\}=2\pi.
\end{equation}
The corresponding distortion coefficient, following the recipe of \eqref{eq:defbeta}, is then
\begin{multline}
[0,1]\times [0,+\infty]\ni (t,\theta) \mapsto \\ \beta^{\H^d}_t(\theta) =
\begin{cases}
 t^{2d+3} &\theta =0, \\
t\dfrac{ \sin\left(\frac{t\theta}{2}\right)^{2d-1}\left[\sin\left(\frac{t\theta}{2}\right)-\frac{t\theta}{2} \cos\left(\frac{t\theta}{2}\right)\right]}{ \sin\left(\frac{\theta}{2}\right)^{2d-1}\left[\sin\left(\frac{\theta}{2}\right)-\frac{\theta}{2} \cos\left(\frac{\theta}{2}\right)\right]}& 0<\theta <2\pi,\\
 +\infty & \theta \geq 2\pi, t\neq 0, \\
0 & \theta \geq 2\pi, t= 0. \\
 \end{cases}
\end{multline}
We highlight in particular the following identity
\begin{equation}
\beta^{\H^d}_t(\theta)=(\tau_{t}^{d}(\theta))^{2d+1}, \qquad \forall\, t\in [0,1],\quad \forall\, \theta \in [0,2\pi],
\end{equation}
where $\tau_{s}^{d}(\theta)$ is the function considered in \cite[Eq.\@ (1.6)]{BKS}.

The main result \cite[Thm.\@ 1.1]{BKS} is a Jacobian determinant inequality which permits by standard manipulations to obtain an interpolation inequality for optimal transport, see \cite[Eq.\@ (3.17), p.\@ 61]{BKS}. More precisely, for all  $\mu_0,\mu_1\in \Prob_{bs}^{*}(\mathbb{H}^{d},\sfd_{cc},\mathscr{L}^{2d+1})$, there exists $\nu\in\OptGeo(\mu_0,\mu_1)$ associated with a $W_2$-geodesic $(\mu_t)_{t\in [0,1]}$ such that $\mu_t \ll \mathscr{L}^{2d+1}$ for all $t\in (0,1]$, for $\nu$-a.e.\@ $\gamma$ it holds $(\gamma_0, \gamma_t)\notin\Cut(\H^d)$ for all $t\in (0,1]$, and
\begin{equation}\label{eq:interpolationineq4}
\frac{1}{\rho_t(\gamma_t)^{1/(2d+1)}}\geq \frac{\beta^{\H^d}_{1-t}(\theta^{\gamma_1,\gamma_0})^{1/(2d+1)}}{\rho_0(\gamma_0)^{1/(2d+1)}} + \frac{\beta^{\H^d}_t(\theta^{\gamma_0,\gamma_1})^{1/(2d+1)}}{\rho_1(\gamma_1)^{1/(2d+1)}}, \qquad \forall\, t\in [0,1],
\end{equation}
where $\rho_t = \frac{\di \mu_t}{\di \mathscr{L}^{2d+1}}$, and $\theta^{x,y}$ is the vertical norm of the covector associated with the geodesic joining $x$ and $y$, for $(x,y)\notin \mathrm{Cut}(\mathbb{H}^d)$.  In \eqref{eq:interpolationineq4} we use the convention that if $\mu_0$ is not absolutely continuous then the first term in the right hand side is omitted.

Using the definition of $\U_n$, the convexity of the map $\R^{2}\ni(x,y)\mapsto \log(e^{x}+e^{y})$ and Jensen's inequality, one establishes the validity of the $\CD(\beta,2d+1)$, with $\beta=\beta^{\H^d}$.

We stress that the argument in the r.h.s.\@ of \eqref{eq:interpolationineq4} does not depend on the distance. It is then natural to set as a gauge function any map $\sfG:\mathbb{H}^d\times \mathbb{H}^d\to [0,+\infty]$ such that
\begin{equation}
\sfG(x,y):=\theta^{x,y},\qquad \forall\, (x,y)\notin \mathrm{Cut}(\mathbb{H}^d).
\end{equation} 
Then the results of \cite{BKS} can be restated within our framework as follows.
\begin{theorem} \label{t:HeisisCD}
The gauge m.m.s.\@ $(\mathbb{H}^{d},\sfd_{cc},\mathscr{L}^{2d+1},\sfG)$ satisfies the $\CD(\beta^{\H^d},2d+1)$.
\end{theorem}
It is not hard to check that $\beta^{\H^d}_t(\theta) \geq t^{2d+3}$ for all $t\in [0,1]$ and $\theta \in [0,2\pi]$. Therefore, by Proposition~\ref{prop:hierarchy}\ref{i:hierarchy3}, the $\CD(\beta^{\H^d},2d+1)$ implies the $\CD(t^{2d+3},2d+1)$. In turn, the latter implies the classical $\MCP(0,2d+3)$, originally established for $\H^d$ by Juillet in \cite{NJ09}.

\section{Geometric consequences}\label{sec:geometricconsequences}

In this section we establish some geometric properties implied by the $\CD(\beta, n)$ condition (compare with \cite{EKS, sturm:II, lottvillani:metric} for the Lott-Sturm-Villani's $\CD(K,N)$ spaces).

\subsection{Generalized Brunn-Minkowski inequality}\label{sec:Brunn-Mink}

Given two Borel sets $A_{0}, A_{1}\subset X$, denote with $A_{t}$ the set of $t$-intermediate points, i.e.
\begin{equation}\label{eq:defAt}
A_{t}:=Z_{t}(A_{0},A_{1})=\{\gamma_{t} \mid \gamma\in \Geo(X),\, \gamma_{0}\in A_{0},\, \gamma_{1}\in A_{1} \}.
\end{equation}
It is well-known that $A_{t}$ is a Suslin set. For non-empty Borel sets $A_0,A_1\subset X$, denote
\begin{equation}\label{eq:betainf}
\beta_{t}(A_{0}, A_{1}):= \sup_{\tilde{A}_i\subseteq A_i} \inf\{\beta_{t}(\sfG(x_0,x_1)) \mid x_i\in \tilde{A}_i \cap \supp\mm \},
\end{equation}
where the sup is taken over all full-measure non-empty subsets $\tilde{A}_i\subseteq A_i$, for $i=0,1$.
It is clear from the definition \eqref{eq:betainf} that 
\begin{equation}\label{eq:Eqbetatriangle}
\mm(A_{0}\triangle \tilde{A}_{0})=\mm(A_{1}\triangle \tilde{A}_{1})=0 \quad \Longrightarrow \quad \beta_{t}(A_{0}, A_{1})=\beta_{t}(\tilde{A}_{0},  \tilde{A}_{1})\,,
\end{equation}
where $A_{0}, \tilde{A}_{0}, A_{1}, \tilde{A}_{1}\subset X$ are Borel sets and $\triangle$ denotes the symmetric difference of sets.

\begin{theorem}[Generalized Brunn-Minkowski inequality]\label{thm:Brunn-Mink}
Let $(X,\sfd,\mm)$ be a m.m.s.\@ with gauge function $\sfG$, satisfying the $\CD(\beta, n)$ condition, with $n\in [1,+\infty)$ and $\beta$ as in \eqref{eq:defbeta}. Let $A_{0}, A_{1}\subset X$ be Borel sets with $\mm(A_{0}), \mm(A_{1})>0$  and $\mm(A_{0}\cap A_{1})=0$. Then
\begin{equation}\label{eq:BrunMink}
\bar{\mm}(A_{t})^{1/n} \geq \beta_{1-t}(A_{1}, A_{0})^{1/n} \cdot \mm(A_{0})^{1/n}+ \beta_{t}(A_{0}, A_{1})^{1/n}\cdot \mm(A_{1})^{1/n}, \qquad \forall\, t\in (0,1),
\end{equation}
where $\bar \mm$ is the completion of $\mm$, $A_{t}$ is the set of $t$-intermediate points \eqref{eq:defAt}, and with the convention that $0\cdot \infty = 0$.
\end{theorem}

\begin{proof} Assume that $\mm(A_i)<+\infty$ for $i=0,1$. We first reduce the proof to the bounded and disjoint case. Since $\mm(A_i)<+\infty$ and $\mm(A_{0}\cap A_{1})=0$, for any $\varepsilon>0$ there exists $A_{i,\ve} \subset A_{i}$ which is closed and bounded, $A_{0,\ve}\cap A_{1,\ve}=\emptyset$,   and $\mm(A_{i,\ve})\to \mm(A_i)$ as $\ve \to 0$. Furthermore, setting $A_{t,\ve}:=Z_{t}(A_{0,\ve}, A_{1,\ve})$ the set of $t$-intermediate points from $A_{0,\ve}$ to $A_{1,\ve}$, it holds
\begin{equation}
\beta_t(A_{0,\varepsilon},A_{1,\varepsilon}) \geq \beta_t(A_0,A_1), \qquad \beta_t(A_{1,\varepsilon},A_{0,\varepsilon}) \geq \beta_t(A_1,A_0), \qquad A_t \supset A_{t,\varepsilon}.
\end{equation}
Thus it is sufficient to prove \eqref{eq:BrunMink} for  closed and bounded sets $A_0,A_1$ with $A_{0}\cap A_{1}=\emptyset$.
Let $\mu_{i}:=\tfrac{1}{\mm(A_{i})}\mm\llcorner_{A_{i}}\in \Prob_{bs}^{*}(X,\sfd,\mm)$ for $i=0,1$ and note that $\supp\mu_{0}\cap \supp \mu_{1}=\emptyset$. Observe that $\U_{n}(\mu_{i})=\mm(A_{i})^{1/n}$.  By assumption, there exists a $W_2$-geodesic $(\mu_t)_{t\in [0,1]}$ joining them, and satisfying the inequality \eqref{eq:defCDbetan}. We remark that $\mu_t$ is concentrated on $A_t$, which is a Suslin set. Using Jensen's inequality twice, we obtain
\begin{equation}
\U_{n}(\mu_{t})=\exp\left(- \frac{1}{n} \int_{X} \log \, \rho_{t} \, \di\mu_{t} \right) \leq \int_{X} \rho_{t}^{-1/n}\,  \di \mu_{t}=\int_{A_{t}} \rho_{t}^{1-\frac{1}{n}} \di \bar{\mm} \leq \bar{\mm}(A_{t})^{1/n}.
\end{equation}
Then \eqref{eq:BrunMink} for sets of finite measure follows directly from the $\CD(\beta,n)$ inequality \eqref{eq:defCDbetan}, and from optimizing with respect to full-measure subsets of $A_i$.

Since $\mm$ is $\sigma$-finite, we can prove the result for sets of possibly infinite measure by approximating them with a monotone sequence of sets with finite measure.
\end{proof}

With analogous arguments one obtains the next generalized half-Brunn-Minkowski inequality for $\MCP(\beta)$, which will play a pivotal role in geometric applications.

\begin{theorem}[Generalized half-Brunn-Minkowski inequality]\label{thm:half-Brunn-Mink}
Let $(X,\sfd,\mm)$ be a m.m.s.\@ with gauge function $\sfG$, satisfying the $\MCP(\beta)$ condition, with $\beta$ as in \eqref{eq:defbeta}. Let $A$ be a Borel set with $\mm(A)>0$, and $\bar x\in \supp \mm$  with $\mm(\{\bar{x}\}\cap A)=0$. Then
\begin{equation}\label{eq:halfBrunMink}
\bar{\mm}(A_{t}) \geq \beta_{t}(\bar{x}, A) \cdot \mm(A), \qquad \forall\, t\in (0,1),
\end{equation}
where $\bar \mm$ is the completion of $\mm$, $A_{t}$ denotes the set of $t$-intermediate points  \eqref{eq:defAt} between $A_0 = \{\bar{x}\}$ and $A_1 = A$, and with the convention that $0\cdot \infty = 0$.
\end{theorem}

\begin{remark}
In case  $\supp \mm$ is not a singleton, the condition  $\mm(\{x\})=0$ for every $x\in X$ is very natural: indeed it is satisfied for all smooth measures typically employed in sub-Riemannian geometry, and it is implied by the generalized Bishop-Gromov Theorem \ref{thm:BG1}\ref{i:BG2} under the assumption that $\sfG$ is meek (see Definition \ref{def:meek}) and that the natural assumption \eqref{eq:Gfinite} holds (see also Remark \ref{rem:onCond3.15}). Recall that $\MCP(K,N)$ spaces (for $N\in [1,+\infty)$, $K\in \R$), thus including all $\CD(K,N)$ ones, satisfy this condition. 
\end{remark}

\begin{remark}\label{rmk:concentratedimprove}
Let $(\mu_t)_{t\in [0,1]}$ be a $W_2$-geodesic between the uniform distributions on $A_0$ and $A_1$, for which the $\CD(\beta,n)$ inequality \eqref{eq:defCDbetan} holds. Then in the l.h.s.\@ of \eqref{eq:BrunMink} one can replace $A_t$ with any (possibly smaller) measurable set on which $\mu_t$ is concentrated. An analogous remark applies for the l.h.s.\@ of \eqref{eq:halfBrunMink}.
\end{remark}

If the distortion coefficients are non-increasing or non-decreasing, then the above inequalities can be expressed in terms of the minimal/maximal value of the gauge function $\sfG$ on an appropriate set. More precisely, if $\theta\mapsto\beta_t(\theta)$ is non-decreasing or non-increasing for all $t\in (0,1)$, then letting
\begin{equation}\label{eq:defTheta}
\Theta:=
\begin{cases}
\inf\{\sfG(x_{0},x_{1}) \mid x_i \in A_i\cap\supp\mm,\; i=1,2 \} & \text{if  $\theta\mapsto \beta_{t}(\theta)$ is non-decreasing}, \\
\sup\{ \sfG(x_{0},x_{1})\mid x_i \in A_i\cap\supp\mm,\; i=1,2 \} & \text{if  $\theta\mapsto \beta_{t}(\theta)$ is non-increasing},
\end{cases}
\end{equation}
it holds 
\begin{equation}
\beta_t(A_0,A_1) \geq \beta_t(\Theta).
\end{equation}
This observation, together with the fact that the classical distortion coefficients 
\begin{equation}
(\beta_{K,N}^\tau)_t(\theta) = \tau_{K,N}^{(t)}(\theta)^N, \qquad \text{resp.} \qquad (\beta_{K,N}^\sigma)_t(\theta) = \sigma_{K/N}^{(t)}(\theta)^N,
\end{equation}
are non-increasing or non-decreasing (according to the sign of $K$), yields the familiar forms of the generalized Brunn-Minkowski inequalities for classical $\CD(K,N)$  spaces \cite{sturm:II} (resp.\@ for classical $\CD^{*}(K,N)$ or $\CD^{e}(K,N)$ spaces \cite{BS10, EKS}; see also Theorem \ref{thm:CDbetaVsLSV} and Remark \ref{rem:CDbetaVsLSV}).

\subsection{MCP implies CD for spaces supporting interpolation inequalities}\label{s:MCPtoCD}

The implication $\CD \Rightarrow\MCP$ is trivial with our definitions. Here we discuss the converse implication.  The result of this section can be applied to ideal sub-Riemannian manifolds of topological dimension $n$, that support interpolation inequalities for densities as in \cite{BRInv}. This includes, of course, all $n$-dimensional Riemannian manifolds \cite{CEMS}.

\begin{definition}\label{def:mcptocd}
We say that a m.m.s.\@ $(X,\sfd, \mm)$ \emph{supports interpolation inequalities for densities, with  dimensional parameter $n\in [1,+\infty)$}, if for any $\mu_0\in \Prob_{bs}(X,\sfd,\mm)$, $\mu_1 \in \Prob_{bs}^{*}(X,\sfd,\mm)$ with $\supp \mu_{0}\cap \supp \mu_{1}=\emptyset$ there exists a $W_2$-geodesic $(\mu_t)_{t\in [0,1]}$, induced by $\nu\in\OptGeo(\mu_0,\mu_1)$ such that $\mu_t\ll \mm$ for all $t\in (0,1]$, and letting $\rho_t:=\tfrac{\di \mu_t}{\di \mm}$ be the corresponding density, it holds for all $t\in (0,1)$:
\begin{equation}\label{eq:interpolationineq}
\frac{1}{\rho_t(\gamma_t)^{1/n}}\geq \frac{\beta_{1-t}^{(X,\sfd,\mm)}(\gamma_1,\gamma_0)^{1/n}}{\rho_0(\gamma_0)^{1/n}} + \frac{\beta_t^{(X,\sfd,\mm)}(\gamma_0,\gamma_1)^{1/n}}{\rho_1(\gamma_1)^{1/n}},\qquad \nu\text{-a.e.}\ \gamma \in \Geo(X),
\end{equation}
with the understanding that the first term in the right hand side of \eqref{eq:interpolationineq} is omitted if $\mu_0\notin \mathcal{P}_{ac}(X,\mm)$.  
\end{definition}

In \eqref{eq:interpolationineq}, $\beta_t^{(X,\sfd,\mm)}$ denotes the ``true'' distortion coefficient of the metric measure space $(X,\sfd,\mm)$, of which we recall the definition.
\begin{definition}\label{def:truedistcoeff}
The \emph{distortion coefficient} of $(X,\sfd,\mm)$ is the map $\beta_{(\cdot)}^{(X,\sfd,\mm)}(\cdot,\cdot): [0,1]\times X \times X \to [0,+\infty]$ defined by
\begin{equation}\label{eq:true}
\beta_t^{(X,\sfd,\mm)}(x,y):=\limsup_{r\to 0^+}\frac{\bar{\mm}(Z_t(\{x\},B_r(y)))}{\mm(B_r(y))}, \qquad \forall\, x,y\in X,\quad t\in [0,1],
\end{equation}
where $Z_t(A,B)$ denotes the set of $t$-midpoints between two Borel sets as in \eqref{eq:defAt}, and $\bar{\mm}$ denotes the completion of $\mm$.
\end{definition}

\begin{theorem}[$\MCP$ to $\CD$]\label{thm:mcptocd}
Let $(X,\sfd, \mm)$ be a m.m.s.\@ that supports interpolation inequalities for densities, with dimensional parameter $n\in [1,+\infty)$. Let $\sfG:X\times X \to [0,+\infty]$ be a gauge function. Assume that:
\begin{itemize}
\item  each $\nu\in\OptGeo(\mu_0,\mu_1)$ occurring in Definition \ref{def:mcptocd} is concentrated on a set of geodesics whose endpoints $(\gamma_0,\gamma_1)$ and $(\gamma_1,\gamma_0)$ are continuity points for $\sfG$;
\item $(X,\sfd, \mm,\sfG)$ satisfies the $\MCP(\beta)$, with $\beta$ as in \eqref{eq:defbeta}, and with $N\geq n$.
\end{itemize}
Then $(X,\sfd, \mm,\sfG)$ also satisfies the $\CD(\beta,n)$.
\end{theorem}

\begin{proof}
Let $x\neq y\in \supp \mm$  with $(x,y)$ a continuity point of $\sfG$, and let $r<\sfd(x,y)$.
We use the half-Brunn-Minkowski inequality for $\MCP(\beta)$ spaces of Proposition \ref{thm:half-Brunn-Mink}, for $\bar{x}=x$, $A=B_r(y)$. It holds, for all $t\in [0,1]$,
\begin{equation}
\frac{\bar{\mm}(Z_t(\{x\},B_r(y)))}{\mm(B_r(y))}\geq \beta_t(x,B_r(y)) \geq \inf \{\beta_t(\sfG(x,z))\mid z \in B_r(y)\}.
\end{equation}
Taking the $\limsup$ for $r\to 0^+$, and using the continuity of $\sfG$ at $(x,y)$, we find a sequence $\theta_j\to \sfG(x,y) \in [0,+\infty]$ such that
\begin{equation}\label{eq:mcptocd-bound1}
\beta_t^{(X,\sfd,\mm)}(x,y) \geq \lim_j \beta_t(\theta_j) = \beta_t(\sfG(x,y)), \qquad \forall\, t\in[0,1].
\end{equation}

Let $\nu\in \OptGeo(\mu_0,\mu_1)$ be as in Definition \ref{def:mcptocd}. Assume for simplicity that $\mu_0 \in \mathcal{P}_{ac}(X,\mm)$ (the argument works verbatim if $\mu_0\notin\mathcal{P}_{ac}(X,\mm)$ by omitting the corresponding term in all inequalities). By our assumption, $\nu$ is concentrated on a set of geodesics $\gamma$ such that $\sfG$ is continuous at $(\gamma_0,\gamma_1)$ and $(\gamma_1,\gamma_0)$.

For $\nu$-a.e.\@ $\gamma$ we can plug \eqref{eq:mcptocd-bound1} in \eqref{eq:interpolationineq}, finding that for all $t\in (0,1)$ it holds
\begin{equation}
\frac{1}{\rho_t(\gamma_t)^{1/n}}\geq \frac{\beta_{1-t}(\sfG(\gamma_1,\gamma_0))^{1/n}}{\rho_0(\gamma_0)^{1/n}} + \frac{\beta_t(\sfG(\gamma_0,\gamma_1))^{1/n}}{\rho_1(\gamma_1)^{1/n}},\qquad \nu\text{-a.e.}\ \gamma.
\end{equation}

Using the definition of $\U_n$,  the convexity of the map $\R^{2}\ni(x,y)\mapsto \log(e^{x}+e^{y})$ and Jensen's inequality, we obtain the $\CD(\beta,n)$ inequality.
\end{proof}

\begin{remark}
For Riemannian manifolds, the classical $\CD(K,n)$ (resp.\;$\MCP(K,n)$) corresponds to the $\CD(\beta^\tau_{K,n},n)$ (resp. $\MCP(\beta^\tau_{K,n})$), see Remark \ref{rem:CDbetaVsLSV}. Thus Theorem \ref{thm:mcptocd} is consistent with the well-known fact that, on $n$-dimensional Riemannian manifolds, $\MCP(K,N)$ implies $\CD(K,N)$ if $N=n$.
\end{remark}

\subsection{Gauge diameter estimate}\label{sec:diam}

\begin{definition}[$\sfG$-diameter]\label{def:Gdiam}
The \emph{$\sfG$-diameter} of a non-empty set $S$ is defined by:
\begin{equation}
\diam_\sfG(S) = \sup_{x\in S} \left(\mm\shortminus\esssup_{y\in S, \,  y\neq x} \sfG(x,y)\right).
\end{equation}
We say that $S\subset X$ is \emph{$\sfG$-bounded} if it is non-empty and it has finite $\sfG$-diameter.
\end{definition}

Recall that the parameter $\cD$ in \eqref{eq:defcD}, when finite, characterizes the blow-up of $\theta\mapsto \beta_t(\theta)$, for $t\in (0,1)$.

\begin{proposition}[$\sfG$-diameter estimate]\label{prop:Gdiamestimate}
Let $(X,\sfd,\mm)$ be a m.m.s.\@ with gauge function $\sfG$. Let $\beta$ be as in \eqref{eq:defbeta} and assume that $\beta_t(+\infty) = +\infty$ for all $t\in (0,1)$. Note that this is in particular the case if $\cD<+\infty$.

\begin{itemize}
\item If $(X,\sfd,\mm,\sfG)$ satisfies the $\MCP(\beta)$, then for all $x\in\supp\mm$ it holds
\begin{equation}
\sfG(x,y)<\cD, \qquad \mm\text{-a.e. } y \in X,  \quad y\neq x.
\end{equation}
In particular, $\diam_{\sfG}(\supp\mm) \leq \cD$.
\item If $(X,\sfd,\mm,\sfG)$ satisfies the $\CD(\beta,n)$ for some $n\in [1,+\infty)$, then 
\begin{equation}
\sfG(x,y) < \cD, \qquad \pi\text{-a.e. }(x,y)\in X\times X.
\end{equation}
for all $\pi = (\ee_0, \ee_1)_\sharp\nu$, where $\nu \in \OptGeo(\mu_0,\mu_1)$ is such that inequality  \eqref{eq:defCDbetan} holds.
\end{itemize}
\end{proposition}
\begin{remark}\label{rmk:MCPimpliesGNonTriv}
We remark that the case $\cD = +\infty$ is included in the statement of Proposition \ref{prop:Gdiamestimate}. In this case, $\mm(\{y\in X\mid y\neq x,\,\sfG(x,y)=+\infty\})=0$ for all $x\in \supp\mm$. This implies the non-triviality condition \eqref{eq:GNonTriv} considered below. 
\end{remark}
\begin{proof}
Fix $x\in\supp\mm$. We show that the set $A=\{y\in X \mid  y\neq x, \, \sfG(x,y)\geq \cD\}\cap B_R(x)$ has zero measure for all $R>0$. By contradiction, assume that $\mm(A)>0$. Then we apply the half-Brunn-Minkowski inequality of Theorem \ref{thm:half-Brunn-Mink} to the set $A$. Notice that $\beta_t(x,A) = +\infty$ for all $t\in (0,1)$, and that $A_t \subset (B_R(x))_t \subset B_{tR}(x)$. We obtain hence
\begin{equation}
\mm(B_{tR}(x)) \geq \bar{\mm}(A_t) = +\infty,\qquad \forall\, t\in(0,1),
\end{equation}
contradicting the local finiteness of $\mm$.

To prove the second part of the proposition, we argue in a similar way: let $\nu\in \OptGeo(\mu_0,\mu_1)$ as in the statement (in particular $\mu_0,\mu_1$ have bounded support) such that $\pi(\{(x,y)\in X\times X \mid \sfG(x,y) \geq \cD\})>0$, and using \eqref{eq:defCDbetan} we obtain that, for the corresponding $W_2$-geodesic $\mu_t = (\ee_t)_\sharp \nu$ it holds
\begin{equation}
\U_n(\mu_t|\mm) = +\infty,\qquad \forall\, t\in(0,1).
\end{equation}
This implies,  by \eqref{eq:jensenEnt},  that $\mm(\supp\mu_t)=+\infty$. However, since $\supp\mu_0$ and $\supp\mu_1$ are bounded, then also $\supp\mu_t$ is bounded, yielding a contradiction.
\end{proof}

For some interesting cases (cf.\@ Section \ref{sec:nonsmooth}), it holds $\sfG \geq \sfd$. In this case, Proposition \ref{prop:Gdiamestimate} admits the following Bonnet-Myers type result.

\begin{corollary}[Bonnet-Myers]\label{cor:BonMy}
Let $(X,\sfd,\mm)$ be a m.m.s.\@ with gauge function $\sfG$, satisfying the $\MCP(\beta)$ with $\beta$ as in \eqref{eq:defbeta}. Assume that $\sfG \geq \sfd$. Then $\diam(\supp\mm) \leq \cD$, and if $\cD<+\infty$ then $\supp\mm$ is compact.
\end{corollary}
\begin{proof}
The estimate on the diameter is immediate by continuity of $\sfd$. If $\cD<+\infty$, then $\supp\mm$ is bounded, and by Proposition \ref{prop:Gdiamestimate} it is also $\sfG$-bounded. By Corollary \ref{cor:totallyboundedclassic}, $\supp\mm$ is totally bounded, complete, and thus compact.
\end{proof}

\subsection{Doubling inequalities for balls}\label{Sec:Doubling}

We first prove a local doubling inequality for metric balls under minimal assumptions. We stress that ``locality'' in Proposition \ref{prop:doublingsimple} is measured with respect to the gauge function, which also determines the doubling constant. For this reason, and without additional relations between $\sfG$ and $\sfd$, the general form of the statement is subtly different with respect to the classical one.

\begin{proposition}[Local doubling]\label{prop:doublingsimple}
Let $(X,\sfd,\mm)$ be a m.m.s.\@ with gauge function $\sfG$, satisfying the $\MCP(\beta)$ with $\beta$ as in \eqref{eq:defbeta}.
Then any $\sfG$-bounded Borel subset $S\subseteq \supp\mm$ satisfies the following inequality: for all $t\in (0,1)$ there exists $C_{S,t}>0$ such that 
\begin{equation}\label{eq:doublingsimple}
\mm(B_{r}(x_0)\cap S)\leq C_{S,t}\cdot \mm(B_{tr}(x_0)),\qquad \forall\, r\geq 0,\; \forall\, x_0\in S.
\end{equation}
The constant $C_{S,t}$ can be estimated in terms of $\beta$ and $\diam_{\sfG}(S)$:
\begin{equation}\label{eq:crR}
\frac{1}{C_{S,t}} =\inf\{\beta_{t}(\theta)\mid \theta \in [0,\diam_{\sfG}(S)]\}  \in (0,t^N].
\end{equation}
\end{proposition}

\begin{proof}
Fix $S\subseteq \supp\mm$ and let $\sfG_0:=\diam_{\sfG}(S) <\infty$. We let
\begin{equation}
\frac{1}{C_{S,t}}:=\inf\{\beta_{t}(\theta)\mid \theta \in [0,\sfG_0]\} \in (0,t^N].
\end{equation}
Here, the upper bound is a consequence of the value of $\beta_t(0)=t^N$ in \eqref{eq:defbeta}, while the lower bound follows from Proposition \ref{prop:propertiesbeta}\ref{i:propertiesbeta4}. The case $r=0$ of \eqref{eq:doublingsimple} is trivial, so let $r>0$. 

Let  $x_0\in  S$, and let $\check{B}_r(x_0) := B_r(x_0)\setminus \{x_0\}$. Let $A=\check{B}_{r}(x_0)\cap S$. If $\mm(A)=0$ there is nothing to prove. Assume then that $\mm(A)>0$, and note that $\mm(A \cap \{x_0\})=0$. Thus we can apply the half-Brunn-Minkowski estimate of Theorem \ref{thm:half-Brunn-Mink} with $\bar{x}=x_0$ and $A=\check{B}_{r}(x_0)\cap S$. Since $A_t\subseteq \check{B}_{tr}(x_0)$, we have for all $t\in(0,1)$:
\begin{equation}
\mm(\check{B}_{tr}(x_0)) \geq \beta_t(x_0,\check{B}_{r}(x_0)\cap S  ) \cdot \mm(\check{B}_{r}(x_0)\cap S) 
\geq \tfrac{1}{C_{S,t}} \cdot \mm(\check{B}_{r}(x_0)\cap S).
\end{equation}
We justify the bound $\beta_t(x_0,\check{B}_{r}(x_0)\cap S )\geq 1/C_{S,t}$. To do it, for any $x\in S$ and any non-empty Borel set $U\subseteq S$, recall the definition \eqref{eq:betainf}, which for our case reduces to
\begin{equation}
\beta_t(x,U) = \sup_{\tilde{U}\subseteq U} \inf \{\beta_t(\sfG(x,y))\mid y\in \tilde{U}\},
\end{equation}
where $\tilde{U}$ ranges over all full-measure and non-empty  subsets of $U$. Assume that $x\in S\setminus U$. Since $U\subseteq S$ and the latter has bounded $\sfG$-diameter (see Definition \ref{def:Gdiam}), there exists a full-measure set $U'\subseteq U$ such that $\sfG(x,y) \leq \sfG_0$ for all $y \in U'$, and in particular it holds
\begin{equation}
\beta_t(x,U) \geq \inf \{\beta_t(\theta)\mid \theta \in [0,\sfG_0]\}  = \frac{1}{C_{S,t}}.
\end{equation}
In particular the above inequality holds for $x=x_0$ and $U=\check{B}_r(x_0) \cap S$  as required.

Thus we have proved \eqref{eq:doublingsimple} for $\check{B}_r(x_0)$ in place of $B_r(x_0)$, that is
\begin{equation}
\mm(\check{B}_{r}(x_0)\cap S)\leq C_{S,t}\cdot \mm(\check{B}_{tr}(x_0)),\quad \forall\, r\geq 0,\; \forall\, x_0\in S.
\end{equation}
Since $C_{S,t}\geq 1/t^N \geq 1$, we add $\mm(\{x_0\}\cap S)$ to both sides to obtain \eqref{eq:doublingsimple}.
\end{proof}

In some natural applications, bounded sets are $\sfG$-bounded, cf.\@ Section~\ref{sec:nonsmooth}. In this case Proposition \ref{prop:doublingsimple} can be restated as a classical doubling inequality. 

\begin{corollary}[Classical local doubling]\label{cor:localdoubling}
Let $(X,\sfd,\mm)$ be a m.m.s.\@ with gauge function $\sfG$, satisfying the $\MCP(\beta)$ with $\beta$ as in \eqref{eq:defbeta}. Assume that bounded subsets of $\supp\mm$ are $\sfG$-bounded. Then for any bounded set $S\subseteq \supp\mm$ and $R>0$ there exists $C=C(S,R)>0$ such that
\begin{equation}
\mm(B_{2r}(x_0))\leq C \cdot \mm(B_{r}(x_0)),\qquad \forall\, r \in [0,R],\; \forall\, x_0\in S.
\end{equation}
\end{corollary}
\begin{proof}
Let $R_0$ be the diameter of $S$, and let $S':=B_{2R+R_0}(\bar{x})$ be a ball centered in any point $\bar{x}\in S$. Of course any ball $B_{2r}(x_0)$ with $r\in [0,R]$ and $x_0\in S$ is contained in $S'$. By our additional hypothesis, $S'$ is also $\sfG$-bounded. We apply then Proposition \ref{prop:doublingsimple} to the set $S'$  with $t=1/2$ and get the claim with $C=C_{S',1/2}$.
\end{proof}

Finally, if $\diam_\sfG(\supp\mm)$ is finite, or if $\theta\mapsto\beta_t(\theta)$ is non-decreasing, then we have a classical global doubling inequality.

\begin{corollary}[Classical global doubling]\label{cor:classicaldoubling}
Let $(X,\sfd,\mm)$ be a m.m.s.\@ with gauge function $\sfG$, satisfying the $\MCP(\beta)$ with $\beta$ as in \eqref{eq:defbeta}. Assume that $\diam_\sfG(\supp\mm)$ is finite, or that $\theta\mapsto\beta_t(\theta)$ is non-decreasing. Then there exists a constant $C>0$ such that
\begin{equation}
\mm(B_{2r}(x_0))\leq C \cdot \mm(B_{r}(x_0)),\qquad \forall\, r \in [0,+\infty),\; \forall\, x_0\in \supp\mm.
\end{equation}
\end{corollary}
\begin{proof}
In both cases, using the same argument of the proof of Proposition \ref{prop:doublingsimple}, taking $S=\supp\mm$, we have that
\begin{equation}\label{eq:thisfact}
\beta_t(x_0,B_r(x_0)) \geq \frac{1}{C(t)} >0,
\end{equation}
for a constant $C(t)>0$ that does not depend on $x_0\in\supp\mm$ or $r\in [0,+\infty)$: if $\diam_\sfG(\supp\mm)$ is finite \eqref{eq:thisfact} follows from the definition of $\beta_t(x_0,B_r(x_0))$ and Proposition \ref{prop:propertiesbeta}\ref{i:propertiesbeta4}, while if $\theta\mapsto\beta_t(\theta)$ is non-decreasing the lower bound is given by $\beta_t(0)=t^N$.
\end{proof}

We also record a standard consequence that will be useful in the following.

\begin{corollary}[Proper geodesic space]\label{cor:totallyboundedclassic}
Let $(X,\sfd,\mm)$ be a m.m.s.\@ with gauge function $\sfG$, satisfying the $\MCP(\beta)$ with $\beta$ as in \eqref{eq:defbeta}. Assume that bounded subsets of $\supp\mm$ are $\sfG$-bounded, or $\theta\mapsto\beta_t(\theta)$ is non-decreasing. Then $(\supp\mm,\sfd)$ is a proper, geodesic space.
\end{corollary}
\begin{proof}
The proof follows the classical strategy. Assume first that bounded subsets of $\supp\mm$ are $\sfG$-bounded. Let $x_0\in\supp\mm$ and $L>0$. Let $x_1,\dots,x_k$ be a collection of $k\in \N$ points in $B_L(x_0)$, and $\varepsilon>0$. If $B_\varepsilon(x_i)\subseteq B_{L}(x_0)$ for all $i=1,\dots,k$  are pairwise disjoint, it holds
\begin{equation}
\mm(B_{L}(x_0)) \geq \sum_{i=1}^k \mm(B_{\varepsilon}(x_i))\geq \frac{1}{C_{S,t}} \sum_{i=1}^k \mm(B_{2L}(x_i)\cap S)\geq \frac{k}{C_{S,t}} \mm(B_{L}(x_0)),
\end{equation}
where we applied Proposition \ref{prop:doublingsimple} with $S=B_{3L}(x_0)$ and $t=\varepsilon/2L$. Thus, the maximal number of pairwise disjoint balls of radius $\varepsilon>0$ contained in $B_{L}(x_0)$ is bounded by
\begin{equation}
k \leq C_{S,t}<+\infty, \qquad\text{with}\quad S= B_{3L}(x_0),\quad t=\varepsilon/2L,
\end{equation}
where $C_{S,t}$ is the constant of \eqref{eq:crR}. As a consequence, the ball $B_{L}(x_0)$ can be covered by $C_{S,t}$ metric balls of radius $2\varepsilon$. Since $x_0\in\supp\mm$ and $L>0$ were arbitrary, we obtain that  balls in $\supp\mm$ are totally bounded, hence their closure is compact. In other words the metric space $(\supp\mm,\sfd)$ is proper. Since $(\supp\mm,\sfd)$ is a length space by the $\MCP(\beta)$ condition (cf. Remark \ref{rem:LengthSpace}), it is also geodesic.

If $\theta\mapsto\beta_t(\theta)$ is non-decreasing, the result is obtained by the same argument but using, instead, the classical global doubling inequality of Corollary \ref{cor:classicaldoubling}.
\end{proof}

\begin{remark}\label{rmk:totallybdduniform}
We notice that as a by-product of the proof of Corollary \ref{cor:totallyboundedclassic}, any $\sfG$-bounded metric ball $B_L(x_0)$ with $x_0\in\supp\mm$ can be covered by a finite number of balls of radius $\varepsilon/L$, and such a number depends only on $\varepsilon$ and the $\sfG$-diameter of $B_{3L}(x_0)$.
\end{remark}

\subsection{Geodesic dimension estimates}\label{Sec:Dim}

The parameter $N\in [1,+\infty)$ occurring in the definition of $\beta$ as in \eqref{eq:fftcN} turns out to be a sharp upper bound for a different notion of dimension, called \emph{geodesic dimension}. The latter was introduced in \cite{ABR-curvature} for sub-Riemannian manifolds, and extended to metric measure spaces in \cite{R-MCP}. We recall its definition.

\begin{definition}[Geodesic dimension]\label{def:geodim}
Let $(X,\sfd,\mm)$ be a metric measure space. For any $x\in X$ and $s>0$, define
\begin{equation}\label{eq:defCsx}
C_s(x):=\sup\left\{ \limsup_{t\to 0^+} \frac{1}{t^s} \frac{\bar{\mm}(A_t)}{\mm(A_1)} \mid \text{$A_1$ Borel,\, bounded, $x\notin A_{1}$,\,  $\mm(A_1)\in (0,+\infty)$} \right\},
\end{equation}
where $A_t$ is the set of $t$-intermediate points between $A_0 = \{x\}$ and $A_1$. We define the \emph{geodesic dimension} at $x\in X$ of $(X,\sfd,\mm)$ as the non-negative real number
\begin{equation}\label{eq:defNx}
\NN(x):=\inf \{s>0 \mid C_s(x) = +\infty \} = \sup\{s>0\mid C_s(x) =0 \}\, ,
\end{equation}
with the convention $\inf\emptyset = +\infty$ and $\sup\emptyset =0$. Finally, for any subset $S\subset X$, we denote
\begin{equation}
\NN(S):=\sup\{\NN(x)\mid x \in S\}.
\end{equation}
\end{definition}

We remark that if $x\notin \supp\mm$ we have $\NN(x) = +\infty$ so it makes sense to consider only the geodesic dimension at points in the support of $\mm$.

\begin{lemma}\label{l:geodgedhaus}
Let $(X,\sfd,\mm)$ be a m.m.s. Then for any Borel set $S\subset X$ it holds
\begin{equation}
\NN(S) \geq \NH(S).
\end{equation}
\end{lemma}
\begin{proof}

We first observe that if there exists $z\in S$ with $z\notin\supp\mm$, we have $\NN(z) = +\infty$, so that there is nothing to prove. Thus, we assume $S\subseteq \supp\mm$. 

Let $S_0\subseteq S$ be any Borel subset of $S$ with $\mm(S_0)<+\infty$. We claim that there exists $x\in S_0$ such that $\NN(x)\geq \NH(S_0)$. If $\NH(S_0)=0$ the claim is true, so assume that $\NH(S_0)>0$. If $k < \NH(S_0)$, then for the $k$-dimensional Hausdorff measure it holds $\mathcal{H}^k(S_0) = +\infty$, so that by using \cite[Thm.\@ 2.4.3]{AT-measure} there exists $x\in S_0$ such that
\begin{equation}
\limsup_{r\to 0^+} \frac{\mm(B_r(x))}{r^k}<+\infty.
\end{equation}

Let $A_1$ be a bounded Borel set  not containing $x$ with $\mm(A_1)\subset (0,+\infty)$, and consider the corresponding set of $t$-midpoints between $A_1$ and $A_0 = \{x\}$. We have $A_1\subseteq B_R(x)$ for some large $R>0$, so that $A_t \subseteq B_{tR}(x)$. Hence for all $\varepsilon>0$ it holds
\begin{equation}
\limsup_{t\to 0^+} \frac{1}{t^{k-\varepsilon}}\frac{\bar{\mm}(A_t)}{\mm(A_1)} \leq \limsup_{t\to 0^+} \frac{1}{t^{k-\varepsilon}}\frac{\mm(B_{tR}(x))}{\mm(A_1)} =0, \qquad \forall\,\varepsilon>0.
\end{equation}
We deduce that $C_{k-\varepsilon}(x) =0$ for all $\varepsilon>0$ (cf.\@ Definition \ref{def:geodim}), and hence $\NN(x) \geq \NH(S_0)$. In particular, for any Borel subset $S_0\subseteq S$ with $\mm(S_0)<+\infty$ it holds
\begin{equation}
\NN(S) \geq \NH(S_0).
\end{equation}

Next, by $\sigma$-finiteness, $X=\cup_{j\in \N} B_j$ for a countable family of Borel sets $(B_j)_{j\in \N}$ with $\mm(B_j)<+\infty$. Hence, setting $S_j:=S\cap B_j$, we have $\NN(S)\geq \NH(S_j)$ for all $j\in \N$. Finally, we obtain
\begin{equation}
\NH(S) = \NH\left(\bigcup_{j\in \N}S_j\right)  = \sup_{j\in\N} \NH(S_j) \leq \NN(S),
\end{equation}
concluding the proof.
\end{proof}

The next result establishes a (sharp) geodesic dimension estimate. We will assume the following \emph{non-triviality condition}:
\begin{equation}\label{eq:GNonTriv}
\mm(\{z\in X\mid z\neq x, \,  \sfG(x, z)<+\infty\})>0, \qquad \forall\, x\in   \supp\mm.
\end{equation}
In other words, $\supp\mm$ is not a singleton and $\sfG(x,\cdot)$ does not take the value $+\infty$ almost everywhere, for all $x\in \supp\mm$. Notice that \eqref{eq:GNonTriv} is implied by the $\MCP(\beta)$ provided that $\beta_t(+\infty) = +\infty$, see Remark \ref{rmk:MCPimpliesGNonTriv}.

\begin{theorem}[Geodesic dimension estimate]\label{thm:GeodDimEst}
Let $(X,\sfd,\mm)$ be a m.m.s.\@ with gauge function $\sfG$ satisfying the non-triviality condition \eqref{eq:GNonTriv}. Assume that $(X,\sfd,\mm, \sfG)$ satisfies the $\MCP(\beta)$ condition, with $\beta$ as in \eqref{eq:defbeta}. Recall that $N \in [1,+\infty)$ is the parameter given in \eqref{eq:fftcN}. Then the geodesic dimension of $(X,\sfd,\mm)$ satisfies
\begin{equation}\label{eq:estimate-geod}
\NN(x) \leq N, \qquad \forall\, x\in \supp\mm.
\end{equation}
In particular it holds for the Hausdorff dimension:
\begin{equation}\label{eq:estimate-haus}
\NH(\supp \mm) \leq  \NN(\supp\mm) \leq N.
\end{equation}
\end{theorem}
\begin{proof}
Thanks to Lemma \ref{l:geodgedhaus}, estimate \eqref{eq:estimate-geod} implies \eqref{eq:estimate-haus}, so it is sufficient to prove the former. We use the half Brunn-Minkowski inequality (Theorem \ref{thm:half-Brunn-Mink}) implied by the $\MCP(\beta)$. 
Fix $x\in \supp\mm$. Using that $\mm$ is locally finite and the non triviality \eqref{eq:GNonTriv} of the gauge function, it is immediate to see that there exists a bounded Borel set $A_1\subset X$,  with $\mm(A_1) \in (0,+\infty)$, and disjoint from $x$, satisfying
\begin{equation}\label{eq:supGfinite}
\sup_{y\in A_1} \sfG(x,y)<\infty.
\end{equation}
By the half Brunn-Minkowski inequality (Theorem \ref{thm:half-Brunn-Mink}), it holds
\begin{equation}
\frac{\bar{\mm}(A_t)}{\mm(A_1)}\geq  \mm\shortminus\essinf_{y\in A_1} \,  \beta_t(\sfG(x,y)), \qquad \forall\, t \in (0,1).
\end{equation}
Let $(\theta_j)_{j\in \N}$ with $\theta_j \in [0,\sup_{y\in A_1}\sfG(x,y)]$ be such that
\begin{equation}
\lim_{j\to\infty} \beta_t(\theta_j) =  \mm\shortminus\essinf_{y\in A_1} \,\beta_t(\sfG(x,y)).
\end{equation}
Assume first that $\theta_j$ is definitely bounded away from zero, so that we can extract a subsequence such that $\lim_{j\to\infty}\theta_j = \theta_\infty \in  (0,+\infty)$, where the finiteness of $\theta_\infty$ follows from \eqref{eq:supGfinite}. Since $\beta_t$ is continuous in this range (cf.\@ Proposition \ref{prop:propertiesbeta}), the essential infimum above is achieved at $\beta_t(\theta_\infty)$. If $\theta_\infty \in (0,\cD)$ then, by the properties \eqref{eq:fftcN} and \eqref{eq:defbeta}, we infer that
\begin{equation}
\limsup_{t\to 0^+} \frac{1}{t^{N+\ve}} \frac{\bar{\mm}(A_t)}{\mm(A_1)} \geq  \limsup_{t\to 0^+} \frac{1}{t^{N+\ve}} \beta_t(\theta_\infty) = \limsup_{t\to 0^+} \frac{1}{t^{N+\ve}} \frac{\sfs(t\theta_\infty)}{\sfs(\theta_\infty)}=+\infty.
\end{equation}
If $\theta_{\infty} \in [\cD,+\infty)$, then  $\beta_t(\theta_\infty) =+\infty$  by \eqref{eq:defbeta}, so we have a similar conclusion.

The remaining case is when $(\theta_j)_{j\in\N}$ admits a subsequence converging to zero. So we assume that $\lim_{j\to \infty} \theta_j = 0$. In this case, by the properties \eqref{eq:fftcN} and  \eqref{eq:defbeta}, we have
\begin{equation}
\frac{1}{t^{N+\ve}} \frac{\bar{\mm}(A_t)}{\mm(A_1)} \geq  \lim_{j\to\infty} \frac{1}{t^{N+\ve}} \beta_t(\theta_j) = \lim_{j\to\infty} \frac{1}{t^{N+\ve}} \frac{\sfs(t\theta_j)}{\sfs(\theta_j)} =\lim_{j\to \infty}\frac{(t\theta_j)^N+o(t\theta_j)^N}{t^{N+\ve}\left(\theta_j^N+ o(\theta_j)^N\right)}= \frac{1}{t^{\ve}}.
\end{equation}
Taking the $\limsup$ for $t\to 0^+$ gives $+\infty$. Combining all the cases, we deduce that $C_{N+\varepsilon}(x) = +\infty$ for all $\varepsilon>0$ (cf.\@ Definition \ref{def:geodim}), that implies \eqref{eq:estimate-geod}.
\end{proof}

\begin{remark}\label{rmk:sharpnessgeodim}
The Heisenberg group $\mathbb{H}^{d}$, equipped with its Carnot-Carath\'eodory distance and a Haar measure, is an example where the geodesic dimension is strictly larger than the Hausdorff dimension (the former being equal to $2d+3$, while the latter is equal to $2d+2$, while the topological dimension is $2d+1$). This is true, more in general, for all sub-Riemannian structures that are not Riemannian. Furthermore, when equipped with a suitable gauge function (see Section \ref{sec:howtorecover2}), the estimate provided by Theorem \ref{thm:GeodDimEst} is attained.
\end{remark}

\subsection{Generalized Bishop-Gromov inequality}\label{s:GBGI}

Let $(X,\sfd,\mm)$ be a  metric measure space with gauge function $\sfG$.  Observe that the sublevel sets of a gauge function $\sfG(x_0,\cdot) \leq r$ may be unbounded. This occurs with  the natural gauge in the Heisenberg group (here, even $\{\sfG(x_0,\cdot) = 0\}$ is unbounded). For this reason we work with the intersection with standard balls. 

\begin{definition}\label{def:gaugesets}
For a point $x_{0}\in \supp \mm$ and for $r,\rho\geq 0$, we set:
\begin{align}
\rv_{\sfG}(x_0,r,\rho)& :=\mm\left(\{x\in X \mid \sfG(x_0,x)\leq r,\;  \sfd(x_0,x)\leq \rho\}\right), \\
\rs_{\sfG}(x_0,r,\rho)& :=\limsup_{\delta\downarrow 0}\frac{1}{\delta} \mm\left(\{x\in X\mid \sfG(x_0,x)\in (r-\delta,r],\; \sfd(x_0,x)\leq \rho\}\right), \quad r\neq 0.
\end{align}
Consider also the following measure of ``gauge balls'' and ``gauge spheres'':
\begin{align}
\rv_{\sfG}(x_0,r) & :=  \mm\left( \{ x\in X \mid \sfG(x_0,x)\leq r,\; \sfd(x_0,x)\leq r\}\right), \label{eq:defvGr} \\
\rs_{\sfG}(x_0,r) & := \limsup_{\delta\downarrow 0}\frac{1}{\delta}\left[\rv_{\sfG}(x_0,r)-\rv_{\sfG}(x_0,r-\delta)\right], \quad r\neq 0. \label{eq:defsGr}
\end{align}
Notice also that $\rv_{\sfG}(x_0,r) = \rv_{\sfG}(x_0,r,r)$ but $\rs_{\sfG}(x_0,r) \neq \rs_{\sfG}(x_0,r,r)$. 
\end{definition}
In Figure \ref{fig:gaugevsball} we sketch these sets for a natural gauge function in the Heisenberg group. 

\begin{figure}
\centering
\includegraphics[width=.95\textwidth]{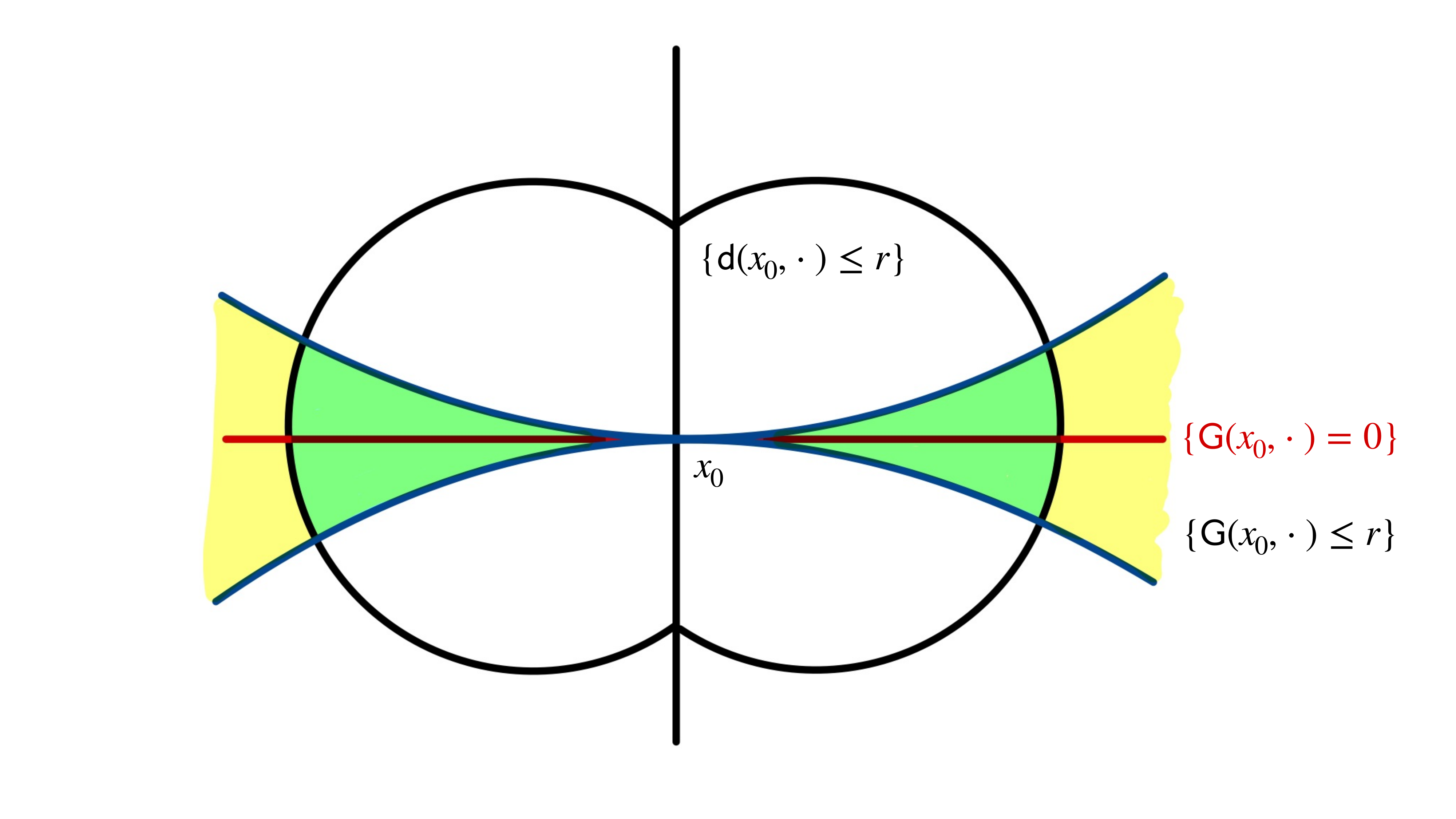}
\caption{Typical gauge sets in the Heisenberg group for the natural choice of gauge function described in Section \ref{sec:howtorecover2}. The green region corresponds to the intersection of the ball $\{\sfd(x_{0},\cdot)\leq r\}$ and the yellow one $\{\sfG(x_{0},\cdot)\leq r\}$. Notice the red line which represents the set $\{\sfG(x_{0},\cdot)=0\}$.
}\label{fig:gaugevsball}
\end{figure}

\begin{definition}[Meek]\label{def:meek}
Let $(X,\sfd,\mm)$ be a m.m.s. The gauge function $\sfG : X \times X \to [0,+\infty]$ is \emph{meek} if for all $\mu_1 \in  \Prob_{bs}^{*}(X,\sfd,\mm)$, $\bar{x}\in\supp\mm$ with $\bar{x}\notin \supp \, \mu_{1}$, and any optimal dynamical plan $\nu \in \OptGeo(\delta_{\bar{x}},\mu_1)$, there exists a Borel set $\Gamma \subset \Geo(X)$ with $\nu(\Gamma)=1$ such that it holds:
\begin{equation} \label{eq:meek1}
\sfG(\gamma_0,\gamma_t) = t\,  \sfG(\gamma_0,\gamma_1), \qquad \forall\, t\in (0,1], \quad \forall\, \gamma \in \Gamma.
\end{equation}
\end{definition}
\begin{remark}\label{rmk:weakmeek}
The meek property is used in the proof of Theorem \ref{thm:BG1} below for gauge m.m.s.\@ satisfying the $\MCP(\beta)$. There, it can be replaced by a weaker assumption where condition \eqref{eq:meek1} is required to hold for one $\nu \in \OptGeo(\delta_{\bar{x}},\mu_1)$ for which the $\MCP(\beta)$ inequality \eqref{eq:defMCPbeta} holds, instead than for \emph{every} $\nu$.  This variant is stable under the convergence of gauge metric measure spaces studied in Section~\ref{sec:stabilityandcompactness}, see proof of item \ref{i:ComMCP3} of Theorem~\ref{thm:precompactMCP}. We refrain from stating Definition \ref{def:meek} in this weaker sense since it would make it dependent on the $\MCP(\beta)$ condition.
\end{remark}

The meek property is satisfied by a large class of natural gauge functions in sub-Riemannian geometry, see  Section \ref{subsec:meek} (in particular, Theorem \ref{t:meeknatural} and Remark \ref{rem:meekexamples})

\begin{theorem}[Generalized Bishop-Gromov] \label{thm:BG1}
Let $(X,\sfd,\mm)$ be a m.m.s.\@ endowed with a gauge function $\sfG$ satisfying the $\MCP(\beta)$ condition, with $\beta$ as in \eqref{eq:defbeta}.  Assume that
\begin{itemize}
\item $\supp \mm$ is not a singleton;
\item $\sfG$ is meek;  (see Remark \ref{rmk:weakmeek} for a slightly weaker condition)
\item for all $x\in\supp\mm$ it holds (see Remark \ref{rem:onCond3.15} for a slightly weaker condition)
\begin{equation}\label{eq:Gfinite}
\mm(\{z\in X\mid  z\neq x,\, \sfG(x, z)=+\infty\})=0.
\end{equation}
\end{itemize}
Then for every $x_0\in\supp\mm$ the following properties hold:
\begin{enumerate}[label = (\roman*)]
\item \label{i:BG1} $\mm(\{x \mid \sfd(x_0,x) = r\})=0$, for every $r >0 $;
\item \label{i:BG2} $\mm(\{x_0\})=0$; in particular, the half-Brunn-Minkowski inequality \eqref{eq:halfBrunMink} holds for every $\bar{x}\in \supp\mm$ and Borel set $A\subset X$ with $\mm(A)>0$;
\item \label{i:BG3} $\mm(\{x \mid \sfG(x_0,x) = r\})=0$, for every $r >0 $;
\item \label{i:BG4} For every $\rho>0$, the functions
\begin{multline}\label{eq:BisGr}
(0,\cD)\ni r \mapsto \frac{\rs_{\sfG}(x_{0}, r,\rho)}{\sfs(r)/r}  \qquad \text{and} \\ (0, \cD]\cap (0,+\infty)\ni r \mapsto\frac{\rv_{\sfG}(x_{0}, r,\rho)-\rv_{\sfG}(x_{0}, 0,\rho)}{\int_{0}^{r} (\sfs(t)/t) \, \di t},
\end{multline}
are monotone non-increasing. Here $\sfs$ is the function defining $\beta$ as in \eqref{eq:defbeta};
\item \label{i:BG5} Assume that $\theta\mapsto \beta_t(\theta)$ is monotone non-increasing for all $t\in [0,1]$, or that for all $x_0\in\supp\mm$ it holds $\sfG(x_0,\cdot) \geq \sfd(x_0,\cdot)$ $\mm$-a.e.. Then the functions
\begin{equation}\label{eq:BisGr2}
(0,\cD)\ni r \mapsto \frac{\rs_{\sfG}(x_{0}, r)}{\sfs(r)/r}  \qquad \text{and} \qquad (0, \cD]\cap (0,+\infty)\ni r \mapsto\frac{\rv_{\sfG}(x_{0}, r)}{\int_{0}^{r} (\sfs(t)/t) \, \di t}
\end{equation}
are monotone non-increasing.
\item  \label{i:BG6} For every $R_0>0$ there exists a constant
\begin{equation}
\frac{1}{C_{r/R,0}}:=\inf\{\beta_{r/R}(\theta)\mid \theta \in [0,R_0]\} \in (0,(r/R)^N],
\end{equation}
such that it holds
\begin{equation}
 \rs_{\sfG}(x_0,R) \leq \frac{r}{R} C_{r/R,0}\cdot \rs_{\sfG}(x_0,r), \qquad \forall\, 0<r\leq R \leq R_0.
\end{equation}
As a consequence, the following doubling estimate for gauge balls holds: there exists a constant $C_0 = C_{1/2,0}>0$ such that
\begin{equation}
\rv_{\sfG}(x_0,2r)\leq C_0 \cdot \rv_{\sfG}(x_0,r), \qquad \forall\, r \in [0,R_0/2].
\end{equation}
\end{enumerate}
\end{theorem}

\begin{remark}\label{rem:onCond3.15}
For the proof of Thorem \ref{thm:BG1} condition \eqref{eq:Gfinite} can be weakened as follows: for all $x\in\supp\mm$  and all $r>0$ it holds
\begin{equation}\label{eq:Gfinite3}
\mm(\{z\in X\mid   \sfd(x,z)=r,\, \sfG(x, z)=+\infty\})=0.
\end{equation}
Both conditions \eqref{eq:Gfinite} and  \eqref{eq:Gfinite3} follow from the $\MCP(\beta)$, provided that $\beta_t(+\infty) = +\infty$, see Remark \ref{rmk:MCPimpliesGNonTriv}.
\end{remark}
\begin{remark}
Item \ref{i:BG4} can be seen as a parametric version of the Bishop-Gromov inequality (the parameter being the ``scale'' $\rho$). The Bishop-Gromov estimate for ``gauge balls'' and ``spheres'' as in item \ref{i:BG5} holds under some additional assumptions.
\end{remark}
\begin{remark}
We are not aware of any situation in which $\rv_{\sfG}(x_0,0,\rho)$ does not vanish, but it does not seem to be a consequence of the $\MCP$ condition. Thus it seems that the extra term in item \ref{i:BG4} cannot be in general removed.
\end{remark}

\begin{proof} Fix $x_0\in\supp\mm$ throughout the proof.

\textbf{Proof of \ref{i:BG1}.} Let $\delta,\rho>0$, and $0<r-\delta r<r<R-\delta R<R<+\infty$. Let
\begin{equation}
E:=\{ x\mid \sfG(x_0,x) \leq \rho, \;\sfd(x_0,x) \in (R-R\delta, R]  \},
\end{equation}
and note that $E$ is Borel. Assume $\mm(E)>0$. Let $\mu_0=\delta_{x_0}$ and $\mu_1 = \mm(E)^{-1}\mm\llcorner_{E}$. Let $\nu\in \OptGeo(\mu_0,\mu_1)$ be such that the $W_2$-geodesic $\mu_t = (\ee_t)_\sharp \nu$ satisfies the $\MCP(\beta)$ inequality. Let $\Gamma\subseteq \Geo(X)$ be the corresponding set of geodesics in the meek assumption, for which $\nu(\Gamma)=1$. Observe that $\supp\nu \subseteq \Gamma_E:= \ee_1^{-1}(E)\cap \ee_0^{-1}(x_0)$ so that, letting $\Gamma'= \Gamma\cap \Gamma_E$, it holds that $\nu(\Gamma')=1$ and the measure $\mu_t$ is concentrated on $\ee_t(\Gamma')$. Since $\Gamma'$ is Borel, $\ee_t(\Gamma')$ is Suslin. Then, by Theorem \ref{thm:half-Brunn-Mink} and Remark \ref{rmk:concentratedimprove}, we have
\begin{align}
\bar{\mm}(\ee_{r/R}(\Gamma'))  \geq \beta_{r/R}(x_0,E) \cdot \mm(E) \geq C_{r/R,\rho} \cdot \mm(E),
\end{align}
where we set
\begin{equation}
C_{r/R,\rho}:=\inf\{\beta_{r/R}(\theta)\mid \theta \in [0,\rho]\}>0.
\end{equation}
The same inequality holds true trivially also if $\mm(E)=0$. An important remark is that any $\nu \in \OptGeo(\mu_0,\mu_1)$ is supported on a set of non-trivial geodesics, since by construction $x_0\notin E$ and $(\ee_0,\ee_1)_\sharp \nu = \delta_{x_0}\otimes \mu_1$, so that all geodesics in $\Gamma'$ are non-trivial. Then,  by the meek assumption, for all $\gamma \in \Gamma'$ it holds
\begin{align}
\sfG(x_0,\gamma_{r/R}) & = \frac{r}{R}\sfG(\gamma_0,\gamma_1) \in [0,r\rho/R], \\
\sfd(x_0,\gamma_{r/R}) & = \frac{r}{R}\sfd(\gamma_0,\gamma_1) \in (r-\delta r,r].
\end{align}
In other words, we have
\begin{equation}
\bar{\mm}\left(\ee_{r/R}(\Gamma') \setminus \{x \mid \sfG(x_0,x) \leq r\rho/R,\; \sfd(x_0,x) \in (r- \delta r,r] \}\right)=0.
\end{equation}
Since $\rho > r\rho/R$, for the given range of parameters, we have:
\begin{multline}\label{eq:multi0}
\mm\left(\{x \mid \sfG(x_0,x) \leq \rho,\; \sfd(x_0,x) \in (r- \delta r,r]\}\right) \geq  \\
C_{r/R,\rho}  \cdot \mm\left(\{x\mid\sfG(x_0,x)\leq\rho,\;  \sfd(x_0,x) \in (R-\delta R, R]\}\right),
\end{multline}
where we noted that the set in the left hand side of \eqref{eq:multi0} is $\mm$-measurable. Thus for fixed $\rho \in (0,+\infty)$ the function
\begin{equation}\label{eq:342}
(0,+\infty) \ni r \mapsto \rv_{\sfG}(x_0,\rho,r) = \mm\left(\{x\mid \sfG(x_0,x) \leq \rho,\; \sfd(x_0,x) \leq r \}\right),
\end{equation}
by \eqref{eq:multi0} satisfies
\begin{equation}\label{eq:multi00}
\rv_{\sfG}(x_0,\rho,r)-\rv_{\sfG}(x_0,\rho,r-\delta r) \geq C_{r/R,\rho}\cdot \left[\rv_{\sfG}(x_0,\rho,R)-\rv_{\sfG}(x_0,\rho,R-\delta R)\right].
\end{equation}
Recall that, by standing assumption, $\mm$ is finite on bounded sets, thus
\begin{equation}
\rv_{\sfG}(x_0,\rho,r) < +\infty,\quad \text{ for any $r,\rho>0$.}
\end{equation}
Furthermore, for fixed $\rho\in (0,+\infty)$, the function $\rv_{\sfG}(x_0,\rho,\cdot)$ is right-continuous, non-decreasing, with at most countably many discontinuities. In particular there exists arbitrarily small $r_0>0$ for which $\rv_{\sfG}(x_0,\rho,\cdot)$ is continuous at $r_0$. Choosing $r=r_0$ and sending $\delta \downarrow 0$ in \eqref{eq:multi00}, we obtain that $\rv_{\sfG}(x_0,\rho,\cdot)$ is also left-continuous on $(0,+\infty)$ and thus continuous. In particular, 
\begin{equation}\label{eq:vanish}
\mm\{x \mid \sfG(x_0,x)\leq \rho,\; \sfd(x_0,x) = r\}=0, \qquad \text{for any $r,\rho>0$}.
\end{equation}
It follows that for any $r>0$
\begin{multline}
\mm\left(\{x\mid \sfd(x_{0}, x)= r\}\right) =  \mm\left(\bigcup_{n\in \N}\{x\mid \sfG(x_0,x)\leq n,\; \sfd(x_{0}, x)= r\}\right) \\ + \mm\left(\{x\mid \sfG(x_0,x)=\infty,\; \sfd(x_{0}, x)= r\}\right) = 0.
\end{multline}
The first term vanishes by \eqref{eq:vanish}, while the second one vanishes by \eqref{eq:Gfinite}.

\textbf{Proof of \ref{i:BG2}.} By the previous item, and since $\supp\mm$ is not a singleton, any $x_0 \in \supp\mm$ is on a sphere of positive radius of another point in $\supp\mm$, hence $\mm(\{x_0\})=0$.

\textbf{Proof of \ref{i:BG3}.} Let $\ve, \delta,\rho>0$ and $0<r(1-\delta)<r<R<R(1+\delta) <\cD$.

The argument is similar to \ref{i:BG1} exchanging the roles of $\sfG$ and $\sfd$, with the important difference that while the \emph{metric} annulus $\{\sfd(x_0, \cdot)\in (r,R)\}$ does not contain $x_0$, the \emph{gauge} annulus $\{\sfG(x_0, \cdot)\in (r,R)\}$ may contain $x_0$ (see the typical situation pictured in Figure \ref{fig:gaugevsball}); for this reason we will argue by approximation by cutting off a small metric ball $B_{\ve}(x_0)$ out of the gauge annulus, perform the transport argument using the $\MCP(\beta)$ and the meek conditions, and then finally let $\ve\downarrow 0$.
 Consider the measurable set
\begin{equation}
 C_{\ve} := \{x  \mid \sfG(x_0,x) \in (R-\delta R,R],\; \sfd(x_0,x)\in [\ve, \rho] \}.
\end{equation}
Assume $\mm(C_{\ve})>0$, and set $\mu_1 = \mm( C_{\ve})^{-1}\mm\llcorner_{C_{\ve}}$, $\mu_0=\delta_{x_0}$. As in the proof of \ref{i:BG1}, we find a $W_2$-geodesic $(\mu_t)_{t\in [0,1]}$ for which the $\MCP(\beta)$ inequality holds, and a Borel set  $ \Gamma \subset \Geo(X)$ of non-trivial geodesics verifying the meek condition \eqref{eq:meek1}, $(\ee_0,\ee_1)(\Gamma) = \{x_0\}\times  C_{\ve}$, and $\mu_t$ is concentrated on $\ee_t(\Gamma)$ for all $t\in [0,1]$. Then by Theorem \ref{thm:half-Brunn-Mink}, and Remark \ref{rmk:concentratedimprove} we have
\begin{align}
\bar{\mm}(\ee_{r/R}(\Gamma)) & \geq \beta_{r/R}(x_0, C_{\ve}) \cdot \mm( C_{\ve}),
\end{align}
and the same inequality holds trivially if $\mm( C_{\ve})=0$.

 By the meek condition, for all $\gamma\in \Gamma$ we have 
\begin{align}
\sfG(x_0,\gamma_{r/R})  = \frac{r}{R}\sfG(\gamma_0,\gamma_1) \in (r-\delta r,r], \qquad \sfd(x_0,\gamma_{r/R})  = \frac{r}{R}\sfd(\gamma_0,\gamma_1) \in [ \ve r/R,\rho r/R].
\end{align}
In other words, we have
\begin{equation}
\bar{\mm}\left(\ee_{r/R}(\Gamma) \setminus \{x \mid \sfG(x_0,x) \in (r-\delta r,r],\; \sfd(x_0,x)  \in [ \ve r/R,\rho r/R]\}\right) =0.
\end{equation}
So that,  for all $0<r<R<\cD$ and $ \ve,\delta>0$ sufficiently small, we have
\begin{multline}\label{eq:multi1}
\mm\left(\{x\mid \sfG(x_0,x) \in (r-\delta r,r],\;  \sfd(x_0,x)  \in [ \ve r/R,\rho r/R]\}\right) \geq  \\
\beta_{r/R}(x_0,C_{\ve}) \cdot \mm\left(\{x\mid \sfG(x_0,x) \in (R-\delta R,R],\;  \sfd(x_0,x)\in [\ve, \rho]\}\right),
\end{multline}
where we noted that the set $\{x\mid \sfG(x_0,x) \in (r-\delta r,r]\}$ is measurable. We also remark that by definition of $\beta$ and $C_{\ve}$, for all $0<r<R<\cD$ it holds:
\begin{align}\label{eq:LBbetaCeps}
\beta_{r/R}(x_0,C_{\ve}) \geq \inf \{\beta_{r/R}(\theta)\mid \theta \in (R-\delta R,R]\} \in (0,+\infty).
\end{align}
 Letting  $\ve\downarrow 0$ in \eqref{eq:multi1}, using \eqref{eq:LBbetaCeps} and that $\mm(\{x_{0}\})=0$ by  \ref{i:BG2}, we infer that
\begin{multline}\label{eq:multi1bis}
\mm\left(\{x\mid \sfG(x_0,x) \in (r-\delta r,r],\;  \sfd(x_0,x)  \leq \rho r/R\}\right) \geq  \\
\inf_{\theta \in (R-\delta R,R]} \beta_{r/R}(\theta) \cdot \mm\left(\{x\mid \sfG(x_0,x) \in (R-\delta R,R],\;  \sfd(x_0,x) \leq \rho \}\right).
\end{multline}
Finally, since $\rho>  \rho r/R$, we obtain from \eqref{eq:multi1bis}
\begin{multline}\label{eq:multi3}
\rv_{\sfG}(x_0,r,\rho) -\rv_{\sfG}(x_0,r-\delta r,\rho) \geq \\
\inf_{\theta \in (R-\delta R,R]} \beta_{r/R}(\theta) \cdot \left[\rv_{\sfG}(x_0,R,\rho) - \rv_{\sfG}(x_0,R-\delta R,\rho)\right].
\end{multline}

Since $\mm$ is finite on bounded sets, it holds $\rv_{\sfG}(x_0,R,\rho)<+\infty$  for any $\rho>0$ and $R \in (0,\cD)$. Furthermore, by construction, $\rv_{\sfG}(x_0,\cdot,\rho)$ is right-con\-tinuous, non-decreasing, with at most countably many discontinuities. In particular, there exists arbitrarily small $r \in (0,\cD)$ for which $\rv_{\sfG}(x_0,\cdot,\rho)$ is continuous at $r$. Since also $\beta_{r/R}(\cdot)$ is continuous on $(0,\cD)$, we obtain from \eqref{eq:multi3} that $\rv_{\sfG}(x_0,\cdot,\rho)$ is left continuous and thus continuous on $(0,\cD)$. Note that the argument requires $R>0$.

It follows that $\mm\left(\{x\mid \sfG(x_{0}, x)= R,\;\sfd(x_0,x)\leq \rho\} \right)=0$ for every $R\in (0,\cD)$ and $\rho>0$, and thus $\mm(\{x\mid \sfG(x_{0}, x)= R \}) =0$.  If $\cD = +\infty$, this concludes the proof of \ref{i:BG3}.  If $\cD<+\infty$, then we invoke Proposition \ref{prop:Gdiamestimate}, so that also in this case the equality keeps holding also for $R\geq \cD$.

\textbf{Proof of \ref{i:BG4}.} From \eqref{eq:multi3}, we obtain:
\begin{multline}\label{eq:PfBG3}
\frac{\rv_{\sfG}(x_0,r,\rho)-\rv_{\sfG}(x_0,r-\delta r,\rho)}{\delta r} \\
\geq \frac{R}{r}\left( \inf_{\theta \in (R-\delta R,R]} \beta_{r/R}(\theta)\right)\cdot \frac{\rv_{\sfG}(x_0,R,\rho)-\rv_{\sfG}(x_0,R-\delta R,\rho)}{\delta R}.
\end{multline}
Taking the $\limsup$ for $\delta \downarrow 0$, and using the continuity of $\beta_{r/R}(\cdot)$ on $(0,\cD)$, we obtain:
\begin{align}
\rs_{\sfG}(x_0,r,\rho)& =\limsup_{\delta \downarrow 0} \frac{\rv_{\sfG}(x_0,r,\rho)-\rv_{\sfG}(x_0,r-\delta r,\rho)}{\delta r} \\
& \geq \frac{R}{r} \limsup_{\delta \downarrow 0} \left(\inf_{\theta \in (R-\delta R,R]} \beta_{r/R}(\theta)\right) \cdot \rs_{\sfG}(x_0,R,\rho) \\
& = \frac{R}{r} \beta_{r/R}(R) \cdot \rs_{\sfG}(x_0,R,\rho).
\end{align}
Using the fact that $\beta_t(\theta) = \frac{\sfs(t\theta)}{\sfs(\theta)}$ for $\theta \in (0,\cD)$ and $t\in [0,1]$, from \eqref{eq:defbeta} we infer the left inequality in \eqref{eq:BisGr}.

To prove its integrated counterpart, we claim that the map $(0,\cD)\ni r\mapsto \rv_{\sfG}(x_{0}, r,\rho)$ is locally Lipschitz,  for fixed $\rho>0$. If by contradiction it were not locally Lipschitz, then there would exist $R>0$ and a sequence $h_{n}\downarrow 0$ such that at least one of the two conditions hold
\begin{align}
&\rv_{\sfG}(x_{0}, R,\rho) -  \rv_{\sfG}(x_{0}, R-h_{n},\rho) \geq n h_{n}, \qquad \forall\, n\in \N, \text{ or } \label{eq:contLip}\\
& \rv_{\sfG}(x_{0}, R+h_{n},\rho) -  \rv_{\sfG}(x_{0}, R,\rho)  \geq n h_{n}, \qquad \forall\, n\in \N. \nonumber
\end{align}
We assume that \eqref{eq:contLip} holds. The argument for the second case is analogous.

For every fixed $n\in \N$, observing that $c:=\inf\{\beta_{t}(\theta) \mid t\in [1/2, 1],\, \theta\in [0,R]\}>0$, the combination of \eqref{eq:contLip} and \eqref{eq:PfBG3}  gives that
\begin{equation}
\rv_{\sfG}(x_{0}, R- (j-1) h_{n},\rho)- \rv_{\sfG}(x_{0}, R- j h_{n},\rho)\geq c  n h_{n}, \quad \forall\, n\in \N, \;\forall \,  j=1,\ldots, \left\lfloor \frac{R}{2 h_{n}} \right \rfloor.
\end{equation}
Summing up over $j$ for every fixed $n$, we obtain
\begin{align*}
\rv_{\sfG}(x_{0}, R,\rho) &\geq \rv_{\sfG}(x_{0}, R,\rho) - \rv_{\sfG}\left(x_{0}, R- \left\lfloor \frac{R}{2 h_{n}}  \right \rfloor h_{n},\rho\right)\\
&= \sum_{j=1}^{ \left\lfloor \frac{R}{2 h_{n}}  \right \rfloor} \rv_{\sfG}(x_{0}, R-(j-1) h_{n},\rho)-  \rv_{\sfG}(x_{0}, R- j h_{n},\rho) \\
&\geq c n h_{n} \left\lfloor \frac{R}{2 h_{n}}  \right \rfloor.
\end{align*}

The right hand side diverges to $+\infty$ as $n\to \infty$, contradicting that  bounded sets have finite measure. Thus the map $(0,\cD)\ni r\mapsto \rv_{\sfG}(x_{0}, r,\rho)$ is locally Lipschitz and finite.

It follows that the function  $r\mapsto \rv_{\sfG}(x_{0}, r,\rho)$ is differentiable a.e.\@ on $(0,\cD)$, and it coincides with the integral of its a.e.\@ derivative $r\mapsto \rs_{\sfG}(x_{0}, r,\rho)$. We can thus apply the classical Gromov's Lemma (see \cite[Lemma III.4.1]{Chav06})  to infer that the first claim in \eqref{eq:BisGr} implies the integrated version given in the second of \eqref{eq:BisGr} for the open interval $(0,\cD)$.

We finally show that the second in \eqref{eq:BisGr} holds on $(0,\cD]$, for $\cD<+\infty$. Since $\mm(\{ x\mid \sfG(x_{0}, x) =\cD \}) =0$ from Proposition \ref{prop:Gdiamestimate}, by monotone convergence we get
\begin{align*}
\rv_{\sfG}(x_{0}, \cD,\rho) - \rv_{\sfG}(x_{0}, 0,\rho)&=\lim_{R\uparrow \cD} \rv_{\sfG}(x_{0}, R,\rho) - \rv_{\sfG}(x_{0}, 0,\rho)  \\
&\leq \left[\rv_{\sfG}(x_{0}, r,\rho) -\rv_{\sfG}(x_{0}, 0,\rho)\right] \lim_{R\uparrow \cD} \frac{\int_{0}^{R} (\sfs(t)/t) \, \di t  }{\int_{0}^{r} (\sfs(t)/t) \, \di t }  \\
&  =\left[\rv_{\sfG}(x_{0}, r,\rho) -\rv_{\sfG}(x_{0}, 0,\rho)\right]   \frac{\int_{0}^{\cD} (\sfs(t)/t) \, \di t  }{\int_{0}^{r} (\sfs(t)/t) \, \di t }, \quad \forall\, r\in (0,\cD).
\end{align*}
This concludes the proof of \ref{i:BG4}.

\textbf{Proof of \ref{i:BG5}.} Recall that  $\rv_{\sfG}(x_0,r)$, $\rs_{\sfG}(x_0,r)$ were defined in  \eqref{eq:defvGr} and \eqref{eq:defsGr}, for $r\geq 0$.
Observe that, by item \ref{i:BG2}, it holds $\rv_{\sfG}(x_0,0)=0$. Moreover,
\begin{multline}
\rv_{\sfG}(x_0,r) -\rv_{\sfG}(x_0,r-\delta r) = \mm\left(\{x\mid\sfG(x_0,x) \in (r-\delta r,r],\; \sfd(x_0,x)\leq r\} \right)\\ +\mm\left(\{x\mid \sfG(x_0,x) \leq r-\delta r,\; \sfd(x_0,x) \in (r-\delta r,r]\} \right).
\end{multline}
Arguing as above we obtain for the two addends, for $0<r<R<\cD$ and sufficiently small $\delta>0$
\begin{multline}\label{eq:lastboundBG1}
\mm\left(\{x\mid\sfG(x_0,x) \in (r-\delta r,r],\; \sfd(x_0,x)\leq r\} \right) \geq \left( \inf_{\theta \in (R-\delta R,R]}\beta_{r/R}(\theta) \right)\\\cdot \mm\left(\{x\mid\sfG(x_0,x) \in (R-\delta R,R],\; \sfd(x_0,x)\leq R\} \right),
\end{multline}
\begin{multline}\label{eq:lastboundBG2}
\mm\left(\{x\mid \sfG(x_0,x) \leq r-\delta r,\; \sfd(x_0,x) \in (r-\delta r,r]\} \right)  \geq \left( \inf_{\theta \in [0,R-\delta R]}\beta_{r/R}(\theta) \right)\\\cdot \mm\left(\{x\mid \sfG(x_0,x) \leq R-\delta R,\; \sfd(x_0,x) \in (R-\delta R,R]\} \right).
\end{multline}
If $\theta \mapsto \beta_t(\theta)$ is non-increasing, then $ \inf_{\theta \in [0,R-\delta R]}\beta_{r/R}(\theta)  \geq \beta_{r/R}(R)$. In case $\sfG(x_0,x)\geq \sfd(x_0,x)$ for $\mm$-a.e.\@ $x\in X$, then both sides of \eqref{eq:lastboundBG2} vanish.  In both cases we obtain
\begin{equation}
\frac{\rv_{\sfG}(x_0,r) -\rv_{\sfG}(x_0,r-\delta r)}{\delta r}   \geq \frac{R}{r}\beta_{r/R}(R) \frac{\rv_{\sfG}(x_0,R) -\rv_{\sfG}(x_0,R-\delta R)}{\delta R}.
\end{equation}
We now can conclude arguing exactly as in item \ref{i:BG4}, concluding the proof of \ref{i:BG5}.

\textbf{Proof of \ref{i:BG6}.} We argue as in the proof of \ref{i:BG5}, but in \eqref{eq:lastboundBG1} and \eqref{eq:lastboundBG2} we bound from below the factors depending on the distortion coefficient by the minimum on $[0,R_0]$:
\begin{equation}
\frac{1}{C_{r/R,0}}:=\inf\{\beta_{r/R}(\theta)\mid \theta \in [0,R_0]\} \in (0,(r/R)^N],
\end{equation}
which is positive by Proposition \ref{prop:propertiesbeta} and satisfies the above upper bound by inspection of the explicit expression \eqref{eq:defbeta}. Therefore we obtain for $0<r\leq R\leq R_0$ and $\delta>0$:
\begin{equation}
\frac{\rv_{\sfG}(x_0,r) -\rv_{\sfG}(x_0,r-\delta r)}{\delta r}   \geq \frac{R}{r} \frac{1}{C_{r/R,0}}\cdot \frac{\rv_{\sfG}(x_0,R) -\rv_{\sfG}(x_0,R-\delta R)}{\delta R}.
\end{equation}
Taking the $\limsup$ for $\delta \downarrow 0$, we obtain the first inequality of \ref{i:BG6}:
\begin{equation}
 \rs_{\sfG}(x_0,R) \leq \frac{r}{R}C_{r/R,0}\cdot \rs_{\sfG}(x_0,r), \qquad \forall \, 0<r\leq R \leq R_0.
\end{equation}
As in the previous steps, we remark that $r\mapsto \rv_{\sfG}(x_0,r)$ is locally Lipschitz continuous for $r\in (0,\cD)$ and its a.e.\@ derivative is $r\mapsto\rs_{\sfG}(x_0,r)$. Integrating the above inequality, and since $\rv_{\sfG}(x_0,0)=0$, we obtain
\begin{equation}
\rv_{\sfG}(x_0,R) \leq C_{r/R,0}\cdot \rv_{\sfG}(x_0, r), \qquad \forall\, 0<r \leq R \leq R_0,
\end{equation}
concluding the proof.
\end{proof}


We can estimate the maximal number of pairwise disjoint gauge balls of radius $\varepsilon>0$ whose centres are contained in a bounded and $\sfG$-bounded set.
\begin{proposition}[Local total boundedness for gauge balls]\label{prop:number-of-butt}
With the same assumptions of Theorem \ref{thm:BG1}, let $S\subseteq \supp\mm$ be bounded and $\sfG$-bounded with $\mm(S)>0$. Then the maximal number of pairwise disjoint gauge balls of radius $\varepsilon>0$ (i.e.\@ sets $\{x\mid \sfG(x_0,x)\leq \ve,\, \sfd(x_0,x)\leq \ve \}$) contained in $S$ is bounded above by
\begin{equation}
k_0= C_0 \left(\frac{\max\{\diam(S),\diam_{\sfG}(S)\}}{\varepsilon}\right)^{\log_2 C_0},
\end{equation}
where $C_0>0$ can be estimated in terms of $\diam_{\sfG}(S)$ and $\beta$ defining the $\MCP(\beta)$ class.
\end{proposition}
\begin{proof}
Fix $\varepsilon>0$ and let $x_1,\dots,x_k \in S\subset \supp \mm$ be such that 
\begin{equation}
\{x\in X \mid \sfG(x_i,x)\leq \varepsilon,\;\sfd(x_i,x)\leq \ve \} \subseteq S
\end{equation}
are pairwise disjoint. Since $S$ is bounded, $\mm(S)<+\infty$. Let 
\begin{equation}
R: = \max\{\diam(S),\diam_{\sfG}(S)\}.
\end{equation}
Let $M\in \N$ be the smallest integer such that $2^M\varepsilon \geq R$. By the doubling property for gauge balls (cf. Theorem \ref{thm:BG1} item \ref{i:BG6}), we have
\begin{equation}\label{eq:sumfortotbdd}
\mm(S)  \geq \sum_{i=1}^k \rv_{\sfG}(x_i,\varepsilon)  \geq C_{0}^{-M} \sum_{i=1}^k \rv_{\sfG}(x_i, R),
\end{equation}
where
\begin{equation}
\frac{1}{C_{0}} = \inf \{\beta_{1/2}(\theta)\mid \theta \in [0,R]\}\in (0,1/2^N].
\end{equation}
By construction,  $\mm(S\setminus \{x\in X\, :\, \sfd(x_{i}, x)\leq R, \, \sfG(x_{i}, x)\leq R \})=0$,  so that the last sum in \eqref{eq:sumfortotbdd} can be estimated from below  by $k \mm(S)>0$. Continuing the chain of inequalities, and since $\mm(S)<+\infty$, we have
\begin{equation}
k\leq k_0:=C_{0}^{M} \leq C_{0}^{\log_{2}(R/\ve)+1}=  C_0 \left(\frac{R}{\varepsilon}\right)^{\log_{2}C_{0}},
\end{equation}
proving the claim.
\end{proof}

\begin{remark}
If $\sfG$ is symmetric and there exists $\alpha\geq 1$ such that 
\begin{equation}\label{eq:ngf}
\sfG(x,z)\leq \alpha(\sfG(x,y)+\sfG(y,z)),\qquad \forall\, x,y,z \in X,
\end{equation}
so in particular if $\sfG$ is a quasi-metric, then we can estimate the number of gauge balls needed to cover a bounded and $\sfG$-bounded set. Indeed the number of gauge balls of radius $2\alpha\varepsilon$ covering $S$ is bounded above by $k_0$. This is the case when $\sfG=\sfd$. However notice that \eqref{eq:ngf} is not satisfied by natural gauge functions in sub-Riemannian geometry.
\end{remark}
\subsection{Parametric doubling for gauge balls}

We close this section with the following doubling property, for intersection of sublevel sets of $\sfG$ with arbitrarily large balls of radius $\rho$.
 
\begin{proposition}[Parametric doubling]\label{prop:pardoubling}
With the same assumptions of Theorem \ref{thm:BG1}, for every $R>0$ there exists $C_R=C_R(\beta)\geq 1$ such that
\begin{equation}
\frac{\rv_{\sfG}(x, 2r,\rho)- \rv_{\sfG}(x, 0,\rho)}{\rv_{\sfG}(x, r,\rho)- \rv_{\sfG}(x, 0,\rho)}\leq  2^N C_R, \qquad \forall\, r\in (0,R),\quad\forall\, \rho>0, \quad \forall\, x\in \supp\mm,
\end{equation}
where $N>1$ is given by  \eqref{eq:fftcN}. Furthermore, $C_R \downarrow 1$ as $R\downarrow 0$.

If $\theta\mapsto \beta_t(\theta)$ is monotone non-increasing for all $t\in [0,1]$, or if for all $x_0\in\supp\mm$ it holds $\sfG(x_0,\cdot) \geq \sfd(x_0,\cdot)$ $\mm$-a.e., then the previous inequality holds true for gauge balls, with the same constants, that is
\begin{equation}
\frac{\rv_{\sfG}(x, 2r)}{\rv_{\sfG}(x, r)}\leq  2^N C_R, \qquad \forall\, r\in (0,R), \quad \forall\, x\in \supp\mm.
\end{equation}
\end{proposition}
\begin{proof}
Fix $\rho>0$. Assume first that $0<r<\mathcal{D}/2$. The second claim in \eqref{eq:BisGr} combined with \eqref{eq:fftcN} gives that
\[
\frac{\rv_{\sfG}(x, 2r,\rho) - \rv_{\sfG}(x, 0,\rho)}{\rv_{\sfG}(x, r,\rho)  - \rv_{\sfG}(x, 0,\rho) }\leq \frac{ \int_{0}^{2r} \sfs(t)/t \, \di t}{ \int_{0}^{r} \sfs(t)/t \, \di t }= \frac{ (2r)^{N}(1+ o(1))}{ r^{N}(1 + o(1))} \leq 2^N C_R.
\]
The constant $C_R$ depends only on the values of  $\sfs$ on the interval $[0,R]$, so that it does not depend on $\rho$.
If $\mathcal{D} =+\infty$, there is nothing left to prove since we covered already the whole range of $r\in (0,R)$. If $\mathcal{D} < +\infty$, notice that by Proposition \ref{prop:Gdiamestimate} and Theorem \ref{thm:BG1}\ref{i:BG2}, $\sfG(x,\cdot)< \cD(x,\cdot)$ $\mm$-almost everywhere, so that $\rv_{\sfG}(x, 2r,\rho)$ is constant in the range $r \geq \cD/2$.

If $\theta\mapsto \beta_t(\theta)$ is monotone non-increasing for all $t\in [0,1]$, or if for all $x_0\in\supp\mm$ it holds $\sfG(x_0,\cdot) \geq \sfd(x_0,\cdot)$ $\mm$-a.e., then we perform the same proof but starting from \eqref{eq:BisGr2} that holds in this case, instead of \eqref{eq:BisGr}.
\end{proof} 			
\section{Stability and compactness}\label{sec:stabilityandcompactness}

The aim of this section is to prove stability results for curvature-dimension conditions for sequences of m.m.s.\@ converging in the pmGH sense. First, we recall the definition of  pointed
measured Gromov Hausdorff convergence (pmGH for short; we follow the convention from \cite[Def. 4.1]{GMS} adapted from \cite[Def. 8.1.1]{Vil}).

\begin{definition}[pmGH convergence]\label{def:pmGHconv}

Let $(X_k,\sfd_k,\mm_k,\star_k)$, $k\in\bar\N$, be pointed metric measure spaces.
We say that $(X_k,\sfd_k,\mm_k, \star_k)$ converges to
$(X_\infty,\sfd_\infty,\mm_\infty, \star_\infty)$ in the pmGH sense if for any $\ve,R>0$ there exists $N({\ve,R})\in \N$ such that for all $k\geq N({\ve,R})$ there exists a Borel map $f^{R,\ve}_k:B_R(\star_k)\to X_\infty$ such that 
\begin{enumerate}[label = (\roman*)$^\prime$]
\item $f^{R,\ve}_k(\star_k)=\star_\infty$;
\item $\sup_{x,y\in B_R(\star_k)}|\sfd_k(x,y)-\sfd_\infty(f^{R,\ve}_k(x),f^{R,\ve}_k(y))|\leq\ve$;
\item the $\ve$-neighborhood of $f^{R,\ve}_k(B_R(\star_k))$ contains $B_{R-\ve}(\star_\infty)$;
\item  $(f^{R,\ve}_k)_\sharp(\mm_k\llcorner{B_R(\star_k)}) \rightharpoonup \mm_\infty\llcorner{B_R(\star_\infty)}$ as $k\to\infty$, for a.e. $R>0$ (weak convergence against bounded continuous functions with bounded support).
\end{enumerate}
\end{definition}

\begin{remark}\label{rem:pmghfkglob}
It is not hard to check (cf. \cite[Prop. 3.28]{GMS}) that the pmGH convergence of $(X_k,\sfd_k,\mm_k, \star_k)$ to $(X_\infty,\sfd_\infty,\mm_\infty, \star_\infty)$  is equivalent to the following condition: there are sequences $R_k\uparrow+\infty$, $\eps_k\downarrow 0$ and Borel maps $f_k:X_k\to X_\infty$ such that 
\begin{enumerate}[label = (\roman*)]
\item $f_k(\star_k)=\star_\infty$;
\item $\sup_{x,y\in B_{R_k}(\star_k)}|\sfd_k(x,y)-\sfd_\infty(f_k(x),f_k(y))|\leq\eps_k$;
\item the $\eps_k$-neighborhood of $f_k(B_{R_k}(\star_k))$ contains $B_{R_k-\eps_k}(\star_\infty)$;
\item $(f_k)_{\sharp}(\mm_{k}) \rightharpoonup \mm_{\infty}$ as $k\to \infty$ (weak convergence against bounded continuous functions with bounded support). 
\end{enumerate}
Having globally defined approximation maps slightly simplifies the notation; thus, in the following, we will make use of this formulation of pmGH convergence. 
\end{remark}

For sequences of gauge m.m.s.\@ it is natural to add extra assumptions to ensure that ``$\sfG_k \to \sfG_\infty$''. The natural extension of (ii) to gauge functions would be:
\begin{equation}\label{eq:iiforgauge}
|\sfG_k(x,y) - \sfG_\infty(f_k(x),f_k(y))|\leq \varepsilon_{k}, \qquad \forall \, x,y\in B_{R_k}(\star_k).
\end{equation}
However, due to the a priori low regularity of the gauge function (which in general is not Lipschitz continuous), this would be a too strong requirement in geometric situations (e.g.\ convergence to the tangent cone for sub-Riemannian structures; cf. Section \ref{sec:ex:tangent}).  It turns out that a significantly weaker condition is sufficient for the stability of our curvature-dimension inequalities.

\begin{definition}\label{def:weakL1andregularity}
Let $(X_k, \sfd_k,\mm_k,\sfG_k)$, $k\in \bar{\N}$ be gauge metric measure spaces such that $(X_k, \sfd_k,\mm_k,\star_k) \to (X_\infty,\sfd_\infty,\mm_\infty,\star_\infty)$ in the pmGH sense, with approximating maps $f_k:X_k \to X_\infty$ as in  Remark \ref{rem:pmghfkglob}. We introduce the following conditions:
\begin{itemize}
\item \textbf{$L^1_{\loc}$ convergence of gauge functions:} if for all sequences $x_k\in \supp\mm_k$ such that $f_k(x_k)$ is convergent in $\supp\mm_\infty$  it holds
\begin{equation}\label{eq:weakL1condition}
\lim_{k\to \infty}\int_{B_{R}(\star_k){\setminus \{x_k\}}}|\sfG_k(x_k,z)-\sfG_\infty(f_k(x_k),f_k(z))| \, \mm_k(\di z) =0 ,  \qquad \forall\, R>0.
\end{equation}
\item \textbf{regularity condition:} if for all $x\in \supp\mm_\infty$ and $\mm_\infty$-a.e.\ $y\in X_\infty$, $y\neq x$, with exceptional set depending on $x$, $\sfG_\infty$ is continuous at $(x,y)$.
\end{itemize}
\end{definition}
\begin{remark}
The removal of $\{x_k\}$ in the domain of integration in \eqref{eq:weakL1condition} and the condition $y\neq x$ in the regularity condition can be omitted if $\mm_k$ and $\mm$ do not give mass to points.
\end{remark}
In Section \ref{sec:regularityGExamples} we show that the regularity condition above is satisfied in natural classes of examples.

\begin{lemma}[Lower semi-continuity of $\sfG$-diameter]\label{lem:contGdiam}
Let $(X_k, \sfd_k,\mm_k,\sfG_k)$, $k\in \bar{\N}$ be gauge metric measure spaces such that $(X_k, \sfd_k,\mm_k,\star_k) \to (X_\infty,\sfd_\infty,\mm_\infty,\star_\infty)$ in the pmGH sense. Assume the $L^1_{\loc}$ convergence of gauge functions and the regularity condition as in Definition \ref{def:weakL1andregularity}. Then it holds:
\begin{equation}
\liminf_{k\to \infty} \diam_{\sfG}(\supp\mm_k) \geq \diam_{\sfG}(\supp\mm_\infty).
\end{equation}
\end{lemma}
For a counter-example to upper semi-continuity, take the trivial sequence given by the Euclidean structure and the Lebesgue measure on $\R$. Choose a sequence of gauge functions $\sfG_k(x,y) =\sfG_k(y)$, and its pointwise limit $\sfG_\infty(y) =\lim_{k\to\infty}\sfG_k(y)$ as in the picture.
\begin{wrapfigure}[8]{l}{.4\textwidth}
\includegraphics[width=.4\textwidth]{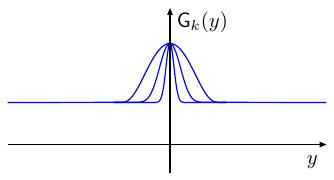}
\end{wrapfigure}
The difference $|\sfG_k(y)-\sfG_\infty(y)|$ is a small a.e.\ continuous bump localized in a shrinking neighborhood of the origin, that decreases monotonically to zero pointwise a.e. One can check in particular that the $L^1_{\loc}$ convergence in Definition \ref{def:weakL1andregularity} is satisfied by dominated convergence. In this situation, we have for all $k\in \N$: $2 = \diam_{\sfG}(\supp\mm_k) > \diam_{\sfG}(\supp\mm) = 1$.
\begin{proof}
We drop the subscript $\infty$ in the notation for the limit objects. Let $D:=\liminf_{k\to \infty} \diam_{\sfG}(\supp\mm_k)$, and suppose that $\diam_{\sfG}(\supp\mm)=D+3\eta$ for $\eta>0$. Recalling Definition \ref{def:Gdiam} of $\sfG$-diameter, there exists $x_o\in \supp\mm$ and a Borel set $P\subseteq \supp\mm$ with $\mm(P)>0$  and $x_o\notin P$ such that
\begin{equation}
\sfG(x_o,y) \geq D +2\eta, \qquad \forall\, y \in P.
\end{equation}
By the regularity condition, removing a zero $\mm$-measure set from $P$, we can assume that $\sfG$ is continuous at all points $(x_o,y)$ for $y\in P$. In particular there is $(x_o,y_o)\in \supp\mm\times \supp\mm$ and $\ve>0$ such that  $B_{2\ve}(x_o)\cap B_{2\ve}(y_o)=\emptyset$ and 
\begin{equation}
\sfG(x,y) \geq D +\eta, \qquad \forall\, (x,y)\in B_\ve(x_o)\times B_\ve(y_o).
\end{equation}
Recall the approximating maps $f_k:X_k\to X$ from Remark \ref{rem:pmghfkglob}. 
Since $(f_k)_\sharp(\mm_k)\rightharpoonup \mm$, we can find a sequence $(x_k)_{k\in \N}$ with $f_k(x_k) \to x_o$. We can assume that $\sfd(f_k(x_k),x_o)\leq \ve$ for all $k\in \N$, and thus
\begin{equation}\label{eq:geq}
\sfG(f_k(x_k),y) \geq D + \eta ,\qquad \forall\, y \in B_\ve(y_0),\quad \forall\, k\in\N.
\end{equation}
Let $h:X\to [0,1]$ be a continuous function such that $h\equiv 1$ on $\overline{B_{\ve/2}(y_o)}$ and $h\equiv 0$ on $X\setminus B_{\ve}(y_o)$. Consider the sequence $\mu_k \in \mathcal{P}_{ac}(X_k,\mm_k)$ given by
\begin{equation}
\mu_k:=\frac{1}{\int_{X_k} h(f_k(z))\,\mm_k(\di z)} h\circ f_k\, \mm_k, \qquad \forall\, k\in \N.
\end{equation}
The support of $\mu_k$ is contained in the Borel set $f_k^{-1}(B_\ve(y_o))$. 
Note also that  $x_k\notin \supp\, \mu_k$, for $k$ sufficiently large. Since $(f_k)_\sharp\mm_k\rightharpoonup\mm$ we have
\begin{equation}
\lim_{k \to \infty} \int_{X_k} h(f_k(z))\,\mm_k(\di z) = \int_X h(z)\, \mm(\di z).
\end{equation}
Furthermore, by local boundedness of $\mm$ and since $y_o\in\supp\mm$, it holds
\begin{equation}
0<\mm(B_{\ve/2}(y_o)) \leq \int_X h(z)\, \mm(\di z) \leq \mm(B_{\ve}(y_o)) <+\infty.
\end{equation}

Thus, we have:
\begin{multline}\label{eq:aboveformula}
\int_{X_k} \sfG_k(x_k,z)\, \mu_k(\di z)  = \int_{X_k} \sfG(f_k(x_k),f_k(z))\, \mu_k(\di z) \\ + \int_{X_k}\left[ \sfG_k(x_k,z)- \sfG(f_k(x_k),f_k(z))\right]\, \mu_k(\di z).
\end{multline}

 We claim that the second integral in the r.h.s.\@ of \eqref{eq:aboveformula} converges to $0$ by the $L^1_{\loc}$ convergence condition. Indeed recall that $x_k\notin \supp\mu_k$ and furthermore $\supp\mu_k\subseteq f_k^{-1}(B_\varepsilon(y_o))$. Using the properties of $f_k$, there exists $R>0$ such that $\supp\mu_k\subseteq B_R(\star_k)$ for all sufficiently large $k\in \N$. Then, using the construction of $\mu_k$, we obtain:
\begin{multline}
\limsup_{k\to \infty}\int_{X_k}\left[ \sfG_k(x_k,z)- \sfG(f_k(x_k),f_k(z))\right]\, \mu_k(\di z)  \leq \\ \limsup_{k\to \infty} \frac{1}{\mm(B_{\varepsilon/2}(y_o))}\int_{B_R(\star_k)\setminus \{x_k\}}| \sfG_k(x_k,z)- \sfG(f_k(x_k),f_k(z))|\, \mm_k(\di z) = 0,
\end{multline}
proving our claim.

Since $\supp\mu_k\subseteq f_k^{-1}(B_\ve(y_o))$, by \eqref{eq:geq} the first integral in the right hand side of \eqref{eq:aboveformula} is $\geq D +\eta$. Therefore
\begin{equation}
\liminf_{k\to \infty} \int_{X_k} \sfG_k(x_k,z)\, \mu_k(\di z) \geq D +\eta >D.
\end{equation}
On the other hand, by definition of $\sfG$-diameter,  and since $x_k \notin \supp\mu_k$ it holds:
\begin{equation}
\liminf_{k\to \infty} \int_{X_k} \sfG_k(x_k,z)\, \mu_k(\di z) \leq \liminf_{k\to \infty} \diam_{\sfG}(\supp\mm_k) = D,
\end{equation}
yielding a contradiction.
\end{proof}

\subsection{Stability of \texorpdfstring{$\MCP(\beta)$}{MCP}}

\begin{theorem}[Stability of $\MCP$]\label{thm:stabMCP}
Let $(X_k, \sfd_k,\mm_k,\sfG_k)$, $k\in \bar{\N}$ be gauge m.m.s.\@ with $(X_k, \sfd_k,\mm_k,\star_k) \to (X_\infty,\sfd_\infty,\mm_\infty,\star_\infty)$ in the pmGH sense. Let $\beta$ be distortion coefficients as in \eqref{eq:defbeta}. Assume that:
\begin{enumerate}[label = (\roman*)]
\item\label{i:stabMCP1} $(\supp\mm_k, \sfd_k)$ is locally compact and geodesic for all $k\in\bar{\N}$;
\item\label{i:stabMCP2} the supports are  locally uniformly $\sfG$-bounded:  for any $R>0$ there exists $D=D(R) >0$ such that $\diam_{\sfG}(\supp\mm_k \cap B_R(\star_k))\leq D$ for all $k\in \N$ (and thus for all $k\in \bar{\N}$ by Lemma \ref{lem:contGdiam});
\item\label{i:stabMCP3} the $L^1_{\loc}$ convergence and regularity conditions hold (cf.\ Definition \ref{def:weakL1andregularity});
\item\label{i:stabMCP4} for all $t\in (0,1)$ the function $\beta_t: [0,\cD) \to [0,\infty)$ is locally Lipschitz;
\item\label{i:stabMCP5} $(X_k, \sfd_k,\mm_k,\sfG_k)$ satisfies the $\MCP(\beta)$ for all $k\in \N$.
\end{enumerate}
Then, also $(X_\infty,\sfd_\infty,\mm_\infty,\sfG_\infty)$ satisfies the $\MCP(\beta)$.
\end{theorem}
\begin{remark}
 Assumption \ref{i:stabMCP1} for \emph{finite} $k\in \N$ is implied by the other ones. More precisely, by \ref{i:stabMCP2}, bounded subsets of $\supp\mm_k$ are $\sfG$-bounded, and we can apply Corollary \ref{cor:totallyboundedclassic}.
\end{remark}
\begin{remark}
The local Lipschitz assumption on the function $\beta_t(\cdot)$ is not restrictive. In all cases of interest, $\beta_t(\cdot)$ is real-analytic on $[0,\cD)$.
\end{remark}
\begin{proof}
The first part of the proof follows the blueprint of the classical one (cf.\ \cite{Vil}), however, later in the argument,  we will need to deal with additional difficulties caused by  the non-continuity of $\sfG$ and the general distortion coefficients.

We drop the $\infty$ from the notation for the limit so that $(X_\infty,\sfd_\infty,\mm_\infty,\sfG_\infty,\star_\infty)$ will be denoted by $(X,\sfd,\mm,\sfG,\star)$. 
Let $f_k :X_k \to X$ be the approximations as in Remark \ref{rem:pmghfkglob}. 
Let $\mu_1 \in \Prob_{bs}^{*}(X,\sfd,\mm)$  and fix $\bar{x} \in \supp\mm \setminus \supp \mu_1$.  Since $\supp \mu_1$ is compact (recall that $(\supp\mm,\sfd)$ is assumed to be locally compact), the distance between $\bar{x}$ and $\supp \mu_1$ is positive. Up to extraction and relabeling, for all $k\in \N$ we find $\bar{x}_k\in \supp\mm_k$ such that $f_k(\bar{x}_k)\to \bar{x}$ as $k\to \infty$.

Furthermore, letting $\mu_1 = \rho_1\, \mm$ for $\rho_1 \in L^1(X,\mm)$ with bounded support, we define
\begin{equation}
\mu_{1,k}:=\frac{\rho_1 \circ f_k \, \mm_k}{\int_{X_k} \rho_1 \circ f_k\, \mm_k}.
\end{equation}
We assume from now on that $\rho_1$ is continuous, and we will deal with the discontinuous case at the end of the proof. Notice that 
\begin{equation}
Z_k:=\int_{X_k} \rho_1 \circ f_k \, \mm_k = \int_{X}\rho_1 \, (f_k)_\sharp \mm_k \to 1,
\end{equation}
so that, $\mu_{1,k}\in \Prob_{bs}^{*}(X_{k},\sfd_{k},\mm_{k})$  and $\bar{x}_k\notin \supp  \mu_{1,k}$, for sufficiently large $k$.

By the $\MCP(\beta)$ we find a $W_2(X_k,\sfd_k)$-geodesic $(\mu_{t,k})_{t\in [0,1]}$ from $\delta_{\bar{x}_k}$ to $\mu_{1,k}$ such that
\begin{equation}\label{eq:tolimit}
\Ent(\mu_{t,k}|\mm_k)\leq \Ent(\mu_{1,k}|\mm_k) - \int_{X_k} \log\beta_t(\sfG_k(\bar{x}_k,x))\, \mu_{1,k}(\di x).
\end{equation}
By construction, $\supp\mu_1$ and $\supp(f_k)_\sharp \mu_{1,k}$ are contained in a common bounded set, say $B_R(\star)\subset X$, for all $k\in \N$ and some $R>0$. Thus, using the properties of the approximating maps $f_k$ and the fact that $(\supp\mm_k,\sfd_k)$ are locally compact and geodesic, one can show that the supports of the family $\{(f_k)_\sharp \mu_{t,k}\}_{k\in \N,t\in [0,1]}$ are contained in the common compact set $\bar{B}_{2R}(\star)$. By stability of optimal transport \cite[Thm.\@ 28.9]{Vil}, up to extraction, there exists a $W_2(X,\sfd)$-geodesic $(\mu_t)_{t\in[0,1]}$ with
\begin{equation}
\sup_{t\in [0,1]} W_2((f_k)_\sharp \mu_{t,k},\mu_t) \to 0.
\end{equation}
In particular by Proposition \ref{prop:W2convNarrow} we also have $(f_k)_\sharp\mu_{t,k} \rightharpoonup \mu_t$ weakly. We can now pass to the limit in the l.h.s.\@ of \eqref{eq:tolimit}:
\begin{align}
\Ent(\mu_t|\mm) & \leq \liminf_{k\to \infty} \Ent((f_k)_\sharp\mu_{t,k}|(f_{k})_\sharp \mm_k)  \leq \liminf_{k\to \infty} \Ent(\mu_{t,k}|\mm_k),
\end{align}
where in the first inequality we use the joint lower semi-continuity of the Boltzmann-Shannon entropy, and in the second one we used the fact that the push-forward via a Borel map does not increase the entropy (cf. \cite[Thm.\@ 29.20 (i), (ii)]{Vil}).

We now study the r.h.s.\@ of \eqref{eq:tolimit} as $k\to \infty$. We first observe that
\begin{align}
\Ent(\mu_{1,k}|\mm_k) =\frac{1}{Z_k} \int_{X} \rho_1\log\rho_1\, (f_k)_\sharp \mm_k  - \log Z_k.
\end{align}
Since $\rho_1$ is continuous and with bounded support, then $\rho_1\log\rho_1$ is bounded, continuous, and with bounded support. Furthermore, since $(f_k)_\sharp \mm_k \rightharpoonup \mm$ and $Z_k\to 1$, we have
\begin{align}
\lim_{k\to\infty} \Ent(\mu_{1,k}|\mm_k)  = \Ent(\mu_{1}|\mm).
\end{align}
This settles the first term in \eqref{eq:tolimit}. 

The limit of the second term in \eqref{eq:tolimit} is quite delicate as a consequence of the low regularity of $\sfG$ and the almost-everywhere nature of the $\sfG$-diameter bound,  here is where the proof departs from the classical one. 

 As a first remark notice that $\bar{x}_k$ (resp.\ $\bar{x}$) and $\supp\mu_{1,k}$ (resp.\ $\supp\mu_1$) are contained in a common bounded set of $\supp\mm_k$ (resp.\ $\supp\mm$), say $B_{\bar{R}}(\star_k)\cap \supp\mm_k$ (resp.\ $B_{\bar{R}}(\star)\cap \supp\mm$) for some $\bar{R}>0$ and all $k\in \N$. In the rest of the proof, only the values of $\sfG_k$ (resp.\ $\sfG$) on these sets play a role. Therefore without loss of generality we can strengthen hypothesis \ref{i:stabMCP2} by assuming, instead, that $\supp\mm_k$ is $\sfG$-bounded, uniformly w.r.t.\ $k\in \bar{\N}$. In other words there exists $D>0$ such that
\begin{equation}\label{eq:diamleqD}
\mm(\{y\mid  y\neq x,\,  \sfG(x,y)\geq D+\ve\}) =0, \qquad \forall\, x\in \supp\mm,\quad \forall\, \ve>0.
\end{equation}
On the one hand, we have no information for $x\notin\supp\mm$. On the other hand, for fixed $x\in\supp\mm$, it may happen that $\sfG(x,y)\geq D+\ve$ on a set of zero $\mm$-measure which may have positive $(f_k)_\sharp \mm_k$ measure. For this reason, fix $\ve>0$ once for all, and introduce the following modification of the limit gauge function $\sfG$:
\begin{equation}\label{eq:defTrunc}
\sfG^{\ve}(x,y):=\min\{\sfG(x,y),D+\ve\}.
\end{equation}
The modified gauge function has the following properties:
\begin{enumerate}[(a)]
\item $\sfG^{\ve}(x,y) \leq D+\ve$ for all $x,y\in X$;
\item $\sfG^{\ve}(x,y) = \sfG(x,y)$ for all $x\in \supp\mm$ and $\mm$-a.e.\ $y\in X$,  $y\neq x$;
\item The $L^1_{\loc}$ convergence of the gauge functions of Definition \ref{def:weakL1andregularity} remains true replacing $\sfG$ with $\sfG^{\ve}$. In particular it holds:
\begin{equation}
\lim_{k\to \infty}\int_{B_R(\star_k){\setminus \{\bar{x}_k\}}}|\sfG_k(\bar{x}_k,z)-\sfG^\ve(f_k(\bar{x}_k),f_k(z))| \, \mm_{k}(\di z) =0, \qquad \forall\, R>0.
\end{equation}
\end{enumerate}
Item (a) is trivial, while (b) follows from the condition $\diam_{\sfG}(X) \leq D$ that is \eqref{eq:diamleqD}. Item (c) follows from the observation that, since $\diam_{\sfG}(X_k)\leq D$, the integrand in \eqref{eq:weakL1condition} decreases when replacing $\sfG$ with $\sfG^{\ve}$, for all fixed $k\in \N$.

We can now study the second term in \eqref{eq:tolimit} as $k\to \infty$. We have:
\begin{equation}\label{eq:twoterms}
\begin{multlined}
\int_{X_k} \log\beta_t\big(\sfG_k(\bar{x}_k,x)\big)\,\mu_{1,k}(\di x)  = \int_{X} \log\beta_t\big(\sfG^{\ve}(f_k(\bar{x}_k),x)\big)\,((f_k)_\sharp\mu_{1,k})(\di x)\\
+\int_{X_k} \Big[\log\beta_t \big(\sfG_k(\bar{x}_k,x)\big)-\log\beta_t \big(\sfG^{\ve} (f_k(\bar{x}_k),f_k(x))\big) \Big]\,\mu_{1,k}(\di x).
\end{multlined}
\end{equation}

To proceed with the proof, assume first that $\beta_t:[0,+\infty]\to [0,+\infty]$ is bounded on bounded intervals. This is the case precisely when $\cD=+\infty$.

We claim first that the second integral in the r.h.s.\@ of \eqref{eq:twoterms} converges to zero for any $t\in (0,1)$ as $k\to \infty$. First observe that $\beta_t(\cdot)$, restricted to $[0,D+\ve]$, is bounded from below away from zero (cf.\ Proposition \ref{prop:propertiesbeta}\ref{i:propertiesbeta4}), bounded from above, and Lipschitz continuous on $[0,D+\ve]$. Thus $\log\beta_t(\cdot)$ is Lipschitz continuous on $[0,D+\ve]$, with Lipschitz constant $C>0$ (the constant depends on $\beta_t$ and $D+\ve$ only). Furthermore, the arguments of $\log\beta_t(\cdot)$ in the second integral of \eqref{eq:twoterms} take values on $[0,D+\ve]$ $\mu_{1,k}$-almost everywhere: the first one since $\diam_{\sfG}(X_k)\leq D$ (recall that $\bar{x}_k\notin \supp\mu_{1,k}$), while the second one by construction of $\sfG^{\ve}$. Hence for the second integral in the r.h.s.\@ of \eqref{eq:twoterms} we have:
\begin{align}
\lim_{k\to \infty}\int_{X_k} &\left\lvert\log\beta_t\big(\sfG_k(\bar{x}_k,x)\big)-\log\beta_t\big(\sfG^{\ve}(f_k(\bar{x}_k),f_k(x))\big)\right\rvert\, \mu_{1,k}(\di x)  \\
 &\leq \lim_{k\to \infty} C \int_{X_k}\left\lvert\sfG_k(\bar{x}_k,x)-\sfG^{\ve}(f_k(\bar{x}_k),f_k(x))\right\rvert\, \mu_{1,k}(\di x) \label{eq:lipschitzproperty} \\
& \leq \lim_{k\to \infty} \frac{C \|\rho_{1}\|_{\infty}}{ Z_k}  \int_{B_{\bar{R}}(\star_{k})\setminus \{\bar{x}_k\}}\left\lvert\sfG_k(\bar{x}_k,x)-\sfG^{\ve}(f_k(\bar{x}_k),f_k(x))\right\rvert\, \mm_{k}(\di x)  =0,
\end{align}
where $\bar{R}>0$ is such that $\supp\mu_{1,k}\subset B_{\bar{R}}(\star_k)$ for all $k\in \N$.

It remains to discuss the limit of the first term in the r.h.s.\@ of \eqref{eq:twoterms}. Notice that $(f_k,f_k)_\sharp (\delta_{\bar{x}_k}\otimes \mu_{1,k}) \rightharpoonup \delta_{\bar{x}}\otimes \mu_1$. Thus we have
\begin{multline}
\lim_{k\to \infty}   \int_{X} \log\beta_t\big(\sfG^{\ve}(f_k(\bar{x}_k),x)\big)\,((f_k)_\sharp\mu_{1,k})(\di x)  \\ 
 = \lim_{k\to \infty}  \int_{X\times X} \log\beta_t\big(\sfG^{\ve}(x,y)\big)\,(f_k,f_k)_\sharp \left(\delta_{\bar{x}_k}\otimes \mu_{1,k}\right)(\di x \di y) \\
 = \int_{X} \log\beta_t\big(\sfG(\bar{x},x)\big)\,\mu_1(\di x).\label{eq:quiusiamoboga}
\end{multline}
In \eqref{eq:quiusiamoboga}, we used the fact that $\log\beta_t\circ\sfG^\ve = \log\beta_t\circ \sfG$ and is continuous $\delta_{\bar{x}}\otimes\mu_1$-a.e., bounded by construction, so that we can apply \cite[Cor.\ 2.2.10]{Bogachev-Weak}.

This concludes the proof in the case when $\beta_t(\cdot)$ is finite on bounded intervals (that is when $\cD=+\infty$), and $\rho_1=\di \mu_1/\di \mm$ is continuous. 

We deal now with the case $\cD<+\infty$. In this case what fails is \eqref{eq:lipschitzproperty}, as the arguments of the $\log$ can attain infinite value. We will argue by approximation. For $\lambda \in \N$ we let
\begin{equation}
\beta_t^{(\lambda)}(\theta):=\min\{\beta_t(\theta),\lambda\}, \qquad \forall\,\lambda\in \N,\,t\in[0,1],\,\theta \in [0,+\infty].
\end{equation}
Since $\beta_t \geq \beta_t^{(\lambda)}$, then $(X_k,\sfd_k,\mm_k,\sfG_k)$ satisfies the $\MCP(\beta^{(\lambda)})$ inequality \eqref{eq:defMCPbeta2}.

Furthermore, the functions $\beta_t^{(\lambda)}:[0,+\infty]\to [0,+\infty]$ are bounded from below away from zero by a positive constant, bounded from above, and Lipschitz continuous. To see the global Lipschitz continuity, recall that $\beta_t$ is locally Lipschitz on $[0,\cD)$ and $\liminf_{\theta \to \cD^-}\beta_t(\theta) = +\infty$, so that $\beta_t^{(\lambda)}$ can be written as the minimum of two globally Lipschitz functions. It follows that $\log\beta^{(\lambda)}_t$ is also Lipschitz continuous with Lipschitz constant that depends on $t$, $\lambda$, $\cD$, but not on $k$. 

Thus we can argue as in the previous part of the proof, and we obtain that for any $\mu_1\in \Prob_{bs}^{*}(X,\sfd,\mm)$ with continuous density and any $\bar{x}\in \supp\mm\setminus \supp\mu_1$    there exists a $W_2(X,\sfd)$-geodesic $(\mu_t)_{t\in[0,1]}$ between $\delta_{\bar{x}}$ and $\mu_1$ such, that for all $\lambda\in \N$:
\begin{equation}
\Ent(\mu_t|\mm)\leq \Ent(\mu_1|\mm)-\int_X \log\beta_t^{(\lambda)}(\sfG(\bar{x},x))\,\mu_1(\di x), \qquad\forall\, t\in(0,1).
\end{equation}
By construction, the sequence of functions $\log\beta_t^{(\lambda)}(\sfG(\bar{x},\cdot))$ is monotone w.r.t.\ $\lambda$, measurable, uniformly bounded from below, and $\log\beta_t^{(\lambda)}(\sfG(\bar{x},x))\uparrow \log\beta_t(\sfG(\bar{x},x))$ as $\lambda \to \infty$ for all $x\in X$. We conclude by the monotone convergence theorem.

This ends the proof in the case in which the density $\rho_1=\di \mu_1/\di \mm$ is continuous. The strategy to reduce the general case to the continuous case is the same as \cite{Vil}: we find a regularized measure $\mu_1^\ve$ with continuous density $\rho_1^\ve$ and bounded support using \cite[Ch.\ 29, First Appendix]{Vil}, and we prove the desired $\MCP$ inequality between $\bar{x}$ and $\mu_1^\ve$:
\begin{equation}
\Ent(\mu_{t}^\eps|\mm)\leq \Ent(\mu_{1}^\eps|\mm) - \int_{X} \log\beta_t(\sfG(\bar{x},x))\mu_{1}^\eps(\di x)\, ,
\end{equation}
where $(\mu_t^\eps)_{t\in[0,1]}$ is a $W_2(X,\sfd)$-geodesic between $\delta_{\bar{x}}$ and $\mu_1^\ve$ for $\ve>0$. The sequence $(\mu_t^\ve)_\ve$ converges in the $W_2$ distance to a $W_2(X,\sfd)$-geodesic between $\delta_{\bar{x}}$ and $\mu_1$. Then we pass to the limit for $\ve \to 0$ using similar arguments as before: the lower semi-continuity of the Boltzmann entropy in the l.h.s., its upper semi-continuity of the first term in the r.h.s. along this particular regularizing sequence, and again the regularity property of $\sfG$ and \cite[Cor.\ 2.2.10]{Bogachev-Weak} for the second term in the r.h.s.
\end{proof}

\begin{remark}[More general coefficients]\label{rmk:moregenfunctions}
The proof works for any family of non-negative functions $\beta_t : [0,+\infty] \to [0,+\infty]$, not necessarily defined as in \eqref{eq:defbeta}, but such that, for all $t\in (0,1)$:
\begin{itemize}
\item are bounded from below away from zero on any compact set;
\item are finite on $[0,\cD)$ for some $\cD \in (0,+\infty]$ and, if $\cD<+\infty$, it holds $\liminf_{\theta \to \cD^-}\beta_t(\theta) = +\infty$, and $\beta_t(\theta) = +\infty$ for $\theta \geq \cD$;
\item $\beta_t(\cdot)$ is locally Lipschitz on $[0,\cD)$.
\end{itemize}
\end{remark}
\begin{remark}[Variable coefficients]
Theorem \ref{thm:stabMCP} can be further extended to the case in which each space $(X_k,\sfd_k,\mm_k,\sfG_k,\star_k)$ satisfies a $\MCP(\beta^k)$ condition, with coefficients $\beta^k$ as in \eqref{eq:defbeta} for all $k\in \N$ and with the following properties: 
\begin{itemize}
\item $\liminf_{k\to\infty}\cD_k=+\infty$ (for simplicity);
\item for all fixed $t\in (0,1)$, the family $(\beta_t^k)_{k\in\N}$ is  definitely bounded below by a positive constant on any compact set, uniformly w.r.t.\ $k$;
\item for all fixed $t\in (0,1)$, the family $(\beta_t^k)_{k\in\N}$, restricted on any compact interval, is definitely Lipschitz, uniformly w.r.t.\ $k$;
\end{itemize}
In fact, up to extraction, and by Arzelà-Ascoli's Theorem, we can assume that there exists $\beta_t^\infty : [0,+\infty]\to [0,+\infty]$ such that $\beta_t^k \to \beta_t^\infty$ uniformly on compact sets, and for all fixed $t\in (0,1)$. Then under the same assumptions of Theorem \ref{thm:stabMCP} the limit space satisfies the $\MCP(\beta^\infty)$. To prove it, take the function $\mathfrak{B}_t^m(\theta):=\inf_{k\geq m} \beta_t^k(\theta)$, for $m\in \N$, $\theta \in [0,+\infty]$, $t\in [0,1]$. By construction, $\mathfrak{B}_t^m(\cdot)$ is finite on $[0,+\infty)$, locally Lipschitz, and bounded from below away from zero. We then use Remark \ref{rmk:moregenfunctions}, so that the limit space satisfies $\MCP(\mathfrak{B}^m)$. Then take the limit for $m\to\infty$, using Fatou's lemma.
\end{remark}

\begin{remark}[A semi-continuous variant]
Instead of assuming $\sfG_\infty$ to satisfy the ``$\mm_\infty$-a.e.\ continuity''  in the sense of the \emph{regularity condition} (cf.\ Definition \ref{def:weakL1andregularity}), one could alternatively assume that $\beta_{t}\circ \sfG_{\infty}$ is lower semi-continuous on $\supp \mm_{\infty} \times \supp \mm_{\infty}$, for all fixed $t\in (0,1)$. In the proof, \eqref{eq:quiusiamoboga} would be replaced by the inequality
\begin{multline}
\liminf_{k\to \infty}   \int_{X} \log\beta_t\big(\sfG^{\ve}(f_k(\bar{x}_k),x)\big)\,((f_k)_\sharp\mu_{1,k})(\di x)  \\ 
 = \liminf_{k\to \infty}  \int_{X\times X} \log\beta_t\big(\sfG^{\ve}(x,y)\big)\,(f_k,f_k)_\sharp \left(\delta_{\bar{x}_k}\otimes \mu_{1,k}\right)(\di x \di y) \\
 \geq \int_{X} \log\beta_t\big(\sfG(\bar{x},x)\big)\,\mu_1(\di x).\label{eq:quiusiamobogaLSC}
\end{multline}
which would follow from \cite[Cor.\ 2.2.6]{Bogachev-Weak}.  Note that  \eqref{eq:quiusiamobogaLSC} is sufficient for the proof.
\end{remark}

\subsection{Stability of \texorpdfstring{$\CD(\beta,n)$}{CD}}

The analogous stability statement for the $\CD(\beta,n)$ condition requires slightly stronger assumptions with respect to the ones in Definition \ref{def:weakL1andregularity}. 

The main difference is that the optimal transport plans used in the $\MCP$ condition are of the form $\pi_k=\delta_{x_k}\otimes \mu_{1,k}$ with $\mu_{1,k} \ll \mm_k$, and thus one can reduce integrals over $\pi_k$ to integrals with respect to $\mm_k$. In the corresponding construction for the $\CD$ case, the sequences of optimal plans $\pi_k$ are in general not absolutely continuous with respect to $\mm_k\otimes \mm_k$. Thus, it is convenient to strengthen the $L^1_{\loc}$ convergence of gauge functions and regularity conditions in the following way, that involves the optimal transport plans.

\begin{definition}\label{def:weakL1andregularity-CD}
Let $(X_k, \sfd_k,\sfG_k,\mm_k)$, $k\in \bar{\N}$ be gauge metric measure spaces such that $(X_k, \sfd_k,\mm_k,\star_k) \to (X_\infty,\sfd_\infty,\mm_\infty,\star_\infty)$ in the pmGH sense, with approximating maps $f_k:X_k \to X_\infty$  as in Remark \ref{rem:pmghfkglob}. We introduce the following conditions:
\begin{itemize}
\item \textbf{$L^1_{\loc}$ convergence of gauge functions over plans:}
 for all sequences $\mu_{0,k}\in \mathcal{P}_{bs}(X_k,\sfd_k,\mm_k)$, $\mu_{1,k}\in \Prob_{bs}^{*}(X_{k}, \sfd_{k}, \mm_{k})$  with $\supp\mu_{0,k}\cap \supp\mu_{1,k}=\emptyset$,  and all $\nu_k\in \OptGeo(\mu_{0,k},\mu_{1,k})$ such that \eqref{eq:defCDbetan} holds, for the corresponding sequence of optimal plans $\pi_k = (\ee_0,\ee_1)_\sharp \nu_k$ such that $(f_k,f_k)_\sharp\pi_k$ is weakly convergent in $\mathcal{P}(X_\infty\times X_\infty)$, we have
\begin{equation}\label{eq:weakL1condition-CD}
\lim_{k\to \infty}\int_{B_R((\star_k,\star_k))}|\sfG_k(x,z)-\sfG_\infty(f_k(x),f_k(z))| \, \pi_k(\di x \di z) =0, \qquad \forall\, R>0.
\end{equation}
\item \textbf{regularity condition for $\sfG_\infty$ over plans:} for all $\mu_{0}\in \mathcal{P}_{bs}(X_\infty,\sfd_\infty,\mm_\infty)$, $\mu_{1}\in \Prob_{bs}^{*}(X_{\infty}, \sfd_{\infty}, \mm_{\infty})$  with $\supp \mu_{0}\cap \supp \mu_{1}=\emptyset$, $\pi \in \Opt(\mu_0,\mu_1)$, it holds that $\sfG_\infty$ is continuous $\pi$-a.e.
\end{itemize}
\end{definition}
\begin{remark}
Since $\pi_k$ and $\pi$ are supported out of the diagonal, the conditions above can be written in a fashion similar to Definition \ref{def:weakL1andregularity}, by removing the diagonal.
\end{remark}
In Section \ref{sec:regularityGExamples} we show that the regularity condition above is satisfied in natural classes of examples.

\begin{theorem}[Stability of $\CD$]\label{thm:stabCD}
Let $(X_k, \sfd_k, \mm_k,\sfG_k)$, $k\in \bar{\N}$ be gauge m.m.s.\@ with $(X_k, \sfd_k,\mm_k,\star_k) \to (X_\infty,\sfd_\infty,\mm_\infty,\star_\infty)$ in the pmGH sense. Let $\beta$ be distortion coefficients as in \eqref{eq:defbeta}, and $n\in [1,+\infty$). Assume that:
\begin{enumerate}[label = (\roman*)]
\item\label{i:stabCD1} $(\supp\mm_k, \sfd_k)$ is locally compact and geodesic for all $k\in\bar{\N}$;
\item\label{i:stabCD2}  for any $R>0$ there exists $D = D(R) >0$ such that $\sfG_k \leq D$ on $ B_R(\star_k)\cap \supp \mm_k\times B_R(\star_k)\cap \supp\mm_k$, for all $k\in \N$;
\item\label{i:stabCD3} the $L^1_{\loc}$ convergence and regularity conditions over plans hold (cf.\ Definition \ref{def:weakL1andregularity-CD});
\item\label{i:stabCD4} for all $t\in (0,1)$ the function $\beta_t: [0,\cD) \to [0,+\infty)$ is locally Lipschitz;
\item\label{i:stabCD5} $(X_k, \sfd_k,\mm_k,\sfG_k)$ satisfies the $\CD(\beta,n)$ for all $k\in \N$;
\end{enumerate}
Then, also $(X_\infty,\sfd_\infty,\mm_\infty,\sfG_\infty)$ satisfies the $\CD(\beta,n)$.
\end{theorem}
\begin{remark}
The local boundedness of $\sfG_k$ (i.e.\@ \ref{i:stabCD2})  can be weakened by asking that $\sfG_k$ is locally $\pi_k$-essentially bounded, uniformly with respect to $k\in \N$ and all optimal plans $\pi_k$ for which the $\CD$ inequality \eqref{eq:defCDbetan} holds. More precisely, for all $R>0$ there exists $D\in (0,+\infty)$ such that for all $k\in \N$ and for all $\pi_k$ for which the $\CD$ inequality is required to hold on $(X_k,\sfd_k,\mm_k,\sfG_k)$, we have $\sfG_k\leq D$ $\pi_k$-a.e. on $B_{R}(\star_k)\times B_R(\star_k)$.
\end{remark}
\begin{proof}
We drop the $\infty$ from the notation for the limit so that $(X_\infty,\sfd_\infty,\mm_\infty,\sfG_\infty,\star_\infty)$ will be denoted by $(X,\sfd,\mm,\sfG,\star)$.
Let $f_k :X_k \to X$ be the approximations as in Remark \ref{rem:pmghfkglob}.
Let $\mu_0\in \mathcal{P}_{bs}(X,\sfd,\mm)$, and $\mu_1 \in \Prob_{bs}^{*}(X,\sfd,\mm)$  with $\supp \mu_{0}\cap \supp \mu_{1}=\emptyset$.  Since $\supp  \mu_i$, $i=0,1$, is compact (recall that $(\supp\mm,\sfd)$ is assumed to be locally compact), the distance between $\supp \mu_0$ and $\supp \mu_1$ is positive. We start by assuming that $\mu_0,\mu_1 \ll \mm$, with continuous density $\mu_i = \rho_i\, \mm$, and bounded support. Define
\begin{equation}
\mu_{i,k}:=\frac{\rho_i \circ f_k \, \mm_k}{\int_{X_k} \rho_i \circ f_k\, \mm_k}, \qquad i=0,1.
\end{equation}
Notice that 
\begin{equation}
Z_{i,k}:=\int_{X_k} \rho_i \circ f_k \, \mm_k = \int_{X}\rho_i \, (f_k)_\sharp \mm_k \to 1, \qquad i=0,1,
\end{equation}
so that $\mu_{i,k}\in \Prob_{bs}^{*}(X_{k}, \sfd_{k}, \mm_{k})$ and  $\supp \mu_{0,k}\cap \supp \mu_{1,k}=\emptyset$,   for sufficiently large $k$.

By the $\CD(\beta,n)$ we find a $W_2(X_k,\sfd_k)$-geodesic $(\mu_{t,k})_{t\in [0,1]}\subset \Prob_2(X_k,\sfd_k)$, induced by $\nu_k\in\OptGeo(\mu_{0,k},\mu_{1,k})$ such that
\begin{multline}\label{eq:tolimit-CD}
\U_{n}(\mu_{t,k}|\mm_k) \geq   \exp\left( \frac{1}{n}\int_{\Geo(X_k)} \log \beta_{1-t}\big(\sfG_k(\gamma_1,\gamma_0))\, \nu_k(\di\gamma) \right) \U_{n}(\mu_{0,k}|\mm_k)  \\
+\exp\left( \frac{1}{n}\int_{\Geo(X_k)} \log \beta_{t}\big(\sfG_k(\gamma_0,\gamma_1)\big) \, \nu_k(\di\gamma) \right)   \U_{n}(\mu_{1,k}|\mm_k), \quad \forall\, t\in (0,1).
\end{multline}
By construction, $\supp\mu_i$ and $\supp(f_k)_\sharp \mu_{i,k}$, $i=0,1$, are contained in a common bounded set, say $B_R(\star)\subset X$, for all $k\in \N$ and some $R>0$. Thus, using the properties of the approximating maps $f_k$ and the fact that $(\supp\mm_k,\sfd_k)$ are locally compact and geodesic, one can show that the supports of the family $\{(f_k)_\sharp \mu_{t,k}\}_{k\in \N,t\in [0,1]}$ are contained in the common compact set $\bar{B}_{2R}(\star)$. By stability of optimal transport \cite[Thm.\@ 28.9]{Vil}, up to extraction, there exists a $W_2(X,\sfd)$-geodesic $(\mu_t)_{t\in[0,1]}$, and corresponding optimal plan $\pi\in \Opt(\mu_0,\mu_1)$ with
\begin{equation}
\sup_{t\in [0,1]} W_2((f_k)_\sharp \mu_{t,k},\mu_t) \to 0, \qquad \text{and} \qquad (f_k,f_k)_\sharp \pi_k \rightharpoonup \pi.
\end{equation}
In particular by Proposition \ref{prop:W2convNarrow} we also have $(f_k)_\sharp\mu_{t,k} \rightharpoonup \mu_t$ weakly.

We can now pass to the limit in the l.h.s.\@ of \eqref{eq:tolimit-CD}:
\begin{equation}
\U_n(\mu_t|\mm)  \geq \limsup_{k\to \infty} \U_n((f_k)_\sharp\mu_{t,k}|(f_{k})_\sharp \mm_k) \geq \limsup_{k\to \infty} \U_n(\mu_{t,k}|\mm_k),
\end{equation}
where in the first inequality we use the joint upper semi-continuity of $\U_n$, and in the second one we used the fact that the push-forward via a Borel map does not increase the entropy (cf. \cite[Thm.\@ 29.20 (i), (ii)]{Vil}).

We now study the r.h.s.\@ of \eqref{eq:tolimit-CD} as $k\to \infty$. We first observe that, for $i=0,1$
\begin{align}
\Ent(\mu_{i,k}|\mm_k)  =\frac{1}{Z_{i,k}} \int_{X} \rho_i\log\rho_i\, (f_k)_\sharp \mm_k  - \log Z_{i,k}.
\end{align}
Since $\rho_i$ is continuous and with bounded support, then $\rho_i\log\rho_i$ is bounded, continuous, and with bounded support. Furthermore, since $(f_k)_\sharp \mm_k \rightharpoonup \mm$ and $Z_{i,k}\to 1$, we have
\begin{equation}
\lim_{k\to\infty} \U_n(\mu_{i,k}|\mm_k)  = \U_n(\mu_{i}|\mm),\qquad i=0,1.
\end{equation}
This settles the entropic factors in the right hand side of \eqref{eq:tolimit-CD}. 

To deal with the remaining factors in the right hand side of \eqref{eq:tolimit-CD} we first claim that 
\begin{equation}\label{eq:piesssup}
\pi\shortminus\esssup \sfG \leq D.
\end{equation}
To prove it, notice that $\supp\pi_k \subseteq B_{\bar{R}}((\star_k,\star_k))\cap\supp\mm_k\otimes\mm_k$ for sufficiently large $\bar{R}>0$ and all $k\in \N$. Therefore by assumption \ref{i:stabCD2} there exists $D>0$ such that $\sfG_k \leq D$ on $\supp\pi_k$ for all $k\in\N$. If there exists set $A\subset X\times X$ with $\pi(A)>0$ such that $\sfG(x,y) \geq D+\varepsilon$ for all $(x,y)\in A$, then, letting $A_k:=(f_k,f_k)^{-1}(A)$, we obtain
\begin{equation}
\liminf_{k\to \infty}\int_{A_k}|\sfG_k(x,z)-\sfG(f_k(x),f_k(z))| \, \pi_k(\di x \di z) \geq\liminf_{k\to \infty} \varepsilon\int_{A_k}\pi_k(\di x \di z) >0,
\end{equation}
contradicting the $L^1_\loc$ convergence over plans of Definition \ref{def:weakL1andregularity-CD}.

In particular it holds
\begin{equation}\label{eq:diamleqD-CD}
\pi(\{(x,y)\mid \sfG(x,y) \geq D+\varepsilon\}) = 0, \qquad \forall\, \varepsilon>0,
\end{equation}
proving \eqref{eq:piesssup}. Fix $\ve>0$ once for all, and let:
\begin{equation}
\sfG^{\ve}(x,y):=\min\{\sfG(x,y),D+\ve\}, \qquad \forall\,(x,y)\in X\times X.
\end{equation}
The modified gauge function has the following properties:
\begin{enumerate}[(a)]
\item $\sfG^{\ve}(x,y) \leq D+\ve$ for all $x,y\in X$;
\item $\sfG^{\ve}(x,y) = \sfG(x,y)$ for $\pi$-a.e.\ $(x,y)\in X\times X$;
\item The $L^1_{\loc}$ convergence of the gauge functions for plans of Definition \ref{def:weakL1andregularity-CD} remains true replacing $\sfG$ with $\sfG^{\ve}$. In particular it holds:
\begin{equation}
\lim_{k\to \infty}\int_{B_{R}((\star_k,\star_k))}|\sfG_k(x,z)-\sfG^\ve(f_k(x),f_k(z))| \, \pi_{k}(\di x\di z) =0, \qquad \forall\, R>0.
\end{equation}
\end{enumerate}
Item (a) is trivial, while (b) follows from \eqref{eq:piesssup}. Item (c) follows from the observation that, since  $\sfG_k\leq D$ on $\supp\pi_k$ for large $k$, the integrand in \eqref{eq:weakL1condition-CD} decreases when replacing $\sfG$ with $\sfG^{\ve}$, for all fixed $k\in \N$.

We now study the second factor in the right hand side of \eqref{eq:tolimit-CD} (the other one is similar). We have:
\begin{multline}\label{eq:twoterms-CD}
\int_{\Geo(X_k)} \log\beta_t\big(\sfG_k(\gamma_0,\gamma_1)\big)\,\nu_{k}(\di \gamma)  = \int_{X\times X} \log\beta_t\big(\sfG^{\ve}(x,z)\big)\,((f_k,f_k)_\sharp\pi_{k})(\di x \di z)\\
+\int_{X_k\times X_k} \Big[\log\beta_t \big(\sfG_k(x,z)\big)-\log\beta_t \big(\sfG^{\ve} (f_k(x),f_k(z))\big) \Big]\,\pi_{k}(\di x \di z).
\end{multline}
To proceed with the proof, assume that $\beta_t:[0,+\infty]\to [0,+\infty]$ is bounded on bounded intervals. This is the case precisely when $\cD=+\infty$.

We claim first that the second integral in the r.h.s.\@ of \eqref{eq:twoterms-CD} converges to zero for any $t\in (0,1)$ as $k\to \infty$. First observe that $\beta_t(\cdot)$, restricted to $[0,D+\ve]$, is bounded from below away from zero (cf.\ Proposition \ref{prop:propertiesbeta}\ref{i:propertiesbeta4}), bounded from above, and Lipschitz continuous on $[0,D+\ve]$. Thus $\log\beta_t(\cdot)$ is Lipschitz continuous on $[0,D+\ve]$, with Lipschitz constant $C>0$ (the constant depends on $\beta_t$ and $D+\ve$ only). Furthermore, the arguments of $\log\beta_t(\cdot)$ take values on $[0,D+\ve]$ everywhere: the first one since $\sfG_k\leq D$ on  $\supp\pi_k$ for large $k$, while the second one by construction of $\sfG^{\ve}$. Hence for the second integral in the r.h.s.\@ of \eqref{eq:twoterms-CD} we have:
\begin{align}\label{eq:lipschitzproperty-CD}
\lim_{k\to \infty} \int_{X_k\times X_k} &\left\lvert\log\beta_t\big(\sfG_k(x,z)\big)-\log\beta_t\big(\sfG^{\ve}(f_k(x),f_k(z))\big)\right\rvert\, \pi_{k}(\di x \di z) \nonumber  \\
 &\leq C \lim_{k\to \infty} \int_{X_k\times X_k}\left\lvert\sfG_k(x,z)-\sfG^{\ve}(f_k(x),f_k(z))\right\rvert\, \pi_{k}(\di x\di z)  =0.
\end{align}

It remains to discuss the limit of the first term in the r.h.s.\@ of \eqref{eq:twoterms-CD}. We have
\begin{equation}\label{eq:quiusiamoboga-CD}
\lim_{k\to \infty}   \int_{X\times X} \log\beta_t\big(\sfG^{\ve}(x,z)\big)\,(f_k,f_k)_\sharp\pi_{k}(\di x \di z)  = \int_{X\times X} \log\beta_t\big(\sfG^{\ve}(x,z)\big)\,\pi(\di x\di z).
\end{equation}
In \eqref{eq:quiusiamoboga-CD}, we used the fact that $\log\beta_t\circ\sfG^\ve = \log\beta_t\circ \sfG$ and is continuous $\pi$-a.e., bounded by construction, so that we can apply \cite[Cor.\ 2.2.10]{Bogachev-Weak}.

Furthermore we observe that, if $\nu \in \OptGeo(\mu_0,\mu_1)$ is the dynamic optimal plan representing the $W_2$-geodesic $(\mu_t)_{t\in [0,1]}$, we have
\begin{equation}
\int_{X\times X} \log\beta_t\big(\sfG^{\ve}(x,z)\big)\,\pi(\di x\di z) = \int_{\Geo(X)} \log\beta_t\big(\sfG^{\ve}(\gamma_0,\gamma_1)\big)\,\nu(\di\gamma).
\end{equation}
This concludes the proof in the case when $\beta_t(\cdot)$ is finite on bounded intervals (that is when $\cD=+\infty$), and both $\rho_i=\di \mu_i/\di \mm$ are continuous. The extension to the case $\cD<+\infty$ is done verbatim as in the proof of Theorem \ref{thm:stabMCP}.

The strategy to reduce the general case (i.e.\@ $\mu_0\in \mathcal{P}_{bs}(X,\sfd,\mm)$, $\mu_1\in \Prob_{bs}^{*}(X,\sfd,\mm)$ with $\supp \mu_0\cap \supp \mu_1=\emptyset$) to the previous one is also along the same lines. As in the proof of Theorem \ref{thm:stabMCP}, we regularize $\mu_0,\mu_1$ to obtain a.c.\ measures with continuous density \cite[Ch.\ 29, First Appendix]{Vil}. The only difference in the argument is that, when passing to the limit, one uses \eqref{eq:piesssup} in order to apply \cite[Cor.\ 2.2.10]{Bogachev-Weak}.
\end{proof}

\begin{remark}[A semi-continuous variant]
Instead of assuming $\sfG_\infty$ to satisfy the \emph{regularity condition over plans} (cf. Definition \ref{def:weakL1andregularity-CD}), one could alternatively assume that $\beta_{t}\circ \sfG_{\infty}$ is lower semi-continuous on $\supp \mm_{\infty} \times \supp \mm_{\infty}$, for all fixed $t\in (0,1)$. In the proof, \eqref{eq:quiusiamoboga-CD} would be replaced by the inequality
\begin{equation}\label{eq:quiusiamoboga-CDLSC}
\liminf_{k\to \infty}   \int_{X\times X} \log\beta_t\big(\sfG^{\ve}(x,z)\big)\,(f_k,f_k)_\sharp\pi_{k}(\di x \di z)  \geq \int_{X\times X} \log\beta_t\big(\sfG^{\ve}(x,z)\big)\,\pi(\di x\di z).
\end{equation}
which would follow from \cite[Cor.\ 2.2.6]{Bogachev-Weak}.  Note that  \eqref{eq:quiusiamoboga-CDLSC} is sufficient for the proof.
\end{remark}

\subsection{Compactness  for \texorpdfstring{$\MCP(\beta)$}{MCP} and \texorpdfstring{$\CD(\beta,n)$}{CD} spaces}

The goal of this section is to establish some compactness results for $\MCP(\beta)$ and $\CD(\beta,n)$ spaces. First, let us recall some notation. Let $(X,\sfd,\mm)$ be a metric measure space.
Given an open subset $\Omega\subset X$, we denote with $\mathrm{LIP}(\Omega)$ the set of Lipschitz functions $u:\Omega\to \R$. The space $\mathrm{LIP}_{\loc}(\Omega)$ is the set of functions $u:\Omega\to \R$ such that $u|_A \in  \mathrm{LIP}(A)$ for every open set $A \Subset \Omega$.

For $u\in \mathrm{LIP}_{\loc}(\Omega)$, we define the \emph{modulus of the gradient} of $u$ at $x\in \Omega$ as
\begin{equation}
\| \nabla u \| (x) := \liminf_{r \to 0} \frac{1}{r} \sup_{y\in \overline{B_r (x)}} |u(y)-u(x)|.
\end{equation}
Following \cite[Sec. 3]{Miranda}, for $u\in L^1_{loc}(\Omega, \mm\llcorner \Omega)$, we define the \emph{total variation} of $u$ on $\Omega$ as
\begin{equation}
\|Du\|(\Omega):=\inf \left \{ \liminf_{k\to \infty} \int_\Omega \| \nabla u_k\| \mm \mid (u_k)_k\subset \mathrm{LIP}_{\loc}(\Omega), \; u_k \to u \text{ in } L^1_{loc}(\Omega,\mm\llcorner \Omega) \right\} .
\end{equation}
A function $u\in L^1_{loc}(\Omega, \mm\llcorner \Omega)$ is said to have bounded total variation on $\Omega$ provided $\|Du\|(\Omega)<\infty$. For $u\in L^1(\Omega, \mm\llcorner \Omega)$ with bounded total variation on $\Omega$, we denote
\begin{equation}
\|u\|_{BV(\Omega)}:=\|u\|_{L^1(\Omega,  \mm\llcorner \Omega)}+\|Du \|(\Omega).
\end{equation}

The next result gives pre-compactness criteria for gauge metric measure spaces satisfying the $\MCP$, that generalizes the well-known one for metric measures spaces. 

\begin{theorem}[Pre-compactness for $\MCP$]\label{thm:precompactMCP}
Let $\beta$ be distortion coefficients as in \eqref{eq:defbeta}. Let $\{(X_k,\sfd_k,\mm_k, \sfG_{k},\star_k)\}_{k\in \N}$ be a family of pointed gauge metric measure spaces, with $\supp\mm_k = X_k$, satisfying the $\MCP(\beta)$. Assume that for all $L>0$ there exists $D=D(L)>0$ such that
\begin{equation}
\diam_{\sfG}(B_L(\star_k)) \leq D(L),\qquad \forall\, k\in \N.
\end{equation}
Then, the following hold:
\begin{enumerate}[label = (\roman*)]
\item \label{i:ComMCP1} If there exists $M>0$ such that $\tfrac{1}{M}\leq \mm_k(B_1(\star_k))\leq M$ for all $k\in \N$, then the family $\{(X_k,\sfd_k,\mm_k,\star_k)\}_{k\in \N}$ is pre-compact in the pmGH topology.  

Moreover, if $(X_\infty,\sfd_\infty,\mm_\infty,\star_\infty)$ is any pmGH limit space, then $(\supp\mm_\infty, \sfd_\infty)$ is proper and geodesic.
\item  \label{i:ComMCP2} Assume that $(X_k,\sfd_k,\mm_k,\star_k)\to (X_\infty,\sfd_\infty,\mm_\infty,\star_\infty)$ in the pmGH topology, by means of approximating maps $f_k:X_k \to X_\infty$.  Assume that the family $\{\sfG_k\}_{k\in\N}$ is locally asymptotically equi-continuous, in the sense that for all $L>0$ and $\ve>0$ there exist $\delta=\delta(\ve,L)>0$ and $\bar{k}=\bar{k}(\ve,L)\in \N$ such that for all $k\geq \bar{k}$ it holds:
\begin{equation}\label{eq:GkEquicont}
\sfd_k(x_k,x_k')\leq \delta, \quad \sfd_k(y_k,y_k')\leq \delta \quad \Longrightarrow \quad |\sfG_k(x_k,y_k)-\sfG_k(x'_k,y'_k)| \leq \ve,
\end{equation}
for all $x_k,x_k',y_k,y_k'\in B_L(\star_k)$.

Then, up to extraction of a subsequence, there exists a continuous map $\sfG_\infty:X_\infty\times X_\infty \to [0,+\infty)$, and $\varepsilon_k \downarrow 0$, such that
\begin{equation}
|\sfG_k(x,y) - \sfG_\infty(f_k(x),f_k(y))|\leq \varepsilon_k, \qquad \forall\, x,y\in B_{L}(\star_k).
\end{equation}
In particular the conditions of Definition \ref{def:weakL1andregularity} hold, and if the functions $\beta_t:[0,\cD)\to \R$ are locally Lipschitz, then $(X_\infty,\sfd_\infty, \mm_\infty, \sfG_\infty)$ satisfies the $\MCP(\beta)$.
\item  \label{i:ComMCP3} With the same assumption as in item \ref{i:ComMCP2}, if the $\sfG_k$ are meek (cf.\ Definition \ref{def:meek}), then the generalized Bishop-Gromov Theorem \ref{thm:BG1} and Propositions \ref{prop:number-of-butt}--\ref{prop:pardoubling} hold for $(X_\infty,\sfd_\infty,\mm_{\infty}, \sfG_\infty)$.
\item  \label{i:ComMCP4} Assume that $(X_k,\sfd_k,\mm_k,\star_k)\to (X_\infty,\sfd_\infty,\mm_\infty,\star_\infty)$ in the pmGH topology, by means of approximating maps $f_k:X_k \to X_\infty$.  Assume that
\begin{itemize}
\item for every $x\in X_\infty$ there exist $x_{k}\in X_{k}$ such that $f_k(x_k) \to x$  in $X_\infty$ and
\begin{equation}\label{eq:BVlocUnif}
\sup_{k\in \N} \| \sfG_k (x_{k}, \cdot) \|_{{\rm BV}(B_{L} (\star_{k}))} <\infty, \qquad \forall\, L>0;
\end{equation}
\item the spaces $(X_k,\sfd_k,\mm_k)$ satisfy a weak $L^1$-Poincar\'e inequality, with uniform constants in $k$: there exist $c\geq 1, C>0, R>0$ such that for every $f\in \mathrm{LIP}_{\loc}(X_k,\sfd_k)$, every $r\in (0,R]$, and every $x\in X_k$, it holds
\begin{multline}\label{eq:L1weakpoinc}
\frac{1}{\mm_{k}(B^{X_{k}}_r(x))} \int_{B^{X_{k}}_r(x)}  \int_{B^{X_{k}}_r(x)} |f(y)-f(z)| \, \mm_{k}(\di y) \, \mm_{k}(\di z) \\
 \leq C r   \int_{B^{X_{k}}_{cr}(x)} \|\nabla f \|(y) \, \mm_{k}(\di y).
\end{multline}
\end{itemize}
Then, up to extraction of a subsequence, there exists a  map $\sfG_\infty:X_{\infty}\times X_{\infty} \to [0,+\infty]$ with $\sfG_\infty(x, \cdot)\in {\rm BV_{loc}}(X_{\infty}, \sfd_\infty,\mm_\infty)$ for every $x\in X_{\infty}$ such that  $\sfG_k\to \sfG_\infty$ in the $L^1_\loc$ sense of Definition \ref{def:weakL1andregularity}. If in addition $\sfG_\infty$ satisfies the regularity condition of Definition \ref{def:weakL1andregularity} and the functions $\beta_t:[0,\cD)\to \R$ are locally Lipschitz, then $(X_\infty,\sfd_\infty,\sfG_\infty,\mm_\infty)$ satisfies the $\MCP(\beta)$.
\end{enumerate}
\end{theorem}
\begin{proof}
From $\supp\mm_k=X_k$, the $\MCP(\beta)$ condition, and the local $\sfG$-boundedness assumption, each $(X_k,\sfd_k)$ is a locally compact geodesic space (see Corollary \ref{cor:totallyboundedclassic}).  We adopt the convention of Remark \ref{rem:pmghfkglob} for the pmGH convergence. 

\textbf{Proof of \ref{i:ComMCP1}.} Thanks to our assumptions on the $\sfG$-diameters, we have that the family of metric spaces $\{(B_L(\star_k),\sfd_k)\}_{k\in\N}$ are totally bounded for any fixed $L>0$, uniformly with respect to $k$ (cf.\ Corollary \ref{cor:totallyboundedclassic} and Remark \ref{rmk:totallybdduniform}); hence, the pre-compactness in the pGH topology follows from the classical characterization in terms of nets (cf.\ for example \cite[Thm.\@ 27.10, Def.\ 27.13]{Vil}). Using the local doubling inequality of Corollary \ref{cor:localdoubling}, we have that the assumption $\mm_k(B_1(\star_k))\leq M$ propagates to any scale, that is there exists $M = M(R)>0$ such that $\mm_k(B_R(\star_k))\leq M(R)$ for all $k\in \N$. Thus we obtain via a diagonal argument that $(f_k)_\sharp\mm_k$ is pre-compact in the weak topology.  The assumption $ \mm_k(B_1(\star_k))\geq \tfrac{1}{M}$ prevents the  limit measures to vanish identically.

Now, let  $(X_\infty,\sfd_\infty,\mm_\infty,\star_\infty)$ be any pmGH limit space. Then $(\supp\mm_\infty, \sfd_\infty, \mm_\infty)$ satisfies a local doubling inequality with the same (uniform) constants as the approximating sequence, and thus $(\supp\mm_\infty,\sfd_\infty)$ is proper. Moreover, $\supp\mm_\infty$ is a length space, as GH-limit of length spaces. Since any proper length space is geodesic, we conclude that $(\supp\mm_\infty, \sfd_\infty)$ is a geodesic space.

\textbf{Proof of \ref{i:ComMCP2}}. We can apply Arzelà-Ascoli theorem for pmGH converging sequences (see for instance \cite[Prop.\@ 27.20]{Vil}). Letting $f'_k : X_\infty \to X_k$ be the approximate inverses of $f_k$ (i.e.\@ $\sfd_\infty(f_k\circ f'_k(y),y) \to 0$ and $\sfd_k(f'_k\circ f_k(x),x) \to 0$ as $k\to \infty$), there exists a continuous $\sfG_\infty : X_\infty \times X_\infty \to [0,+\infty)$ such that
\begin{equation}
\lim_{k\to \infty}\sup_{x,y\in B_L(\star_\infty)}|\sfG_k(f_k'(x),f_k'(y))-\sfG_\infty(x,y)| =0, \qquad \forall\, L>0.
\end{equation}
Using the approximate inverse property and the equi-continuity of $\{\sfG_k\}_{k\in\N}$, we obtain
\begin{equation}\label{eq:strongconv}
\lim_{k\to\infty}\sup_{x,y \in B_L(\star_k)} |\sfG_k(x,y)-\sfG_\infty(f_k(x),f_k(y))| =0, \qquad \forall\, L >0.
\end{equation}
Clearly,  \eqref{eq:strongconv} implies the $L^1_\loc$ convergence of Definition \ref{def:weakL1andregularity}.

Furthermore, by construction, $\sfG_\infty$ is continuous. All the assumptions of Theorem \ref{thm:stabMCP} are then met, and thus the limit $(X_\infty,\sfd_\infty,\sfG_\infty,\mm_\infty)$ satisfies the $\MCP(\beta)$.

\textbf{Proof of \ref{i:ComMCP3}.} Fix $\mu_1 \in \Prob_{bs}^{*}(X,\sfd,\mm)$ and  $\bar{x}\in\supp\mm\setminus \supp \mu_{1}$. We claim that there exists $\nu \in \OptGeo(\delta_{\bar{x}},\mu_1)$ satisfying the $\MCP$ inequality \eqref{eq:defMCPbeta} and there exists a Borel set $\Gamma \subset \Geo(X)$ with $\nu(\Gamma)=1$, such that it holds
\begin{equation} \label{eq:meek1PfStab}
\sfG(\gamma_0,\gamma_t) = t\, \sfG(\gamma_0,\gamma_1), \qquad \forall\, t\in (0,1], \quad \forall\, \gamma \in \Gamma.
\end{equation}
Note that \eqref{eq:meek1PfStab} is weaker than the meek property for $\sfG$, but it is sufficient for the proof of the generalized Bishop-Gromov Theorem \ref{thm:BG1}, as explained in Remark \ref{rmk:weakmeek}.

We prove now the claim. By the argument used in the proof of Theorem \ref{thm:stabMCP}, we know that there exist $\mu_{1,k}  \in  \Prob_{bs}^{*}(X_{k},\sfd_{k},\mm_{k})$, $\bar{x}_k\in\supp\mm_k$ with $\bar{x}_{k}\notin  \supp \mu_{1,k}$, $\nu_k \in \OptGeo(\delta_{\bar{x}_k},\mu_{1,k})$, and $\nu \in \OptGeo(\delta_{\bar{x}},\mu_1)$  such that $\nu_k$ and $\nu$ satisfy the $\MCP$ condition and moreover (up to a extraction of a subsequence)   
\begin{equation}\label{eq:fketnuk}
(f_k)_{\sharp}\big( (\ee_t)_{\sharp} \nu_k \big)\rightharpoonup (\ee_t)_{\sharp} \nu,\qquad \forall\,t\in [0,1]. 
\end{equation}
Since each $\sfG_k$ is meek, for every $k\in \N$ there exists $\Gamma_k \subset \Geo(X_k)$ with $\nu_k(\Gamma_k)=1$, such that for all $\gamma \in \Gamma_k$ it holds
\begin{equation} \label{eq:meek1PfStabk}
\sfG_k(\gamma_0,\gamma_t) = t\, \sfG_k(\gamma_0,\gamma_1), \qquad \forall\, t\in (0,1].
\end{equation}
From \eqref{eq:fketnuk} we infer that for $\nu$-a.e.\ $\gamma$, there exist $\gamma_k\in \Gamma_k$ such that 
\begin{equation}\label{eq:fkgammakconvQ}
f_k(\gamma_{k,t})\to \gamma_t, \qquad \forall\, t\in [0,1]\cap \Q.
\end{equation}
 Since $\mu_1$ has bounded support, one can choose the
$\nu_k \in \OptGeo(\delta_{\bar{x}_k},\mu_{1,k})$ with the following property: there exists $R>0$ such that  $\cup_{t\in [0,1]} \supp \big( (\ee_t)_{\sharp} \nu_k \big)\subset B_R(\bar x_k)$. It follows that the geodesics $\gamma_k$ have a uniform bound on the length and thus are equi-Lipschitz.  Hence, by the Arzelà-Ascoli Theorem for varying spaces converging in the GH sense (see for instance \cite[Prop.\@ 27.20]{Vil}), we can promote \eqref{eq:fkgammakconvQ} to 
\begin{equation}\label{eq:fkgammakconv}
f_k(\gamma_{k,t})\to \gamma_t \; \text{ for every $t\in [0,1]$}.
\end{equation}
Combining \eqref{eq:meek1PfStabk} and \eqref{eq:fkgammakconv}  with the asymptotic equi-continuity of the Gauge functions \eqref{eq:GkEquicont}, we obtain for $\nu$-a.e.\ $\gamma\in\Geo(X)$
\begin{equation}
\sfG(\gamma_0,\gamma_t)= \lim_{k\to \infty} \sfG_k(\gamma_{k,0},\gamma_{k,t})= t \lim_{k\to \infty} \sfG_k(\gamma_{k,0},\gamma_{1,t})= t\,  \sfG(\gamma_0,\gamma_t), \qquad \forall\, t\in (0,1],
\end{equation}
as claimed.

\textbf{Proof of \ref{i:ComMCP4}.}  By assumption, the spaces $\{(X_{k}, \sfd_{k}, \mm_{k})\}_{k\in \N}$ satisfy local doubling (see Corollary \ref{cor:localdoubling}) and weak $L^{1}$-Poincar\'e inequalities, with constants uniform in $k\in \N$; moreover they converge in pmGH sense to $(X_\infty,\sfd_\infty,\mm_\infty,\star_\infty)$. Under such conditions, it is well-known that BV compactness in varying spaces holds (see for instance \cite[Thm.\@ 4.15]{KuwaeShyoiaTAMS208}). It follows that  there exists a  map $\sfG_\infty:X_{\infty}\times X_{\infty} \to [0,+\infty)$ with $\sfG_\infty(x, \cdot)\in {\rm BV_{loc}}(X_{\infty}, \sfd_\infty,\mm_\infty)$ for every $x\in X_{\infty}$ such that  $\sfG_k\to \sfG_\infty$ in the $L^1_\loc$ sense of Definition \ref{def:weakL1andregularity}. If in addition $\sfG_\infty$ satisfies the regularity condition of Definition \ref{def:weakL1andregularity} and the functions $\beta_t:[0,\cD)\to \R$ are locally Lipschitz, then the assumptions of Theorem \ref{thm:stabMCP} are met; thus we can conclude that the limit $(X_\infty,\sfd_\infty,\mm_\infty, \sfG_\infty)$ satisfies the $\MCP(\beta)$. 
\end{proof}

We next comment on the assumptions of Theorem \ref{thm:precompactMCP}. 

\begin{remark} The weak $L^{1}$-Poincar\'e inequality \eqref{eq:L1weakpoinc} is a very natural assumption in the framework of this paper. Indeed it was proved to hold with $c=2$, any $R>0$, for an explicit $C=C(K,N,R)>0$ in the class of  $\MCP(K,N)$ spaces $(X,\sfd, \mm)$ satisfying the following mild  ``a.e. non-branching'' assumption \cite[Cor.\ p.\ 28]{VRN}: the set
\begin{equation}
\{y\in X \mid \text{ there exist } \gamma^{1}\neq \gamma^{2}\in {\rm Geo}(X) \text{ such that } x=\gamma_{0}^{1}=\gamma_{0}^{2}, \, y=\gamma_{1}^{1}=\gamma_{1}^{2}\}
\end{equation}
has $\mm$-measure zero for $\mm$-a.e. $x\in X$. It is well-known that such a property holds for essentially non-branching $\MCP(K,N)$ spaces (see for instance \cite[Rmk.\ 2.6]{CavaMondAPDE}).
\end{remark}

\begin{remark}The classical framework of Lott-Sturm-Villani's theory corresponds to the choice of gauge function as $\sfG(x_{0}, \cdot):=\sfd(x_{0}, \cdot)$, for  $x_{0}\in X$. In this case, $\sfG(x_{0}, \cdot)$ is (trivially) 1-Lipschitz; thus the asymptotic equi-continuity assumed in item \ref{i:ComMCP2} holds trivially,   and the uniform local BV bounds \eqref{eq:BVlocUnif} assumed in  item \ref{i:ComMCP4} hold as a consequence of the uniform upper bound on the volume of metric balls. 
\end{remark}

Since $\CD(\beta,n)$ implies $\MCP(\beta)$, it is clear that \ref{i:ComMCP1} in Theorem \ref{thm:precompactMCP}  implies a corresponding pre-compactness result for the class of $\CD(\beta,n)$ spaces. Also a compactness result analogous to item \ref{i:ComMCP2} can be proved along the same lines, so  we only state it.

\begin{theorem}[Pre-compactness for $\CD$]\label{thm:precompact-CD}
Let $\beta$ be the distortion coefficients as \eqref{eq:defbeta}, and $n\in [1,+\infty)$. Let $\{(X_k,\sfd_k,\mm_k, \sfG_k,\star_k)\}_{k\in \bar{\N}}$ be a family of pointed gauge metric measure spaces, with $(X_k,\sfd_k,\mm_k,\star_k)\to (X_\infty,\sfd_\infty,\mm_\infty,\star_\infty)$ in the pmGH topology, by means of approximating maps $f_k:X_k \to X_\infty$  as in  Remark \ref{rem:pmghfkglob}. Assume that
\begin{itemize}
\item  for any $L>0$ there exists $D=D(L)>0$ such that $\sfG_k\leq D$ on $B_L(\star_k)\times B_L(\star_k) \cap \supp\mm_k\times \supp\mm_k$;
\item the family $\{\sfG_k\}_{k\in\N}$ is locally asymptotically equi-continuous, in the sense that for all $L>0$ and $\ve>0$ there exist $\delta=\delta(\ve,L)>0$ and $\bar{k}=\bar{k}(\ve,L)\in \N$ such that for all $k\geq \bar{k}$ it holds:
\begin{equation}
\sfd_k(x_k,x_k')\leq \delta, \quad \sfd_k(y_k,y_k')\leq \delta \qquad \Longrightarrow \qquad |\sfG_k(x_k,y_k)-\sfG_k(x'_k,y'_k)| \leq \ve,
\end{equation}
for all $x_k,x_k',y_k,y_k'\in B_L(\star_k)$.
\end{itemize}
 Then, up to extraction of a subsequence, there exists a continuous map $\sfG_\infty:X_\infty\times X_\infty \to [0,+\infty)$, and $\varepsilon_k \downarrow 0$, such that
\begin{equation}
|\sfG_k(x,y) - \sfG_\infty(f_k(x),f_k(y))|\leq \varepsilon_k, \qquad \forall\, x,y\in B_{L}(\star_k).
\end{equation}
In particular the conditions of Definition \ref{def:weakL1andregularity-CD} hold. If the functions $\beta_t:[0,\cD)\to \R$ are locally Lipschitz, and $(X_k,\sfd_k,\mm_k, \sfG_k)$ satisfy the $\CD(\beta,n)$, then also $(X_\infty,\sfd_\infty,\mm_\infty, \sfG_\infty)$ satisfies the $\CD(\beta,n)$.
\end{theorem}

\begin{remark}
Let $(X,\sfd,\mm)$ be a Carnot group, equipped with its Haar measure. It is well-known that they support Poincaré inequalities (see e.g.\@ \cite[Prop.\@ 11.17]{SobmetPoincare} and references within). If the Carnot group is ideal (which means, in the special case of Carnot groups, that it is also fat and thus of step $\leq 2$) and it is equipped with a natural gauge function (see Section \ref{sec:nonsmooth}), then $\sfG$ is locally bounded (see Figure \ref{fig:implications}). Furthermore $\sfd$ is locally semiconcave in charts outside of the diagonal \cite{riffordbook}. As a consequence one can prove that natural gauge functions satisfy
\begin{equation}\label{eq:labelfake}
\sup_{x\in \Omega} \| \sfG(x,\cdot)\|_{BV(\Omega)} < +\infty,
\end{equation}
for any bounded set $\Omega$. This means that one can apply the compactness Theorem \ref{thm:precompactMCP}\ref{i:ComMCP4} to sequences of ideal Carnot groups equipped with natural gauge functions provided that the Poincaré inequality, the $\mathrm{BV}$ bound on $\sfG$, and the local $\sfG$-diameter bound hold with uniform constants w.r.t.\@ to the sequence.
\end{remark}
\section{Vector-valued gauge functions}\label{sec:vectorial}

In this section we extend the previous theory to the case of a vector-valued gauge function. Let $\R^m_{+}:=[0,+\infty)^m$. We consider the real projective ``positive'' space:
\begin{equation}
\RP^m_+ :=  \Big(\R^{m+1}_{+} \setminus \{0\}\Big)/ \sim\,,
\end{equation}
where $x\sim y$ if $x=\lambda y$ for $\lambda > 0$.  To convey the correct intuition, in this section $\RP^m_{+}$ shall be thought as a compactification of $\R^m_{+}$, adding one point at infinity for every direction, rather than space of directions in $\R^{m+1}_{+}$.

The map $x \mapsto [x:1]$ embeds $\R^m_+$ in $\RP^m_+$.  For $\theta\in  \RP_+^m\setminus  \R^m_+$, we set $|\theta|=+\infty$; otherwise, if $\theta = [x:1]$ for some $x\in \R^m_+$, $|\theta|$ is the standard norm of the vector $x \in  \R^m_+$.

Finally, dilations on $\R^m_+$ extend to $\RP^m_+$: if $\theta = [x:a]$ for $x\in \R^m$ and $a\in \R$, and $t>0$, we denote $t\theta = [t x: a]$.

A vector valued gauge function on a metric measure space $(X,\sfd,\mm)$ is  a Borel function
\begin{equation}\label{eq:defGVec}
\sfG : X\times X \to \RP^m_+\, .
\end{equation}

Let $N\in [1,+\infty)$, and $\sfs:[0,+\infty)^{m}\to \R$ be a continuous function such that
\begin{equation}\label{eq:fftcNVec}
\sfs(\theta)=c  |\theta|^{N}+o(|\theta|^{N}) \qquad \text{ for some } c >0.
\end{equation}

The \emph{positivity domain} of $\sfs$ is the largest open and star-shaped (with respect to the origin $0\in \R^m$) subset $\DOM \subseteq \R^m_+$ such that $\sfs>0$ on $\DOM\setminus\{0\}$. Since $\DOM$ is star-shaped then, for every $\theta \neq 0\in \R^m_+$, there is a unique non-zero limit point of $\DOM$ in $\RP^m_+$, along the  half-line defined by $\theta$, which we denote by $\cD_{\theta} \in \RP^m_+$ (it can be either in the affine part $\R^{m}_+\subset \RP^m_+$ or at infinity, see Figure \ref{fig:DOM-cdtheta}).

\begin{figure}[t]
\centering
\includegraphics[width=.85\textwidth]{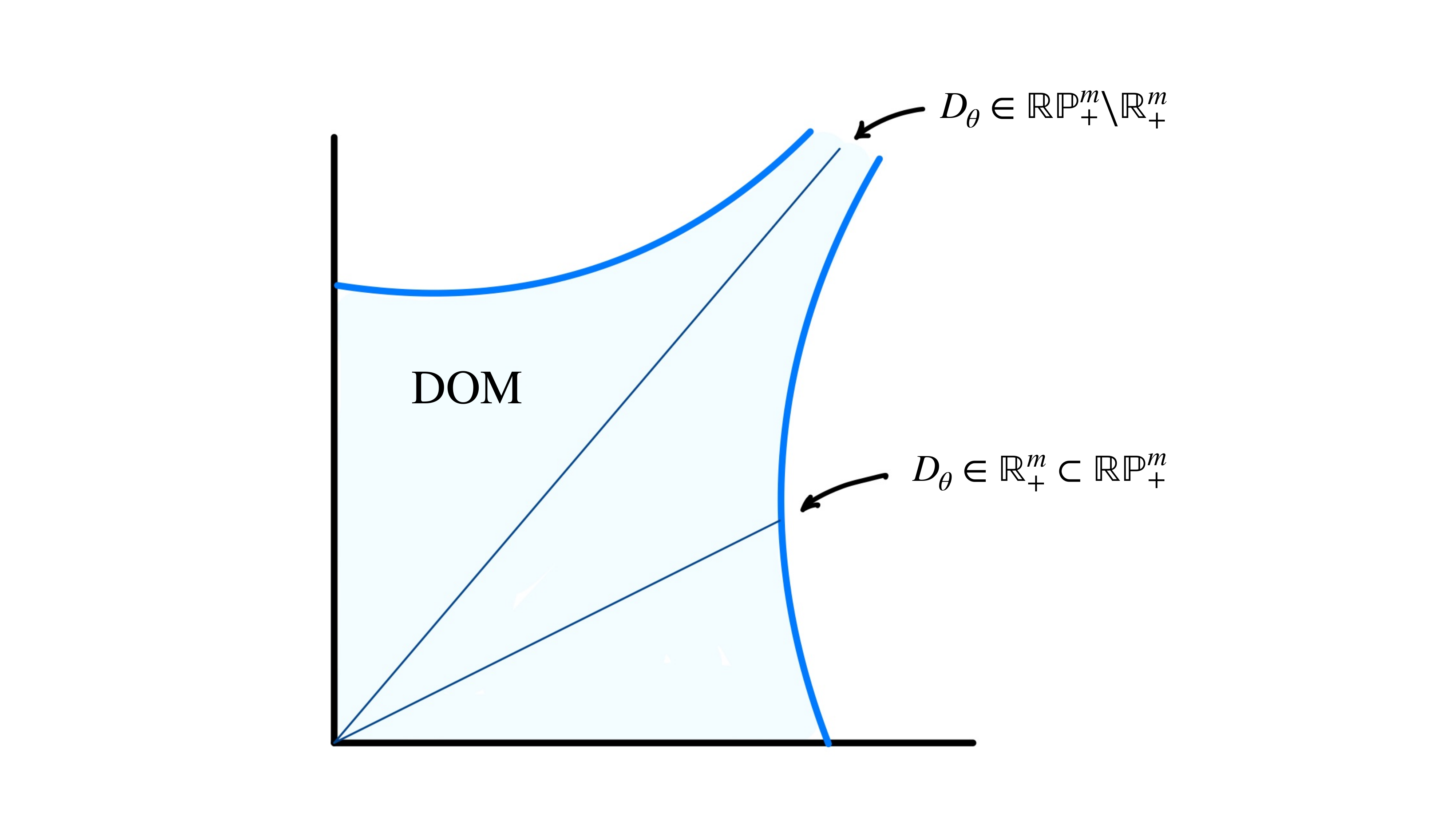}
\caption{An example of positivity domain $\DOM$ and the corresponding $\cD_{\theta} \in \RP^m_+$.}\label{fig:DOM-cdtheta}
\end{figure}

Define the \emph{distortion coefficient} $\beta_{(\cdot)}(\cdot):[0,1]\times \RP^{m}_+\to [0,+\infty]$ as
\begin{align}\label{eq:defbetaVec}
 \beta_{t}(\theta):=
 \begin{cases}
 t^{N}\quad &|\theta| =0,  \\
\dfrac{\sfs(t\theta )}{\sfs(\theta)}& |\theta|\neq 0 \text{ and } \theta \in \mathrm{DOM},\\
 \displaystyle  \liminf_{\phi \to \cD_{\theta}^-}  \dfrac{\sfs(t\phi)}{\sfs(\phi)} & \theta \notin \DOM,
 \end{cases}
\end{align}
where the liminf is a directional limit, in the direction of $\theta$, from the inside of $\DOM$.

For a scalar gauge function, in some statements it was useful to consider the case when $\theta\mapsto \beta_t(\theta)$ is monotone. When  the gauge function is vector-valued, the appropriate counterpart is the \emph{radial monotonicity}:  we say that $\beta_t(\cdot)$ is \emph{radially non-increasing} (resp.\@ \emph{radially non-decreasing}) if, for every fixed $\theta\in \DOM\setminus \{0\}$ and $t\in (0,1)$, the function  $\lambda\mapsto \beta_t(\lambda\theta)$ is non-increasing  (resp.\@ non-decreasing).  Likewise, we say that  $\sfs(\cdot)$ is \emph{radially real-analytic on $\overline{\DOM}$} if for any $\theta \in \R^m_+\setminus\{0\}$ the function $[0,|\cD_{\theta}|] \ni t \mapsto \sfs(t\hat{\theta})$ is real-analytic, where $\hat{\theta}:= \theta/|\theta|$.

Upon a suitable reformulation, we have the following elementary properties that generalize the ones of Proposition \ref{prop:propertiesbeta} to the vector case.

\begin{proposition}\label{prop:propertiesbetaVec}
Any distortion coefficient satisfies the following:
\begin{enumerate}[label=(\roman*)]
\item \label{i:propertiesbetaVec-1}  $\beta_0(\theta)=0$ and $\beta_1(\theta) = 1$ for all $\theta \in \RP^m_+$;
\item \label{i:propertiesbetaVec0}  if $|\cD_\theta|<+\infty$, then $\beta_t(\theta) = +\infty$ for all $\theta \notin\DOM$ and $t\in (0,1)$;
\item \label{i:propertiesbetaVec1} $\beta_{t}(\theta)=0$ for some $|\theta| < +\infty$ if and only if $t=0$;
\item \label{i:propertiesbetaVec2} for every $t\in [0,1]$,  $\theta_{j}, \theta\in  \R^m_+$,  if $\theta_{j}\to \theta$ along the direction of $\theta$, then $\beta_{t}(\theta_{j})\to \beta_{t}(\theta)$ in $[0,+\infty]$;  if moreover $\theta, \theta_j\in \DOM$ for all $j\in \N$ and $\theta_{j}\to \theta$, then $\beta_{t}(\theta_{j})\to \beta_{t}(\theta)$ in $[0,+\infty)$, for every $t\in [0,1]$;
\item \label{i:propertiesbetaVec3} $\beta$ is  continuous and finite when restricted to $[0,1]\times \DOM $;
\item \label{i:propertiesbetaVec4} for all fixed $t\in (0,1]$,  the distortion coefficient $\beta_t(\cdot)$ is bounded from below away from zero on any bounded set of $\R^m_+$;
\item \label{i:propertiesbetaVec5} assume that $\DOM$ is a bounded subset of $\R^m_+$, that $\sfs \in C^\infty(\overline{\DOM})$ and $\sfs$ is radially real analytic on $\overline{\DOM}$. Then there exists $N'\geq N$ such that $\beta_t(\theta) \geq t^{N'}$ for all $t\in [0,1]$ and $\theta \in\RP^{m}_+$.
\end{enumerate}
\end{proposition}

\begin{proof}
The only item requiring a proof is \ref{i:propertiesbetaVec5}. Denote by $\sfs^{(k)}_{rad}(x):=\left.\frac{d^k}{dt^k}\right|_{t=1}\sfs(tx)$ the $k^{th}$-order derivative of $\sfs$ in the radial direction.  Assume by contradiction that there exists a sequence $(x_k)\subset \partial \DOM$ such that $\sfs^{(j)}_{rad}(x_k)=0$ for all $j\leq k$. Since by assumption $\DOM$ is bounded and $\sfs$ satisfies  \eqref{eq:fftcNVec} then, up to a non-relabeled subsequence, there exists $x_\infty\in \partial \DOM {\setminus\{0\}}$ such that $x_k\to x_\infty$.  Then, by the smoothness of $\sfs$, we infer that $\sfs^{(k)}_{rad}(x_\infty)=0$ for all $k\in\N$. But then, since $\sfs$ is radially real-analytic on $\overline{\DOM}$, it must be that $\sfs\equiv 0$ on the half line $\{t x_\infty \mid t>0\}\cap \overline\DOM$, contradicting \eqref{eq:fftcNVec}.
\end{proof}

The definitions of $\CD(\beta,n)$ and $\MCP(\beta)$ for a metric measure space $(X,\sfd,\mm)$  endowed with a vector-valued gauge $\sfG:X\times X\to \RP^m_+$ are verbatim as in Definition \ref{def:CDMCPbetan}, once the distortion coefficients are as in \eqref{eq:defbetaVec}.

\subsection{Geometric consequences}\label{sec:geometricconsequencesVec}

\subsubsection*{Generalized Brunn-Minkowski}

The generalized Brunn-Minkowski inequality for $\CD(\beta, n)$ spaces  (Theorem \ref{thm:Brunn-Mink}) and the generalized half-Brunn-Minkowski inequality for $\MCP(\beta)$ spaces (Theorem \ref{thm:half-Brunn-Mink}) hold with verbatim the same statements and proofs in case of a vector-valued gauge function.

\subsubsection*{Gauge diameter estimate}\label{sec:diamVec}

The following extends Definition \ref{def:Gdiam} to the vector-valued case.

\begin{definition}\label{def:GdiamVec}
The $\sfG$-diameter for a non-empty set $S\subset X$ is defined as
\begin{equation}
\diam_\sfG(S)=  \sup_{x\in S} \left(\mm\shortminus\esssup_{y\in S, \,  y\neq x} |\sfG(x,y)|\right).
\end{equation}
A non-empty set $S\subset X$ is $\sfG$-bounded if it has finite $\sfG$-diameter.
\end{definition}

Recall  that $\DOM\subseteq \R^m_+\subset \RP^m_+$ is defined as the largest open and star-shaped (with respect to the origin) subset such that $\sfs>0$ on $\DOM\setminus \{0\}$.

\begin{proposition}[$\sfG$-diameter estimate]\label{prop:GdiamestimateVec}
Let $(X,\sfd,\mm)$ be a m.m.s.\@ with gauge function $\sfG$. Let $\beta$ as in \eqref{eq:defbetaVec}, and assume that $\beta_t(\theta)=+\infty$ for all $\theta \notin \DOM$, $t\in (0,1)$. Note that this is in particular the case if $\DOM$ is bounded.
\begin{itemize}
\item If $(X,\sfd,\mm,\sfG)$ satisfies the $\MCP(\beta)$, then for all $x\in\supp\mm$ it holds
\begin{equation}
\sfG(x,y)\in \DOM, \qquad \mm\text{-a.e. $y \in X$,\, $y\neq x$}.
\end{equation}
In particular, $\diam_{\sfG}(\supp\mm) \leq \diam_{\R^m} (\DOM)$.

\item If $(X,\sfd,\mm,\sfG)$ satisfies the $\CD(\beta,n)$ for some $n\in [1,+\infty)$, then
\begin{equation}
\sfG(x,y) \in \DOM,\qquad \pi\text{-a.e. }(x,y)\in X\times X,
\end{equation}
for all $\pi = (\ee_0, \ee_1)_\sharp\nu$, where $\nu \in \OptGeo(\mu_0,\mu_1)$ is such that inequality  \eqref{eq:defCDbetan} holds.
\end{itemize}
\end{proposition}
\begin{proof}
We argue by contradiction. Fix $R>0$, $x\in\supp\mm$ and let 
\begin{equation}
A=\{y\mid  y\neq x, \, \sfG(x,y)\notin \DOM \}\cap B_R(x).
\end{equation}
Assume by contradiction that $\mm(A)>0$. Then we apply the half-Brunn-Minkowski inequality to the set $A$. Notice that $\beta_t(x,A) = +\infty$ for all $t\in (0,1)$, and that $A_t \subset (B_R(x))_t \subset B_{tR}(x)$. We obtain hence
\begin{equation}
\mm(B_{tR}(x)) \geq \bar{\mm}(A_t) = +\infty,\qquad \forall\, t\in(0,1),
\end{equation}
contradicting the local finiteness of $\mm$.

To prove the second part of the proposition, we argue in a similar way: assume that there exists $\nu\in \OptGeo(\mu_0,\mu_1)$ as in the statement (in particular $\mu_0,\mu_1$ have bounded and disjoint supports) such that 
\begin{equation}
\pi(\{(x,y)\in X\times X \mid \sfG(x,y) \notin \DOM \})>0.
\end{equation}
Using \eqref{eq:defCDbetan} we obtain that, for the corresponding $W_2$-geodesic $\mu_t = (\ee_t)_\sharp \nu$ it holds
\begin{equation}
\U_n(\mu_t|\mm) = +\infty,\qquad \forall\, t\in(0,1).
\end{equation}
This implies that $\mm(\supp\mu_t)=+\infty$. However, since $\supp\mu_0$ and $\supp\mu_1$ are bounded, then also $\supp\mu_t$ is bounded, yielding a contradiction.
\end{proof}

\subsubsection*{Local doubling}

The following proposition can be proved along the same lines of Proposition \ref{prop:doublingsimple}.

\begin{proposition}[Local doubling]\label{prop:doublingsimpleVec2}
Let $(X,\sfd,\mm)$ be a m.m.s.\@ with gauge function $\sfG$, satisfying the $\MCP(\beta)$ with $\beta$ as in \eqref{eq:defbetaVec}. Then any $\sfG$-bounded Borel subset $S\subseteq \supp\mm$ satisfies the following inequality: for all $t\in (0,1)$ there exists $C_{S,t}>0$ such that
\begin{equation}\label{eq:doublingsimpleVec2}
\mm(B_{r}(x_0)\cap S)\leq C_{S,t}\cdot \mm(B_{tr}(x_0)),\qquad \forall\, r\geq 0,\;  \forall\, x_0\in S.
\end{equation}
The constant $C_{S,t}$ can be estimated in terms of $\beta$ and $\diam_{\sfG}(S)$:
\begin{equation}\label{eq:crRVec2}
\frac{1}{C_{S,t}} =\inf\{\beta_{t}(\theta)\, :\,  |\theta| \in [0,\diam_{\sfG}(S)] \}  \in (0,t^N].
\end{equation}
\end{proposition}

Corollaries \ref{cor:localdoubling}, \ref{cor:classicaldoubling} and \ref{cor:totallyboundedclassic} hold with verbatim the same statements and proofs replacing, when necessary, the monotonicity of $\beta_t$ with the radial monotonicity.

\subsubsection*{Geodesic dimension estimates}

Recall the role of the parameter $N\in [1,+\infty)$ in \eqref{eq:fftcNVec} and the Definition \ref{def:geodim} of geodesic dimension $\NN(x)$ at a point $x$  of a metric measure space. Since the proof of the geodesic dimension estimate (see Theorem \ref{thm:GeodDimEst})
\begin{equation}
\NN(x) \leq N,\qquad  \text{for all }  x\in \supp\mm,
\end{equation}
built on top of the half-Brunn Minkowski inequality, which holds in the case of a vector-valued gauge function, the same result holds with analogous proof,  provided that we assume the following non-triviality condition
\begin{equation}\label{eq:GNonTrivVec}
\mm(\{z\in X \mid  z\neq x, \,  |\sfG(x, z)|<+\infty\})>0, \qquad \forall\, x\in \supp\mm.
\end{equation}
which replaces \eqref{eq:GNonTriv} in the vector case.

\subsubsection*{Generalized Bishop-Gromov inequality}

Let $(X,\sfd,\mm)$ be a  metric measure space with gauge function $\sfG:X\times X\to \RP^{m}_+$.  In the following, for $r\in \R^m_+,\sfG\in \RP^m_+$, the notation
\begin{equation}
\sfG \leq r, \qquad \text{(resp.\@ $\sfG\geq r$)},
\end{equation}
means that $\sfG  =[\sfG_1:\dots:\sfG_m:1]$ is in the affine part $\R^m_+ \subset \RP^m_+$ and $\sfG_i \leq r_i$  (resp.\@ $\sfG_i \geq r_i$) for all $i=1,\dots,m$.  In particular, the notation $\sfG \in (r,R)$ for $r,R\in \R^{m}_{+}$ means that $\sfG\leq R$ and $r\leq\sfG$. In other words, the inequalities for vector-valued functions are meant component-wise. Furthermore, for $r\in \R^m_+\setminus 0$, we set $\hat{r} := \tfrac{r}{|r|}$; if $r=0$ then $\hat{r} =0$. 

\begin{definition}\label{def:gaugesetsVec}
For a point $x_{0}\in \supp \mm$ and for $r\in \R^m_+$, $\rho\geq 0$, we set:
\begin{align}
\rv_{\sfG}(x_0,r,\rho)& :=\mm\left(\{x\in X \mid \sfG(x_0,x)\leq r,\;  \sfd(x_0,x)\leq \rho\}\right), \\
\rs_{\sfG}(x_0,r,\rho)& :=\limsup_{\delta\downarrow 0}\frac{1}{\delta} \mm\left(\left\{x\in X \mid \sfG(x_0,x) \in (r-\delta\hat{r}, r],\; \sfd(x_0,x)\leq \rho\right\}\right), \quad r\neq 0.
\end{align}
Consider also the following measure of ``gauge balls'' and ``gauge spheres'':
\begin{align}
\rv_{\sfG}(x_0,r) & :=  \mm\left( \{ x\in X \mid \sfG(x_0,x)\leq r,\; \sfd(x_0,x)\leq |r|\}\right), \label{eq:defvGrVec} \\
\rs_{\sfG}(x_0,r) & := \limsup_{\delta\downarrow 0}\frac{1}{\delta}\left[\rv_{\sfG}(x_0,r)-\rv_{\sfG}(x_0,r-\delta \hat{r})\right], \quad r\neq 0. \label{eq:defsGrVec}
\end{align}
Notice that $\rv_{\sfG}(x_0,r) = \rv_{\sfG}(x_0,r,|r|)$ but $\rs_{\sfG}(x_0,r) \neq \rs_{\sfG}(x_0,r,|r|)$.
\end{definition}

The definition of the meek property is formulated verbatim for vector-valued gauge functions, see Definition \ref{def:meek}.

The proof of the generalized Bishop-Gromov inequality for a vector-valued gauge function can be adapted from the scalar case (Theorem \ref{thm:BG1}), by arguing along rays.

\begin{theorem}[Generalized Bishop-Gromov] \label{thm:BG1Vec}
Let $(X,\sfd,\mm)$ be a m.m.s.\@ endowed with a gauge function $\sfG$ satisfying the $\MCP(\beta)$ condition, with $\beta$ as in \eqref{eq:defbetaVec}.
Assume that 
\begin{itemize}
\item $\supp \mm$ is not a singleton;
\item $\sfG$ is meek;  (see Remark \ref{rmk:weakmeek} for a slightly weaker condition)
\item for all $x\in \supp\mm$ it holds (see Remark \ref{rem:onCond3.15bis} for a slightly weaker condition)
\begin{equation}\label{eq:GfiniteVec}
\mm(\{z\in X \mid x\neq z, \, |\sfG(x, z)|=+\infty\})=0.
\end{equation}
\end{itemize}
Then for every $x_0\in\supp\mm$ the following properties hold:
\begin{enumerate}[label = (\roman*)]
\item \label{i:BG1Vec} $\mm(\{x \mid \sfd(x_0,x) = r\})=0$, for every $r >0 $;
\item \label{i:BG2Vec} $\mm(\{x_0\})=0$; in particular, the half-Brunn-Minkowski inequality \eqref{eq:halfBrunMink} holds for every $\bar{x}\in \supp\mm$ and Borel set $A\subset X$ with $\mm(A)>0$;
\item \label{i:BG3Vec} $\mm(\{x \mid \sfG(x_0,x) = r\})=0$, for every $r \in \R^m_+$, $r\neq 0$;
\item \label{i:BG4Vec}  For every $\rho>0$, and  $r,R\in \mathbb{R}^{m}_+$ lying on the same direction $\theta\in \mathbb{S}^{m-1}_+$ with $0<|r|\leq |R|<\cD_{\theta}$ , the following holds:
\begin{align}
\frac{\rs_{\sfG}(x_{0}, r,\rho)}{ \rs_{\sfG}(x_{0}, R,\rho)} \geq \frac{|R|}{|r|}  \frac{\sfs( r )}{ \sfs(R)}. \label{eq:BisGr1VecS}
\end{align}
Furthermore, it holds
 \begin{equation} \label{eq:BisGr1VecV}
  \frac{\rv_{\sfG}(x_{0}, r,\rho)-\rv_{\sfG}(x_{0}, 0,\rho)}{ \int_{0}^{|r|} \sfs(t\theta)/t \, \di t} \geq \frac{\rv_{\sfG}(x_{0}, R,\rho)-\rv_{\sfG}(x_{0}, 0,\rho)}{ \int_0^{|R|} \sfs(t\theta)/t\, \di t}; 
 \end{equation}
\item \label{i:BG5Vec}  Assume that  $\theta \mapsto \beta_t(\theta)$ is radially  non-increasing for all $t\in [0,1]$, or that for all $x_0\in\supp\mm$ it holds $|\sfG(x_0,\cdot)| \geq \sfd(x_0,\cdot)$ $\mm$-a.e.. Then, for every $x_{0}\in \supp\mm$, $r,R\in \mathbb{R}^{m}_+$ lying on the same direction $\theta\in \mathbb{S}^{m-1}_+$ with $0<|r|\leq |R|<\cD_{\theta}$, the following holds:
\begin{align}
\frac{\rs_{\sfG}(x_{0}, r)}{ \rs_{\sfG}(x_{0}, R)} \geq \frac{|R|}{|r|} \frac{\sfs(r)}{ \sfs(R)}. \label{eq:BisGr2VecA} 
\end{align}
Furthermore, it holds:
 \begin{equation} \label{eq:BisGr2VecV}
 \frac{\rv_{\sfG}(x_{0}, r)}{ \int_{0}^{|r|}   \sfs(t\theta)/t \, \di t} \geq \frac{\rv_{\sfG}(x_{0}, R)}{ \int_0^{|R|} \sfs(t\theta)/t \, \di t};
 \end{equation}

\item  \label{i:BG6Vec} For every compact set $K_0\subseteq \R^m_+$, and $r,R\in K_0$ there exists a constant 
\begin{equation}
\frac{1}{C_{|r|/|R|,0}}:=\inf\{\beta_{|r|/|R|}(\theta) \mid \theta\in K_0 \} \in (0,(|r|/|R|)^N],
\end{equation}
such that,  if $r,R \in K_0$ lie along the same direction, it holds
\begin{equation}
 \rs_{\sfG}(x_0,R) \leq \frac{|r|}{|R|} C_{|r|/|R|,0}\cdot \rs_{\sfG}(x_0,r).
\end{equation}
As a consequence, the following doubling estimate for gauge balls holds: there exists a constant $C_0 = C_{1/2,0}>0$ such that
\begin{equation}
\rv_{\sfG}(x_0,2r)\leq C_0 \cdot \rv_{\sfG}(x_0,r), \qquad \forall\, r \in K_0/2.
\end{equation}
\end{enumerate}
\end{theorem}

The estimate on the maximal number of disjoint gauge balls (Proposition \ref{prop:number-of-butt}) holds with verbatim the same statement and proof in case of a vector-valued gauge function (replacing $\sfG$ by $|\sfG|$ in its statement and its proof).

\begin{remark}\label{rem:onCond3.15bis}
For the proof of Thorem \ref{thm:BG1Vec} condition \eqref{eq:GfiniteVec} can be weakened as follows: for all $x\in\supp\mm$  and all $r>0$ it holds
\begin{equation}\label{eq:GfiniteVec2}
\mm(\{z\in X\mid   \sfd(x,z)=r,\, |\sfG(x, z)|=+\infty\})=0.
\end{equation}
Both conditions \eqref{eq:GfiniteVec} and  \eqref{eq:GfiniteVec2} follow from the $\MCP(\beta)$, provided that $\beta_t(+\infty) = +\infty$, see Remark \ref{rmk:MCPimpliesGNonTriv}.
\end{remark}
\subsubsection*{Parametric doubling for gauge balls}
 
Finally, the parametric doubling property (Proposition \ref{prop:pardoubling}) holds also in the case of vector-valued gauge functions. The proof follows verbatim making use of the appropriate vector-valued versions of the Bishop-Gromov theorem (Theorem \ref{thm:BG1Vec}) and gauge diameter estimates (Proposition \ref{prop:GdiamestimateVec}).
\begin{proposition}[Parametric doubling property]\label{prop:pardoublingVec}
With the same assumptions of Theorem \ref{thm:BG1Vec}, for every bounded set $K\subseteq\DOM$ there exists $C_K=C_K(\beta)\geq 1$ such that
\begin{equation}
\frac{\rv_{\sfG}(x, 2r,\rho)-\rv_{\sfG}(x,0,\rho)}{\rv_{\sfG}(x, r,\rho)-\rv_{\sfG}(x,0,\rho)}\leq  2^N C_K, \qquad \forall\, r \in K,\quad\forall\, \rho>0, \quad \forall\, x\in \supp\mm,
\end{equation}
where $N>1$ is given by \eqref{eq:fftcNVec}. Furthermore, $C_K \downarrow 1$ as $\diam_{\R^m}(K)\downarrow 0$.

If $\theta\mapsto \beta_t(\theta)$ is radially monotone non-increasing for all $t\in [0,1]$, or if for all $x_0\in\supp\mm$ it holds $|\sfG(x_0,\cdot)| \geq \sfd(x_0,\cdot)$ $\mm$-a.e., then the previous inequality holds true for gauge balls, with the same constants, that is
\begin{equation}
\frac{\rv_{\sfG}(x, 2r)}{\rv_{\sfG}(x, r)}\leq  2^N C_{K}, \qquad \forall\, r\in K, \quad \forall\, x\in \supp\mm.
\end{equation}
\end{proposition}

\subsection{Stability and compactness}

We adopt verbatim the same notion of $L^1_{\loc}$ convergence of gauge functions and regularity condition for the limit gauge function (resp.\@ over plans) as in Definition \ref{def:weakL1andregularity} (resp.\@ as in Definition  \ref{def:weakL1andregularity-CD}), in case of vector-valued gauge functions.
Under this convention, verbatim the same statements for the stability of $\MCP(\beta)$ (i.e.\@ Theorem \ref{thm:stabMCP}) and of $\CD(\beta,n)$ (i.e.\@ Theorem \ref{thm:stabCD}) hold. The proofs hold verbatim, with the only difference being the definition of the truncated gauge function $\sfG^{\ve}$, which is now vector-valued; however, in order to follow the arguments, it is enough to replace \eqref{eq:defTrunc} with 
\begin{equation}\label{eq:defTruncGVec}
\sfG^{\ve}(x,y):=
\begin{cases}
 \sfG & \text{ if } |\sfG|\leq D+\ve, \\
 \frac{D+\ve}{|\sfG|} \cdot \sfG  & \text{ if } |\sfG|>D+\ve.
\end{cases}
\end{equation}

Also the statements and proofs of the pre-compactness criteria for $\MCP(\beta)$ (i.e.\@ Theorem \ref{thm:precompactMCP}) and $\CD(\beta,n)$ (i.e.\@ Theorem \ref{thm:precompact-CD}) hold verbatim in case of vector-valued gauge functions. 
 
\section{Natural gauge functions}\label{sec:nonsmooth}

The purpose of this section is to introduce natural gauge functions for metric spaces equipped with an additional reference metric. This construction is inspired by sub-Riemannian geometry, where the sub-Riemannian distance can always be seen as the restriction of a Riemannian length structure on a subset of admissible curves.

\begin{definition}[Reference and extension]\label{def:extension}
Let $(X,\sfd)$ be a complete metric space. We say that a complete length metric $\sfd_R$ on $X$ is a \emph{reference} for $\sfd$ if $\sfd_R\leq \sfd$, and $\sfd,\sfd_R$ induce the same topology on $X$. Furthermore, we say that the reference metric $\sfd_R$ is an \emph{extension} of $\sfd$ if for all $\sfd$-rectifiable curves $\gamma$ it holds
\begin{equation}\label{eq:extension}
L_{\sfd_R}(\gamma) = L_{\sfd}(\gamma).
\end{equation}
\end{definition}
\begin{remark}
The inequality $\sfd_R\leq \sfd$ implies that each rectifiable curve in  $(X,\sfd)$ is rectifiable also in  $(X,\sfd_R)$. The fact that $\sfd_R$ is complete is not implied by the other requirements. For example take the standard Carnot-Carathéodory metric on the Heisenberg group. Extend it to a Riemannian metric by setting $\|\partial_z\| = 1/(1+z^2)$. Any point on the $z$-axis is at extended distance $\leq \pi/2$ from the origin, and the extension is not complete.
\end{remark}

\begin{definition}[Asymptotic Lipschitz number]\label{def:asLipnumb}
Let $(X,\rho)$ be a metric space. For a Borel function $f: X \to [0,+\infty]$ we define the asymptotic Lipschitz number with respect to $\rho$ at $x\in X$ as
\begin{equation}
\mathrm{Lip}_a^{\rho} [f](x): =\limsup_{z,w\to x} \frac{|f(z)-f(w)|}{\rho(z,w)},
\end{equation}
with the convention that $\mathrm{Lip}_a^{\rho} [f](x) = 0$ if $x$ is an isolated point.
\end{definition}

\subsection{The \texorpdfstring{$\sfD$}{D} function}\label{sec:D}

We introduce the building block for natural gauge functions.

\begin{definition}[The $\sfD$ function]\label{def:D}
Let $(X,\sfd)$ be a metric space with reference $\sfd_R$. Let $c: = \frac{1}{2}\sfd^2$. We define:
\begin{equation}
\sfD :X\times X \to [0,+\infty], \qquad \sfD(x,y) := \mathrm{Lip}_a^{\sfd_R} [c(\cdot,y)](x).
\end{equation}
\end{definition}

For the next proposition we say that $f,g:X\times X \to [0,+\infty]$ are locally equivalent if for any $o\in X$ there exists a neighborhood $\mathcal{O}$ of $o$ and a constant $C=C_{\mathcal{O}}>0$ such that
\begin{equation}
C^{-1} g(x,y)\leq f(x,y)\leq C g(x,y),\qquad \forall\, x,y\in\mathcal{O}.
\end{equation}

\begin{proposition}[Properties of the $\sfD$ function]\label{prop:propertiesD}
Let $(X,\sfd)$ be a length metric space with reference $\sfd_R$. The $\sfD$ function has the following properties:
\begin{enumerate}[label = (\roman*)]
\item \label{i:propertiesD1} it holds  $\sfd \leq \sfD$;
\item \label{i:propertiesD2} $\sfd_R$ and $\sfd$ are locally equivalent if and only if $\sfD$ and $\sfd$ are locally equivalent;
\item \label{i:propertiesD3} $\sfD= \sfd$ if and only if the reference is the trivial one, that is $\sfd_R=\sfd$;
\item \label{i:propertiesD4} assume that $\sfd_R$ is an extension of $\sfd$; then $\sfd$ is locally equivalent to $\sfd_R$ if and only if $\sfD=\sfd$;
\item \label{i:propertiesD5} For every fixed $y\in X$, the map $\sfD(\cdot, y) :X  \to [0,+\infty]$ is upper semi-continuous.
\end{enumerate}
\end{proposition}
\begin{proof}
\textbf{Proof of \ref{i:propertiesD1}.} By definition $\sfd_R \leq \sfd$. Therefore for all $x,y\in X$ it holds
\begin{equation}
\sfD(x,y)=\mathrm{Lip}_a^{\sfd_R}[c(\cdot,y)](x)\geq  \mathrm{Lip}_a^{\sfd}[c(\cdot,y)](x) =  \sfd(x,y),
\end{equation}
where we used the assumption that $(X,\sfd)$ is a length metric space in the last identity.

\textbf{Proof of \ref{i:propertiesD2}.} Observe first that for any $z,w,y\in X$ it holds
\begin{align}
\frac{|c(z,y)-c(w,y)|}{\sfd_R(z,w)} & = \frac{|\sfd(z,y)+\sfd(w,y)||\sfd(z,y)-\sfd(w,y)|}{2 \sfd_R(z,w)} \label{eq:formulaD} \\
& \leq \left(\frac{\sfd(z,y)+\sfd(w,y)}{2}\right) \frac{\sfd(z,w)}{\sfd_R(z,w)}. \label{eq:formulaDineq}
\end{align}
We prove first the $\Rightarrow$ implication. By assumption, for $o\in X$ there exists a neighborhood $\mathcal{O}$ where it holds $\sfd_R \leq \sfd \leq C \sfd_R$. Plugging the upper bound in \eqref{eq:formulaDineq} and taking the limsup for $z,w\to x$ we obtain $\sfD(x,y)\leq C \sfd(x,y)$, for all $y\in X$ and $x\in \mathcal{O}$. Together with the global inequality of \ref{i:propertiesD1} we obtain that $\sfD$ and $\sfd$ are locally equivalent.

To prove the $\Leftarrow$ implication, assume that $\sfD$ and $\sfd$ are locally equivalent. Thus for $o\in X$ there exists $r>0$ and $C=C_r>0$ such that $\sfD \leq C \sfd$ on $B_{3r}(o)$.  By dividing \eqref{eq:formulaD} by $\sfd(x,y)$, and taking the limsup, we obtain:
\begin{equation}\label{eq:equalto1}
\mathrm{Lip}_a^{\sfd_R}[\sfd(\cdot,y)](x)=\limsup_{z,w\to x} \frac{|\sfd(z,y)-\sfd(w,y)|}{\sfd_R(w,z)} \leq  C, \qquad \forall\, x,y\in B_{3r}(o),\quad x\neq y.
\end{equation}

Let now $\mathcal{O} = B_r(o)$. For $a,b\in \mathcal{O}\setminus \{y\}$, and for any $\varepsilon>0$ take a $\sfd_R$-unit-speed curve $\gamma:[0,\sfd_R(a,b)+\varepsilon]\to X$ joining $a$ and $b$. Observe that $\gamma$ is contained in $B_{3r}(o)$, that is the region where \eqref{eq:equalto1} holds. If $\gamma$ does not cross $y$ it holds
\begin{equation}
|\sfd(a,y)-\sfd(b,y)| \leq \lim_{\varepsilon\to 0} \int_0^{\sfd_R(a,b)+\varepsilon} \mathrm{Lip}_a^{\sfd_R}[\sfd(\cdot,y)](\gamma_t)dt \leq C \sfd_R(a,b), \quad \forall\, a,b\in \mathcal{O}\setminus \{y\},
\end{equation}
where in the second equality we used \eqref{eq:equalto1} since $\gamma$ does not intersect $y$. 

If, instead, $\gamma$ intersect $y$, we can first assume without loss of generality that the intersection occurs at a single time $t \in (0,\sfd_R(a,b)+\varepsilon)$, where $\gamma_t = y$. Then, define $y_\pm =\gamma_{t\pm \eta}$ for sufficiently small $\eta>0$. Consider the two $\sfd_R$-unit-speed curves $\gamma_- = \gamma|_{[0,t-\eta]}$ joining $a$ with $y_-$ and $\gamma_+ = \gamma|_{[t+\eta,\sfd_R(a,b)+\varepsilon]}$ joining $y_+$ with $b$. Neither of them intersects $y$ and we can use the previous estimate for the endpoints of $\gamma_\pm$. Hence
\begin{align}
|\sfd(a,y)-\sfd(b,y)| & \leq |\sfd(a,y)-\sfd(y_-,y)| + |\sfd(y_+,y)-\sfd(b,y)| + \sfd(y_-,y_+) \\
& \leq C(t-\eta) + C(\sfd_R(a,b) +\varepsilon-t-\eta)   + \sfd(y_-,y_+) \\
& = C(\sfd_R(a,b)+\varepsilon-2\eta)  + \sfd(y_-,y_+).
\end{align}
Letting $\varepsilon,\eta\to 0$ we have $y_-,y_+ \to y$ and we obtain that $|\sfd(a,y)-\sfd(b,y)| \leq C \sfd_R(a,b)$ for all $a,b\in \mathcal{O}\setminus \{y\}$. By continuity such an inequality holds for all $a,b\in \mathcal{O}$. Taking for example $b=y$ we obtain $\sfd(a,b) \leq C \sfd_R(a,b)$ for all $a,b\in \mathcal{O}$. Since $\sfd_R$ is a reference, we also have the opposite inequality $\sfd_R\leq \sfd$, and thus $\sfd$, $\sfd_R$ are locally equivalent.

\textbf{Proof of \ref{i:propertiesD3}.} If $\sfd=\sfd_R$ then clearly $\sfD=\sfd$. Conversely, if $\sfD=\sfd$,  one can follow the same steps of the proof of point \ref{i:propertiesD2} with $C=1$ and $\mathcal{O}=X$, yielding that $\sfd\leq \sfd_R$ on $X$. Since $\sfd_R$ is a reference, it holds $\sfd_R\leq \sfd$ so that $\sfd=\sfd_R$.

\textbf{Proof of \ref{i:propertiesD4}.} Observe that if $\sfd_R$ and $\sfd$ are locally equivalent, then they have the same class of rectifiable curves; moreover, if $(X,\sfd_R)$ is an extension of $(X,\sfd)$, on the class of rectifiable curves the length functionals associated with $\sfd$ and $\sfd_R$ coincide. Since $\sfd_R$ is length it follows that $\sfd\leq \sfd_R$ and since $\sfd_R$ is an extension $\sfd_R\leq \sfd$. Thus $\sfd=\sfd_R$ and so $\sfD=\sfd$. On the other hand if $\sfD=\sfd$, then $\sfd_R =\sfd$ by item \ref{i:propertiesD3}.

\textbf{Proof of \ref{i:propertiesD5}.} This is a consequence of the upper semi-continuity of the asymptotic Lipschitz number of a Borel function with respect to a length metric  (cf.\ \cite{ACM-Sobolev}).
\end{proof}

\subsection{Definition of natural gauge functions}\label{sec:ngf}

The functions $\sfd$, $\sfD$, and $\sqrt{\sfD^2-\sfd^2}$ appear as gauge functions for a large class of sub-Riemannian metric spaces, depending on the geometric structure under consideration:

\begin{itemize}
\item $\sfG =\sfd$ for Riemannian manifolds, see Section \ref{sec:howtorecover};
\item $\sfG = \sfD$ for qualitative bounds on fat sub-Riemannian manifolds, see Section \ref{sec:fat};
\item $\sfG = \sqrt{\sfD^2 -\sfd^2}$ for the Heisenberg group, see Section \ref{sec:heisenbergpuro}. This quantity corresponds to the ``absolute value of the vertical part of the covector'', appearing in the study of interpolation inequalities in the Heisenberg group performed in \cite{BKS}, see Section \ref{sec:howtorecover2} (there, such a function was called $\theta$); 
\item $\sfG = (\sfG_1,\sfG_2)$, with $\sfG_1=\sqrt{\sfD^2 -\sfd^2}$ and $\sfG_2 = \sfd$ (this is a vector-valued gauge function, see Section \ref{sec:vectorial}), for left-invariant structures on three-dimensional Lie groups, see Section \ref{sec:ex:vectorial}.
\end{itemize}

This motivates the following definition of natural gauge functions.

\begin{definition}[Natural gauge functions]\label{def:naturalgauge} 
We call a \emph{natural gauge function} any function $\sfG : X\times X\to [0,+\infty]$, such that $\sfG(x,y) = f(\sfd(x,y),\sfD(x,y))$ for all $x,y\in X$, where 
\begin{equation}
f :  \Omega:=\{(a,b) \in [0,+\infty)\times [0,+\infty]\mid a \leq b\} \to [0,+\infty],
\end{equation}
is continuous and homogeneous on $\Omega$, i.e.\@ $f(\lambda a,\lambda b) = \lambda f(a,b)$ for all $\lambda >0$, $(a,b)\in\Omega$ with $f(a,b) < +\infty$. Similarly, natural vector-valued gauge functions are those $\sfG : X\times X \to \RP^m_+$, $m\in \N$, induced by a continuous function $f: \Omega\to\RP^m_+$, such that $f(\lambda a,\lambda b) = \lambda f(a,b)$ for all $\lambda>0$ and $a,b\in  \Omega$ with $|f(a,b)|<+\infty$ (see Section \ref{sec:vectorial} for the definition of dilations on $\RP^m_+$).

To avoid trivialities, we assume that $|f(a,b)|<+\infty$ for $(a,b)\in \Omega$ if $a,b<+\infty$.
\end{definition}
\begin{remark}
The $1$-homogeneity property in Definition \ref{def:naturalgauge} is required to preserve the meek assumption of Definition \ref{def:meek}, provided that $\sfD$  is meek as well.
\end{remark}
\begin{remark}
The function $\sqrt{\sfD^2-\sfd^2}$ is well-defined for general metric spaces equipped with a reference metric, by Proposition \ref{prop:propertiesD}\ref{i:propertiesD1}.
\end{remark}

\subsection{Natural gauge functions in sub-Riemannian metric spaces}\label{sec:naturalinSR}

We assume a basic knowledge of sub-Riemannian geometry. For the benefit of the reader, we collect in Appendix \ref{a:SR} a summary tailored to our purposes.

\begin{definition}[Sub-Riemannian m.m.s.]\label{def:srmms}
 A sub-Riemannian metric measure space $(M,\sfd,\mm)$ is a smooth manifold $M$, equipped with a complete Carnot-Carath\'eodory distance $\sfd$, and a smooth measure $\mm$, i.e.\@ with smooth positive density in local charts.
\end{definition}

Under standard assumptions in sub-Rieman\-nian geometry, natural gauge functions satisfy the conditions required for the validity of the results of Sections \ref{sec:geometricconsequences} and \ref{sec:stabilityandcompactness}. We discuss, in particular:

\begin{itemize}
\item the meek condition for $\sfG$, see Definition \ref{def:meek}, used for the generalized Bishop-Gromov Theorem \ref{thm:BG1};
\item the regularity properties for $\sfG$, see Definitions \ref{def:weakL1andregularity} and \ref{def:weakL1andregularity-CD}, used for the pre-compactness Theorems \ref{thm:precompactMCP} and \ref{thm:precompact-CD};
\item the $\sfG-$boundedness properties, used for the stability Theorems \ref{thm:stabMCP} and \ref{thm:stabCD}.
\end{itemize}

For simplicity we focus on the case of scalar-valued gauge functions, but all the statements of this section hold in general for the vector-valued case, with no modifications, using the corresponding definitions and statements that can be found in Section \ref{sec:vectorial}.

\subsubsection{Meek condition}\label{subsec:meek}

We start with a sufficient condition for the validity of the meek assumption for $\sfD$, in terms of the local regularity of $\sfd$. We remark that it is valid for general m.m.s.\@ with an underlying smooth structure.

\begin{proposition}[Criterion for meekness]\label{p:meek1}
Let $(X,\sfd,\mm)$ be a m.m.s.\@ equipped with a reference metric $\sfd_R$. Assume that:
\begin{enumerate}[label = (\alph*)]
\item \label{i:meek11} $X$ is a smooth manifold and $\sfd_R$ is the distance induced by a Riemannian metric;
\item \label{i:meek12} for any $x \in X$ there exists a Borel set $C_x$ with $\mm(C_x)=0$, such that for all $y\notin C_x$ and geodesic $\gamma$ with $(\gamma_0,\gamma_1)=(x,y)$ the function $z\mapsto \sfd(z,\gamma_t)$ is continuously differentiable in charts in a neighborhood of $x$ for any $t\in (0,1]$.
\end{enumerate}
Then  any natural gauge function $\sfG:X\times X\to [0,+\infty]$ is meek.
\end{proposition}
\begin{proof}
It is sufficient to prove the statement for $\sfG=\sfD$. Fix $\bar{x}\in X$ and $\mu_1\in \mathcal{P}_{bs}^*(X,\sfd,\mm)$, $\bar{x}\notin \supp\mu_1$. Let $\Gamma\subset \Geo(X)$ be the set of geodesics starting at $\bar{x}$ and ending out of $C_{\bar{x}}$:
\begin{equation}
\Gamma := (\ee_0\times \ee_1)^{-1}(\{\bar{x}\}\times X\setminus C_{\bar{x}}),
\end{equation}
so that $\Gamma$ is Borel.

Let $\nu \in \OptGeo(\delta_{\bar{x}},\mu_1)$. It holds $(\ee_0,\ee_1)_\sharp \nu = \pi = (\bar{x}, {\rm id})_\sharp \mu_1$, which is the (unique)  optimal plan between $\delta_{\bar{x}}$ and $\mu_1$. We have
\begin{equation}
1 = \mu_1(X) = \mu_1(X\setminus C_{\bar{x}}) = \pi(\{\bar{x}\}\times X\setminus C_{\bar{x}})= \nu((\ee_0\times \ee_1)^{-1}(\{\bar{x}\}\times X\setminus C_{\bar{x}})) = \nu(\Gamma),
\end{equation}
where we used $\mu_1\ll \mm$ and $\mm(C_{\bar{x}})=0$.

We now show that any $\gamma\in\Gamma$ satisfies the meek equality \eqref{eq:meek1}. Fix an arbitrary $\gamma \in \Gamma$. We will use the following general inequality, which follows from squaring the triangle inequality (recall that $c=\sfd^2/2$):
\begin{equation}\label{eq:squaredtriangle}
c(p,q) \leq (1+\varepsilon^{-1})c(p,r) + (1+\varepsilon)c(r,q), \qquad \forall\, p,q,r\in X,\; \forall\, \varepsilon>0.
\end{equation}
Now take $p=z$, $q=y$, $r=\gamma_t$, $t\in (0,1)$, and $\varepsilon=t/(1-t)$, so that \eqref{eq:squaredtriangle} reads
\begin{equation}
c(z,y) \leq \frac{1}{t}c(z,\gamma_t) + \frac{1}{1-t}c(\gamma_t,y), \qquad\forall\, z\in X,
\end{equation}
with equality at $z=\bar{x}$. It follows that for all $t\in (0,1)$ the function
\begin{equation}
f: z\mapsto c(z,\gamma_t)-tc(z,y),\qquad z\in X,
\end{equation}
has a minimum at $z=\bar{x}$. By assumption \ref{i:meek12}, for all $t\in (0,1)$ it holds
\begin{equation}\label{eq:equalitygrad}
\nabla_{\bar{x}}^R c(\cdot,\gamma_t) = t \nabla_{\bar{x}}^R c(\cdot,y),
\end{equation}
where $\nabla^R_{\bar{x}}$ denotes the Riemannian gradient at $\bar{x}$. Furthermore, we observe that since $c(\cdot,\gamma_t)$ is continuously differentiable in a neighborhood of $\bar{x}$, the asymptotic $\sfd_R$-Lipschitz constant is the norm of the Riemannian gradient, that is
\begin{equation}
\sfD(\bar{x},\gamma_t) = \|\nabla^R_{\bar{x}} c(\cdot,\gamma_t)\|_R, \qquad \forall\, t\in (0,1).
\end{equation}
Using \eqref{eq:equalitygrad} we obtain that $\sfD(\gamma_0,\gamma_t) = t \sfD(\gamma_0,\gamma_1)$ for all $t\in (0,1]$. 
\end{proof}

We use Proposition \ref{p:meek1} to prove that natural gauge functions satisfy the meek property. We first recall the definition of sub-Riemannian cut locus.

\begin{definition}[Smooth points and cut locus]\label{def:cut}
 Let $(M, \sfd)$ be a sub-Riemannian metric space.
We say that $y\in M$ is a \emph{smooth point}, with respect to $x\in M$, if there exists a unique geodesic joining $x$ with $y$, which is not abnormal, and with non-conjugate endpoints. The \emph{cut locus} $\Cut(x)$ is the complement of the set of smooth points with respect to $x$. The \emph{global cut locus} of $M$ is 
\begin{equation}
\Cut(M) := \{(x,y) \in M \times M \mid y \in \Cut(x)\}.
\end{equation}
\end{definition}

We have the following fundamental result \cite{agrasmoothness,RT-MorseSard}.

\begin{theorem}[Agrachev, Rifford, Trélat]\label{thm:agrariffsmooth}
The set of smooth points is open and dense in $M$, and the squared sub-Riemannian distance is smooth on $M\times M \setminus \Cut(M)$.
\end{theorem}

It is an open problem to determine whether $\Cut(x)$ has zero measure. Even in this case, the complement of $\Cut(x)$ may not be geodesically star-shaped: more precisely, if $y\notin \Cut(x)$ and $\gamma$ is the geodesic joining $x$ with $y$, it may happen that $\gamma_t \in \Cut(x)$ for $t\in (0,1)$. This is related with the presence of abnormal non-trivial segments on strictly normal geodesics, and the branching phenomenon in sub-Riemannian geometry \cite{MR-branching}. Therefore we cannot apply Proposition \ref{p:meek1} simply with $C_x=\Cut(x)$. 

To proceed, we recall the classical minimizing Sard property, and we introduce its strengthening that we call ``$*$-minimizing Sard property''.

\begin{definition}[Sard properties]\label{def:Sard} A sub-Riemannian metric space $(M,\sfd)$ satisfies
\begin{itemize}
\item the \emph{minimizing Sard property} if for any $x\in M$ the set of final points of abnormal geodesics has zero measure in $M$;
\item the \emph{$*$-minimizing Sard property} if for any $x\in M$ the set of final points of geodesics from $x$ containing non-trivial abnormal segments has zero measure in $M$.
\end{itemize} 
\end{definition}

It is well-known that, if the minimizing Sard property holds, then $\Cut(x)$ has zero measure for all $x\in M$. We now show that, if the $*$-minimizing Sard property holds, then there exists a set, larger than $\Cut(x)$, but still with zero measure, whose complement is star-shaped. We do not know if such a set is closed, but it is Borel, which is what we need in Proposition \ref{p:meek1}.

\begin{lemma}[Star-shaped set of smooth points]\label{lem:starsardpointofsmoothness}
Let $(M,\sfd)$ be a sub-Riemannian metric space. Then for any $x\in M$ there exists a set $\mathcal{U}_x\subseteq M$ such that:
\begin{enumerate}[label=(\roman*)]
\item \label{i:starsardpointofsmoothness1} $\mathcal{U}_x$ is the countable intersection of open sets, and in particular it is Borel;
\item \label{i:starsardpointofsmoothness2} $\sfd^2$ is smooth in a neighborhood of $(x,y)$ for any $y\in \mathcal{U}_x$, i.e.\@ $\mathcal{U}_x\subseteq M\setminus \Cut(x)$;
\item \label{i:starsardpointofsmoothness3} $\mathcal{U}_x$ is geodesically star-shaped at $x$: for any $y\in \mathcal{U}_x$ there exists a unique geodesic $\gamma$ joining $x$ with $y$, and $\gamma_t \in \mathcal{U}_x$ for all $t\in (0,1]$ ($t=0$ is excluded);
\item \label{i:starsardpointofsmoothness4} if the $*$-minimizing Sard property holds, $\mathcal{U}_x$ has full measure and it is dense.
\end{enumerate}
\end{lemma}
\begin{remark}
If the sub-Riemannian structure is real-analytic, geodesics are either abnormal, or do not contain non-trivial abnormal segments, and then the $*$-minimizing Sard property coincides with the classical minimizing Sard property. For real-analytic sub-Riemannian structures satisfying the minimizing Sard property, one can take $\mathcal{U}_x:=M\setminus \mathrm{Cut}(x)$, and thus $\mathcal{U}_x$ is open, see \cite[Prop.\@ 6]{BR-realanalMCP}.
\end{remark}
\begin{proof} 
Fix $x\in M$. Let $\mathrm{Cut}(x)$ be the sub-Riemannian cut locus, as in Definition \ref{def:cut}, and recall Theorem \ref{thm:agrariffsmooth}. We also recall that if $\gamma$ is a normal geodesic that is also abnormal, then any pair of distinct points along $\gamma$ are conjugate.

For all $n\in \N$, let us define the following monotone family of sets:
\begin{multline}
\mathcal{U}^n_x:=\{\gamma_1\mid \gamma\in \Geo(M),\; \gamma_0 = x, \; \gamma_1 \in M\setminus \Cut(x), \; \text{the distance between} \\ \text{two distinct conjugate points along $\gamma$ is strictly smaller than $1/n$}\}.
\end{multline}
By construction, for any $y\in  \mathcal{U}^{n}_x$ there is a unique geodesic $\gamma$ between $x$ and $y$, $\gamma$ is not abnormal, but $\gamma$ can contain abnormal segments of length strictly smaller than $1/n$.

We claim that the set $\mathcal{U}^n_x$ is open. Let $y\in \mathcal{U}_x^n$, and let $\gamma$ be the unique non-abnormal geodesic between $x$ and $y$, that is $\gamma_t = \exp_x(t\lambda)$ for a unique $\lambda \in T_x^*M$. For any $\tau_1,\tau_2\in [0,1]$, with $|\tau_2-\tau_1|\geq 1/n$ the map 
\begin{equation}
d_{e^{\tau_1\vec{H}}(\lambda)} \exp_{\gamma_{\tau_1}}((\tau_2-\tau_1)\;\cdot\;): T_{\gamma_{\tau_1}}^*M \to T_{\gamma_{\tau_2}}M,
\end{equation}
must be invertible, otherwise $\gamma(\tau_1)$ and $\gamma(\tau_2)$ would be conjugate. Notice also that since $y\notin \Cut(x)$, we have a neighborhood $\mathcal{W}_y$ of $y$ and a neighborhood $\mathcal{V}_\lambda$ of $\lambda$ such that $\exp_x:\mathcal{V}_\lambda \to \mathcal{W}_y$ is a diffeomorphism. It follows by compactness that there exists an open neighborhood $\mathcal{W}'_y$ of $y$ such that for all $y'\in \mathcal{W}'_y$ there exists a unique geodesic $\gamma'$ joining $x$ with $y'$, that is not abnormal, and all pairs of points along $\gamma'$ with distance $\geq 1/n$ are not conjugate (and in particular, $\gamma'$ cannot contain abnormal segments of length $\geq 1/n$). In other words, $\mathcal{W}'_y\subseteq \mathcal{U}^n_x$, proving that $\mathcal{U}_n^x$ is open. Define
\begin{equation}
\mathcal{U}_x:=\bigcap_{n\in \N} \mathcal{U}_x^n,
\end{equation}
which satisfies \ref{i:starsardpointofsmoothness1}. By construction $\mathcal{U}_x\subseteq M\setminus \Cut(x)$, proving \ref{i:starsardpointofsmoothness2}. Furthermore, for all $y \in \mathcal{U}_x$ there exists a unique geodesic $\gamma$ joining $x$ with $y$, which does not contain non-trivial segments with conjugate endpoints (and in particular, it does not contain non-trivial abnormal segments). Thus $\gamma_t \in \mathcal{U}_x$ for all $t\in (0,1]$, proving \ref{i:starsardpointofsmoothness3}.

\textbf{Claim.} We claim that $M\setminus \mathcal{U}_x$ is equal to the union of $\Cut(x)$ with the set of endpoints of geodesics starting from $x$ that contain a non-trivial abnormal segment. Notice that the latter set has zero measure if the $*$-minimizing Sard property holds. Furthermore also $\mathrm{Cut}(x)$ has zero measure\footnote{This is proved for example in \cite[Prop.\@ 6]{BR-realanalMCP}, and the minimizing Sard property is sufficient. A self-contained argument will be provided at the end of this proof.}. Thus, the claim implies that $\mathcal{U}_x$ has full measure. Furthermore, each open set $\mathcal{U}_x^n$ must have full measure, and in particular it must be dense in $M$. By the Baire theorem $\mathcal{U}_x$ is dense, concluding the proof of \ref{i:starsardpointofsmoothness4}. 

It only remains to prove the claim:

$\supseteq:$ If $y\in \Cut(x)$ clearly it does not belong to any of the $\mathcal{U}_x^n$. Let then $y\notin\Cut(x)$, such that there exists a geodesic $\gamma$ with $\gamma_0 = x$ that contains a non-trivial abnormal segment. Observe that $\gamma$ is not abnormal and is the unique geodesic joining $x$ with $y$. Let $\gamma|_{[\tau_1,\tau_2]}$ with $0<|\tau_1-\tau_2|<1$ be the non-trivial abnormal segment. Any pair of distinct points on $\gamma|_{[\tau_1,\tau_2]}$ are conjugate, and thus $y\notin \mathcal{U}_x^n$ for $n$ sufficiently large.

$\subseteq:$ Let $y\notin \mathcal{U}_x^n$ for some $n$. It follows that either $y \in \Cut(x)$, or if not there exist two distinct conjugate points along the unique geodesic joining $x$ with $y$. This can happen only if $\gamma$ contains non-trivial abnormal segments (otherwise it would cease to be a geodesic after meeting the first conjugate point \cite[Thm.\@ 72]{BRInv}).

To make our proof self-contained, we show that $\mathrm{Cut}(x)$ has zero measure if the minimizing Sard conjecture holds, that is when the set $\mathrm{Abn}(x)$ of endpoints of abnormal geodesics from $x$ has zero measure. By completeness, the image of $\exp_x : T_x^*M \to M$ covers $M\setminus \mathrm{Abn}(x)$. The set of conjugate points to $x$ is contained in the set $S(x)$ of singular values of $\exp_x$, which has zero measure by the Sard theorem. Furthermore, one can show that the subset $R_2(x)$ of points $z$ where $z$ is not conjugate to $x$ and there exist at least two normal geodesics joining $x$ with $z$ has zero measure, cf.\ \cite[Lemma 4.8]{RT-MorseSard}. Hence, by construction, for any $y\in M$ which is not in $R_2(x) \cup \mathrm{Abn}(x) \cup S(x)$ there exists a unique strictly normal geodesic $\gamma$ joining $x$ with $y$, with non-conjugate endpoints. In other words $\mathrm{Cut}(x) \subseteq R_2(x) \cup \mathrm{Abn}(x) \cup S(x)$, which is a union of zero measure sets.
\end{proof}

We can now prove that, for a large class of sub-Riemannian structures, the $\sfD$ function (cf.\ Definition \ref{def:D}) satisfies the meek assumption (cf.\ Definition \ref{def:meek}).

\begin{theorem}[Natural gauge functions are meek]\label{t:meeknatural}
Let $(M,\sfd,\mm)$ be a sub-Riemannian m.m.s., equipped with a Riemannian reference $\sfd_R$. Assume that $(M,\sfd)$ satisfies the $*$-minimizing Sard property.  Then any natural gauge function $\sfG$ is meek.
\end{theorem}
\begin{proof}
It is sufficient to prove the property for $\sfG=\sfD$. In this case, apply Proposition \ref{p:meek1} with $C_x := M\setminus \mathcal{U}_x$, with the set $\mathcal{U}_x$ of Lemma \ref{lem:starsardpointofsmoothness}.
\end{proof}

\begin{remark}\label{rem:meekexamples}
The $*$-minimizing Sard property is true for:
\begin{itemize}
\item ideal structures, i.e.\@ those not admitting non-trivial abnormal geodesics, which is generically true for constant rank $\geq 3$ distributions, cf.\ \cite[Thm.\@ 2.8]{CJT-Genericity};
\item real-analytic structures satisfying the minimizing Sard conjecture. This includes Carnot groups of step $2$, or more general Carnot groups, cf.\ \cite[Thm.\@ 1.2]{SardProp}.
\end{itemize}
Thus, for all those structures, then all natural gauge functions are meek.
\end{remark}

\subsubsection{Regularity condition}\label{sec:regularityGExamples}

We first discuss the regularity condition required in the $\MCP(\beta)$ stability Theorem \ref{thm:stabMCP}.

\begin{theorem}[Natural gauge functions are regular I]\label{t:regularityforstabilityMCP}
Let $(M,\sfd,\mm)$ be a sub-Rieman\-nian m.m.s., equipped with a Riemannian reference $\sfd_R$. Then any natural gauge function $\sfG$ is continuous and finite on $M\times M \setminus \Cut(M)$. In particular, if $(M,\sfd)$ satisfies the minimizing Sard property, any such $\sfG$ satisfies the regularity condition of Definition \ref{def:weakL1andregularity}.
\end{theorem}
\begin{proof}
It is sufficient to prove the statement for the case $\sfG = \sfD$. We observe that $\sfd$ is smooth in charts on $M\times M \setminus \Cut(M)$ by Theorem \ref{thm:agrariffsmooth}. At all those points, $\sfD$ coincides with the norm of the gradient of $c=\frac{1}{2}\sfd^2$ computed with respect to the Riemannian reference, cf.\ Proposition \ref{prop:naturalobject}. Then $\sfD$ is finite and continuous on $M\times M \setminus \Cut(M)$. Furthermore, if the minimizing Sard property holds, then $\Cut(x)$ is a closed set with zero measure for all $x\in M$. Thus $\sfD$ is finite and continuous at all points $(x,y)$ with $y\notin \Cut(x)$, which is the regularity condition of Definition \ref{def:weakL1andregularity}.
\end{proof}

We turn now to the regularity condition for plans occurring in $\CD(\beta,n)$ stability in Theorem \ref{thm:stabCD}. We recall the following definition, introduced by Rifford in \cite{RiffordCarnot,riffordbook}.

\begin{definition}[Ideal structures]\label{def:ideal}
A sub-Riemannian metric space $(M,\sfd)$ is ideal if it has no non-trivial abnormal geodesics.
\end{definition}

Generic sub-Riemannian structures are ideal, when the rank of the distribution is constant and at least 3, see \cite[Thm.\@ 2.8]{CJT-Genericity}. For ideal structures, a sub-Riemannian version of the McCann-Brenier theorem holds \cite{FR-mass,riffordbook}. Furthermore, the non-static part of the transport map between suitable measures avoids almost surely the cut locus \cite{BRInv}, so that $\sfd$ is smooth $\pi$-a.e.\ for any optimal plan whose marginals have disjoint support.

\begin{theorem}[Natural gauge functions are regular II]\label{t:regularityforstabilityCD}
Let $(M,\sfd,\mm)$ be a sub-Rieman\-nian m.m.s., equipped with a Riemannian reference $\sfd_R$. Assume that $(M,\sfd)$ is ideal. Then for all $\mu_{0}\in \mathcal{P}_{bs}(M,\sfd)$, $\mu_{1}\in \Prob_{bs}^{*}(M,\sfd, \mm)$, with $\supp\mu_0 \cap \supp\mu_1 = \emptyset$ there is a unique $\pi \in \Opt(\mu_0,\mu_1)$, and any natural gauge function $\sfG$ is continuous and finite $\pi$-a.e. In particular, the regularity condition of Definition \ref{def:weakL1andregularity-CD} is verified.
\end{theorem}
\begin{proof}
It is sufficient to prove the result for the case $\sfG=\sfD$. If $(M,\sfd)$ is ideal, there exists a unique $\pi \in \Opt(\mu_0,\mu_1)$, and it is induced by a transport map $T:M\to M$. Furthermore, for $\mu_0$-a.e.\ $x\in M$, either $T(x)=x$, or $T(x)\notin \Cut(x)$, cf.\ \cite[Cor.\ 38]{BRInv}. Since $\supp\mu_0$ and $\supp\mu_1$ are disjoint, the first case never occurs. It follows that $\sfD$ is continuous at $(x,y)$ for $\pi$-a.e.\ $(x,y) \in M\times M$.
\end{proof}
\begin{remark}
One can appreciate in the above proof the importance of the restriction to measures $\mu_0,\mu_1$ with disjoint support. In general (and more precisely, whenever the sub-Riemannian structure is not Riemannian), the diagonal $\Delta\subset M\times M$ is always part of the cut locus, and $\sfD$ is not continuous there. Thus, the regularity condition of Definition \ref{def:weakL1andregularity-CD} would not be verified without that restriction.
\end{remark}
\begin{remark}
The ideal assumption cannot be removed. In fact, in \cite[Fig.\ 2]{BKS2}, the authors built an explicit example, on corank 1 Carnot groups, and measures $\mu_i \in \mathcal{P}^*_{bs}(X,\sfd,\mm)$ with disjoint support where any point $x\in\supp\mu_0$ is sent to $T(x)\in \Cut(x)$, and more precisely $T(x)$ is the endpoint of an abnormal geodesic from $x$. Here, it turns out that $\sfD$ is not continuous at every $(x,y)\in \supp\pi$.
\end{remark}

\subsubsection{G-boundedness condition}\label{sec:boundedness}

We discuss here the boundedness properties of natural gauge functions. We recall the following definition, see \cite[Sec.\ 3.1]{nostrolibro} for further details.

\begin{definition}[Step]\label{def:step}
A sub-Riemannian metric space has step $k\in \N$ at $x\in M$ if Lie brackets of length $k$ are sufficient to satisfy the bracket-generating condition at $x$, and $k=k(x)$ is the smallest number with this property. The step of a sub-Riemannian metric space is the supremum of the step at all its points.
\end{definition}

The next result is a reformulation of well-known facts about the regularity of the sub-Riemannian distance.

\begin{theorem}[Boundedness properties of natural gauge functions]\label{t:boundedness}
Let $(M,\sfd,\mm)$ be a sub-Riemannian m.m.s. Let $\sfG$ be a natural gauge function. Then the following hold:
\begin{enumerate}[label = (\roman*)]
\item \label{i:boundedness1} The function $\sfG$ is locally bounded on an open and dense set of $M\times M$. In particular for any non-empty open set $\mathcal{O}$ it holds
\begin{equation}
\mm(\{y\in M \; : \; |\sfG(x,y)|<+\infty\}\cap \mathcal{O}) >0, \qquad \forall\, x \in M.
\end{equation}
\item \label{i:boundedness2} If $(M,\sfd)$ satisfies the minimizing Sard property, then for any $x\in M$ the function $\sfG(x,\cdot)$ is locally bounded almost everywhere. In particular it holds
\begin{equation}
\mm(\{y\in M  \; : \; |\sfG(x,y)| = +\infty\}) = 0, \qquad \forall\, x\in M.
\end{equation}
\item \label{i:boundedness3} If $(M,\sfd)$ has step $\leq 2$, then $\sfG$ is locally bounded, that is for every bounded subset $\mathcal{O}$ there exists $C>0$ such that
\begin{equation}
|\sfG(x,y)| \leq C,\qquad \forall\, x,y\in \mathcal{O}.
\end{equation}

\item \label{i:boundedness4} If $\sfd$ is locally Lipschitz in charts outside of the diagonal $\Delta\subset M\times M$, then $\sfG$ is locally bounded away from the diagonal, that is for all $\varepsilon>0$ and bounded subset $\mathcal{O}$ there exists $C>0$ such that 
\begin{equation}
|\sfG(x,y)| \leq C,\qquad \forall\, x,y\in \mathcal{O},\; \sfd(x,y)\geq \varepsilon.
\end{equation}
\end{enumerate}
\end{theorem}
\begin{proof}
It is sufficient to prove the statements for $\sfG=\sfD$. 

\textbf{Proof of \ref{i:boundedness1}.} Fix $x\in M$ and $y\notin\Cut(x)$. Then $\sfd$ is smooth in a neighborhood of $(x,y)$, and $\sfD(x,y) = \|\nabla^R_x \sfd(\cdot,y)\|_R <+\infty$. The set $M\setminus \Cut(x)$ is open and dense by Theorem \ref{thm:agrariffsmooth}, and since $\mm$ is smooth the statement follows.

\textbf{Proof of \ref{i:boundedness2}.} By the previous argument, if $\sfD(x,y) = +\infty$ we must have $y\in \Cut(x)$. If the Sard property holds, then $\Cut(x)$ has zero measure.

\textbf{Proof of \ref{i:boundedness3}.} The step of $(M,\sfd)$ is $\leq 2$ if and only if $\sfd$ is locally Lipschitz in charts \cite[Cor.\ 6.2]{AAPL-transport}. Since $\sfD$ is the asymptotic Lipschitz constant of $\sfd$ with respect to a Riemannian metric $\sfd_R$, then $\sfD$ must be locally bounded.

\textbf{Proof of \ref{i:boundedness4}.} Obvious.
\end{proof}

\begin{remark}
Item \ref{i:boundedness1} yields the non-triviality condition \eqref{eq:GNonTriv} for natural gauge functions. Item \ref{i:boundedness2} yields the finiteness condition \eqref{eq:Gfinite} required in the generalized Bishop-Gromov Theorem \ref{thm:BG1}.
\end{remark}
\begin{remark}\label{rmk:Goh}
Item \ref{i:boundedness4} holds for example when $(M,\sfd)$ is ideal, and more generally if there are no non-trivial Goh paths, see \cite[Thm.\@ 5.5]{AAPL-transport} for details. This is true, for example, if the underlying distribution is medium-fat \cite[Sec.\ 4.5]{FR-mass}, \cite{AS-minimalityVSsubanalyticity}.
\end{remark}
\begin{remark}
The step $2$ assumption in \ref{i:boundedness3} cannot be removed. Notice that, by the Ball-Box theorem \cite{Jeanbook,Bellaiche}, $\sfd(o,\cdot)$ cannot be Lipschitz in charts in a neighborhood of any point $o$ where the step $s(o)\geq 3$, so that $\sfD(o,o)=+\infty$. One can show that, in general, $\sfD$ may not be essentially bounded in any neighborhood of the diagonal around such points.

To illustrate this fact, let us discuss the case of Carnot groups $(\mathbb{G},\sfd)$. These are left-invariant structures on a nilpotent, connected and simply connected Lie group $\mathbb{G} \simeq \R^n$. We denote points as $(x_1,\dots,x_n)$ in exponential coordinates, with $o=(0,\dots,0)$ the identity of the group. Carnot groups are equipped with metric dilations $\delta_\lambda$, $\lambda>0$, that in these coordinates read $\delta_\lambda(x) = (\lambda^{w_1} x_1,\dots,\lambda^{w_n}x_n)$, where $w_i$ denotes the weight of the $i$-th coordinate. Furthermore, we may choose a Riemannian reference $\sfd_R$, left-invariant by the group action of $\mathbb{G}$. Notice that, for any bounded neighborhood $\mathcal{O}\subset \R^n$, there exists a constant $C=C_{\mathcal{O}}>0$ such that it holds
\begin{equation}
\frac{1}{C} \sum_{i=1}^n|x_i-x_i'|\leq \sfd_R(x,x') \leq C \sum_{i=1}^n|x_i-x_i'|, \qquad \forall\, x,x'\in \mathcal{O}.
\end{equation}
Furthermore, using dilations, one checks that there exists a constant $C'>0$ such that 
\begin{equation}
\frac{1}{C'}\sum_{i=1}^n |x_i-x'_i|^{1/w_i}\leq \sfd(x,x') \leq C'\sum_{i=1}^n |x_i-x'_i|^{1/w_i}, \qquad \forall\, x,x'\in \R^n.
\end{equation}
Let now $\mathcal{O}$ be a bounded neighborhood of the origin $o\in \mathbb{G}$. Up to restricting it, we may assume that $p\in \mathcal{O} \Leftrightarrow p^{-1}\in \mathcal{O}$ (recall that, in exponential coordinates, group inversion corresponds to the map $x\mapsto -x$). We have, for all $x\in \mathcal{O}$
\begin{align}
\sfD(o,x^{-1})=\sfD(x,o) & = \frac{1}{2}\limsup_{z,w\to x} \frac{\sfd(z,o)^2-\sfd(w,o)^2}{\sfd_R(z,w)} \\
& \geq \frac{1}{2}\limsup_{\lambda,\lambda'\to 1} \frac{\sfd(\delta_\lambda(x),o)^2-\sfd(\delta_{\lambda'}(x),o)^2}{\sfd_R(\delta_\lambda(x),\delta_{\lambda'}(x))} \\
& \geq \frac{1}{2} \limsup_{\lambda,\lambda'\to 1} \frac{|\lambda^2-\lambda^{'2}|\sfd(x,o)^2}{|\lambda - \lambda'|C^2 \sfd_R(x,o)}\\
& \geq \frac{1}{C^2} \frac{\sfd(x,o)^2}{\sfd_R(x,o)} \geq C'' \frac{\sum_{i=1}^n |x_i|^{2/w_i}}{\sum_{i=1}^n |x_i|},\label{eq:finalestimate-unbound}
\end{align}
where in the first equality we used the left-invariance of $\sfD$ (which follows from the left-invariance of $\sfd$ and $\sfd_R$), and we note that the constant $C''>0$ depends only on $\mathcal{O}$. It follows that, if there is one $w_i\geq 3$ (that is, when the step is $\geq 3$), then it holds
\begin{equation}
\mm\shortminus\esssup_{x\in \mathcal{O}, x\neq o}\sfG(o,x) = +\infty,
\end{equation}
where $\mm$ is any smooth measure on $\R^n$ (this includes the Haar measures of $\mathbb{G}$). Notice that in the last estimate, $\mathcal{O}$ can be any bounded neighborhood of $o$. It follows by left-invariance that for any neighborhood $\mathcal{O}$ in $\mathbb{G}$ we have
\begin{equation}\label{eq:unboundedness}
\diam_{\sfG}(\mathcal{O}) = +\infty.
\end{equation}

Thus, for sub-Riemannian structures of step $3$, natural gauge functions do not satisfy, in general, the $\sfG$-boundedness assumptions of the stability Theorems \ref{thm:stabMCP} and \ref{thm:stabCD}.
\end{remark}

The diagram in Figure~\ref{fig:implications} illustrates most of the properties for natural gauge functions we proved in this section, on sub-Riemannian structures.
\begin{figure}[t]
\[
\begin{tikzcd}[arrows=Rightarrow]
\fbox{\text{step $\leq 2$ (Def.\ \ref{def:step})}} \arrow[d] \arrow[rrr, "\text{Theorem \ref{t:boundedness}}"] & & & \fbox{$\sfG$ \text{is locally bounded}} \arrow[d] \\
\fbox{\begin{tabular}{@{}c@{}}minimizing Sard \\ (Def.\ \ref{def:Sard})\end{tabular}} \arrow[rrrd, "\text{Theorem \ref{t:regularityforstabilityMCP}}", sloped] \arrow[rrr, "\text{Theorem \ref{t:boundedness}}"] \arrow[dd, "\text{real-analytic}", shift left=2] & & &
\fbox{\begin{tabular}{@{}c@{}} $\sfG$ is locally bounded\\ almost everywhere \end{tabular}}\\
& & & \fbox{\begin{tabular}{@{}c@{}}$\sfG$ satisfies the\\ regularity for stability\\ of $\MCP$ (Def.\ \ref{def:weakL1andregularity})\end{tabular}} \\
\fbox{\begin{tabular}{@{}c@{}} $*$-minimizing Sard \\ (Def.\ \ref{def:Sard})\end{tabular}} \arrow[uu, shift left=2] \arrow[rrr, "\text{Theorem \ref{t:meeknatural}}"] & & & \fbox{$\sfG$ is meek (Def.\ \ref{def:meek})}\\
\fbox{\text{ideal (Def.\ \ref{def:ideal})}} \arrow[u] \arrow[d] \arrow[rrr, "\text{Theorem \ref{t:regularityforstabilityCD}}"] & & & 
\fbox{\begin{tabular}{@{}c@{}}$\sfG$ satisfies the \\ regularity for stability \\ of $\CD$   (Def.\ \ref{def:weakL1andregularity-CD})\end{tabular}} \\
\fbox{\begin{tabular}{@{}c@{}}no non-trivial Goh \\ geodesics \end{tabular}} \arrow[rrr, "\text{Remark \ref{rmk:Goh}}"] & & &
\fbox{\begin{tabular}{@{}c@{}}$\sfG$ is locally bounded\\ away from diagonal\end{tabular}}
\end{tikzcd}
\]
\caption{Sub-Riemannian properties (left), and of natural gauge functions $\sfG$ (right). When $\sfG=\sfd$ all properties are trivially verified with no assumptions.}\label{fig:implications}
\end{figure}
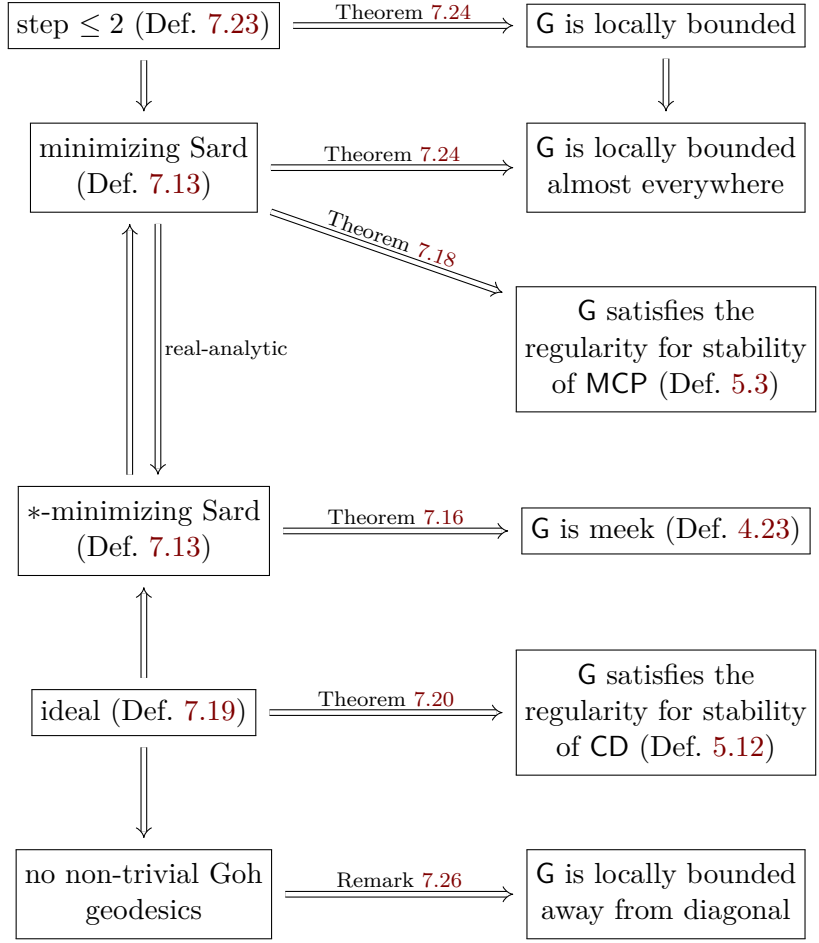

\subsection{Expression of \texorpdfstring{$\sfD$}{D} out of the cut locus}\label{sec:Doutofcut}

In this final section we show how $\sfD$ is related with more familiar objects in sub-Riemannian geometry. It will also serve the purpose to better motivate the expression $\sqrt{\sfD^2-\sfd^2}$ that appears in the list of natural gauge functions in Section \ref{sec:ngf}.

\begin{proposition}\label{prop:naturalobject}
Let $(M,\sfd)$ be a sub-Riemannian metric space, equipped with a reference Riemannian metric $\sfd_R$. Let $\sfD:M\times M\to [0,+\infty]$ be as in Definition \ref{def:D}. For all $x\in M$, let $T_x M = \hor_x \oplus \ver_x$, where $\hor_x$ is the horizontal subspace at $x$ and $\hor_x \perp \ver_x$, and denote with $\pi_x^\ver,\pi_x^\hor$ the corresponding projections. Let also $\nabla^R$ and $\|\cdot\|_R$ denote the gradient and the norm of the Riemannian metric. Then it holds:
\begin{equation}\label{eq:naturalobject1}
\sfD(x,y) = \|\nabla_x^R c(\cdot,y)\|_R,\qquad \forall\, (x,y)\notin\Cut(M).
\end{equation}
If, in addition, the reference $\sfd_R$ is an extension of $\sfd$ (cf.\ Definition \ref{def:extension}), then it holds
\begin{equation}\label{eq:naturalobject2}
\sfD(x,y)^2 = \sfd(x,y)^2 + \|\pi^\ver_x \nabla_x^R c(\cdot,y)\|_R^2, \qquad \forall\, (x,y)\notin\Cut(M).
\end{equation}
Equivalently, identifying $T_xM \simeq T_x^*M$ by means of the Riemannian structure, it holds
\begin{equation}\label{eq:naturalobject3}
\|\pi_x^\ver \lambda^{x,y}\|^2 = \sfD(x,y)^2 - \sfd(x,y)^2, \qquad (x,y)\notin\Cut(M),
\end{equation}
where $\lambda^{x,y} \in T_x^*M$ is the initial covector of the unique geodesic joining $x$ with $y$.
\end{proposition}
\begin{proof}
By Theorem \ref{thm:agrariffsmooth}, on $M\setminus \Cut(y)$ the function $c(\cdot,y)=\tfrac{1}{2}\sfd(\cdot,y)^2$ is smooth. Thus, out of $\Cut(y)$, the asymptotic Lipschitz constant with respect to $\sfd_R$ coincides with the Riemannian norm of the Riemannian gradient, thus proving \eqref{eq:naturalobject1}.

We then prove that, if $\sfd_R$ is a Riemannian extension induced by the Riemannian metric $g_R$, and $g$ denotes the sub-Riemannian metric, for all $x\in X$ and $v_x\in \hor_x$ it holds $\|v_x\|_{g_R} = \|v_x\|_{g}$. In fact, let $X_1,\dots,X_L$ be a generating family for the sub-Riemannian structure. There exist $c_{i}\in \R$ such that $v_x = \sum_{i=1}^N c_i X_i(x)$, then let $V:=\sum_{i=1}^N c_i X_i$. Let also $\gamma_t = e^{tV}(x)$ be the flow of $V$, for $t\in [0,\ve]$. This is by construction a horizontal curve with $\dot{\gamma}_t = V(\gamma_t)$. Furthermore, it holds (cf.\ \cite[Ex.\ 3.51, 3.52]{nostrolibro}):
\begin{equation}
L_{\sfd}(\gamma|_{[0,t]}) = \int_0^t \|V(\gamma_s)\|_g\,\di s, \qquad \Longrightarrow \qquad \|v_x\|_{g} = \left.\frac{d}{dt}\right|_{t=0} L_{\sfd}(\gamma|_{[0,t]}).
\end{equation}
Since $\sfd_R$ is an extension, $\gamma_t$ is rectifiable also for $\sfd_R$, and it holds $L_{\sfd}(\gamma|_{[0,t]}) = L_{\sfd_R}(\gamma|_{[0,t]})$. We deduce that $\|v_x\|_{g_R} = \|v_x\|_{g}$, which proves \eqref{eq:naturalobject2}.

As a consequence, letting $k(x)=\dim\hor_x$, we can find a $g_R$-orthonormal basis $v_1,\dots,v_n \in T_x M$ such that $v_1,\dots,v_{k(x)}$ is an orthonormal basis for $\hor_x$ with respect to the sub-Riemannian metric at the point $x$. Therefore, if $(x,y)\notin \Cut(M)$, it holds
\begin{equation}\label{eq:derivate}
\sfD(x,y)^2 = \|\nabla^R_x c(\cdot,y)\|_{g_R}^2 = \sum_{i=1}^{k(x)} \langle \di_x c(\cdot,y),v_i\rangle^2 + \sum_{i=k(x)+1}^{n} \langle \di_x c(\cdot,y),v_i\rangle^2,
\end{equation}
where $\langle \cdot,\cdot\rangle : T_x^*M\times T_x M \to \R$ denotes the dual action of covectors on vectors. Notice that $\di_x c(\cdot,y)=\lambda^{x,y}$ is the initial covector of the unique geodesic $\gamma:[0,1]\to M$ joining $x$ with $y$: this follows by the characterization of sub-Riemannian geodesics via the Lagrange multipliers rule, cf.\ for example \cite[Prop.\@ 4.3, Lemma 2.20]{ABR-curvature}. The first term in the r.h.s.\@ of \eqref{eq:derivate} is $2H(\lambda^{x,y}) = \sfd(x,y)^2$, concluding the proof of \eqref{eq:naturalobject3}.
\end{proof}

\section{Sub-Riemannian comparison theory}\label{sec:comparison}

We recall from \cite{BRMathAnn} a general comparison theory for lower Ricci curvature bounds in the sub-Riemannian setting. We remind first what comparison models are in this setting.

\subsection{LQ optimal control problems}\label{sec:LQ}

Comparison models are a special class of dynamical systems, called \emph{linear quadratic optimal control problems} (LQ in the following), and are a classical topic in optimal control theory. They arise as variational problems in $\R^n$ with a quadratic cost and linear dynamics. We recall their general features, and we refer to \cite[Ch.\@ 16]{Agrachevbook}, \cite[Ch.\@ 1]{Coron} and \cite[Ch.\@ 7]{Jurdjevicbook} for further details.

Let $A,B$ be $\ell\times \ell$ real matrices, with $B$ symmetric and $B \geq 0$. Letting $k\leq \ell$ be the rank of $B$, there exist $b_1,\dots,b_k\in \R^\ell$, unique up to orthogonal transformations, such that $B = \sum_{i=1}^k b_i b_i^*$. Let also $Q$ be a symmetric $\ell\times \ell$ real matrix, and $T>0$. We are interested in \emph{admissible trajectories}, namely  curves $q:[0,T]\to \mathbb{R}^\ell$ for which there exists a \emph{control} $u \in L^2([0,T],\mathbb{R}^k)$ such that
\begin{equation}\label{eq:lq1}
\dot{q} = Aq + \sum_{i=1}^k u_i b_i.
\end{equation}
Thus, we look for admissible trajectories with fixed endpoints  $q(0)=x$,  $q(T)=y$, that minimize the quadratic functional $C_T: L^2([0,T],\mathbb{R}^k) \to \mathbb{R}$
\begin{equation}\label{eq:lq2}
C_T(u) = \frac{1}{2}\int_{0}^{T} \left(u^* u - q^*Qq \right)dt.
\end{equation}
Admissible trajectories minimizing \eqref{eq:lq2} are called \emph{minimizers}. The vector $Aq$ represents the \emph{drift}, while  $b_1,\dots,b_k$ are the \emph{controllable directions}. The matrix $Q$ is the \emph{potential} of the LQ problem.

We only deal with \emph{controllable} systems, i.e.\@ for which there exists $s >0$ such that
\begin{equation}\label{eq:kalman}
\rank(B,AB,\ldots,A^{s-1}B) = \ell.
\end{equation}
Condition \eqref{eq:kalman} is known as \emph{Kalman condition} in control theory. It is equivalent to the fact that, for any choice of $x,y\in \R^\ell$ and $T>0$, there is a non-empty set of admissible trajectories $q :[0,T] \to \R^\ell$ joining $x$ with $y$.

It is well-known that the admissible trajectories minimizing \eqref{eq:lq2} are projections $(p,q) \mapsto q$ of the solutions of the Hamilton equations
\begin{equation} \label{eq:eqhhh}
\dot{p}  = -\partial_q H, \qquad \dot{q} = \partial_p H, \qquad (p,q) \in T^*\R^\ell = \R^{2\ell},
\end{equation}
where the Hamiltonian function $H: \mathbb{R}^{2\ell} \to \mathbb{R}$ is defined by
\begin{equation}\label{eq:Hamiltonian}
H(p,q) = \frac{1}{2}\left( p^* B p +2 p^* A q + q^* Q q \right).
\end{equation}
We remark that the Hamiltonian flow is defined for all times, since $H$ is quadratic.

\begin{definition}[Conjugate time]
We say that $t_*>0$ is a \emph{conjugate time} if there exists a non-trivial solution of the Hamilton equations \eqref{eq:eqhhh} such that $q(0) = q(t_*) = 0$.
\end{definition}

The first conjugate time $t_c = \inf\{t_*>0 \mid t_* \text{ is a conjugate time}\} \in (0,+\infty]$ determines existence and the uniqueness of solutions of minimizers, as specified by the following proposition (see \cite[Sec.\ 16.4]{Agrachevbook}).
 
\begin{proposition}[Conjugate times and existence of minimizers]\label{p:LQconj}
Let $t_c$ be the first conjugate time of the LQ problem \eqref{eq:lq1}-\eqref{eq:lq2}. Then, for any $x,y\in \R^\ell$,
\begin{itemize}
\item if $T<t_c$ there exists a unique minimizer between $x$ with $y$;
\item if $T>t_c$ there exist no minimizers between $x$ with $y$;
\item if $T= t_c$ existence of minimizers depends on $x$, $y$.
\end{itemize}
\end{proposition}

The minimization of the functional \eqref{eq:lq2} with fixed endpoints and $T>0$ does not define a metric on $\R^{\ell}$, in general. Nevertheless, one can still define a distortion coefficient.

\begin{definition}[LQ distortion coefficient]\label{d:LQdist}
Consider the LQ problem \eqref{eq:lq1}-\eqref{eq:lq2}, with $T=1$. For $x,y \in \R^\ell$ and $t\in[0,1]$, define
\begin{equation}
Z_{t}(x,y)=\{q(t)\mid q : [0,1] \to \R^\ell  \text{ is a minimizer s.t. } q(0) = x,\, q(1) = y\}.
\end{equation}
The \emph{distortion coefficient} of the LQ problem is
\begin{equation}
\beta_{t}^{A,B,Q}(x,y)=\limsup_{r\to 0}\frac{|Z_{t}(x,B_{r}(y))|}{|B_{r}(y)|}, \qquad t \in [0,1],
\end{equation}
where $x,y\in \R^\ell$, $B_r(y)$ denotes the Euclidean ball with center $y$ and radius $r>0$, and $|\cdot|$ denotes the Lebesgue measure of $\R^\ell$.
\end{definition}

The next proposition, proven in \cite[Prop.\@ 27]{BRMathAnn}, links the distortion coefficient with the Hamilton equations of the LQ problem.
\begin{proposition}\label{p:LQdist}
Consider the LQ problem \eqref{eq:lq1}-\eqref{eq:lq2}, with $T=1$, and assume that $t_c>1$. Its distortion coefficient does not depend on the choice of $x,y$, and satisfies
\begin{equation}\label{eq:firstrepformula}
\beta_{t}^{A,B,Q}=\frac{\det N(t)}{\det N(1)}>0, \qquad \forall\, t \in [0,1],
\end{equation}
where $M(t),N(t) :[0,+\infty) \to \mathrm{Mat}(\ell \times \ell)$ are the solutions of the Hamiltonian system
\begin{equation}\label{eq:hamiltoniansystem}
\frac{d}{dt} \begin{pmatrix}
M \\
N
\end{pmatrix} = \begin{pmatrix}
-A^* & - Q \\
B &  A
\end{pmatrix} \begin{pmatrix}
M \\
N
\end{pmatrix}, \qquad \begin{pmatrix}
M(0) \\N(0) 
\end{pmatrix} = \begin{pmatrix}
\mathbb{1} \\ \mathbb{0}
\end{pmatrix}.
\end{equation}
Equivalently, we have
\begin{equation}\label{eq:seconrepformula}
\beta_{t}^{A,B,Q}=\exp\left( -\int_{t}^{1}\tr(BV(s) +A) \di s\right) >0, \qquad \forall\, t \in (0,1],
\end{equation}
where $V :(0,t_c) \to \mathrm{Sym}(\ell\times \ell)$ is the maximal solution of the matrix Riccati equation
\begin{equation}\label{eq:RiccatiLQ}
\dot{V} + A^*V + VA + VBV + Q = \mathbb{0}, \qquad \lim_{t\to 0^+} V^{-1}(t) = \mathbb{0}.
\end{equation}
\end{proposition}
Notice that, by the Kalman condition \eqref{eq:kalman}, the Cauchy problem with limit initial datum \eqref{eq:RiccatiLQ} is well-posed, its solution is well-defined on the maximal interval $(0,t_c)$, where $t_c$ is the first conjugate time of the corresponding LQ problem. In particular, the standing assumption $t_c>1$ makes \eqref{eq:seconrepformula} well-defined (see \cite[Appendix A]{BR-comparison}).

\subsection{Constant curvature models}\label{sec:constcurvmodels}

A LQ problem depends on data $A,B,Q$. In this section we explain how we choose these, yielding a class of constant curvature models. The matrices $A$ and $B$ are determined by a Young diagram. The general case can be reduced to diagrams with one row, which can be seen as basic building blocks. Consider thus the following Young diagram:
\begin{equation}
\ytableausetup{centertableaux}
Y = \begin{ytableau}
1 & 2 & \none[\cdots] & \ell
\end{ytableau}\,,
\end{equation}
of length $\ell\in \N$. The case $\ell=1$ must be thought of as the ``Riemannian case''. We associate the matrices $A,B$ as follows:
\begin{equation}\label{eq:AandB}
A = \Gamma_1^*(Y), \qquad B = \Gamma_2(Y),
\end{equation}
(the transpose is a convention to agree with the one of \cite{BR-comparison}), where
\begin{equation}\label{eq:Gamma}
\Gamma_1(Y) := \begin{pmatrix}
\mathbb{0} & \mathbb{1}_{\ell-1} \\
0 & \mathbb{0}
\end{pmatrix} , 
\qquad \Gamma_2(Y) := \begin{pmatrix}
1 & \mathbb{0} \\
\mathbb{0} & \mathbb{0}_{\ell-1}
\end{pmatrix}.
\end{equation}
The matrices $A=\Gamma_{1}^*(Y)$ and $B=\Gamma_{2}(Y)$ satisfy the Kalman condition \eqref{eq:kalman}. The potential $Q$ is a diagonal matrix, whose entries represent the Ricci curvature bounds: 
\begin{equation}\label{eq:QQ}
Q = \mathrm{diag}(\kappa_1,\dots,\kappa_\ell), \qquad \kappa \in \R^\ell.
\end{equation}
\paragraph*{Notation.} Following the above prescriptions, a basic LQ model is uniquely specified by the choice of $\ell\in \N$, $\kappa \in \R^\ell$. We denote the corresponding distortion coefficient by
\begin{equation}
\beta_t^\kappa:=\beta_t^{A,B,Q}, \qquad \text{$A$, $B$, $Q$ as in \eqref{eq:AandB}-\eqref{eq:QQ}}.
\end{equation}

These LQ problems have constant curvature in the sense discussed in Appendix \ref{a:canonicalframe}. More precisely, one can show that any Jacobi curve of the corresponding Hamiltonian flow has Young diagram $Y$ of one row of length $\ell$, constant canonical curvature $\mathfrak{R}_{\lambda_t} = Q$, and constant canonical Ricci curvatures $\mathfrak{Ric}_{\lambda_t}^{i} = \kappa_{i}$, where $i=1,\dots,\ell$. This can be proved by observing that the restrictions to any extremal of $E_i=\partial_{p_i}$ and $F_i =\partial_{q_i}$, for $i=1,\dots,\ell$ is a canonical frame along that extremal (see Section \ref{sec:can-along-extremal}).

\begin{proposition}[Basic models]\label{p:basic}
Let $\ell\in\N$, and $\kappa\in \R^\ell$. Consider the LQ problem \eqref{eq:lq1}-\eqref{eq:lq2} on $\R^\ell$, with $A$, $B$ as in \eqref{eq:AandB} and $Q$ as in \eqref{eq:QQ}. Then there exist a computable real-analytic function $\sfs_\kappa: [0,+\infty) \to \R$, and $t_\kappa \in (0,+\infty]$, such that:
\begin{enumerate}[(i)]
\item \label{i:basic-1} $\sfs_\kappa$ is strictly positive on $(0,t_\kappa)$, and vanishes at the endpoints;
\item \label{i:basic-2} the first conjugate time of the LQ problem is $t_\kappa$;
\item \label{i:basic-3} there exists $c>0$ depending only on $\ell$ (and not on $\kappa$) and $N=\ell^2$ such that 
\[
\sfs_\kappa(t) \sim c t^{N}, \qquad t \to 0;
\]
\item \label{i:basic-4} assuming that $t_\kappa>1$, then the distortion coefficient of the LQ problem is
\begin{equation}
\beta_t^{\kappa} = \frac{\sfs_\kappa(t)}{\sfs_\kappa(1)}, \qquad \forall\, t \in [0,1];
\end{equation}
\item \label{i:basic-5} for $\lambda\geq 0$, let $\lambda \odot \kappa :=(\kappa_1 \lambda^2,\dots,\kappa_\ell \lambda^{2\ell})$. Then it holds:
\begin{equation}
t_{\lambda\odot\kappa} = \frac{t_\kappa}{\lambda}, \qquad \text{and} \qquad \sfs_{\lambda\odot \kappa}(t) = \frac{\sfs_\kappa(t\lambda)}{\lambda^{\ell^2}} ,\qquad \forall\, t\in \R,
\end{equation}
with the convention that, if $\lambda=0$, then $t_0=+\infty$ and $\sfs_0(t) = t^{\ell^2}$ for all $t\in [0,1]$;
\item \label{i:basic-6} let $m\in \N$, and assume that $\bar{\kappa} : \R^m_+\to \R^\ell$ where each $\bar{\kappa}_i : \R^m_+ \to \R$ is a homogeneous function of degree $2i$, for $i=1,\dots,\ell$. Let $\sfs_{\bar{\kappa}}: \R^m_+ \to \R$ be the function defined by
\begin{equation}\label{eq:link}
\sfs_{\bar{\kappa}}(\theta) := \sfs_{\bar{\kappa}(\theta/|\theta|)}(|\theta|), \qquad \forall\, \theta \in \R^m_+,
\end{equation}
with the convention $\sfs_{\bar{\kappa}}(0) = 0$.

The function $\sfs_{\bar\kappa}$ is radially real-analytic, satisfies $\sfs_{\bar{\kappa}}(\theta) = c|\theta|^N + o(|\theta|^N)$ for some $c>0$ and $N=\ell^2$. Its positivity domain, that is the largest open star-shaped subset $\DOM_{\bar{\kappa}}\subseteq \R^m_+$ such that $\sfs_{\bar{\kappa}}>0$ on $\DOM_{\bar{\kappa}}\setminus \{0\}$, is characterized as
\begin{equation}
\DOM_{\bar{\kappa}} = \{ \theta\in \R^m_+ \mid 1 < t_{\bar{\kappa}(\theta)} \} = \{\theta \in \R^m_+ \mid |\theta| < t_{\bar{\kappa}(\theta/|\theta|)}\},
\end{equation}
and for all $t\in [0,1]$ and $\theta \in \DOM_{\bar{\kappa}}$ it holds
\begin{equation}
\beta_t^{\bar{\kappa}(\theta)} = \begin{cases}
t^N & |\theta|=0, \\
\frac{\sfs_{\bar{\kappa}}(t\theta)}{\sfs_{\bar{\kappa}}(\theta)} & |\theta|\neq 0 \text{ and } \theta \in \DOM_{\bar{\kappa}},
\end{cases}
\end{equation}
with $N=\ell^2$.
\end{enumerate}
\end{proposition}
\begin{remark}[Link with CD theory]\label{r:link} 
Proposition \ref{p:basic}, and in particular item \ref{i:basic-6}, is the bridge between the curvature-dimension theory of Section \ref{sec:CDbeta} and the comparison theory of Section \ref{sec:comparison}. It establishes that model distortion coefficients of LQ problems coincide with model distortion coefficients in the sense of (\ref{eq:defbeta}) (for the scalar-valued case) and (\ref{eq:defbetaVec}) (for the vector-valued case).
\end{remark}
\begin{remark}[Scalar case]\label{rmk:identify}
Item \ref{i:basic-6} is stated in the general case $m\in \N$, in order to describe model coefficients for vector-valued gauge functions. The scalar case corresponds to $m=1$. In this case, we identify $\bar\kappa: \R^m_+\to \R^\ell$ with $\kappa\in \R^\ell$ by the identity $\bar{\kappa}_i(\theta) = \kappa_i \theta^{2i}$ for all $i=1,\dots,\ell$. Our notation is consistent since it holds
\begin{equation}
\sfs_{\bar{\kappa}}(t\theta) = \sfs_\kappa(t\theta) \qquad \text{and} \qquad \DOM_{\bar{\kappa}} = [0,t_{\kappa}).
\end{equation}
In particular, there is no need to introduce the $\sfs_{\bar\kappa}$ functions, and the positivity domain is just a segment. For general $m\in \N$, however
\begin{equation}
\sfs_\kappa :[0,+\infty) \to \R\qquad\text{and}\qquad \sfs_{\bar{\kappa}} : \R^m_+\to \R
\end{equation}
act on different domains. The link between them is provided by \eqref{eq:link}.
\end{remark}
\begin{remark}[Convexity]
If $\ell=m=1$ then $(0,t_\kappa)\ni \theta\mapsto \log\beta_t^{\kappa}(\sqrt{\theta})$ is convex and monotone for fixed $t \in [0,1]$ (the case of Riemannian coefficients, see Section \ref{ss:Riemanniancase}). Convexity, in particular, is important for the deduction of the standard form of the $\CD(K,N)$ inequality. Both properties fail, in general, if $\ell > 1$. Take for example $\kappa = (2, -1)$, and compute the solution of the Hamiltonian system \eqref{eq:hamiltoniansystem} (the values are chosen so that the system has a simple Jordan form, with eigenvalues $\pm i$).
\end{remark}

\begin{proof}
Using the notation of Proposition \ref{p:LQdist}, we set
\begin{equation}
\sfs_\kappa(t) := \det N(t).
\end{equation}
By definition, the first conjugate time of the LQ problem is 
\begin{equation}
t_\kappa :=\inf\{t>0 \mid \sfs_\kappa(t) =0\}.
\end{equation}
The Kalman's condition implies that $\det N(t)\neq 0$ for small $t>0$ so that $t_\kappa >0$. Furthermore one can check, by studying the asymptotic behaviour of the solutions of \eqref{eq:hamiltoniansystem}, that $\sfs_\kappa(t)>0$ for small $t$. Thus \ref{i:basic-1}, \ref{i:basic-2} and \ref{i:basic-4} follow immediately.

To prove \ref{i:basic-3}, we first observe that $\sfs_\kappa(\cdot)$ is real-analytic, and thus there exist\footnote{There should be no confusion between $N\in \mathbb N$ and the family of matrices $N(t)$.} $N\in \N$ and $c \neq 0$ such that $\sfs_\kappa(t) \sim c t^N$ as $t\to 0$. To prove that $N=\ell^2$, we claim that the following asymptotic formula holds:
\begin{equation}\label{eq:requiredasymptotics}
\lim_{t\to 0} \left(\frac{\di}{\di t}\log\det N(t) - \frac{\ell^2}{t}\right) = 0.
\end{equation}
To prove it, observe that the first term on the left hand side is
\begin{equation}
\frac{\di}{\di t}\log\det N(t)= \tr(\dot{N}(t) N(t)^{-1}) = \tr(\Gamma_2 M(t)N(t)^{-1} + \Gamma_1^*).
\end{equation}
The asymptotics of $M(t)N(t)^{-1}$ as $t\to 0$ is obtained in \cite[Thm.\@ 7.4 and Cor.\ 7.5]{ABR-curvature}, yielding \eqref{eq:requiredasymptotics}. We do not repeat the details here. We remark that \eqref{eq:requiredasymptotics} can also be proved by deducing asymptotic formulas for the solutions $M(t),N(t)$ of the Hamiltonian system in Proposition \ref{p:LQdist}, using the properties of $A,B$ in \eqref{eq:AandB}. Such a proof also shows that the constant $c$ depends only on $\ell$ and not on $Q$ (i.e.\@ not on $\kappa\in\R^\ell$), and $c >0$.

To prove \ref{i:basic-5}, for all $\lambda>0$ let
\begin{equation}
Q=\mathrm{diag}(\kappa_1,\dots,\kappa_\ell), \qquad Q_{\lambda} = \mathrm{diag} (\kappa_1 \lambda^2,\dots,\kappa_\ell \lambda^{2\ell}),
\end{equation}
be the potentials corresponding to $\kappa$ and $\lambda \odot \kappa$, respectively. We remark that
\begin{equation}
Q_\lambda = \Omega_\lambda Q \Omega_{\lambda}, \qquad \Omega_\lambda := \mathrm{diag}(\lambda,\dots,\lambda^\ell).
\end{equation}
For the specific choice $A=\Gamma_1^*(Y)$ and $B=\Gamma_2(Y)$ of the statement, see \eqref{eq:AandB}, we denote by $M(t),N(t)$ (resp.\ $M_{\lambda}(t),N_{\lambda}(t)$) the solution of the Hamiltonian system \eqref{eq:hamiltoniansystem} for potential $Q$ (resp.\ $Q_{\lambda}$). We observe from the explicit form \eqref{eq:Gamma} that
\begin{equation}
\Omega_{\lambda} A = \lambda A \Omega_{\lambda} ,\qquad B\Omega_{\lambda} = \Omega_{\lambda} B = \lambda B.
\end{equation}
Using these properties and \eqref{eq:hamiltoniansystem} we immediately see that
\begin{equation}
M_{\lambda}(t) = \Omega_{\lambda} M(t\lambda)\Omega_{\lambda}^{-1},\qquad N_{\lambda}(t)=\lambda \Omega_{\lambda}^{-1} N(t\lambda)\Omega_{\lambda}^{-1},\qquad \forall\, t\in \R.
\end{equation}
It follows that, for conjugate times $t_{\lambda\odot \kappa} = t_{\kappa}/\lambda$, and furthermore
\begin{equation}
\sfs_{\lambda\odot\kappa}(t) =\det N_{\lambda}(t) = \lambda^{-\ell^2} \sfs_{\kappa}(t\lambda),\qquad \forall\, \lambda>0.
\end{equation}
The above rescaling formula yields  \ref{i:basic-5} in the case $\lambda>0$. The case $\lambda=0$ can be seen as a limit of the previous one, or also proved directly by solving \eqref{eq:hamiltoniansystem}; in this case $Q=0$ and thus $M(t)_{ij} = (-1)^{j-i}t^{j-i}/(j-i)!$ for $j\geq i$, and zero otherwise, while $N(t)_{ij} = (-1)^{j+1}t^{i+j-1}/(i+j-1)!$ for all $i,j=1,\dots,\ell$.

The proof of \ref{i:basic-6} is a straightforward consequence of all previous items, unraveling the notation. The fact that $\sfs_{\bar{\kappa}}$ is radially real-analytic follows from the fact that $\sfs_{\kappa}$ is real-analytic for $\kappa\in \R^\ell$. Notice that if the $\bar{\kappa}_i$ are homogeneous of degree $2i$, then for $\theta \neq 0$ we have $\bar{\kappa}(\theta) = |\theta|\odot \bar{\kappa}(\theta/|\theta|)$. Therefore it holds
\begin{equation}
\sfs_{\bar{\kappa}}(\theta) = \sfs_{\bar{\kappa}(\theta/|\theta|)}(|\theta|) \sim c |\theta|^{\ell^2}, \qquad |\theta| \to 0,
\end{equation}
where we used \ref{i:basic-3}, and the fact that the constant $c>0$ in \ref{i:basic-3} does not depend on $\kappa \in \R^\ell$. The characterization of the positivity domain follows from the homogeneity property \ref{i:basic-5}. Finally, by definition, and using \ref{i:basic-5}, it holds
\begin{equation}
\beta_t^{\bar{\kappa}(\theta)} = \frac{\sfs_{\bar{\kappa}(\theta)}(t)}{\sfs_{\bar{\kappa}(\theta)}(1)} =
\frac{\sfs_{|\theta|\odot\bar{\kappa}(\theta/|\theta|)}(t)}{\sfs_{|\theta|\odot\bar{\kappa}(\theta/|\theta|)}(1)}=
\frac{\sfs_{\bar{\kappa}(\theta/|\theta|)}(t|\theta|)}{\sfs_{\bar{\kappa}(\theta/|\theta|)}( |\theta|)}=
 \frac{\sfs_{\bar{\kappa}}(t\theta)}{\sfs_{\bar{\kappa}}(\theta)},
\end{equation}
for all $\theta \in \DOM_{\bar{\kappa}}\setminus\{0\}$ and all $t\in (0,1)$. Finally if $|\theta|=0$, we have $\bar\kappa(0)=0$ and we proceed in a similar fashion reminding that $\sfs_0(t) = t^{\ell^2}$ by item \ref{i:basic-5}. The cases $t=0$ and $t=1$ are trivially verified.
\end{proof}

The proof of Proposition \ref{p:basic} is constructive and the distortion coefficients are obtained by solving the Hamiltonian system \eqref{eq:hamiltoniansystem} associated with corresponding LQ problem. In the next three sections we give explicit examples, recovering all typical comparison functions that have arisen in the recent literature in sub-Riemannian geometry.

\subsubsection{The Riemannian case}\label{ss:Riemanniancase}
Let $\ell=1$, so that $Y=\ytableausetup{smalltableaux}\ytableaushort{\empty}$. According to \eqref{eq:AandB}, we have $A=0$, $B=1$, and we set $Q=\kappa$ for some $\kappa \in \R$. The Hamiltonian of the LQ problem is
\begin{equation}
H(p,q) = \frac{1}{2} \left(p^2+ \kappa q^2\right),
\end{equation}
which is the Hamiltonian of a one-dimensional harmonic oscillator. From Proposition \ref{p:basic} we obtain easily that
\begin{equation}
\sfs_{\kappa}(t) = \begin{cases} \sin(\sqrt{\kappa} t) & \kappa > 0, \\
\sinh(\sqrt{-\kappa}  t) & \kappa \leq 0,
\end{cases} \qquad t_{\kappa} = \begin{cases}
\frac{\pi}{ \sqrt{\kappa}} & \kappa >0, \\
+\infty & \kappa \leq 0.
\end{cases}
\end{equation}
The typical right hand side of Ricci curvature lower bounds has the form $\bar{\kappa}(\sfd) = \kappa \sfd^2$. Thus, according to Proposition \ref{p:basic}, we obtain that for $\theta \in [0,+\infty)$
\begin{equation}
\sfs_{\bar{\kappa}}(\theta) = \sfs_{\kappa}(\theta),  \qquad \DOM_{\bar\kappa} = [0,t_{\kappa}).
\end{equation}
Here we stress that $m=1$. The corresponding distortion coefficient is:
\begin{equation}
\beta_{t}^{\bar{\kappa}(\theta)} = \beta_{t}^{\kappa\theta^2}= \begin{cases}
\frac{\sin (\sqrt{\kappa} \theta t)}{\sin (\sqrt{\kappa}\theta)} & \kappa >0,\\
t & \kappa =0,\\
\frac{\sinh (\sqrt{|\kappa|} \theta t)}{\sinh (\sqrt{|\kappa|}\theta)} & \kappa<0,
\end{cases} \qquad \forall\, t\in [0,1], \quad \theta \in \DOM_{\bar\kappa}= [0,t_{\kappa}).
\end{equation}
The condition $\theta \in \DOM_{\bar\kappa}$ corresponds to $\bar\kappa(\theta) = \kappa \theta^2 < \pi^2$, which is reminiscent of the Bonnet-Myers bound.

\paragraph{Relation with classical coefficients.} This is the building block for the 
construction of Riemannian distortion coefficients. Consider a space form of dimension 
$N$ and (sectional) curvature equal to $K/(N-1)$, and thus Ricci curvature equal to 
$K\in \R$. The potential ``felt'' in the directions orthogonal to the one of the motion, 
along a geodesic of length $\sfd$, is $K\sfd^2/(N-1)$. There are $N-1$ 
orthogonal directions, and one direction where there is no curvature (the direction of the 
motion). Therefore, taking the product of all these factors we obtain
\begin{equation}
\beta_t^0 \prod_{i=1}^{N-1} \beta_t^{K\sfd^2/(N-1)} =  \begin{cases}
t \left(\frac{\sin (t\sfd \sqrt{K/(N-1)})}{\sin (\sfd\sqrt{ K/(N-1)})}\right)^{N-1} & K>0,\\
t^N & K=0,\\
t \left(\frac{\sinh (t\sfd \sqrt{|K|/(N-1)})}{\sinh (\sfd\sqrt{ |K|/(N-1)})}\right)^{N-1} & K<0,
\end{cases}
\quad \forall\, t\in [0,1],\ \sfd < t_{K/(N-1)}.
\end{equation}
This is the usual distortion coefficient for a space form of Ricci curvature $K$ and dimension $N$. Compare with Section \ref{sec:howtorecover}.

\subsubsection{The Sasakian case}\label{ss:Sasakiancase}

In this section we describe the typical coefficients arising in the description of three-dimensional Sasakian space forms, which were the first sub-Riemannian structures where an effective comparison theory was established \cite{HughenPhD,AAPL-Ricci, LeeLiZel-Sasakian}.

Let $\ell=2$, so that $Y=\ytableausetup{smalltableaux}\ytableaushort{\empty\empty}$. We have $A=\left(\begin{smallmatrix}0 & 0 \\ 1 & 0 \end{smallmatrix}\right)$, $B=\left( \begin{smallmatrix} 1 & 0 \\ 0 & 0\end{smallmatrix}\right)$, and we set $Q=\mathrm{diag}(\kappa_1,0)$ for some $\kappa_1 \in \R$ (formally, $\kappa=(\kappa_1,\kappa_2)$ with $\kappa_2=0$, which greatly simplifies this case). Proposition \ref{p:LQdist} can be used to obtain
\begin{equation}
\mathsf{\sfs}_{\kappa_1,0}(t) = 
\frac{2 -2 \cos \left(\sqrt{\kappa_1} t\right)-\sqrt{\kappa_1} t 
\sin \left(\sqrt{\kappa_1} t\right)}{\kappa_1^2}, 
\qquad t_{\kappa_1,0}
 = 
\begin{cases}
\tfrac{2\pi}{\sqrt{\kappa_1}} & \kappa_1>0, \\ +\infty & \kappa_1\leq 0,
\end{cases}
\end{equation}
For brevity, instead of writing explicitly all the cases depending on the sign of $\kappa_1$ as done in Section \ref{ss:Riemanniancase},  here $\mathsf{\sfs}_{\kappa_1,0}(t) $ is understood as real-analytic function of $\kappa_{1}\in \R$, to interpret it for $\kappa_1 \leq 0$. Furthermore we obtain
\begin{equation}
\beta_t^{\kappa_1,0} =  \frac{2-2 \cos \left(\sqrt{\kappa_1} t\right)-\sqrt{\kappa_1} t \sin \left(\sqrt{\kappa_1} t\right)}{2-2 \cos \left(\sqrt{\kappa_1}\right)-\sqrt{\kappa_1} \sin \left(\sqrt{\kappa_1}\right)}.
\end{equation}

On a three-dimensional Sasakian manifold with Tanaka-Webster curvature equal to $K\in \R$, the sharp distortion coefficient along a geodesic $\gamma$ is the product of two factors. One flat factor $t$, arising from a one-dimensional subspace (the direction of the motion), and one factor $\beta_t^{\kappa_1,0}$ as the one we just described, with 
\begin{equation}
\kappa_1=\bar{\kappa}_1(|h_0|,\sfd) := |h_0|^2+K \sfd^2,
\end{equation}
where $\sfd$ is the length of the geodesic and $h_0$ is the Reeb component of its initial covector. In particular $\bar{\kappa}_1 :\R^2_+ \to \R$ is homogeneous of degree $2$ and the distortion coefficient thus obtained is suitable for a theory with vector-valued gauge function (the gauge function being $\sfG = (\sfG_1,\sfG_2)$ with $\sfG_1=|h_0|$ and $\sfG_2=\sfd$).	

We will describe more precisely this case in Section \ref{sec:ex:vectorial}.

\subsubsection{The two columns case}\label{ss:twocolumnscase}

This example generalizes the previous one. Let again $\ell =2$ so that $Y=\ytableausetup{smalltableaux}\ytableaushort{\empty\empty}$, and let $Q = \mathrm{diag}(\kappa_1,\kappa_2)$, with $\kappa_1,\kappa_2 \in \R$. The difference with respect to Section \ref{ss:Sasakiancase} is that now $\kappa_2$ may be non-zero. The Hamiltonian of the corresponding LQ problem is
\begin{equation}
H(p,q) = \frac{1}{2} \left(p_1^2 +2p_{2}q_{1}+\kappa_{1} q_{1}^2+\kappa_{2}q_{2}^{2}\right).
\end{equation}
We can compute the distortion coefficient using Proposition \ref{p:LQdist}. By reduction to Jordan normal form of the Hamiltonian system \eqref{eq:hamiltoniansystem} (see details in \cite[Prop.\@ 28]{RS-3Sas}), one obtains
\begin{equation}
\sfs_{\kappa_1,\kappa_2}(t)=\frac{\xi _-^2 \sin ^2\left(\xi _+ t\right)-\xi _+^2 \sin ^2\left(\xi _- t\right)}{4 \xi _-^2 \xi _+^2(\xi _-^2- \xi _+^2)},
\end{equation} 
where we have set $\kappa_1 = 2 (\xi_+^2 + \xi_-^2)$ and $\kappa_2 = -(\xi_+^2-\xi_2^2)^2$. In other words, choosing the principal branch of the square root:
\begin{equation}
\xi_{\pm}= \frac{1}{2}(\sqrt{x+y} \pm \sqrt{x-y}),\qquad \text{with} \qquad x = \frac{\kappa_1}{2}, \quad y = \frac{\sqrt{4\kappa_2 + \kappa_1^2}}{2}.
\end{equation}
It seems not trivial to deduce an explicit expression for $t_{\kappa_1,\kappa_2}$ from the above expression. Abstract necessary and sufficient conditions in terms of $\kappa_1,\kappa_2$ for the finiteness of $t_{\kappa_1,\kappa_2}$ are obtained via the results in \cite{ARS-LQ}. Estimates for $t_{\kappa_1,\kappa_2}$ are found in \cite[Prop.\@ 1]{RS-3Sas}.

Thus, assuming $t_{\kappa_1,\kappa_2} >1$, the distortion coefficient is
\begin{equation}
\beta_{t}^{\kappa_1,\kappa_2}= \frac{\xi _-^2 \sin ^2\left(\xi _+ t\right)-\xi _+^2 \sin ^2\left(\xi _- t\right)}{\xi _-^2 \sin ^2\left(\xi _+ \right)-\xi _+^2 \sin ^2\left(\xi _- \right)}, \qquad \forall\, t\in [0,1],
\end{equation}
understood with the usual conventions as a real-analytic function of $\xi_{\pm} \in \mathbb{C}$. 

\subsection{Main comparison result}\label{sec:maincomparison}

Recall from Definition \ref{def:truedistcoeff} the distortion coefficient of a metric measure space. For a sub-Riemannian m.m.s\ $(M,\sfd,\mm)$, it can always be written as
\begin{equation}
\beta_t^{(M,\sfd,\mm)}(x,y) = t \beta_t^\perp(x,y), \qquad \forall\, (x,y)\notin \mathrm{Cut}(M).
\end{equation}
The factor $t$ corresponds to the distortion felt in the direction of the geodesic between $x$ and $y$, while the factor $\beta_t^\perp$ represents the distortion felt in the perpendicular directions.

The following result is a reformulation of \cite[Thm.\@ 16]{BRMathAnn}, with $N=n$ (cf.\@ Rmks. 17 and 12 there), with an explicit dependence on a possibly vector-valued gauge function $\sfG:M\times M \to \RP^m_+$, combined with Proposition \ref{p:basic}\ref{i:basic-6}. We will state directly the general vector-valued case, the scalar-valued one is obtained by setting $m=1$,  in which case the statement greatly simplifies, see in particular Remark \ref{rmk:identify}.

\begin{theorem}\label{thm:maincomparison}
Let $(M,\sfd,\mm)$ be a sub-Riemannian metric measure space, with $n=\dim M$, and let $\sfG : M \times M \to \R^m_+$ be a finite Borel function, $m\in \N$. Let $(x,y)\notin \mathrm{Cut}(M)$ and $\gamma :[0,1]\to M$ be a geodesic joining them, with initial covector $\lambda$, and reduced Young diagram $Y$. Assume that there exists a function $\bar{C}: \R^m_+\to \R$, homogeneous of degree $1$, such that
\begin{equation}
\rho_{\mm,\lambda} \leq \bar{C}(\sfG(x,y)).
\end{equation}
Denote with $\Upsilon$ the set of levels of $Y$. Assume that for any level $\alpha\in \Upsilon$, with size $r_\alpha$ and length $\ell_\alpha$, there exist $\bar{\kappa}_\alpha : \R^m_+ \to \R^{\ell_\alpha}$ such that each component $\bar{\kappa}_{\alpha_i} : \R^m_+\to \R$ is homogeneous of degree $2i$, for $i=1,\dots,\ell_{\alpha}$, and such that for any superbox $\alpha_i$ of that level it holds
\begin{equation}
\frac{1}{r_\alpha}\mathfrak{Ric}^{\alpha_i}_{\lambda} \geq \bar{\kappa}_{\alpha_i}(\sfG(x,y)), \qquad i=1,\dots,\ell_\alpha,
\end{equation}
with the convention that if $\alpha$ is the level of length $\ell_{\alpha}=1$ then $r_\alpha$ is replaced by $r_\alpha-1$; and if $r_\alpha =0$ after this replacement, then that level is omitted. Then:
\begin{itemize}
\item denoting with $\DOM_{\bar\kappa_\alpha} \subseteq \R^m_+$ the positivity domain of $\sfs_{\bar{\kappa}_{\alpha}}:\R^m_+\to \R$ as in Proposition \ref{p:basic}, it holds
\begin{equation}
\sfG(x,y)\in \bigcap_{\alpha\in\Upsilon} \DOM_{\bar{\kappa}_\alpha},
\end{equation}
so that the distortion coefficients $\beta_t^{\bar{\kappa}_\alpha(\sfG(x,y))}$ are well-defined for all $\alpha \in \Upsilon$ as in Proposition \ref{p:basic};
\item with the same convention as above, we have that
\begin{equation}
\frac{\beta_t^\perp(x,y)}{\displaystyle\prod_{\alpha \in \Upsilon} \left(\beta_t^{\bar{\kappa}_{\alpha}(\sfG(x,y))}\right)^{r_\alpha}}e^{-t\bar{C}(\sfG(x,y)) } \quad \text{is a non-increasing function of $t\in (0,1]$},
\end{equation}
and in particular, the following holds
\begin{equation}
\beta_t^{(M,\sfd,\mm)}(x,y) \geq t e^{(t-1)\bar{C}(\sfG(x,y))}  \prod_{\alpha \in \Upsilon} \left(\beta_t^{\bar{\kappa}_{\alpha}(\sfG(x,y))}\right)^{r_\alpha}, \qquad \forall\, t\in [0,1].
\end{equation}
\end{itemize}
\end{theorem}
\begin{remark}\label{rmk:linkbetabeta}
We recall from Proposition \ref{p:basic} that, for $\bar{\kappa} : \R^m_+ \to \R^\ell$, and $t\in [0,1]$
\begin{equation}
\beta_t^{\bar{\kappa}(\sfG)} = \begin{cases}
t^N & |\sfG|=0, \\
\frac{\sfs_{\bar{\kappa}}(t\sfG)}{\sfs_{\bar{\kappa}}(\sfG)} & |\sfG|\neq 0 \text{ and } \sfG \in \DOM_{\bar{\kappa}},
\end{cases}
\end{equation}
with $N=\ell^2$, is the distortion coefficient of the LQ model with a Young diagram of one row and $\ell$ columns, with potential $Q=\mathrm{diag}(\bar{\kappa}_1(\sfG),\dots,\bar{\kappa}_\ell (\sfG))$. In particular the functions $\sfs_{\bar{\kappa}}:  \R^m_+ \to \R$ depend only on the $\ell$ homogeneous functions $\bar{\kappa}_1,\dots,\bar{\kappa}_\ell$.
\end{remark}

Theorem \ref{thm:maincomparison} yields a comparison result for distortion along a single geodesic. If its assumptions hold globally, then we obtain the following result, establishing the compatibility of our $\CD(\beta,n)$ theory with the one of sub-Riemannian Ricci bounds. Once again, in the scalar-valued case $m=1$ what follows greatly simplifies, see Remark \ref{rmk:identify}.

\begin{definition}[Sub-Riemannian m.m.s.\@ with Ricci curvature bounded from below]\label{def:srbddbelow}
Let $(M,\sfd,\mm)$ be a sub-Riemannian metric measure space, with $n=\dim M$, equipped with a finite gauge function $\sfG : M \times M \to \R^m_+$, $m\in \N$. We say that $(M,\sfd,\mm,\sfG)$ has \emph{Ricci curvatures bounded from below} if there exist:
\begin{itemize}
\item a Young diagram $Y$,
\item functions $\bar{C}: \R^m_+\to \R$ homogeneous of degree $1$,
\item functions $\bar{\kappa}_\alpha : \R^m_+ \to \R^{\ell_\alpha}$ such that each component $\bar{\kappa}_{\alpha_i} : \R^m_+\to \R$ is homogeneous of degree $2i$, for $i=1,\dots,\ell_{\alpha}$ for any level $\alpha$ of $Y$, with length $\ell_{\alpha}$,
\end{itemize}
such that for all $x\in M$ and $\mm$-a.e.\@ $y$ there exists a unique geodesic $\gamma \in \Geo(M)$ joining $x$ with $y$, with Young diagram $Y$, initial covector $\lambda \in T_x^*M$, for which the following hold:
\begin{itemize}
\item bound on the geodesic volume derivative:
\begin{equation}
\rho_{\mm,\lambda} \leq \bar{C}(\sfG(x,y)),
\end{equation}
\item bound on the Ricci curvatures: for all levels $\alpha$ of the Young diagram $Y$
\begin{equation}
\frac{1}{r_\alpha}\mathfrak{Ric}^{\alpha_i}_{\lambda} \geq \bar{\kappa}_{\alpha_i}(\sfG(x,y)), \qquad i=1,\dots,\ell_\alpha,
\end{equation}
with the convention that if $\alpha$ is the level of length $\ell_{\alpha}=1$ then $r_\alpha$ is replaced by $r_\alpha-1$; and if $r_\alpha =0$ after this replacement, then that level is omitted.
\end{itemize}
The above data determine a function $\sfs:\R^m_+\to \R$ defined by
\begin{equation}\label{eq:sfsbddbelow}
\sfs(\theta) := |\theta|\cdot e^{\bar{C}(\theta)}\cdot \prod_{\alpha\in\Upsilon} \sfs_{\bar{\kappa}_{\alpha}}(\theta)^{r_\alpha},\qquad \forall\, \theta \in \R^m_+,
\end{equation}
where $\sfs_{\bar{\kappa}_{\alpha}} :\R^m_+ \to \R$ are as in Proposition \ref{p:basic}\ref{i:basic-6}, $\Upsilon$ denotes the set of levels of $Y$, and with the convention that if $\alpha$ is the level of length $\ell_{\alpha}=1$ then $r_\alpha$ is replaced by $r_\alpha-1$; and if $r_\alpha =0$ after this replacement, then that level is omitted from the product in \eqref{eq:sfsbddbelow}. Finally, let $\beta$ be the corresponding distortion coefficient defined as in \eqref{eq:defbetaVec}.
\end{definition}
We note that the positivity domain $\DOM$ of $\sfs$ (see Section \ref{sec:vectorial}) is
\begin{equation}
\DOM = \bigcap_{\alpha\in\Upsilon} \DOM_{\bar{\kappa}_\alpha},
\end{equation}
and that $\sfs$ satisfies the asymptotic relation $\sfs(\theta) = c |\theta|^{N} + o(|\theta|^N)$ for some $c>0$ and
\begin{equation}
N=\sum_{\alpha\in \Upsilon} r_{\alpha} \ell_{\alpha}^2.
\end{equation}

\begin{theorem}[Ideal sub-Riemannian structures with Ricci bounded below are $\CD$]\label{thm:RiccbddbelowareCD}
Let $(M,\sfd,\mm)$ be an ideal sub-Riemannian metric measure space, with $n=\dim M$, equipped with a finite gauge function $\sfG : M \times M \to \R^m_+$, $m\in \N$, with Ricci curvatures bounded from below in the sense of Definition \ref{def:srbddbelow}, and let $\beta$ be the corresponding distortion coefficient. Then $(M,\sfd,\mm,\sfG)$ satisfies the $\CD(\beta,n)$ condition.
\end{theorem}
\begin{proof}
\textbf{Step 1: comparison out of the cut locus.} We use the Ricci curvature bounds assumption of Definition \ref{def:srbddbelow}, paired with Theorem \ref{thm:maincomparison}, to deduce that for all $x\in M$ and $\mm$-a.e.\ $y\in M$
\begin{equation}
\sfG(x,y) \in \DOM= \bigcap_{\alpha \in \Upsilon} \DOM_{\bar{\kappa}_\alpha},
\end{equation}
where $\DOM$ is the positivity domain of  $\sfs: \R^m_+ \to \R$ in \eqref{eq:sfsbddbelow}, and furthermore
\begin{equation}\label{eq:comparisonbeta-proof}
\beta_t^{(M,\sfd,\mm)}(x,y) \geq \beta_t^{\mathrm{mod}}(\sfG(x,y)), \qquad \forall\, t\in [0,1],
\end{equation}
where $\beta^{\mathrm{mod}}$ in the right hand side is the distortion coefficient built out of $\sfs$ according to the construction in Section \ref{sec:CDbeta} (Section \ref{sec:vectorial} for the vector case). See Remark \ref{rmk:linkbetabeta}.

\textbf{Step 2: interpolation inequalities for densities.} By \cite[Thm.\@ 4]{BRInv}, for any $\mu_0\in \Prob_{bs}(M,\sfd,\mm)$, $\mu_1 \in \Prob_{bs}^{*}(M,\sfd,\mm)$ there exists a unique $W_2$-geodesic $(\mu_t)_{t\in [0,1]}$, associated with $\nu\in\OptGeo(\mu_0,\mu_1)$ such that $\mu_t\ll \mm$ for all $t\in (0,1]$, and letting $\rho_t:=\tfrac{\di \mu_t}{\di \mm}$ be the corresponding density, it holds for all $t\in (0,1)$:
\begin{equation}\label{eq:interpolationineq-proof}
\frac{1}{\rho_t(\gamma_t)^{1/n}}\geq \frac{\beta_{1-t}^{(M,\sfd,\mm)}(\gamma_1,\gamma_0)^{1/n}}{\rho_0(\gamma_0)^{1/n}} + \frac{\beta_t^{(M,\sfd,\mm)}(\gamma_0,\gamma_1)^{1/n}}{\rho_1(\gamma_1)^{1/n}},\quad \nu\text{-a.e.}\ \gamma \in \Geo(M),
\end{equation}
with the understanding that the first term in the right hand side of \eqref{eq:interpolationineq-proof} is omitted if $\mu_0\notin \mathcal{P}_{ac}(M,\mm)$. We also remark that $\nu\in \OptGeo(\mu_0,\mu_1)$, the corresponding optimal plan $\pi$, and the corresponding $W_2$-geodesic $\mu_t$ are unique, and they are induced by an optimal transport map $T:M\to M$ such that $T_\sharp \mu_0 = \mu_1$.

Furthermore, by \cite[Cor.\ 3.8]{BRInv}, optimal transport on ideal structures almost surely is either static or avoids the cut locus. More precisely:
\begin{equation}
\mu_1\left(M\setminus \{x\in M \mid T(x)=x\} \cup \{x\in M \mid T(x)\notin \Cut(x)\}\right)=0.
\end{equation}
Recall that the $\CD(\beta,n)$ inequality \eqref{eq:defCDbetan} must be tested only for measures such that $\supp\mu_0\cap \supp\mu_1= \emptyset$. In this case, $\nu$ is concentrated on a set of geodesics with endpoints out of the cut locus.

Thus, using \eqref{eq:comparisonbeta-proof}, we have that \eqref{eq:interpolationineq-proof} holds for $\nu$-a.e.\ geodesic $\gamma \in \Geo(M)$, with the replacements:
\begin{equation}
\beta_{1-t}^{(M,\sfd,\mm)}(\gamma_1,\gamma_0)  \mapsto \beta_{1-t}^{\mathrm{mod}}(\sfG(\gamma_1,\gamma_0)),\qquad \beta_{t}^{(M,\sfd,\mm)}(\gamma_0,\gamma_1)  \mapsto \beta_{t}^{\mathrm{mod}}(\sfG(\gamma_0,\gamma_1)).
\end{equation}

Integrating the aforementioned inequality with respect to $\nu$, and using Jensen's inequality, we obtain the $\CD(\beta,n)$ condition \eqref{eq:defCDbetan} for $\beta=\beta^{\mathrm{mod}}$.
\end{proof}
\section{Curvature estimates for fat sub-Riemannian structures}\label{sec:fat}
 
In this section, we establish new Ricci curvature lower bounds for fat sub-Riemannian structures. We address  the reader to Appendix \ref{a:canonicalframe} for a self-contained summary on the canonical curvature in sub-Riemannian geometry that we use extensively.

 We provide here an informal explanation of the basic concepts of Young diagram and Ricci curvatures. To the generic\footnote{More precisely, for an open and dense set of points $x\in M$ there exists a non-empty Zariski open (and thus dense) set $A_x \subseteq T_x^*M$ of initial covectors for which the associated geodesic has a well-defined Young diagram. See \cite[Sec.\@ 5]{ZeLi} and \cite[Sec.\@ 5.2]{ABR-curvature}.} geodesic of a sub-Riemannian manifold we can associate a Young diagram. Each Young diagram must be thought of as an ordered collection of $n=\dim M$ boxes, each one corresponding to a one-dimensional subspace of the tangent space along the given geodesic. The position of the box in the diagram is determined by the number of derivations (aka Lie brackets in the direction of the geodesic) that one must perform to access to that direction starting with the original distribution. According to the general theory, a sub-Riemannian canonical curvature operator can be defined on the tangent space along the geodesic. It behaves in a different way when restricted to each one of these one-dimensional subspaces. It is natural then to group up the one-dimensional subspaces having homogeneous behaviour. We obtain in this way a so-called \emph{reduced} Young diagram, whose boxes are called superboxes. Each superbox of a reduced Young diagram represents a subspace of the tangent space along the geodesic where the curvature behaves in a homogeneous way, and thus it makes sense to trace the curvature operator in these subspaces. This gives rise to Ricci curvatures, one for each superbox. Note that, on Riemannian manifolds, the reduced Young diagram is the same for any geodesic, it is formed by only one superbox, and a large part of the canonical curvature theory explained in Appendix \ref{a:canonicalframe} trivializes.

\subsection{Fat curvature estimates}

A distribution $\distr$ on a smooth manifold $M$ of dimension $n\geq 3$ is called \emph{fat} if it is \emph{strong bracket-generating} \cite{Strichartz,montgomerybook}: for any smooth vector field $X$ on $M$ it holds
\begin{equation}\label{eq:fat-definition}
X|_x \neq 0 \quad \Longrightarrow \quad \distr_x + [X,\distr]_x = T_x M, \qquad \forall\, x\in M.
\end{equation}
Let $g$ be a sub-Riemannian metric on $\distr$.  We call $(\distr,g)$ a \emph{fat sub-Riemannian structure} on $M$. Letting $\sfd$ be the corresponding Carnot-Carathéodory distance. 

 \begin{remark}
We recall that the following popular classes of sub-Riemannian structures are fat: the Heisenberg groups, contact structures, all H-type groups in the sense of Kaplan \cite{Kaplan}, and more generally H-type foliations \cite{BGMR-Htype2}.
\end{remark}

For a fat structure, any non-trivial geodesic has the same Young diagram with two columns.  We label the superboxes of the reduced Young diagram as follows:
\[
\ytableausetup{centertableaux,nosmalltableaux}
\begin{ytableau}
b & a \\
c\\
\end{ytableau}
\]
The sizes of the superboxes are $\size(a)=\size(b) =n-k$ and $\size(c)= 2k-n$. To each superbox, we associate the canonical Ricci curvatures $\mathfrak{Ric}^a$, $\mathfrak{Ric}^b$, $\mathfrak{Ric}^c$ respectively.

Without loss of generality, we assume that $g$ is the restriction to $\distr$ of a Riemannian metric $g_R$. Recall from Definition \ref{def:D} the function $\sfD:M\times M \to [0,+\infty]$ associated with the Riemannian extension $\sfd_{R}=\sfd_{g_R}$. By Proposition \ref{prop:naturalobject}, it holds
\begin{equation}
\sfD(x,y) = \|\nabla^R_x c(\cdot,y)\|_R, \qquad \forall\, (x,y) \notin \mathrm{Cut}(M),
\end{equation}
where $c =\tfrac{1}{2}\sfd^2$. To fix the ideas, in the Heisenberg group, and choosing the canonical Riemannian extension, for $(x,y)\notin \mathrm{Cut}(M)$ it holds
\begin{equation}
\sfD(x,y)^2 = \sfd(x,y)^2 + h_0(\lambda^{x,y})^2,
\end{equation}
where $h_0(\lambda^{x,y})$ is the vertical component of the unique initial covector $\lambda^{x,y}$ of the geodesic between $x$ and $y$.

The next result will be pivotal to show that fat sub-Riemannian structures on a compact $n$-dimensional manifold are $\CD(\beta,n)$ spaces, for suitable $\beta$ (cf.\@ Theorem \ref{thm:fatcptareCD}).

\begin{theorem}[Fat curvature estimates]\label{t:fattheorem}
Let $(M,\sfd)$ be a sub-Riemannian metric space with dimension $n$ and with fat distribution of rank $k< n$. Then, for any compact set $K\subset M$ there exists a constant $\kappa>0$ such that for any $x,y\in K\setminus \mathrm{Cut}(M)$ and for the unique geodesic $\gamma:[0,1]\to M$ joining them, with initial covector $\lambda \in T_x^*M$, it holds
\begin{align}
\mathfrak{Ric}^a_{\lambda} & \geq - \kappa \sfD(x,y)^4, \\
\mathfrak{Ric}^b_{\lambda} & \geq - \kappa \sfD(x,y)^2, \\
\mathfrak{Ric}^c_{\lambda} & \geq - \kappa \sfD(x,y)^2.
\end{align}
Furthermore, if $\mm$ is a smooth measure on $M$, there exists $C>0$ such that for any $x,y\in K\setminus \mathrm{Cut}(M)$ and for the unique geodesic $\gamma:[0,1]\to M$ joining them, with initial covector $\lambda\in T_x^*M$, it holds:
\begin{equation}
|\rho_{\mm,\lambda}| \leq C \sfD(x,y),
\end{equation}
where $\rho_{\mm,\lambda}$ is the geodesic volume derivative (see Section \ref{sec:gvd}).
\end{theorem}
\begin{remark}
If the last block is one-dimensional, i.e.\@ $\size(c)=2k-n=1$, then the lower bound for $\mathfrak{Ric}^c$ is redundant as $\mathfrak{Ric}^c_{\lambda} = 0$.
\end{remark}
\begin{remark}
The constants in Theorem \ref{t:fattheorem} do depend on the choice of $g_R$.
\end{remark}

Recall that the canonical frame and curvature, seen as functions of the initial covector $\lambda \in T^*M$, are not defined when the associated geodesic is the trivial (abnormal) one, corresponding to $\lambda \in \distr^0$ (the annihilator bundle of $\distr$). From an analytic viewpoint, this translates into a singularity as $\lambda \to \distr^0$ of the canonical frame. This does not happen in Riemannian geometry, where the canonical frame is defined even for the zero covector. 

In this section we will prove that even if the canonical frame is singular, the canonical curvature is not. More precisely, we will establish a stronger result on the structure of the canonical frame in a neighborhood of the annihilator bundle, which is the main result of this section (see Theorem \ref{t:fatcurvbound}). The latter yields Theorem \ref{t:fattheorem} as a consequence.

\subsection{Preliminary setup}\label{sec:preliminary}
We start by setting up some notation. Let $X_1,\dots,X_n$ be a local Riemannian orthonormal frame defined on a neighborhood $\mathcal{O}\subset M$, such that $X_1,\dots,X_k$ are horizontal (i.e., smooth sections of $\distr$). Define linear-on-fibers functions $h_i: T^*\mathcal{O} \to \R^n$ via $h_i(\lambda) = \langle\lambda,X_i\rangle$, for $i=1,\dots,n$. These functions define a local trivialization $T^*\mathcal{O} \simeq \mathcal{O}\times \R^n$. Let also $c_{ij}^k : \mathcal{O}\to \R$ be the smooth functions defined by
\begin{equation}
[X_i,X_j] = \sum_{k=1}^n c_{ij}^k X_k, \qquad i,j=1,\dots,n.
\end{equation}
We routinely use symplectic calculus (cf.\@ \cite[Ch.\@ 4]{nostrolibro}). We recall some basic concepts here. To any $a \in C^\infty(T^*\mathcal{O})$ we associate the vector field $\vec{a}$ on $T^*\mathcal{O}$ 
\begin{equation}
\vec{a}(f) = \{a,f\}, \qquad \forall\, f \in C^\infty(T^*\mathcal{O}),
\end{equation}
where $\{\cdot,\cdot\}$ denotes the Poisson bracket induced by the symplectic structure on $T^*M$. Thus, on $T^*\mathcal{O}$, we define a frame for $T(T^*\mathcal{O})$ given by
\begin{equation}\label{eq:frame}
\vec{h}_1,\dots , \vec{h}_n,\partial_{h_1},\dots,\partial_{h_n}.
\end{equation}
Notice that it holds $\pi_* \vec{h}_i = X_i$, where $\pi :T^*M \to M$ is the bundle projection. Let also $\nu_1,\dots,\nu_n$ be the dual frame of one-forms on $\mathcal{O}$ such that $\langle \nu_i,X_j\rangle=\delta_{ij}$ for $i,j=1,\dots,n$. In terms of this frame, the tautological one-form is $\tau = \sum_{\ell=1}^n h_\ell \pi^*\nu_\ell$ on $T^*\mathcal{O}$ and the symplectic form $\sigma = d\tau$ is
\begin{equation}
\sigma = \sum_{\ell=1}^n dh_\ell \wedge \pi^*\nu_\ell + h_\ell \pi^* d\nu_\ell.
\end{equation}
In terms of the frame \eqref{eq:frame} we have, on $T^*\mathcal{O}$ and for $i,j=1,\dots,n$:
\begin{equation}\label{eq:sigmaframe}
\sigma(\partial_{h_i},\partial_{h_j})=0, \qquad \sigma(\partial_{h_i},\vec{h}_j)=\delta_{ij},\qquad \sigma(\vec{h}_i,\vec{h}_j) =  \sum_{\mu=1}^n h_\mu c_{ij}^\mu.
\end{equation}
Let $H : T^*M \to \R$ be the sub-Riemannian Hamiltonian and denote with $\vec{H}$ the corresponding Hamiltonian vector field, which we assume to be complete, and denote with $e^{t\vec{H}}$ its flow. The curves $\lambda :[0,1]\to T^*M$ given by
\begin{equation}
\lambda_t = e^{t\vec{H}}(\lambda), \qquad \forall\, t\in [0,1],\qquad \forall\, \lambda \in T^*M,
\end{equation}
are called normal extremals, and the projections $\gamma = \pi\circ\gamma$ are locally length-minimizing curves parametrized with constant speed $\|\dot\gamma_t\|^2 = 2H(\lambda)$, for all $t\in [0,1]$. In particular if $\gamma$ is a geodesic it holds $\sfd(\gamma_0,\gamma_1)^2 = 2H(\lambda)$.

For any tensor field on $T^*M$, and without risk of confusion, the dot denotes the derivation with respect to $\vec{H}$. For example if $V\in \Gamma(T(T^*M))$:
\begin{equation}
\dot{V}|_{\lambda}= \left.\frac{d}{d\varepsilon}\right|_{\varepsilon=0} e^{-\varepsilon \vec{H}}_*V|_{e^{\varepsilon\vec{H}}(\lambda)}=[\vec{H},V]|_{\lambda}, \qquad \forall\, \lambda \in T^*M,
\end{equation}
and for functions $a :T^*M \to \R$, we have $\dot{a} = \vec{H}(a)$.

The frame \eqref{eq:frame} will serve as a reference for the computation of the canonical frame. It is convenient to adopt the following notation to distinguish between ``horizontal'' and ``vertical'' coordinates. We let $u_i = h_i$ for $i=1,\dots,k$ and $v_j = h_j$ for $j=k+1,\dots,n$. In particular $u,v$ denote tuples $u=(u_1,\dots,u_k)$, $v=(u_{k+1},\dots,u_n)$, and thus $h=(h_1,\dots,h_n) = (u,v)$. In this notation we have the following trivialization of $T^*\mathcal{O}$:
\begin{equation}
T^*\mathcal{O} = \mathcal{O} \times \R^{k} \times \R^{n-k},
\end{equation}
and we denote a point $\lambda \in T^*\mathcal{O}$ as $\lambda = (x;h)=(x;u,v)$. On $T^*\mathcal{O}$, the sub-Riemannian and Riemannian Hamiltonians $H,H_R:T^*M \to \R$, respectively, read:
\begin{equation}
H =\frac{1}{2}|u|^2\qquad \text{and} \qquad H_R = \frac{1}{2}\left(|u|^2+|v|^2\right).
\end{equation}
For any set $\Lambda \subset T^*M$, we adopt the notation (recall that $ \distr^0$ is the annihilator of $\distr$)
\begin{equation}\label{eq:0}
\Lambda_{\neq 0}:= \Lambda \setminus \distr^0 = \{\lambda \in \Lambda \mid H(\lambda)>0\},
\end{equation}
so that, if $\Lambda\subset T^*\mathcal{O}$, in the given trivialization we have
\begin{equation}
\Lambda_{\neq 0} = \{(x;u,v)\in \Lambda \mid u \neq 0\}.
\end{equation}

\subsubsection{Fundamental computations}
The following lemma collects computations and notations that will be used throughout the section. The proofs are routine computation, and are omitted.
\begin{lemma}[Fundamental computations]\label{l:fundamentalcomputations}
The following formulas hold on $T^*\mathcal{O}$:
\begin{align}
\dot{\partial}_{v_\mu} & = -u_\ell c_{\ell A}^\mu \partial_{h_A},\\
\dot{\partial}_{u_i}  & = - \vec{u}_i - u_\ell c_{\ell A}^i \partial_{h_A},\\
\dot{{\vec{h}}}_{A}   &  =  - h_B c_{A \ell}^B \vec{u}_\ell +u_\ell c_{\ell A}^B \vec{h}_B + u_\ell h_B f_{\ell A C}^B \partial_{h_C},
\end{align}
where for lower-case latin indices $i,j,\ell=1,\dots,k$, for greek ones $\mu,\nu,\sigma=k+1,\dots,n$, for upper-case latin ones $A,B,C=1,\dots,n$, repeated indices are understood as summed over their range, and we defined the following smooth functions on $\mathcal{O}$:
\begin{equation}
f_{\ell A C}^B\in C^\infty(\mathcal{O}),\qquad f_{\ell A C}^B := X_\ell(c_{AC}^B) - X_A(c_{\ell C}^B) - c_{\ell A}^D c_{DC}^B +c_{\ell D}^B c_{AC}^D-c_{AD}^B  c_{\ell C}^D.
\end{equation}
We rewrite these equations in a compact notation as follows:\begin{align}
\dot{\partial}_v & = \A \cdot \partial_u + \B \cdot \partial_v, \label{eq:partialvdot}\\
\dot{\partial}_u & = -\vec{u} + \C \cdot \partial_u + \D \cdot \partial_v, \label{eq:partialudot}\\
\dot{\vec{u}} & =  \E \cdot \vec{u} -\A^* \cdot \vec{v} + \F \cdot \partial_u + \G \cdot \partial_v, \label{eq:vecudot} \\ 
\dot{\vec{v}} & = \H \cdot \vec{u}  -\B^* \cdot \vec{v} + \I \cdot \partial_u + \L \cdot \partial_v, \label{eq:vecvdot}
\end{align}
where the dot denotes multiplication of tuples\footnote{Here and in the following, the notation $V = O\cdot W$, where $O$ is an $n\times m$ matrix, $V=(V_1,\dots,V_n)$ is a $n$-tuple of vector fields, and similarly $W=(W_1,\dots,W_m)$, with $n,m \in \N$, means that $V_i = \sum_{j=1}^m O_{ij} W_j$.}, and where we defined the following smooth matrix-valued maps:
\begin{align}
\A & : T^*\mathcal{O} \to \mathrm{M}(n-k, k), & \A_{\mu i} & := - u_\ell c_{\ell i}^\mu,  \\
\B & : T^*\mathcal{O} \to \mathrm{M}(n-k, n-k), & \B_{\mu \nu} & := - u_\ell c_{\ell\nu}^\mu,\\
\C & : T^*\mathcal{O} \to \mathrm{M}(k, k), & \C_{i j} & := - u_\ell c_{\ell j}^i, \\
\D & : T^*\mathcal{O} \to \mathrm{M}(k, n-k),& \D_{i \mu} & := - u_\ell c_{\ell \mu}^i, \\
\E & : T^*\mathcal{O} \to \mathrm{M}(k, k), & \E_{ij} & := -h_Bc_{ij}^B + u_\ell c_{\ell i}^j,   \\
\F & : T^*\mathcal{O} \to \mathrm{M}(k, k), & \F_{ij} & := u_\ell h_B f_{\ell i j}^B, \\
\G & : T^*\mathcal{O} \to \mathrm{M}(k, n-k),& \G_{i \mu} & := u_\ell h_B f_{\ell i \mu}^B, \\
\H & : T^*\mathcal{O} \to \mathrm{M}(n-k, k), & \H_{\mu i} & := - h_B c^{B}_{\mu i}+u_\ell c_{\ell \mu}^i,  \\
\I & : T^*\mathcal{O} \to \mathrm{M}(n-k, k), & \I_{\mu i} & :=  u_\ell h_B f_{\ell \mu i}^B, \\
\L & : T^*\mathcal{O} \to \mathrm{M}(n-k, n-k),& \L_{\mu \nu} & := u_\ell h_B f_{\ell \mu \nu}^B.
\end{align}
\end{lemma}

\subsubsection{Partial dilations}

Recall that we have fixed a Riemannian extension so that $TM = \distr \oplus \ver$, with $\ver \perp \distr$, and thus $T^*M =\ver^0 \oplus \distr^0$. This induces partial dilations on the fibers of the cotangent bundle. For simplicity we only provide local definitions in the trivialization $T^*\mathcal{O}$, even though these maps are defined on the whole $T^*M$.

\begin{definition}[$u$-dilation]
For $a>0$, we define the $u$-dilation $\delta_a^{hor} :T^*\mathcal{O}\to T^*\mathcal{O}$ by
\begin{equation}
\delta_a^{hor}(x;u,v) :=(x;au,v),\qquad (x;u,v)\in T^*\mathcal{O}.
\end{equation}
\end{definition}
\begin{definition}[$u$-homogeneous functions]
Let $V$ be a vector space. A map $f: T^*\mathcal{O} \to V$ is \emph{$u$-homogeneous of degree $d$} if 
\begin{equation}
f(\delta_a^{hor}(\lambda)) = a^d f(\lambda),\qquad \forall\, a>0,\; \lambda \in T^*\mathcal{O}.
\end{equation}
\end{definition}
\begin{definition}[$u$-star shaped set]
A set $\Lambda\subset T^*\mathcal{O}$ is \emph{$u$-star-shaped} if
\begin{equation}\label{eq:strip}
\Lambda = \bigcup_{a\in [0,1]} \delta_a^{hor}(\Lambda).
\end{equation}
\end{definition}
Similar definitions can be given for $v$-dilations, $v$-homogeneous functions, $v$-star-shaped sets but they will not be used, so they are omitted.

\subsubsection{Pseudo-homogeneous maps}

The class of $u$-homogeneous maps is not closed with respect to $\vec{H}$-derivations. Thus, we introduce a larger class of maps with this property. We only give a local definition.

\begin{definition}[$u$-pseudo-homogeneous maps]\label{def:pseudohom}
Let $V$ be a vector space and let $\Lambda\subset T^*\mathcal{O}$ be an open set. We say that $f: \Lambda \to V$ is \emph{$u$-pseudo-homogeneous} of degree $d\in \N$ if there exist
\begin{itemize}
\item a finite set of indices $I$;
\item smooth functions $g_i,c_i:T^*\mathcal{O}\to \R$, $i\in I$ (defined on the whole $T^*\mathcal{O}$);
\end{itemize}
such that it holds
\begin{equation}\label{eq:ordergeqd}
f(x;u,v) = |u|^d \sum_{i\in I} c_i(x;u,v) g_i(x;u/|u|,v), \qquad \forall\, (x;u,v) \in \Lambda_{\neq 0}.
\end{equation}
\end{definition}
The following simple observation will be crucial.
\begin{lemma}\label{l:property}
If $f:\Lambda \to V$ is $u$-pseudo-homogeneous of degree $d$, then $\vec{H}(f)$ is also $u$-pseudo-homogeneous of degree $d$. In particular, if $\Lambda\subset T^*\mathcal{O}$ is relatively compact domain, $j\in \N$, there exists $C=C_{j,\Lambda}$ such that
\begin{equation}
\|\vec{H}^{(j)}(f)(\lambda)\| \leq C H(\lambda)^{d/2}, \qquad \forall\, \lambda \in \Lambda_{\neq 0},
\end{equation}
where $\|\cdot\|$ denotes a fixed arbitrary norm on $V$.
\end{lemma}
\begin{proof}
Notice that $H=\tfrac{1}{2}|u|^2$ so that $\vec{H} = \sum_{i=1}^k u_i\vec{u}_i$, and it holds $\vec{H}(|u|)=0$ on $T^*\mathcal{O}_{\neq 0}$. In particular $\vec{H}$ is a derivation that acts only on $x,v$ and the spherical part of $u$. The statement follows easily from \eqref{eq:ordergeqd}.
\end{proof}
\begin{remark}
All $u$-homogeneous functions of degree $d$ are $u$-pseudo-homogeneous. In this case \eqref{eq:ordergeqd} is satisfied for $I=\{1\}$, $c_1=1$, and $g_1$ is any everywhere-smooth function such that $g_1(x;u,v) = f(x;u,v)$ for $u\in \mathbb{S}^{k-1}$.
\end{remark}

\subsection{Regularity of the canonical frame}

In the next series of lemmas we prove that one can choose in a smooth and locally bounded fashion a canonical frame as a function of the initial covector. The key feature is that the sets $\Lambda$ in the statements are $u$-star-shaped so that the estimates hold arbitrarily close to the singularity at $\distr^0$, corresponding to $2H=0$ or, equivalently, $u=0$.

Recall that, for a fat structure, all non-trivial normal extremals have the same Young diagram, with two columns. The elements of the frame are labeled according to the boxes of the (reduced) Young diagram:
\begin{equation}
\ytableausetup{centertableaux,nosmalltableaux}
Y = \begin{ytableau}
b & a \\
c\\
\end{ytableau} \qquad\qquad \begin{aligned}
\size(a) & = \size(b) = n-k, \\ \size(c) & =2k-n.
\end{aligned}
\end{equation}
We organize a canonical frame along an extremal $\lambda_t=e^{t\vec{H}}(\lambda)$ as a collection of tuples
\begin{equation}
\{E_a,E_b,E_c,F_a,F_b,F_c\},
\end{equation}
where, for $\boxdot=a,b,c$, each $E_{\boxdot}$, $F_{\boxdot}$ in this list is a tuple of smooth vector fields along $\lambda_t$, of size equal to the one of the corresponding superbox $\size(\boxdot)$.

\begin{remark}
In Lemmas \ref{l:lemmaEa}, \ref{l:lemmaEbFb}, $\Lambda$ is a neighborhood of any $\bar{\lambda}\in T^*M$. Furthermore, the $u$-star-shaped property of $\Lambda$ is a triviality since, as it is clear from the proof, $\Lambda$ can actually be any relatively bounded domain contained in a local trivialization $T^*\mathcal{O}$. Starting from Lemma \ref{l:lemmaEcFc} onward, however, $\bar{\lambda} \in T^*M_{\neq 0}$, and furthermore the $u$-star-shaped property is not trivial and must be proved in the construction.
\end{remark}

\begin{lemma}\label{l:lemmaEa}
For any $\bar{\lambda}\in T^*M$ there exist a $u$-star-shaped neighborhood $\Lambda \subset T^*M$ of $\bar{\lambda}$, $T>0$, and a smooth map $\P: [0,T]\times \Lambda_{\neq 0} \to \mathrm{GL}(\R^{n-k})$, such that for all $\lambda \in \Lambda_{\neq 0}$ the tuple
\begin{equation}
E_a|_{\lambda_t} = \frac{1}{\sqrt{2H(\lambda)}}\P(t,\lambda) \cdot\partial_v|_{\lambda_t}, \qquad t\in[0,T],
\end{equation}
is (a choice of) the $E_a$-component of the canonical frame along $\lambda_t = e^{t\vec{H}}(\lambda)$.

Furthermore, the map $\P$ and all its time-derivatives are uniformly bounded, that is for all $j\in \N$ there exists a constant $C=C_j>0$ such that
\begin{equation}
\|\partial_t^{j}\P(t,\lambda)\| \leq C ,\qquad \forall\, t \in [0,T],\; \lambda \in \Lambda_{\neq 0},
\end{equation}
where $\|\cdot\|$ denotes a fixed matrix norm.
\end{lemma}

\begin{proof}
Let $\bar{x}=\pi(\bar{\lambda})$ and let $\mathcal{O}\subseteq M$ be an open neighborhood of $\bar{x}$ equipped with a local orthonormal frame, as explained in Section \ref{sec:preliminary}, so that we have the trivialization
\begin{equation}
T^*\mathcal{O} = \mathcal{O} \times \R^k\times \R^{n-k}, \qquad \lambda = (x;u,v).
\end{equation}
Let $\Lambda\subset T^*\mathcal{O}$ be any $u$-star-shaped and relatively compact neighborhood of $\bar{\lambda}$. Let $\Lambda_t= e^{t\vec{H}}(\Lambda)$. Indeed, the union of all $\Lambda_t$, for $t\in [0,T]$, is relatively compact. Furthermore, $\pi(\Lambda_t)$ is contained in the union of all the sub-Riemannian balls with centers in the bounded set $\pi(\Lambda)$ and radius $t\rho$, where $\rho = \sup\{|u| \mid (x;u,v) \in \Lambda\}$. If $T$ is sufficiently small, then $\Lambda_t \subset T^*\mathcal{O}$ for all $t\leq T$. Hence, in the following, all extremals $\lambda_t$ with $\lambda_0=\lambda \in \Lambda$, are contained in a common compact subset of $T^*\mathcal{O}$.

Now fix $\lambda \in \Lambda_{\neq 0}$ and $t\in [0, T]$. The tuple $E_a$ is determined (up to a constant orthogonal transformation) by the following conditions along $\lambda_t$:
\begin{itemize}
\item[(i)] $E_a$ is vertical (i.e.\@ $\pi_* E_a =0$);
\item[(ii)] $\dot{E}_a$ is vertical;
\item[(iii)] normalization condition: $\mathbb{1}=\sigma(\ddot{E}_a,\dot{E}_a)$;
\item[(iv)] Darboux frame condition: $\mathbb{0}=\sigma(\ddot{E}_a,\ddot{E}_a)$.
\end{itemize}

By Lemma \ref{l:fundamentalcomputations}, points (i) and (ii), $E_a$ must have the form
\begin{equation}\label{eq:ansatzEa}
E_a|_{\lambda_t} = \theta(t) \cdot \partial_v|_{\lambda_t}, \qquad \forall\, t\in [0,T],
\end{equation}
for some smooth $t\mapsto \theta(t) \in \mathrm{GL}(\R^{n-k})$. 

Using Lemma \ref{l:fundamentalcomputations}, we obtain
\begin{equation}\label{eq:partialvddot}
\ddot{\partial}_v = (\dot{\A} +\A\C + \B\A) \cdot \partial_u + (\dot{\B} +\A\D + \B^2) \cdot \partial_v - \A \cdot \vec{u}.
\end{equation}
Here $\dot\A : T^*\mathcal{O} \to \mathrm{M}(n-k\times k)$ is the smooth map given by $\dot{\A}(\lambda) = (\vec{H}\A)(\lambda)$, where the action of $\vec{H}$ is meant component-wise. The notation is consistent  as $\dot{\A}(\lambda_t) = \frac{d}{dt} \A(\lambda_t)$. Similarly for $\dot\B$.

Using \eqref{eq:partialvddot} and \eqref{eq:partialvdot} in the ansatz \eqref{eq:ansatzEa}, we obtain:
\begin{align}
\dot{E}_a & = \theta \A \cdot \partial_u + (\dot\theta +\theta \B) \cdot \partial_v, \label{eq:Eadot} \\
\ddot{E}_a & = [2\dot\theta \A+ \theta (\dot{\A} +\A\C + \B\A)] \cdot \partial_u + [\ddot{\theta} +2\dot\theta \B +\theta (\dot{\B} +\A\D + \B^2)] \cdot \partial_v - \theta \A \cdot \vec{u}, \label{eq:Eaddot} 
\end{align}
where we omitted the explicit evaluation along $\lambda_t$.

We now impose (iii). It yields:
\begin{equation} \label{eq:conditioniii}
\mathbb{1}  = -\sigma(\dot{E}_a,\ddot{E}_a)= \sigma(\theta \A \cdot \partial_u,\theta \A \cdot\vec{u}) =(\theta \A) \sigma(\partial_u,\vec{u})(\theta \A)^* = \theta \A\A^*\theta^*,
\end{equation}
where we used \eqref{eq:sigmaframe} for $\sigma(\partial_u,\vec{u}) = \mathbb{1}$, and the star denotes the transpose. 

Notice that $\A\A^* : T^*\mathcal{O} \to \mathrm{M}(n-k, n-k)$ is given by
\begin{equation}\label{eq:AAstar}
\A\A^*(x;u,v)_{\mu\nu}=-\sum_{i,j,\ell=1}^{k} u_\ell u_j c_{\ell i}^\mu(x) c_{j i}^\nu(x), \qquad \mu,\nu=k+1,\dots,n.
\end{equation}
In particular $\A\A^*$ is $u$-homogeneous of degree $2$. The fat assumption means that $\A=\A(x;u,v)$ has non-trivial kernel if and only if $u=0$, or equivalently that $\A\A^*>0$ on $T^*\mathcal{O}_{\neq 0}$, and thus the restriction
\begin{equation}
\A\A^*|_{T^*\mathcal{O}_{\neq 0}}: T^*\mathcal{O}_{\neq 0} \to \mathrm{GL}(\R^{n-k}),
\end{equation}
is a smooth map taking values in the space of scalar products.

By the Gram-Schmidt process applied to the columns of $\A\A^*|_{T^*\mathcal{O}_{\neq 0}}$, and the fact that the latter is smooth, we deduce the existence of a smooth map
\begin{equation}
\S:  T^*\mathcal{O}_{\neq 0} \to \mathrm{GL}(\R^{n-k}),
\end{equation}
which is $u$-homogeneous of degree $1$, and such that
\begin{equation}
\A\A^*(x;u,v) = \S\S^*(x;u,v), \qquad \forall\, (x;u,v)\in T^*\mathcal{O}_{\neq 0}.
\end{equation}
Notice that $\A$ is a rectangular matrix, while $\S$ is an invertible square matrix.

We define then the smooth map $t\mapsto \phi(t) \in \mathrm{GL}(\R^{n-k})$ by
\begin{equation}
\phi(t) := \theta(t)\S(\lambda_t), \qquad t\in [0,T].
\end{equation}
Condition (iii) in \eqref{eq:conditioniii}, expressed in terms of $\phi$, is equivalent to
\begin{equation}
\phi\phi^* = \theta \S\S^*\theta^* = \theta \A\A^*\theta^* = \mathbb{1},
\end{equation}
omitting evaluation along $\lambda_t$. Since $\phi(t)$ is invertible we obtain that
\begin{equation}
\text{condition (iii)} \qquad \Leftrightarrow \qquad \phi(t) \in \mathrm{O}(\R^{n-k}), \qquad \forall\, t\in[0,T].
\end{equation}
So that we can refine \eqref{eq:ansatzEa} as:
\begin{equation}\label{eq:ansatzEarefined}
E_a|_{\lambda_t} = \phi(t) \S(\lambda_t)^{-1} \cdot \partial_v|_{\lambda_t},\qquad \forall\, t\in [0,T],
\end{equation}
for a smooth map $t\mapsto \phi(t)$ taking values now in $\mathrm{O}(\R^{n-k})$. 

We now impose condition (iv), namely $\mathbb{0} = \sigma(\ddot{E}_a,\ddot{E}_a)$. This yields an ODE which will characterize $t\mapsto \phi(t)$. Using \eqref{eq:Eaddot} and  \eqref{eq:sigmaframe}, we obtain
\begin{equation}
\mathbb{0} = [2\dot\theta \A+ \theta (\dot{\A} +\A\C + \B\A)] (\theta\A)^* - \theta\A [2\dot\theta \A+ \theta (\dot{\A} +\A\C + \B\A)]^*.
\end{equation}
We re-express this condition in terms of $\phi$. To do this, we use again the shorthand $\dot\S = \vec{H}(\mathbb{S}) : T^*\mathcal{O}_{\neq 0} \to \mathrm{GL}(\R^{n-k})$, so that $\dot{\S}(\lambda_t) = \frac{d}{dt}\S(\lambda_t)$, consistently with the previous notation. Omitting $t$ and evaluation along $\lambda_t$, and denoting with $\mathcal{A}(M) = \tfrac{1}{2}(M-M^*)$ the skew-symmetric part of a matrix $M$, we obtain the following ODE for $\phi$:
\begin{align}
2\dot{\phi} & =  \phi \S^{-1}\mathcal{A}\left(2\dot{\S}\S-(\dot{\A}+\A\C+\B\A)\A^* + \tfrac{1}{2}\A\sigma(\vec{u},\vec{u})\A^*  \right) \S^{-1},
\end{align}
where we used $\phi \in \mathrm{O}(\R^{n-k}) \Rightarrow \dot{\phi}\phi^* =- \phi\dot\phi^*$. We can rewrite the above ODE as $\dot{\phi}(t) = \phi(t) \Xi(\lambda_t)$, where the smooth map
\begin{equation}
\Xi : T^*\mathcal{O}_{\neq 0} \to \mathrm{M}(n-k\times n-k),
\end{equation}
takes values in skew-symmetric matrices and is defined by
\begin{equation}\label{eq:Xi}
\Xi := \frac{1}{2}\S^{-1}\mathcal{A}\left[2\dot{\S}\S-(\dot{\A}+\A\C+\B\A)\A^* + \tfrac{1}{2}\A\sigma(\vec{u},\vec{u})\A^*  \right] \S^{-1}.
\end{equation}
The terms $\A,\dot{\A},\B,\C,\D,\sigma(\vec{u},\vec{u})$ are smoothly defined on $T^*\mathcal{O}$, on the other hand $\mathbb{S}$ and $\mathbb{S}^{-1}$ are well-defined and smooth only on $T^*\mathcal{O}_{\neq 0}$, and they blow-up as $H(\lambda)\to 0$. 

We claim that $\Xi$ and all its $\vec{H}$-derivatives $\vec{H}^{(j)}(\Xi)$ remain uniformly bounded on $\Lambda_{\neq 0}$. To prove this claim, observe that $\A$, $\B$, $\C$, $\sigma(\vec{u},\vec{u})$, $\S$ appearing in \eqref{eq:Xi} are $u$-homogeneous of degree one, in particular they are $u$-pseudo-homogeneous of degree $1$. By Lemma \ref{l:property}, the $\vec{H}$-derivatives of these maps are $u$-pseudo-homogeneous of degree $1$. In particular this is the case for $\dot{\A}$ and $\dot{\S}$ appearing in \eqref{eq:Xi}. Taking into account the form \eqref{eq:ordergeqd} of $u$-pseudo-homogeneous maps all possible singularities at $u=0$ in \eqref{eq:Xi} cancel out, and $\Xi$ and its $\vec{H}$-derivatives are $u$-pseudo-homogeneous of degree $d=0$. Another application of Lemma \ref{l:property} proves the claim.

We choose $t  \mapsto \phi(t)$ in \eqref{eq:ansatzEarefined} as the solution of the Cauchy problem
\begin{equation}\label{eq:ODEphi}
\dot{\phi}(t) = \phi(t) \Xi(\lambda_t), \qquad \phi(0)=\mathbb{1}.
\end{equation}
Of course one might choose as the initial condition $\phi(0)$ any orthogonal matrix, and this freedom in the choice of the initial condition is expected as the canonical frame is uniquely defined only up to a constant orthogonal transformation.

By the construction of $\Lambda$ and $T$ all extremals $\lambda_t=e^{t\vec{H}}(\lambda)$, for $t\leq T$ and $\lambda \in \Lambda$, are contained in a common relatively compact subset of $T^*\mathcal{O}_{\neq 0}$. Furthermore, $\Xi$ and its $\vec{H}$-derivatives are uniformly bounded for all $t\in [0,T]$, $\lambda \in \Lambda_{\neq 0}$. In particular the solution to the ODE \eqref{eq:ODEphi} is well-defined on $[0,T]$. Define the map
\begin{equation}
\Phi : [0,T]\times \Lambda_{\neq 0} \to \mathrm{O}(\R^{n-k}),
\end{equation}
such that $\Phi(\cdot,\lambda)$ is the solution of \eqref{eq:ODEphi} with $\lambda_t=e^{t\vec{H}}(\lambda)$. Since $\Xi$ is skew-symmetric then $\Phi$ is orthogonal for $t\in [0,T]$.

Thanks to the aforementioned claim of uniform boundedness property of $\Xi$, we can apply Gronwall's lemma to \eqref{eq:ODEphi}, and we obtain that for any $j\in \N$ there exists $C = C(j,\Lambda,T)>0$ such that
\begin{equation}\label{eq:boundPhi}
\|\partial_t^{j}\Phi(t,\lambda)\| \leq C, \qquad \forall\, t\in [0,T],\; \lambda \in \Lambda_{\neq 0}.
\end{equation}

In order to recover the statement of the lemma, we set
\begin{equation}
\P(t,\lambda) := \Phi(t,\lambda) \sqrt{2H(\lambda)} \S(\lambda_t)^{-1}, \qquad \lambda_t = e^{t\vec{H}}(\lambda),\qquad \forall\, (t,\lambda)\in [0,1]\times \Lambda_{\neq 0}.
\end{equation}
By construction, the map
\begin{equation}
(t,\lambda)\mapsto E_a|_{\lambda_t} = \frac{1}{\sqrt{2H(\lambda)}} \P(t,\lambda) \cdot \partial_v|_{\lambda_t},
\end{equation}
is smooth for $(t,\lambda) \in [0,T]\times \Lambda_{\neq 0}$, and $E_a|_{\lambda_t}$ is the $E_a$-component of a canonical frame along the extremal $\lambda_t$ with initial covector $\lambda$.

It remains to prove that the map $\P$ and its time-derivatives $\partial_t^j\P$ are uniformly bounded as required on $[0,T]\times \Lambda_{\neq 0}$. Notice that $\Phi(t,\lambda)$ and its time-derivatives already satisfy the required property by \eqref{eq:boundPhi}, while $\lambda \mapsto \sqrt{2H(\lambda)}\S(\lambda)^{-1}$ is $u$-homogeneous of degree $0$ and smooth on $T^*\mathcal{O}_{\neq 0}$, so that we can apply Lemma \ref{l:property} for $d=0$, so that also $\sqrt{2H(\lambda_t)}\S(\lambda_t)^{-1}$ and all its time-derivatives are uniformly bounded as required.
\end{proof}

Lemma \ref{l:lemmaEa} for the $E_a$-component of the frame is the starting point for analogous statement for all the other tuples.

\begin{lemma}\label{l:lemmaEbFb}
For any $\bar{\lambda}\in T^*M$ there exist a $u$-star-shaped neighborhood $\Lambda \subset T^*M$ of $\bar{\lambda}$, $T>0$, and smooth maps $\Q_i: [0,T]\times  \Lambda_{\neq 0} \to \mathrm{M}(n_i,m_i)$, with $n_i,m_i \in \N$, $i=1,\dots,5$, such that for all $\lambda \in \Lambda_{\neq 0}$ the tuples
\begin{align}
E_b|_{\lambda_t} & = \Q_1(t,\lambda)\cdot \partial_u|_{\lambda_t}  + \frac{1}{\sqrt{2H(\lambda)}}\Q_2(t,\lambda) \cdot \partial_v|_{\lambda_t} , & t\in [0,T],\\
F_b|_{\lambda_t} & = \Q_3(t,\lambda)\cdot \vec{u} |_{\lambda_t} + \Q_4(t,\lambda)\cdot\partial_u|_{\lambda_t}  + \frac{1}{\sqrt{2H(\lambda)}}\Q_5(t,\lambda)\cdot\partial_v|_{\lambda_t} ,  & t\in [0,T],
\end{align}
are (a choice of) the $E_b$ and $F_b$-components of the canonical frame along $\lambda_t = e^{t\vec{H}}(\lambda)$.

Furthermore, the maps $\Q_i$ and all their time-derivatives are bounded, that is for all $j\in \N$ there exists a constant $C = C_j>0$ such that
\begin{equation}
\|\partial_t^{j}\Q_i(t,\lambda)\| \leq C ,\qquad \forall\, t \in [0,T],\; \lambda \in \Lambda_{\neq 0},\; i=1,\dots,5,
\end{equation}
where $\|\cdot\|$ denotes a fixed matrix norm.
\end{lemma}
\begin{proof}
Let $\Lambda$ and $T>0$ as in Lemma \ref{l:lemmaEa} (cf.\@ first paragraph of its proof). Let $\lambda \in \Lambda_{\neq 0}$. By the structural equations $E_b|_{\lambda_t} = \dot{E}_a|_{\lambda_t}$. Using Lemma \ref{l:fundamentalcomputations}, we obtain
\begin{equation}\label{eq:Ebexplicit}
E_b|_{\lambda_t}  = \frac{1}{\sqrt{2H}} \left[ \P(t,\lambda)\A(\lambda_t) \cdot \partial_u|_{\lambda_t} +  (\dot{\P}(t,\lambda)+ \P(t,\lambda) \B(\lambda_t) )\cdot \partial_v|_{\lambda_t}\right],
\end{equation}
where $\dot{\P}(t,\lambda) =\partial_t \P(t,\lambda)$. We set hence:
\begin{equation}
 \Q_1(t,\lambda):= \frac{1}{\sqrt{2H(\lambda)}}\P(t,\lambda)\A(\lambda_t), \qquad \Q_2(t,\lambda):=\dot{\P}(t,\lambda) + \P(t,\lambda)\B(\lambda_t).
\end{equation}
Notice that $\P$ and all its time-derivatives are uniformly bounded by the estimate in Lemma \ref{l:lemmaEa}. Since $\A$ and $\B$ are $u$-pseudo-homogeneous of degree $1$, an application of Lemma \ref{l:property} yields the required bounds on $\Q_1$ and $\Q_2$ and their time-derivatives.

Likewise, $F_b|_{\lambda_t} = - \dot{E}_b|_{\lambda_t}$, so that, using Lemma \ref{l:fundamentalcomputations}, we obtain:
\begin{multline}
F_b = \frac{1}{\sqrt{2H}} \left[ 
\P\A \cdot \vec{u}
- (\dot{\P} \A + \P\dot{\A}+\P\A\C + \P\B\A)\cdot \partial_u \right. \\
\left. - (\ddot{\P} + 2 \dot{\P} \B + \P\dot{\B} + \P\A\D + \P\B^2)\cdot\partial_v
\right],
\end{multline}
omitting evaluation along $\lambda_t$ from $\A$, $\B$, $\D$, and evaluation at $(t,\lambda)$ for $\P$. To obtain the statement for $F_b$ we set
\begin{align}
\Q_3 & :=  \frac{1}{\sqrt{2H}} \P\A, \\
\Q_4 & := -\frac{1}{\sqrt{2H}} (\dot{\P} \A + \P\dot{\A}+\P\A\C + \P\B\A), \\
\Q_5 & := -(\ddot{\P} + 2 \dot{\P} \B + \P\dot{\B} + \P\A\D + \P\B^2),
\end{align}
and argue as above using Lemma \ref{l:property} to prove the required estimates on the $\Q_i$'s.
\end{proof}

\begin{lemma}\label{l:lemmaEcFc}
For any $\bar{\lambda}\in T^*M_{\neq 0}$ there exist a $u$-star-shaped neighborhood $\Lambda \subset T^*M$ of $\bar{\lambda}$, $T>0$, and smooth maps $\Q_i: [0,T]\times  \Lambda_{\neq 0} \to \mathrm{M}(n_i,m_i)$, with $n_i,m_i\in \N$, $i=6,\dots,10$, such that for all $\lambda \in \Lambda_{\neq 0}$ the tuples
\begin{align}
E_c|_{\lambda_t} & = \Q_6(t,\lambda) \cdot \partial_u|_{\lambda_t}  + \frac{1}{\sqrt{2H(\lambda)}}\Q_7(t,\lambda)\cdot \partial_v|_{\lambda_t} , & t\in [0,T],\\
F_c|_{\lambda_t} & =  \Q_8(t,\lambda) \cdot \vec{u}|_{\lambda_t}  + \Q_9(t,\lambda) \cdot \partial_u|_{\lambda_t}  + \frac{1}{\sqrt{2H(\lambda)}}\Q_{10}(t,\lambda)\cdot \partial_v|_{\lambda_t} , & t\in [0,T],
\end{align}
are (a choice of) the $E_c$ and $F_c$-components of the canonical frame along $\lambda_t = e^{t\vec{H}}(\lambda)$.

Furthermore, the maps $\Q_i$ and all their time-derivatives are bounded, that is for all $j\in \N$ there exists a constant $C = C_j>0$ such that
\begin{equation}
\|\partial_t^{j}\Q_i(t,\lambda)\| \leq C ,\qquad \forall\, t \in [0,T],\; \lambda \in \Lambda_{\neq 0},\; i=6,\dots,10,
\end{equation}
where $\|\cdot\|$ denotes a fixed matrix norm.
\end{lemma}
\begin{proof}
Let $\Lambda$ and $T>0$ as in Lemma \ref{l:lemmaEa} (cf.\@ first paragraph of its proof). Let $\lambda_t = e^{t\vec{H}}(\lambda)$, for $\lambda \in \Lambda_{\neq 0}$ and $t\in [0,T]$. The $(2k-n)$-tuple $E_c|_{\lambda_t}$ is determined, up to a constant orthogonal transformation, by the following conditions along $\lambda_t$:
\begin{itemize}
\item[(i)] $E_c$ is vertical (i.e.\@ $\pi_* E_c =0$);
\item[(ii)] Darboux frame conditions: $\mathbb{0}=\sigma(\dot{E}_b,E_c)$ and $\mathbb{0}=\sigma(\ddot{E}_b,E_c)$;
\item[(iii)] normalization condition: $\mathbb{1} = \sigma(\dot{E}_c,E_c)$;
\item[(iv)] isotropic condition: $\mathbb{0}=\sigma(\dot{E}_c,\dot{E}_c)$.
\end{itemize}
We observe that if $\bar{E}_c$ is a tuple which verifies (i), (ii) and (iii) along $\lambda_t$, then all these conditions will be also verified by
\begin{equation}
E_c|_{\lambda_t} : = \psi(t)\cdot \bar{E}_c|_{\lambda_t}, \qquad t\in [0,T],
\end{equation}
for any smooth map $t\mapsto \psi(t) \in \mathrm{O}(\R^{2k-n})$. Then, we find first smooth tuple $\bar{E}_c$ such that (i), (ii) and (iii) are verified, and subsequently we will fix $\psi$ via (iv). 

By (i), we must have
\begin{equation}\label{eq:ansatzEc}
\bar{E}_c|_{\lambda_t}= U(t)\cdot\partial_u|_{\lambda_t} + V(t)\cdot \partial_v|_{\lambda_t}, \qquad t\in [0,T],
\end{equation}
for some smooth $t\mapsto U(t) \in \mathrm{M}(2k-n,k)$ and $t\mapsto V(t)\in \mathrm{M}(2k-n,n-k)$. From this, and Lemma \ref{l:fundamentalcomputations}, we obtain:
\begin{align}
\dot{\bar{E}}_c &  = -U\cdot \vec{u}+(\dot{U} + U\C + V\A)\cdot \partial_u + (\dot{V} + U\D + V\B)\cdot\partial_v, \\
\ddot{\bar{E}}_c & =  -(2\dot{U} + U\E +  U\C + V\A)\cdot \vec{u} +U\A^* \cdot \vec{v} + \mathrm{span}\{\partial_u,\partial_v\}, 
\end{align}
Similarly from \eqref{eq:Ebexplicit} in the proof of Lemma \ref{l:lemmaEbFb}, and Lemma \ref{l:fundamentalcomputations} we obtain:
\begin{align}
\dot{E}_b & = \frac{1}{\sqrt{2H}}\P\A \cdot \vec{u} + \mathrm{span}\{\partial_u,\partial_v\},\\
\ddot{E}_b & = \mathbb{T}\cdot\vec{u} - \frac{1}{\sqrt{2H}} \P\A\A^* \cdot\vec{v} + \mathrm{span}\{\partial_u,\partial_v\},
\end{align}
where, for brevity of notation, we set $\mathbb{T} = \mathbb{T}(t,\lambda)$, as
\begin{equation}\label{eq:defT}
\mathbb{T}(t,\lambda):= - \frac{1}{\sqrt{2H}}\Big[2(\partial_t\P)\A + 3\P\dot{\A}+\P(\A\E+\A\C+\B\A) \Big],
\end{equation}
where the matrix maps $\A$, $\B$, $\C$, $\E$ in the r.h.s.\@ are evaluated along $\lambda_t = e^{t\vec{H}}(\lambda)$, $\P = \P(t,\lambda)$, for $\lambda \in \Lambda_{\neq 0}$ and $t\in [0,T]$. Notice that $\A$ is $u$-pseudo-homogeneous of degree $1$, together with $\dot{\A}$, by Lemma \ref{l:property}. Therefore we obtain that $\mathbb{T}$ is uniformly bounded with all its time-derivatives, that is for all $j\in \N$ there exists $C = C_j$ such that
\begin{equation}
\|\partial_t^{j}\mathbb{T}(t,\lambda)\|\leq C_j,\qquad \forall\, t\in [0,T],\; \lambda \in \Lambda_{\neq 0}.
\end{equation}

Using relations \eqref{eq:sigmaframe}, and the fact that $\P$ and $\A\A^*$ are invertible on $\Lambda_{\neq 0}$ (cf.\@ proof of Lemma \ref{l:lemmaEa}), conditions (ii) are equivalent to:
\begin{equation}\label{eq:pairofeqs}
\A U^* = 0, \qquad \text{and}\qquad V^* = \sqrt{2H}(\P\A\A^*)^{-1}\mathbb{T}U^*,
\end{equation}
along $\lambda_t$ and where $\P = \P(t,\lambda)$ and $\mathbb{T} = \mathbb{T}(t,\lambda)$.
The second equation of \eqref{eq:pairofeqs} determines $V(t)$ in terms of $U(t)$. The first one is an orthogonality relation between the rows of $\A(\lambda_t)$ and those of $U(t)$. Recall from Lemma \ref{l:fundamentalcomputations} that $\A : T^*\mathcal{O} \to \mathrm{M}(n-k,k)$ has constant rank equal to $n-k$ on $T^*\mathcal{O}_{\neq 0}$, while $U: [0,T]\to \mathrm{M}(2k-n,k)$ must have rank $2k-n$. By applying the Gram-Schmidt process to a local orthogonal complement to the rows of $\A(\bar{\lambda})$ we can find a neighborhood $\Lambda' \subset T^*\mathcal{O}$ of $\bar{\lambda}$ and a smooth map
\begin{equation}
\A^\perp : \Lambda'_{\neq 0} \to \mathrm{M}(2k-n,k).
\end{equation}
such that 
\begin{equation}\label{eq:propertiesA}
\A^\perp \A^* = \mathbb{0}, \qquad \A^\perp \A^{\perp *} = |u|^2 \mathbb{1}.
\end{equation}
Furthermore, since $\A$ is $u$-homogeneous of degree one, we can take $\A^\perp$ to be also $u$-homogeneous of degree $1$ and $\Lambda'$ to be $u$-star-shaped. The $u$-star-shaped neighborhood $\Lambda$ of $\bar{\lambda}$ built at the beginning of the proof of Lemma \ref{l:lemmaEa} was arbitrary. Up to restriction of $\Lambda$ to a smaller $u$-star-shaped set, and up to taking a smaller $T$ (and since the flow of $\vec{H}$ is smooth and preserves $|u|$), we can assume that for all $\lambda \in \Lambda_{\neq 0}$ and $t\in [0,T]$, the extremals $\lambda_t = e^{t\vec{H}}(\lambda)$ take value in a common bounded neighborhood contained in $\Lambda'_{\neq 0}$, so that $\A^\perp(\lambda_t)$ is well-defined. We then set:
\begin{align}\label{eq:UandV}
U(t) := \frac{1}{\sqrt{2H}}\A^{\perp}, \qquad V(t)  :=-\A^{\perp}\mathbb{T}^*(\P\A\A^*)^{-1*},
\end{align}
where the matrix maps $\A$, $\A^\perp$ in the r.h.s.\@ are evaluated along $\lambda_t = e^{t\vec{H}}(\lambda)$ and $\P = \P(t,\lambda)$, $\mathbb{T}=\mathbb{T}(t,\lambda)$ for $\lambda \in \Lambda_{\neq 0}$ and $t\in [0,T]$. Since $\A$ and $\A^\perp$ are $u$-homogeneous of degree $1$, $\P(t,\lambda)$ and $\mathbb{T}(t,\lambda)$ are uniformly bounded with all time-derivatives and $\lambda \in \Lambda_{\neq 0}$, and using Lemma \ref{l:property}, we observe that for all $j\in \N$ there exists a $C=C_j>0$ independent on $\lambda$ such that
\begin{equation}\label{eq:estimatesUV}
\|\partial_t^j U(t)\| \leq C, \qquad \|\partial_t V(t)\| \leq C H(\lambda)^{-1/2}, \qquad \forall\, t\in [0,T],\; \lambda \in \Lambda_{\neq 0}.
\end{equation}
Estimates \eqref{eq:estimatesUV} will be used later.

With the definition \eqref{eq:UandV}, conditions (i) and (ii) are fully verified, while (iii) corresponds to the following normalization:
\begin{equation}
\mathbb{1} = UU^* = \frac{1}{2H}\A^{\perp} \A^{\perp *},
\end{equation}
along $\lambda_t$, which is verified by our choice in the normalization of $\A^\perp$ in \eqref{eq:propertiesA}. This concludes the construction of the tuple $\bar{E}_c$.

Set hence $E_c|_{\lambda_t} = \psi(t)\cdot \bar{E}_c|_{\lambda_t}$, for $\psi(t) \in \mathrm{O}(\R^{2k-n})$. Condition (iv) yields an ODE that determines $\psi$, once the initial condition is fixed, for example $\phi(0) = \mathbb{1}$. We obtain
\begin{equation}\label{eq:ODEpsi}
\dot{\psi}(t) = \frac{1}{2}\psi(t)\sigma(\dot{\bar{E}}_c|_{\lambda_t},\dot{\bar{E}}_c|_{\lambda_t}), \qquad \psi(0)=\mathbb{1}.
\end{equation}
Since $(t,\lambda)\mapsto \sigma(\bar{E}_c|_{\lambda_t},\bar{E}_c|_{\lambda_t})$ is skew-symmetric and smooth, and $\mathrm{O}(\R^{2k-n})$ is compact, \eqref{eq:ODEpsi} will have a unique solution in $\mathrm{O}(\R^{2k-n})$ for all $t\in [0,T]$, and for any $\lambda \in \Lambda_{\neq 0}$. We define then the smooth map
\begin{equation}
\Psi : [0,T]\times \Lambda_{\neq 0} \to \mathrm{O}(\R^{2k-n}),
\end{equation}
such that $t\mapsto \Psi(t,\lambda)$ is the solution of \eqref{eq:ODEpsi} on $[0,T]$ corresponding to $\lambda_t = e^{t\vec{H}}(\lambda)$.

The matrix in the r.h.s.\@ of \eqref{eq:ODEpsi} is given more explicitly by\label{L: check this}
\begin{align}
\sigma(\dot{\bar{E}}_c|_{\lambda_t},\dot{\bar{E}}_c|_{\lambda_t}) & = U\sigma(\vec{u},\vec{u}) U^* -2\mathcal{A}\left[(\dot{U}+U\C+V\A)U^*\right].
\end{align}
where $\mathcal{A}(M)=\tfrac{1}{2}(M-M^*)$ denotes the skew-symmetric part of a matrix $M$, and $U$, $V$ are defined in \eqref{eq:UandV}. By the uniform estimates for $U(t)$ and $V(t)$ in \eqref{eq:estimatesUV}, and since $\A$ is $u$-homogeneous of degree $1$, we see that $\sigma(\dot{\bar{E}}_c|_{\lambda_t},\dot{\bar{E}}_c|_{\lambda_t})$ and all its time-derivatives remain uniformly bounded for all extremals $\lambda_t =e^{t\vec{H}}(\lambda)$ for $t\in [0,T]$ and $\lambda \in \Lambda_{\neq 0}$. By Gronwall's lemma we deduce an analogous uniform boundedness property of $\Psi$ and its time-derivatives, that is for all $j\in \N$ there exists $C=C_j >0$ such that
\begin{equation}
\|\partial_t^{j}\Psi(t,\lambda)\| \leq C, \qquad \forall\, t\in [0,T]\;  \lambda \in \Lambda_{\neq 0}.
\end{equation}

Finally, to obtain the statement of the lemma, we set then
\begin{align}
\Q_6(t,\lambda) & := \frac{1}{\sqrt{2H(\lambda)}}\Psi(t,\lambda) \A^\perp(\lambda_t), \\
\Q_7(t,\lambda) & :=  -\sqrt{2H(\lambda)}\Psi(t,\lambda)\A^\perp(\lambda_t)\mathbb{T}^*(t,\lambda)\left[\P(t,\lambda)\A(\lambda_t)\A(\lambda_t)^*\right]^{-1*}.
\end{align}
Using the fact that $\A^\perp$ is $u$-homogeneous of degree one, that $\Psi,\mathbb{T},\P$ are uniformly bounded on $[0,T]\times \Lambda_{\neq 0}$ with all their time-derivatives, and applying Lemma \ref{l:property}, we obtain the desired uniform bounds on $\Q_6,\Q_7$.

To prove the analogous statement for the tuple $F_c$, we recall that by the structural equations
$F_c|_{\lambda_t} = -\dot{E}_c|_{\lambda_t}$. Then, making use of the Leibniz rule, we obtain the corresponding statement of the lemma, for suitably defined matrix-valued maps $\Q_8,\Q_9,\Q_{10}$ defined on $[0,T]\times \Lambda_{\neq 0}$, satisfying the required uniform bounds.
\end{proof}

We recall first, that for a general reduced Young diagram $Y$ and for two superboxes $\boxtimes,\boxplus$ of $Y$, the corresponding curvature map can be computed as
\begin{equation}
R_t(\boxtimes,\boxplus) = \sigma(\dot{F}_{\boxtimes}|_{\lambda_t},F_{\boxplus}|_{\lambda_t}).
\end{equation}
Hence a first series of estimates for some of the curvature maps comes as a consequence of the previous ones for the canonical frame.

\begin{remark}\label{rmk:hilbert-schmidt}
Notice that $R_{t}(\boxtimes,\boxplus)$ is a one-parameter family of $\size(\boxtimes)\times \size(\boxplus)$ matrices representing the canonical curvature operator for the given choice of the canonical frame along $\lambda_t$. By the uniqueness part of Theorem \ref{t:ZeLi}, and since the Hilbert-Schmidt norm $\|M\|^2=\tr(MM^*)$ is invariant under this action of the orthogonal group, the quantities $\|\partial_t^j R_t(\boxtimes,\boxplus)\|$ do not depend on choice of the canonical frame and are thus well-defined estimates for the corresponding canonical curvature operator $\mathfrak{R}_{\lambda_t}^{\boxtimes\boxplus}$ and their time-derivatives.
\end{remark}

\begin{lemma}[Curvature estimates I]\label{l:lemmaRfirst}
For any $\bar{\lambda} \in T^*M_{\neq 0}$ there exist a $u$-star-shaped neighborhood $\Lambda \subset T^*M$ of $\bar{\lambda}$ and $T>0$ such that, for all $j\in \N$, there exists a constant $C=C_j>0$ such that for the canonical curvature map $R_{t}$ along the extremal $\lambda_t = e^{t\vec{H}}(\lambda)$ it holds:
\begin{equation}
\|\partial_t^j R_{t}(\boxtimes,\boxplus)\| \leq C, \qquad \forall\, t\in [0,T],\, \lambda \in \Lambda_{\neq 0},
\end{equation}
where $\|\cdot\|$ denotes the Hilbert-Schmidt norm, and $\boxtimes,\boxplus = a,b,c$, with the exclusion of the pairs $(\boxtimes,\boxplus)=(a,a)$, $(a,c)$ and $(c,a)$.
\end{lemma}
\begin{proof}
Let $T>0$ and $\Lambda$ as in all previous lemmas (\ref{l:lemmaEa}, \ref{l:lemmaEbFb}, \ref{l:lemmaEcFc}). Take $\lambda \in \Lambda_{\neq 0}$, $t\in [0,T]$. We routinely omit the evaluation along the extremal $\lambda_t$ for vector fields, and on $(t,\lambda)$ for maps $\Q_i$. Recall from the previous lemma that
\begin{align}
F_b & = \Q_3 \cdot \vec{u} + \Q_4 \cdot \partial_u + \frac{1}{\sqrt{2H}}\Q_5 \cdot \partial_v, \\
F_c & =  \Q_8 \cdot \vec{u} + \Q_9 \cdot \partial_u + \frac{1}{\sqrt{2H}}\Q_{10}\cdot \partial_v,
\end{align}
where $\Q_i=\Q_i(t,\lambda)$ are uniformly bounded with all their time-derivatives for $(t,\lambda)\in [0,T]\times \Lambda_{\neq 0}$. Thanks to these estimates, the only possible singularities as $H\to 0$ can arise from the $\partial_v$-term in the above expressions. Recalling the symplectic product in the static basis $\partial_u,\partial_v,\vec{u},\vec{v}$ in \eqref{eq:sigmaframe}, and using Leibniz rule we see that, for example
\begin{align}
R(b,b)  = \sigma(\dot{F}_b,F_b) = \text{regular part} + \frac{1}{\sqrt{2H}}\left(\Q_3 \sigma(\dot{\vec{u}},\partial_v)\Q_5^* + \Q_5 \sigma(\dot{\partial_v},\vec{u}) \Q_3^*\right),
\end{align}
where the regular part is the multiplication of a finite number of matrix maps $\Q_i$ and their time-derivatives, evaluated at $(t,\lambda)$, and the smooth matrix maps $\A,\B,\dots$ from Lemma \ref{l:fundamentalcomputations} evaluated along $\lambda_t$. In particular the regular part will be bounded with all its time-derivatives, uniformly for $(t,\lambda)\in [0,T]\times \Lambda_{\neq 0}$.

For what concerns the possibly singular part, recall from Lemma \ref{l:fundamentalcomputations} that $\sigma(\dot{\vec{u}},\partial_v) = -\A^*$, and $\sigma(\dot{\partial_v},\vec{u}) = \A$. Since $\A : T^*\mathcal{O}_{\neq 0} \to \mathrm{M}(n-k,n-k)$ is $u$-homogeneous of degree $1$, by Lemma \ref{l:property} the term
\begin{equation}
\frac{1}{\sqrt{2H}}\left(\Q_3 \sigma(\dot{\vec{u}},\partial_v)\Q_5^* + \Q_5 \sigma(\dot{\partial_v},\vec{u}) \Q_3^*\right),
\end{equation}
when evaluated along the extremal $\lambda_t$, is bounded with all its time-derivatives, uniformly for $(t,\lambda)\in [0,T]\times \Lambda_{\neq 0}$.

Therefore $R(b,b)$ is uniformly bounded with all its time-derivatives as required. The same argument proves the similar estimates for $R(c,c)$, $R(c,b)$ and $R(b,c)= R(c,b)^*$. 

The estimate of $R(a,b)$ is more delicate. A direct estimate requires the knowledge of $F_a$, which are not in position to estimate yet. However, using the normal condition in the construction of the canonical frame, one can determine $R(a,b)$ only using the estimates for $F_b$ appearing in Lemma \ref{l:lemmaEbFb}. Indeed, recall that for a Young diagram with two columns, the normal condition is (cf.\@ Appendix \ref{a:canonicalframe}):
\begin{equation}
R(b,a) = -R(b,a)^*.
\end{equation}
Therefore, using \eqref{eq:formulaFa}, and the fact that the canonical frame is Darboux, we obtain
\begin{align}
R(b,a) & = \sigma(\dot{F}_b,F_a)  = -\sigma(\dot{F}_b,\dot{F}_b) + \sum_{\mu=a,b,c}\sigma(\dot{F}_b,E_\mu)R(b,\mu)^* \\
& = -\sigma(\dot{F}_b,\dot{F}_b) + R(b,a)^*  =-\sigma(\dot{F}_b,\dot{F}_b) - R(b,a),
\end{align}
from which it follows that
\begin{equation}
R(b,a) =  -\frac{1}{2}\sigma(\dot{F}_b,\dot{F}_b).
\end{equation}
Using the explicit estimates for $F_b$ and its derivative, one can now estimate $R(b,a)$. Using Leibniz rule and arguing as above we obtain that
\begin{multline}\label{eq:108}
R(b,a) = \text{regular part} + \frac{1}{\sqrt{2H}}\left[\Q_3 \sigma(\dot{\vec{u}},\partial_v)\dot{\Q}_5^* + \Q_3 \sigma(\dot{\vec{u}},\dot{\partial_v})\Q_5^* + \dot{\Q}_3 \sigma(\vec{u},\dot{\partial_v})\Q_5^* \right. \\
\left. + \dot{\Q}_5 \sigma(\partial_v,\dot{\vec{u}})\Q_3^* + \Q_5\sigma(\dot{\partial_v},\dot{\vec{u}})\Q_3^* + \Q_5\sigma(\dot{\partial_v},\vec{u})\dot{\Q}_3^*\right],
\end{multline}
where $\dot{\Q}_i = \partial_t\Q_i(t,\lambda)$. Using Lemma \ref{l:fundamentalcomputations} we observe that
\begin{align}
\sigma(\dot{\vec{u}},\partial_v)  = -\sigma(\partial_v,\dot{\vec{u}})^* &  =  -\A^*, \\
\sigma(\dot{\vec{u}},\dot{\partial_v})  = -\sigma(\dot{\partial_v},\dot{\vec{u}})^* & = -\E\A^* + \A^*\B^*,\\
\sigma(\vec{u},\dot{\partial_v})  = -\sigma(\dot{\partial_v},\vec{u})^* & = -\A^*.
\end{align}
Using again the fact that $\A$ is $u$-homogeneous of degree $1$, and applying Lemma \ref{l:property} to the possibly singular part in \eqref{eq:108}, we see that all possible singularities cancel out and also $R(b,a)$ remains uniformly bounded with all derivatives along $\lambda_t$, as claimed.

The estimates from the missing pairs of superboxes in the statement of the lemma follow by the global symmetry of the curvature maps $R(\boxtimes,\boxplus) = R(\boxplus,\boxtimes)^*$.
\end{proof}

We can now prove the estimate for the last element of the canonical frame, which is the $n-k$-tuple $F_a$. Using the structural equations, we have the general formula:
\begin{equation}\label{eq:formulaFa}
F_a|_{\lambda_t} = -\dot{F}_b|_{\lambda_t} + \sum_{\boxtimes =a,b,c} R_t(b,\boxtimes)\cdot E_{\boxtimes}|_{\lambda_t}.
\end{equation}
The curvatures $R_t(b,\boxtimes)$ in \eqref{eq:formulaFa} are precisely those estimated in Lemma \ref{l:lemmaRfirst}.

\begin{lemma}\label{l:lemmaFa}
For any $\bar{\lambda} \in T^*M_{\neq 0}$ there exist a $u$-star-shaped neighborhood $\Lambda \subset T^*M$ of $\bar{\lambda}$, $T>0$, and smooth maps $\Q_i: [0,T]\times  \Lambda_{\neq 0} \to \mathrm{M}(n_i,m_i)$, with $n_i,m_i\in \N$, $i=11,\dots,14$, such that for all $\lambda \in \Lambda_{\neq 0}$ the tuple
\begin{multline}
F_a|_{\lambda_t}  = \Q_{11}(t,\lambda)\cdot \vec{u}|_{\lambda_t} + \sqrt{2H(\lambda)} \Q_{12}(t,\lambda)\cdot \vec{v}|_{\lambda_t} \\
+  \Q_{13}(t,\lambda) \cdot \partial_u|_{\lambda_t} + \frac{1}{\sqrt{2H(\lambda)}}\Q_{14}(t,\lambda)\cdot \partial_v|_{\lambda_t},
\end{multline}
for $t\in [0,T]$, is (a choice of) the $F_a$-component of the canonical frame along $\lambda_t = e^{t\vec{H}}(\lambda)$.

Furthermore, the maps $\Q_i$ and all their time-derivatives are bounded, that is for all $j\in \N$ there exists a constant $C = C_j>0$ such that
\begin{equation}
\|\partial_t^{j}\Q_i(t,\lambda)\| \leq C ,\qquad \forall\, t \in [0,T],\; \lambda \in \Lambda_{\neq 0},\; i=11,\dots,14,
\end{equation}
where $\|\cdot\|$ denotes a fixed matrix norm.
\end{lemma}
\begin{proof}
By the structural equations we have, along any given extremal $\lambda_t=e^{t\vec{H}}(\lambda)$:
\begin{equation}\label{eq:formulaFa2}
F_a = -\dot{F}_b + R(b,b)\cdot E_b +R(b,a)\cdot E_a + R(b,c)\cdot E_c.
\end{equation}
The claim is now a consequence of lemmas \ref{l:lemmaEa}, \ref{l:lemmaEbFb}, \ref{l:lemmaEcFc} for the previous parts of the canonical frame and Lemma \ref{l:lemmaRfirst} for estimate on the curvature maps appearing in \eqref{eq:formulaFa2}. Using Leibniz rule and omitting as usual evaluation along $\lambda_t$ and $(t,\lambda)$ we obtain
\begin{multline}
F_a  = \left(\dot{\Q}_3 + \Q_3\E - \Q_4\right) \cdot \vec{u} - \Q_3\A^*\cdot \vec{v} 
+\bigg(\Q_3 \F + \dot{\Q}_4 + \Q_4\C + \frac{1}{\sqrt{2H}}\Q_5 \A  \\
+ R(b,b)\Q_1 + R(b,c)\Q_6 \bigg)\cdot\partial_u 
+\bigg( \Q_3 \G + \Q_4 \D + \frac{1}{\sqrt{2H}}\dot{\Q}_5 + \frac{1}{\sqrt{2H}}\Q_5\B \\
+ \frac{1}{\sqrt{2H}}R(b,b)\Q_2 + \frac{1}{\sqrt{2H}}R(b,a)\P+ \frac{1}{\sqrt{2H}}R(b,c)\Q_7\bigg)\cdot \partial_v.
\end{multline}
Using the $u$-homogeneity of $\A$ coupled with Lemma \ref{l:property}, we obtain the required uniform estimates, for $(t,\lambda) \in [0,T]\times \Lambda_{\neq 0}$, where $T$ and $\Lambda$ are as in the previous lemmas.
\end{proof}

We are ready to complete the estimates on the canonical curvature maps.

\begin{lemma}[Curvature estimates II]\label{l:lemmaRlast}
For any $\bar{\lambda} \in T^*M$ there exist a $u$-star-shaped neighborhood $\Lambda \subset T^*M_{\neq 0}$ of $\bar{\lambda}$, and $T>0$ such that, for all $j\in \N$ there exists a constant $C=C_j>0$ such that for the canonical curvature maps $R_{t}$ along the extremal $\lambda_t = e^{t\vec{H}}(\lambda)$ it holds:
\begin{equation}
\|\partial_t^j R_{t}(\boxtimes,\boxplus)\| \leq C, \qquad \forall\, t\in [0,T],\; \lambda \in \Lambda_{\neq 0},
\end{equation}
where $\|\cdot\|$ denotes the Hilbert-Schmidt norm, and $\boxtimes,\boxplus = a,b,c$.
\end{lemma}
\begin{proof}
We only need to prove the statement for the pairs $\boxtimes,\boxplus$ excluded in Lemma \ref{l:lemmaRfirst}, that is the pairs $(a,a)$, $(a,c)$ and $(c,a)$. We sketch the proof since the argument is analogous to the one in Lemma \ref{l:lemmaRfirst} for the first curvature estimates, coupled with the new piece of information on $F_a$ coming from Lemma \ref{l:lemmaFa}. 

Let $T>0$ and $\Lambda$ so that all previous lemmas and estimates about the canonical frame hold. Then we decomposing, for example
\begin{equation}
R(a,a)  = \sigma(\dot{F}_a,F_a)=\text{regular part} + \text{possibly singular part},
\end{equation}
where the possibly singular part contains all terms with a factor containing a positive power of $H^{-1/2}$. Then, using the estimates for $F_a$, Lemma \ref{l:fundamentalcomputations} and Lemma \ref{l:property}, one can prove that all singularities cancel out. We stress in particular that the presence of the factor $\sqrt{2H}$ in the $\vec{v}$-component of $F_a$ of Lemma \ref{l:lemmaFa} is necessary for these cancellations. One then argues similarly for $R(c,a) = \sigma(\dot{F}_c,F_a)$ and $R(a,c) = \sigma(\dot{F}_a,F_c)$.
\end{proof}

We now estimate the geodesic volume derivative. Recall that the fat sub-Riemannian manifold $M$ comes equipped with a Riemannian extension. We denote with $\mathrm{vol}$ the associated Riemannian density, while $M$ is equipped with a smooth reference volume $\mm = e^{-V}\mathrm{vol}$, where $V:M \to \R$ is a smooth function. The geodesic volume derivative with respect to $\mm$ is the smooth function $\rho_{\mm}:T^*M_{\neq 0} \to \R$, defined by
\begin{equation}\label{eq:volder}
\rho_{\mm,\lambda} := \left.\frac{\di}{\di t}\right|_{t=0}\log\mm(\pi_* F_a|_{\lambda_t},\pi_* F_b|_{\lambda_t},\pi_*F_c|_{\lambda_t}), \qquad \forall\, \lambda \in T^*M_{\neq 0},
\end{equation}
for any choice of canonical frame along $\lambda_t=e^{t\vec{H}}(\lambda)$, see Section \ref{sec:gvd}.

\begin{lemma}[Volume derivative estimates]\label{l:canvorder}
For any $\bar{\lambda} \in T^*M_{\neq 0}$ there exist a $u$-star-shaped neighborhood $\Lambda \subset T^*M_{\neq 0}$ of $\bar{\lambda}$ such that, for all $j\in N$ there exists a constant $C=C_j>0$ such that for the geodesic volume derivative $\rho_{\mm,\lambda}$ it holds:
\begin{equation}
\vec{H}^{(j)}\rho_{\mm,\lambda} \leq C, \qquad \forall\, \lambda \in \Lambda_{\neq 0}.
\end{equation}
\end{lemma}
\begin{proof}
Let $\Lambda$ and $T>0$ as in the previous lemmas. By construction, for all $\lambda \in \Lambda_{\neq 0}$ the extremal $\lambda_t$ remains for all $t\in [0,T]$ in a common bounded subset of a trivial neighborhood $T^*\mathcal{O}$, and the corresponding geodesic $\gamma_t = \pi(\lambda_t)$ remains in $\mathcal{O}$. Recall that from Lemmas \ref{l:lemmaEbFb}, \ref{l:lemmaEcFc} and \ref{l:lemmaFa} we have
\begin{align}
\pi_* F_a|_{\lambda_t} & = \Q_{11}(t,\lambda) X|_{\gamma_t} + \sqrt{2H(\lambda)} \Q_{12}(t,\lambda) Z_{\gamma_t}, \\
\pi_* F_b|_{\lambda_t} & = \Q_{3}(t,\lambda) X|_{\gamma_t}, \\
\pi_* F_c|_{\lambda_t} & = \Q_{8}(t,\lambda) X|_{\gamma_t},
\end{align}
for all $t\in [0,T]$, and where $X=(X_1,\dots,X_k)$, $Z=(X_{k+1},\dots,X_{n})$ are tuples forming an adapted orthonormal frame for $g$ on $\mathcal{O}$. In particular since $\mm(X,Z)=1$, we have
\begin{align}\label{eq:calcolorho}
\rho_{\mm,\lambda_t} = -\vec{H}(V)(\lambda_t) + \frac{\di}{\di t} \log \left\lvert\det \Q_{12}(t,\lambda)\det \begin{matrix} \Q_{3}(t,\lambda) \\ \Q_{8}(t,\lambda) \end{matrix}\right\rvert,
\end{align}
for all $\lambda \in \Lambda_{\neq 0}$ and $t\in[0,T]$. Notice that by construction $\mathbb{Q}_{12}$ is a $n-k\times n-k$ matrix of full rank, while $\left(\begin{smallmatrix} \Q_3 \\ \Q_8\end{smallmatrix}\right)$ is a $k\times k$ matrix of full rank. To prove the statement, observe that $\vec{H}^{(j)}\rho_{\mm,\lambda} = \partial_t^j\rho_{\mm,\lambda_t}|_{t=0}$, so that one has only to bound the time-derivatives of \eqref{eq:calcolorho} at $t=0$ and $\lambda \in \Lambda_{\neq 0}$. The first term $\vec{H}(V)(\lambda_t)$ is uniformly bounded with all time-derivatives since all curves $\lambda_t$ remain in a common bounded subset of the trivial neighborhood $T^*\mathcal{O}$ by construction, and $V$ is smooth. The second term in \eqref{eq:calcolorho} is also uniformly bounded with all time-derivatives, by the properties of the $\Q_i$'s.
\end{proof}

\subsection{Regularity of the canonical curvature maps}

We collect the local estimates from Lemmas \ref{l:lemmaRfirst}--\ref{l:lemmaRlast} in a unified statement for the canonical curvature operators defined as in Section \ref{sec:Jaccurve}.

\begin{remark}
The fact that the neighborhoods $\Lambda$ of Lemmas from \ref{l:lemmaEa} to \ref{l:lemmaRlast} are $u$-star-shaped will be used now.
\end{remark}
\begin{theorem}[Fat curvature estimates]\label{t:fatcurvbound}
Let $(\distr,g)$ be a fat sub-Riemannian structure on $M$. Then for any compact set $B\subset M$ and $j\in \N$ there exists a constant $C=C_j(B)>0$ such that, for all superboxes $\boxtimes,\boxplus =a,b,c$ of the Young diagram $Y$ it holds
\begin{equation}
\|\vec{H}^{(j)}\mathfrak{R}_{\lambda}^{\boxtimes\boxplus}\| \leq C H_R(\lambda)^{\tfrac{j+\col(\boxtimes)+\col(\boxplus)}{2}}, \qquad \forall\, \lambda \in T^*B_{\neq 0},
\end{equation}
where $\|\cdot\|$ denotes the Hilbert-Schmidt norm, and $\col(\boxtimes)$ is the column index of the superbox. As a consequence, for some $C'= C'_j(B)>0$ it also holds
\begin{equation}
|\vec{H}^{(j)}\mathfrak{Ric}_{\lambda}^{\boxtimes}| \leq C' H_R(\lambda)^{\tfrac{j}{2}+\col(\boxtimes)}, \qquad \forall\, \lambda \in T^*B_{\neq 0}.
\end{equation}
Furthermore, if $\mm = e^{-V}\mathrm{vol}$, for the geodesic volume derivative it holds
\begin{equation}
|\vec{H}^{(j)}\rho_{\mm,\lambda}| \leq C H_R(\lambda)^{\tfrac{j+1}{2}}, \qquad \forall\, \lambda \in T^*B_{\neq 0}.
\end{equation}
\end{theorem}
\begin{proof}
With no loss of generality we assume that $B$ is properly contained in a trivial neighborhood $\mathcal{O}$ such that
\begin{equation}
T^*\mathcal{O} = \mathcal{O}\times \R^{k}\times \R^{n-k} , \qquad \lambda = (x;u,v),
\end{equation}
and for  the sub-Riemannian and Riemannian Hamiltonians it holds:
\begin{equation}
H(x;u,v) = \frac{1}{2}|u|^2,\qquad H_R(x;u,v) = \frac{1}{2}(|u|^2+|v|^2).
\end{equation}
Take the following compact set $K \subseteq T^*B$:
\begin{equation}\label{eq:Kset}
K = \{(x;u,v)\in B\times \R^k\times \R^{n-k}\mid |u|^2+|v|^2 = 1\}.
\end{equation}
We can find a finite number of neighborhoods $\Lambda^i$, $i=1,\dots,N$ as in Lemmas \ref{l:lemmaRfirst}--\ref{l:lemmaRlast}, that cover $K_{\neq 0}$. This is possible, even if $K_{\neq 0}$ is non-compact, since the neighborhoods $\Lambda$ obtained in the previous lemmas are $u$-star-shaped (otherwise, their size in the $u$-coordinate might get smaller and smaller as $u\to0$). We deduce that for all $j\in \N$ there exists $C=C_j>0$ such that
\begin{equation}
\|\partial_t^j\mathfrak{R}_{\lambda_t}^{\boxtimes\boxplus}|_{t=0}\| \leq C, \qquad  \forall\, t\in [0,T],\; \lambda \in K_{\neq 0},
\end{equation}
where $\lambda_t = e^{t\vec{H}}(\lambda)$. Recall the homogeneity properties of the curvature from \cite[Thm.\@ 4.7]{BR-Connection}: for any $c>0$, and up to constant congruence (cf.\@ Remark \ref{rmk:hilbert-schmidt}), it holds
\begin{equation}
\mathfrak{R}_{\lambda^c_t}^{\boxtimes\boxplus} \simeq c^{\col(\boxtimes) + \col(\boxplus)}\mathfrak{R}_{\lambda_{ct}}^{\boxtimes\boxplus}, \qquad \forall\, \lambda \in T^*M_{\neq 0},
\end{equation}
where $\lambda^c_t = e^{t\vec{H}}(c\lambda)$. In particular taking derivatives with respect to $t$, and evaluating at $t=0$, we have
\begin{equation}
\partial_t \mathfrak{R}_{\lambda_t}^{\boxtimes\boxplus}|_{t=0} \simeq c^{\col(\boxtimes) + \col(\boxplus) + j}\partial_t \mathfrak{R}_{\eta_t}^{\boxtimes\boxplus}|_{t=0}, \qquad \forall\, \lambda \in T^*M_{\neq 0},
\end{equation}
where $\eta = \lambda/c$. The statement then follows by reduction to $K_{\neq 0}$, since any point $\lambda \in T^*\mathcal{O}_{\neq 0}$ has the form $c \eta$ with $\eta \in K_{\neq 0}$ and $c= \sqrt{2H_R(\lambda)}$.

The statement concerning the canonical Ricci curvatures follows immediately as, for any superbox $\boxtimes \in Y$ and $\lambda \in T^*M_{\neq 0}$ it holds by definition: $\mathfrak{Ric}_\lambda^\boxtimes = \tr (\mathfrak{R}_\lambda^{\boxtimes\boxtimes})$.

The final statement about the geodesic volume derivative follows in a similar fashion, starting from Lemma \ref{l:canvorder} and using the fact that the following property holds:
\begin{equation}
\rho_{\mm,\lambda^c_t} = c\,\rho_{\mm,\lambda_{ct}}, \qquad \forall\, \lambda \in T^*M_{\neq 0},
\end{equation}
which is immediately obtained by using \cite[Prop.\@ 4.9]{BR-Connection} in \eqref{eq:volder}.
\end{proof}

\subsection{Proof of Theorem \ref{t:fattheorem}}

We can now prove Theorem \ref{t:fattheorem}. In fact, it is a corollary of Theorem \ref{t:fatcurvbound}.
\begin{proof}[Proof of Theorem \ref{t:fattheorem}]
If $x,y\notin  \Cut(M)$, we have by Proposition \ref{prop:naturalobject} that
\begin{equation}
\sfD(x,y)^2 = \|\nabla^R_x c(\cdot,y)\|_R^2  = 2H_R( \di_x c(\cdot,y)) = 2H_R(\lambda^{x,y}),
\end{equation}
where $\lambda^{x,y} = \di_x c(\cdot,y) \in T_x ^*M_{\neq 0}$ is the initial covector of the unique strictly normal geodesic $\gamma :[0,1]\to M$ joining $x$ with $y$. Notice that the first equality follows by the very definition of $\sfD$ in the sub-Riemannian case and by the fact that $c$ is smooth outside of $\mathrm{Cut}(M)$. The second equality is obvious by the usual duality between norm and Hamiltonian. In the third equality we used the fact that the initial covector can be recovered by the differential of the squared distance (cf.\@ for example \cite[Lemma 2.20(ii)]{ABR-curvature}).

Let $K\subset M$ be a compact set as in the statement of Theorem \ref{t:fattheorem}. Let $B$ be a larger compact set such that for any $x,y\in K$ all geodesics joining them belong to $B$. On $B$ we apply Theorem \ref{t:fatcurvbound}. In particular if $\lambda\in T_x^*M$ is the initial covector of the unique geodesic $\gamma$ joining $x,y \in K \setminus \mathrm{Cut}(M)$ there exists $C' >0$ such that
\begin{equation}
|\mathfrak{Ric}_{\lambda}^\boxdot| \leq C' H_R(\lambda)^{\col(\boxdot)} = \frac{C'}{2^{\col(\boxdot)}} \sfD(x,y)^{2\col(\boxdot)},
\end{equation}
and where $\boxdot =a,b,c$. This (stronger) inequality in particular implies the desired lower bounds on $\mathfrak{Ric}$ by recalling that $\col(a) =\col(b) =2$ and $\col(c)=1$.

In an analogous way one proves the statement about the geodesic volume derivative starting from the corresponding estimate in Theorem \ref{t:fatcurvbound}.
\end{proof}

\subsection{Compact fat structures satisfy the \texorpdfstring{$\CD(\beta,n)$}{CD} condition}\label{sec:fatcptareCD}

We combine the curvature estimates of Theorem \ref{t:fattheorem} for fat structures, and Theorem \ref{thm:RiccbddbelowareCD}, to obtain the main result of this section.

Recall that $\sfs_\kappa:[0,+\infty)\to \R$, for $\kappa\in\R^\ell$, $\ell\in \N$, are the basic comparison functions of Proposition \ref{p:basic}, associated with a Young diagram of length $\ell$. In particular:
\begin{itemize}
\item for $\ell=1$, $\sfs_{\kappa}$ is a Riemannian-type function (see Section \ref{ss:Riemanniancase});
\item for $\ell=2$, $\sfs_{\kappa_1,\kappa_2}$ is a two-columns-type function (see Section \ref{ss:twocolumnscase}).
\end{itemize}
In all cases, the $\sfs_\kappa$ are real-analytic and have order $N=\ell^2$ at zero.

\begin{theorem}\label{thm:fatcptareCD}
Let $(M,\sfd,\mm)$ be a compact $n$-dimensional sub-Riemannian metric measure space, with fat distribution of rank $k<n$. Let $\sfD$ be the natural gauge function associated with a Riemannian extension. There exist constants $C,\kappa_a,\kappa_b,\kappa_c\in \R$ with the following property: let $\sfs : [0,+\infty) \to \R$ be the real-analytic function
\begin{equation}
\sfs(\theta) := \theta \cdot e^{C\theta} \cdot \sfs_{\kappa_b,\kappa_a}(\theta)^{n-k} \cdot\sfs_{\kappa_c}(\theta)^{2k-n-1} ,
\end{equation}
and let $\beta$ be defined accordingly as in \eqref{eq:defbeta}. Then $(M,\sfd,\mm,\sfD)$ satisfies the $\CD(\beta,n)$.
\end{theorem}
Since $\sfs$ is real-analytic, using Proposition \ref{prop:propertiesbeta}\ref{i:propertiesbeta5}, we obtain the following.
\begin{corollary}\label{cor:fatcptareMCP}
Let $(M,\sfd,\mm)$ be a compact $n$-dimensional sub-Riemannian metric measure space, with fat distribution of rank $k<n$. Then there exists $N'\geq 3n-2k$ such that $(M,\sfd,\mm)$ satisfies the classical $\MCP(0,N')$.
\end{corollary}
\begin{remark}
Both these results can be seen as a sub-Riemannian  counterpart of the fact that any compact Riemannian manifold has Ricci curvature bounded from below.
\end{remark}
\begin{remark}
Corollary \ref{cor:fatcptareMCP} removes the real-analytic assumption of the analogous result proved in  \cite[Thm.\@ 1.3]{BR-realanalMCP}, in the case when the underlying distribution is fat.
\end{remark}
\begin{remark} 
Fat sub-Riemannian structures are ideal and their step is $2$. In particular $\sfD$ is locally bounded, $\sfD$ is meek, and the regularity conditions for the stability of the $\CD$ and $\MCP$ hold. See Figure \ref{fig:implications}.
\end{remark}
\begin{proof}[Proof of Theorem \ref{thm:fatcptareCD}]
For a fat structure, all normal extremals corresponding to non-trivial geodesics have the same Young diagram, with two columns. It has two levels, denoted by $I$ and $II$, the first one with length $\ell_{I}=2$ and size $r_{I} = n-k$, while the second one with length $\ell_{II}=1$ and size $r_{II}=2k-n$. We label the superboxes as follows:
\begin{equation}
\ytableausetup{nosmalltableaux}
\ytableausetup{centertableaux}
Y = \begin{ytableau}
b & a \\
c\\
\end{ytableau}.
\end{equation}
By Theorem \ref{t:fattheorem}, and since $M$ is compact, we find $\kappa_a,\kappa_b,\kappa_c, C\in \R$ such that, for any $(x,y)\notin \Cut(M)$ and the corresponding geodesic $\gamma:[0,1]\to M$ joining them with initial covector $\lambda \in T^*_x M$, it holds
\begin{gather}
\frac{1}{n-k}\mathfrak{Ric}_{\lambda}^a  \geq \kappa_a \sfD(x,y)^4, \qquad  \frac{1}{n-k}\mathfrak{Ric}^{b}_\lambda \geq \kappa_b \sfD(x,y)^2, \\
\frac{1}{2k-n-1} \mathfrak{Ric}_{\lambda}^c  \geq \kappa_c \sfD(x,y)^2, \qquad \rho_{\mm,\lambda} \leq C \sfD(x,y),
\end{gather}
the bound on $\mathfrak{Ric}_{\lambda}^c$ being omitted if $2k-n=1$. Thus it follows that $(M,\sfd,\mm)$, equipped with the gauge function $\sfD$, has Ricci curvatures bounded from below according to Definition \ref{def:srbddbelow}, with the above Young diagram and
\begin{align}
\bar{\kappa}_I & : [0,+\infty) \to \R^2, &  \bar{\kappa}_I(\sfD)  & := (\kappa_b\sfD^2,\kappa_a \sfD^4),\\
\bar{\kappa}_{II} & : [0,+\infty) \to \R, & \bar{\kappa}_{II}(\sfD) & := \kappa_c\, \sfD^2,\\
\bar{C} & : [0,+\infty) \to \R, & \bar{C}(\sfD) & := C \sfD.
\end{align}
We conclude by applying Theorem \ref{thm:RiccbddbelowareCD} obtaining that $(M,\sfd,\mm,\sfD)$ is a $\CD(\beta,n)$ space. The distortion coefficient $\beta$ is induced by the function $\sfs :\R^m_+ \to \R$ of \eqref{eq:sfsbddbelow}. To describe the latter explicitly, we recall that we are in the scalar case $m=1$, and thus by Remark \ref{rmk:identify} the function $\sfs : [0,+\infty) \to \R$ is given by
\begin{equation}
\sfs(\theta) = \theta \cdot e^{C\theta} \cdot \sfs_{\kappa_b,\kappa_a}(\theta)^{n-k} \cdot\sfs_{\kappa_c}(\theta)^{2k-n-1}, \qquad \theta \in [0,+\infty).
\end{equation}
The positivity domain of $\sfs$ is given by
\begin{equation}
\DOM=  \DOM_{\bar{\kappa}_I}\cap\DOM_{\bar{\kappa}_{II}} = [0,\min\{t_{\kappa_c},t_{\kappa_b,\kappa_a}\}],
\end{equation}
and $\sfs$ has order $N=1+ (n-k)\ell_I^{2} + (2k-n-1)\ell_{II}^{2} =3n-2k$ as $\theta \to 0$.
\end{proof}		
\section{Examples and applications}\label{sec:applications}

\subsection{Heisenberg group}\label{sec:heisenbergpuro}

In this section we describe the results already presented in Section~\ref{sec:howtorecover2}, in the special three-dimensional case, in order to make a bridge with the comparison theory developed in the previous sections.

The first Heisenberg group $\mathbb{H}^1$ is the non-commutative group structure on $\R^{3}$ given by the law
\begin{equation}
(x,y,z) \star (x',y',z') = \left(x+x',y+y', z+z'+\frac{1}{2}(xy'-yx')\right).
\end{equation}

Consider the left-invariant vector fields 
\begin{equation}
X = \frac{\partial}{\partial x} -\frac{y}{2}\frac{\partial}{\partial z}, \qquad Y =\frac{\partial}{\partial y} + \frac{x}{2}\frac{\partial}{\partial z},
\end{equation}
and the left-invariant metric making $\{X,Y\}$ a global orthonormal frame. We equip $\mathbb{H}^{1}$ with the associated sub-Riemannian (Carnot-Carath\'eodory) distance, denoted by $\sfd_{0}$ (this notation is slightly different from the one used in Section \ref{sec:howtorecover2}, but it is more convenient in what follows). We also equip $\mathbb{H}^1$ with the Lebesgue measure $\mathscr{L}^{3}$, which is a Haar measure. We set
\begin{equation}
\sfs(\theta) =\theta \sin\left(\frac{\theta}{2}\right)\left[\sin\left(\frac{\theta}{2}\right)-\frac{\theta}{2} \cos\left(\frac{\theta}{2}\right)\right], \qquad \theta\in [0,+\infty).
\end{equation} 
Notice that $\sfs(\theta)$ has finite order equal to $N=5$ as $\theta \to 0$. Moreover, it is clear that
\begin{equation}
\cD=\inf\{\theta>0\mid \sfs(\theta)=0\}=2\pi.
\end{equation}
The corresponding model distortion coefficient defined as in \eqref{eq:defbeta} is then
\begin{equation}
[0,1]\times [0,\infty]\ni (t,\theta) \mapsto \beta^{\mathbb{H}^{1}}_t(\theta) =
\begin{cases}
 t^{5}\quad &\theta =0, \\
t\dfrac{ \sin\left(\frac{t\theta}{2}\right)\left(\sin\left(\frac{t\theta}{2}\right)-\frac{t\theta}{2} \cos\left(\frac{t\theta}{2}\right)\right)}{ \sin\left(\frac{\theta}{2}\right)\left(\sin\left(\frac{\theta}{2}\right)-\frac{\theta}{2} \cos\left(\frac{\theta}{2}\right)\right)}& 0<\theta <2\pi,\\
 +\infty & \theta \geq 2\pi, t\neq 0, \\
0 & \theta \geq 2\pi, t= 0. \\
 \end{cases}
\end{equation}

The main result in \cite[Thm.~1.1]{BKS} is a Jacobian determinant inequality, see \cite[Eq.\@ (3.17), p.\@ 61]{BKS},  which yields by standard manipulations an interpolation inequality for optimal transport. More precisely, for all  $\mu_0,\mu_1\in \Prob_{bs}^{*}(\mathbb{H}^{1},\sfd_0,\mathscr{L}^3)$, there exists $\nu\in\OptGeo(\mu_0,\mu_1)$ associated with a $W_2$-geodesic $(\mu_t)_{t\in [0,1]}$ such that $\mu_t \ll \mathscr{L}^3$ for all $t\in (0,1]$, for $\nu$-a.e.\@ $\gamma$ it holds $(\gamma_0,\gamma_t)\notin \Cut(\mathbb{H}^1)$ for all $t\in (0,1]$, and
\begin{equation}\label{eq:interpolationineq3}
\frac{1}{\rho_t(\gamma_t)^{1/3}}\geq \frac{\beta^{\mathbb{H}^{1}}_{1-t}(\theta^{\gamma_1,\gamma_0})^{1/3}}{\rho_0(\gamma_0)^{1/3}} + \frac{\beta^{\mathbb{H}^{1}}_{t}(\theta^{\gamma_0,\gamma_1})^{1/3}}{\rho_1(\gamma_1)^{1/3}},\qquad \forall\,t\in [0,1],
\end{equation}
where $\rho_t = \frac{\di \mu_t}{\di \mathscr{L}^3}$, and $\theta^{\gamma_0,\gamma_1}$ is the vertical norm of the covector associated with the  unique geodesic with joining $\gamma_0$ with $\gamma_1$.  In \eqref{eq:interpolationineq3} we use the convention that if $\mu_0$ is not absolutely continuous then the first term in the right hand side is omitted.

Notice that the argument in the r.h.s.\@ of \eqref{eq:interpolationineq3} does not depend on the distance. It is then natural to set as a gauge function any map $\sfG:\mathbb{H}^1\times \mathbb{H}^1\to [0,+\infty]$ such that
\begin{equation}\label{eq:sopradacancellare}
\sfG(q,q')= \theta^{q,q'}, \qquad \forall\,(q,q')\notin\Cut(\mathbb{H}^1).
\end{equation}
We remark that $\Cut(\mathbb{H}^1)$ has zero $(\ee_0,\ee_1)_\sharp\nu$-measure for all optimal dynamical plans $\nu$ involved, so that the values of $\sfG$ on that set are irrelevant.  By standard arguments (cf.\@ again Section~\ref{sec:howtorecover2}), the following holds.

\begin{theorem}
The gauge m.m.s.\@ $(\mathbb{H}^{1},\sfd_0,\mathscr{L}^3,\sfG)$ satisfies the $\CD(\beta^{\mathbb{H}^{1}},3)$ condition.
\end{theorem}

One can check that $\beta^{\H^1}_t(\theta) \geq t^{5}$ for all $t\in [0,1]$ and $\theta \in [0,2\pi]$. Thus, by Proposition~\ref{prop:hierarchy}\ref{i:hierarchy3}, $\CD(\beta^{\H^1},3) \Rightarrow \CD(t^{5},3)$. In turn, the latter implies the classical $\MCP(0,5)$.

\subsubsection{Application of the comparison theory to the Heisenberg group}
To make the link with the curvature computations of Sections \ref{sec:comparison}-\ref{sec:fat}, notice that for every pair  $(q,q')\in \mathbb{H}^1\times \mathbb{H}^1\setminus\Cut(\mathbb{H}^1)$ and for the geodesic $\gamma:[0,1]\to \mathbb{H}^1$ joining them, the (reduced) Young diagram $Y$ along $\gamma$ has two columns which can be labeled as follows
\[
Y=\ytableausetup{centertableaux}
\begin{ytableau}
b & a \\
c\\
\end{ytableau}
\]
with all sizes of the superboxes equal to 1. The sub-Riemannian Ricci curvatures have been computed in \cite{AAPL-Ricci}, and are:
\begin{align}
\mathfrak{Ric}^a_{\lambda} =\mathfrak{Ric}^c_{\lambda} = 0, \qquad
\mathfrak{Ric}^b_{\lambda} & = \theta^{2}, 
\end{align}
where
\begin{equation}
\theta = \theta^{q,q'}= |\langle \lambda^{q,q'}, \partial_z\rangle|,
\end{equation}
is the absolute value of the $z$-component of the unique initial covector $\lambda^{q,q'}$ of the geodesic $\gamma :[0,1]\to \mathbb{H}^1$ joining  $q$ and $q'$. We stress that, for such geodesics, $\theta<2\pi$.
We have (cf.\@ Section~\ref{ss:Sasakiancase})
\begin{equation}
\mathsf{\sfs}_{\kappa_1,0}(t) = 
\frac{2 -2 \cos \left(\sqrt{\kappa_1} t\right)-\sqrt{\kappa_1} t 
\sin \left(\sqrt{\kappa_1} t\right)}{\kappa_1^2}, 
\qquad t_{\kappa_1,0}
 = 
\begin{cases}
\tfrac{2\pi}{\sqrt{\kappa_1}} & \kappa_1>0, \\ +\infty & \kappa_1\leq 0.
\end{cases}
\end{equation}

From Theorem~\ref{thm:maincomparison}, one obtains the following comparison theorem
\begin{align}
\beta_{t}^{(\mathbb{H}^1,\sfd_0,\mathcal{L}^3)}(q,q')& \geq 
t\frac{2 -2 \cos \left(t\theta \right)-t\theta  
\sin \left(t\theta \right) }
{2 -2 \cos \left(\theta \right)-\theta  
\sin \left(\theta \right)}\\
& = t \frac{\sin\left(\frac{t\theta}{2}\right)\left[\sin\left(\frac{t\theta}{2}\right)-\frac{t\theta}{2}\cos\left(\frac{t\theta}{2}\right)\right]}{\sin\left(\frac{\theta}{2}\right)\left[\sin\left(\frac{\theta}{2}\right)-\frac{\theta}{2}\cos\left(\frac{\theta}{2}\right)\right]}\\
& = \beta^{\mathbb{H}^1}_t(\theta),
\end{align}
where $\beta^{\mathbb{H}^1}$ is the expression given in Section~\ref{sec:howtorecover2}.

\subsection{The Grushin plane}\label{sec:Grushin}

The Grushin plane is a quotient of the Heisenberg group. More precisely, there is a dilation-invariant (and thus non-compact) subgroup $H$ of the Heisenberg group $\mathbb{H}^1$ and a sub-Riemannian structure on the space of \emph{right} cosets $H\backslash \mathbb{H}^1$ such that $\pi : \mathbb{H}^1 \to H\backslash \mathbb{H}^1$ is a submetry, and $H\backslash \mathbb{H}^1$ is isometric to the Grushin plane, see \cite[Sec.\ 5]{Bellaiche}. It is well-known that the classical synthetic theory of Ricci curvature bounds descends to suitable quotients \cite{GKMS-quotients}. We expect that a similar theory can be developed for gauge metric measure spaces, but this is out of the scope of the present work.\footnote{One should be aware, however, that the Grushin plane is not a homogeneous space in the classical sense: the Heisenberg group acts transitively (from the right) on $\mathbb{G}_2$, but this action is not by isometries, see the discussion in \cite[p.\ 53]{Bellaiche}.}

In this section we illustrate how the Grushin plane, equipped with a suitable gauge function, satisfies the same curvature-dimension inequalities of the Heisenberg group.

To this purpose, we introduce a presentation of the Grushin plane. The Grushin plane $\mathbb{G}_2$ is the sub-Riemannian structure on $\R^2$ defined by the global generating frame
\begin{equation}
X_1 = \partial_x, \qquad X_2 = x \partial_y.
\end{equation}
We equip the Grushin plane with the corresponding sub-Riemannian distance $\sfd_{\mathbb{G}_2}$ and the Lebesgue measure $\mm=\mathscr{L}^2$ of $\R^2$, making it a metric measure space $(\mathbb{G}_2,\sfd_{\mathbb{G}_2},\mathcal{L}^2)$.

Fix $q = (x,y)\in \R^2$ and $q'\notin \Cut(q)$. Let $\gamma:[0,1]\to \R^2$ be the geodesic from $q$ to $q'$. Let $\lambda = u\, dx+v\, dy \in T^*_{(x,y)}\R^2$ be its initial covector. We set $\sfG_{\mathbb{G}_2}:\mathbb{G}_2\times \mathbb{G}_2 \to [0,\infty)$:
\begin{equation}\label{eq:GGrushin}
\sfG_{\mathbb{G}_2}(q,q') = v,\qquad \forall (q,q')\notin \Cut(\mathbb{G}_2),
\end{equation}
arbitrarily extending it to $0$ at the cut-locus. It is well-known that $|v|<\pi$, and conversely if $|v|<\pi$ it holds $\exp_{(x,y)}(u\, dx+v\, dy) \notin \Cut(q)$, see \cite[Ch.\ 13]{nostrolibro}. See Figure \ref{f:Grushin}.

\begin{figure}[t]
\centering
\includegraphics[width=.4\textwidth]{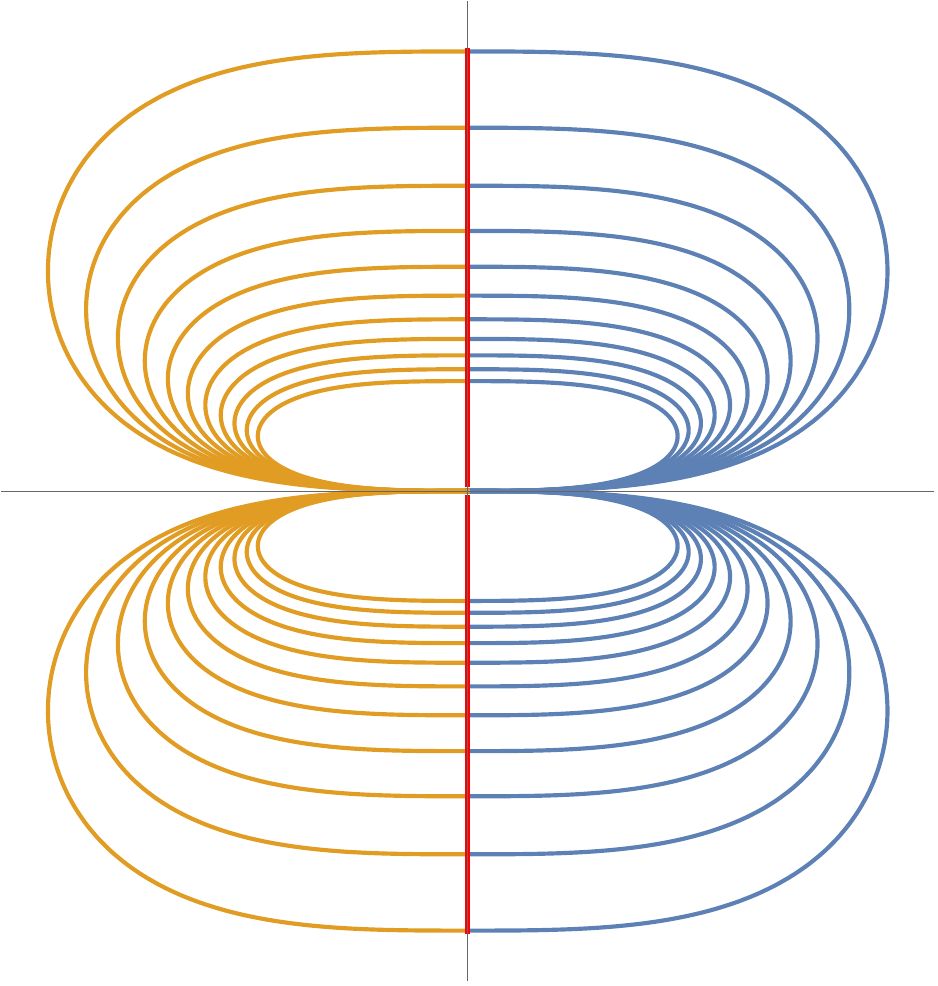} \qquad\qquad \includegraphics[width=.4\textwidth]{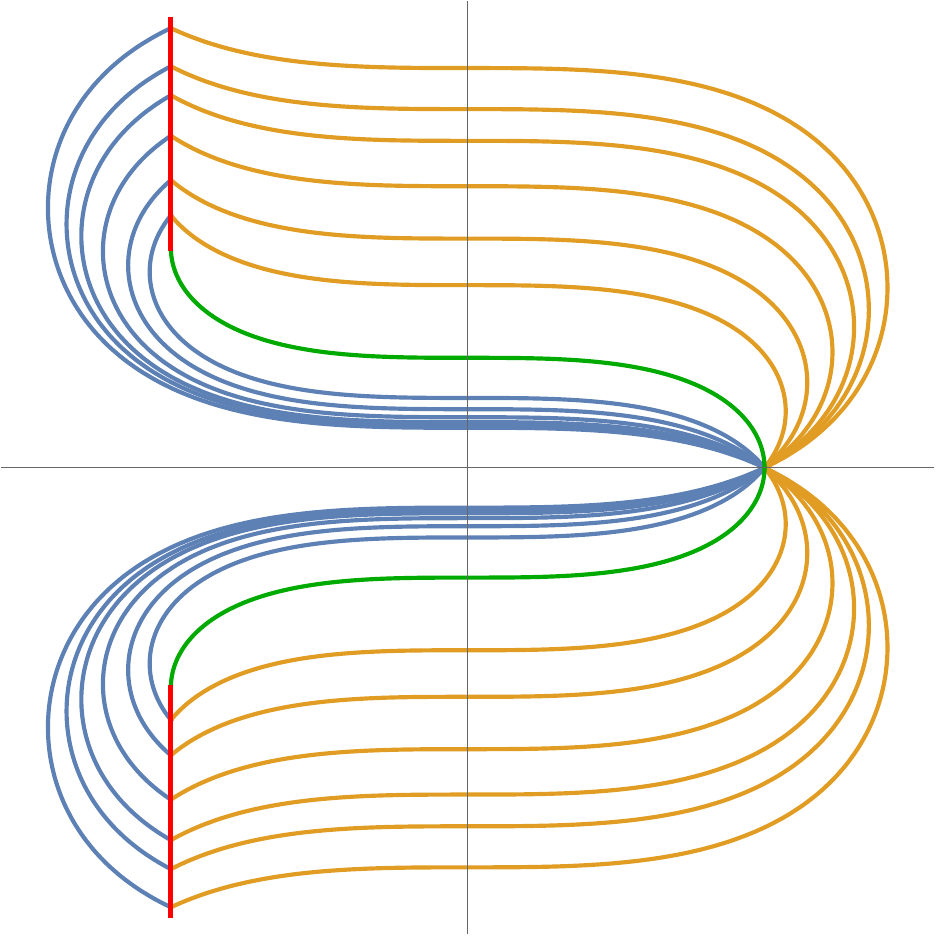}
\caption{Geodesics and cut-locus (in red) of the Grushin plane starting from the origin and from $q=(1,0)$.  Displayed geodesics have initial covector $\lambda=u\, dx + v\, dy$, with $v=\pm \pi$ and different values of $u$.} \label{f:Grushin}
\end{figure}

It is easy to see that \eqref{eq:GGrushin} is a meek gauge function. The Grushin plane is an ideal structure, and hence by \cite{BRInv} it supports interpolation inequalities for densities (with dimensional parameter $n=2$). More precisely, for all  $\mu_0,\mu_1\in \Prob_{bs}^{*}(\mathbb{G}_2,\sfd,\mathscr{L}^2)$, there exists $\nu\in\OptGeo(\mu_0,\mu_1)$ associated with a $W_2$-geodesic $(\mu_t)_{t\in [0,1]}$ such that $\mu_t \ll \mathscr{L}^2$ for all $t\in [0,1]$ and for $\nu$-a.e.\@ $\gamma$ it holds $(\gamma_0,\gamma_t)\notin \Cut(\mathbb{G}_2)$ and
\begin{equation}\label{eq:interpolationineq-grushin}
\frac{1}{\rho_t(\gamma_t)^{1/2}}\geq \frac{\beta^{(\mathbb{G}_2,\sfd,\mathscr{L}^2)}_{1-t}(\gamma_1,\gamma_0)^{1/2}}{\rho_0(\gamma_0)^{1/2}} + \frac{\beta^{(\mathbb{G}_2,\sfd,\mathscr{L}^2)}_{t}(\gamma_0,\gamma_1)^{1/2}}{\rho_1(\gamma_1)^{1/2}},\qquad \forall\,t\in [0,1],
\end{equation}
where $\rho_t = \frac{\di \mu_t}{\di \mathscr{L}^2}$, and $\beta^{(\mathbb{G}_2,\sfd,\mathscr{L}^2)}$ is the ``true'' distortion coefficient of the Grushin plane, see Definition \ref{def:truedistcoeff}. The latter was computed in \cite[Sec.\ 7.3]{BRInv}, and it is given by
\begin{equation}\label{eq:grushintruedistortion}
\beta_t^{(\mathbb{G}_2,\sfd,\mathscr{L}^2)}(q,q') = t \frac{(u^2+t u v^2 x + v^2 x^2)\sin(t v) - t u^2 v \cos(t v)}{(u^2+u v^2 x + v^2 x^2)\sin(v) - u^2 v \cos(v)}, \qquad \forall t\in [0,1],
\end{equation}
where $\lambda = u\, dx + v\, dy \in T_{(x,y)}^*\R^2$ is the initial covector of the geodesic joining $q$ with $q'$.
\begin{lemma}\label{l:GvsH}
For all $(q,q')\notin\Cut(\mathbb{G}_2)$ it holds
\[
\beta_t^{(\mathbb{G}_2,\sfd_{\mathbb{G}_2},\mathscr{L}^2)}(q,q')\geq \beta_t^{\mathbb{H}^1}(\sfG_{\mathbb{G}_2}(q,q')),
\]
where we recall, by definition, $\sfG_{\mathbb{G}_2}(q,q') = v$ in the above notation.
\end{lemma}
\begin{proof}
This can be proven by direct computation, passing to inequalities between logarithmic derivative similarly to  \cite[Sec.\ 7.3]{BRInv}.
\end{proof}

Replacing the inequality of Lemma \ref{l:GvsH} in \eqref{eq:interpolationineq-grushin}, using the definition of $\U_n$, the convexity of the map $\R^{2}\ni(x,y)\mapsto \log(e^{x}+e^{y})$ and Jensen's inequality, one establishes the validity of the corresponding curvature-dimension inequality.
\begin{theorem}\label{t:GrushinCD}
The gauge m.m.s.\@ $(\mathbb{G}_2,\sfd_{\mathbb{G}_2},\mathscr{L}^2,\sfG_{\mathbb{G}_2})$ satisfies the $\CD(\beta^{\mathbb{H}^{1}},2)$ condition.
\end{theorem}
\begin{remark}
We stress that the Heisenberg group satisfies the $\CD(\beta^{\mathbb{H}^1},3)$, while Grushin correctly satisfies the $\CD(\beta^{\mathbb{H}^1},2)$. The correct dimensional parameter appears through the interpolation inequalities for densities, namely \eqref{eq:interpolationineq-grushin}.
\end{remark}

As we already remarked, $\beta^{\H^1}_t(\theta) \geq t^{5}$ for all $t\in [0,1]$ and $\theta \in [0,2\pi]$. Thus, by Proposition~\ref{prop:hierarchy}\ref{i:hierarchy3}, $\CD(\beta^{\H^1},3) \Rightarrow \CD(t^{5},3)$. In turn, the latter implies the classical $\MCP(0,5)$. Thus Theorem \ref{t:GrushinCD} can be seen as a strengthening in the $\CD$ sense of the fact that the Grushin plane satisfies the classical $\MCP(0,5)$.

\subsection{Canonical variation}\label{sec:ex:canvar}

 In this section we consider what is called the \emph{canonical variation} of the Heisenberg group, namely the one-parameter family of left-invariant Riemannian distances $\sfd_\varepsilon$ on $\H^1$, which are those associated with a left-invariant Riemannian metric for which
\begin{equation}
\{X,Y,\varepsilon Z\}, \qquad \varepsilon>0,
\end{equation}
is orthonormal, where
\begin{equation}
X = \partial_x -\frac{y}{2}\partial_z, \qquad Y =\partial_y + \frac{x}{2}\partial_z, \qquad Z=\partial_{z}.
\end{equation}
The terminology canonical variation is used in the context of Riemannian submersions \cite{Besse}. Only for this subsection, it will be handy to denote by $\sfd_0$ the Carnot-Carathéodory metric, so that we study the family of metric measure spaces $(\H^1,\sfd_\varepsilon,\mathscr{L}^3)$ for all $\varepsilon\geq 0$.

 We remark that the Riemannian Ricci curvature of $(\H^1,\sfd_\varepsilon)$ is unbounded as $\varepsilon\to 0$. Therefore it is convenient to study the distortion coefficients of the corresponding metric measure spaces by direct methods. To this purpose, let $\beta_t^\varepsilon : \H^1\times \H^1 \to [0,+\infty]$ be the distortion coefficient of $(\H^1,\sfd_\varepsilon,\mathscr{L}^3)$, for $\varepsilon\geq 0$. It can be computed explicitly, out of the cut locus, denoted by $\Cut_\varepsilon$, for $\varepsilon\geq 0$.

This was done, for example, in \cite[Sec.\@ 4.2]{RiffordCarnot}: letting $\exp_0^{\varepsilon}: T^*_0\H^1 \to \H^1$ be the exponential map for the metric $\sfd_\varepsilon$ starting from $0\in \H^1$, we use cylindrical  coordinates on $T^*_0\H^1$ so that $\lambda = (\theta,\rho,p_z)$. It holds that, for a general point $q\in \H^1$:
\begin{equation}\label{eq:cut-can-var-eps}
q = \exp_0(\theta,\rho,p_z), \qquad q\notin\Cut_\varepsilon(0) \Rightarrow |p_z|<2\pi,
\end{equation}
and furthermore for all $q\notin \Cut_\varepsilon(0)$:
\begin{equation}\label{eq:beta-can-var-eps}
\beta_t^\varepsilon(0,q)= t^2\frac{ h(tp_z/2)\left[\varepsilon^2 t p_z^3 h(tp_z/2) +2\rho^2 k(tp_z/2)\right]}{h(p_z/2)\left[\varepsilon^2 p_z^3 h(p_z/2) + 2\rho^2 k(p_z/2)\right]}, \qquad \forall\, t\in [0,1],
\end{equation}
with the convention that, if $p_z=0$, then we take the limit in the above formula so that $\beta_t^\varepsilon(0,q) = t^5$. In \eqref{eq:beta-can-var-eps}, we have set
\begin{equation}
h(\varsigma) := \frac{\sin(\varsigma)}{\varsigma}, \qquad k(\varsigma):=\sin(\varsigma) - \varsigma\cos(\varsigma)\qquad \varsigma \in (0,\pi).
\end{equation}
The fact that $q\notin \Cut_\varepsilon(0)$, and thus $|p_z|< 2\pi$, implies that the right hand side of \eqref{eq:beta-can-var-eps} is well-defined, for all $\varepsilon\geq 0$. One can check via elementary methods that for all $\varepsilon\geq 0$:
\begin{equation}
\beta_t^\varepsilon(0,q) \geq \beta_t^0(0,q) = t \frac{2-2\cos(p_z t)- p_z t \sin(p_z t)}{2-2\cos(p_z)-p_z \sin(p_z)}, \qquad \forall\, t \in [0,1],\,q\notin \Cut_\varepsilon(0).
\end{equation}

The initial covector $\lambda$ of the unique $\sfd_\varepsilon$-geodesic joining the origin with $q\notin \Cut_\varepsilon(0)$ is given by $\lambda = d_0\left[ \tfrac{1}{2}\sfd_\varepsilon(\cdot,q)^2\right]$, where $d_0$ denotes the differential at zero, and thus $p_z = \lambda(\partial_z)$. We define the left-invariant gauge function $\sfG_\varepsilon: \H^1 \times \H^1 \to [0,2\pi]$ such that:
\begin{equation}\label{eq:prescribed}
\sfG_\varepsilon(0,q):=\begin{cases} \tfrac{1}{2}\left|\left.\tfrac{\partial}{\partial z}\right|_{0}\sfd_\varepsilon(\cdot,q)^2\right|  & q\notin \Cut_\varepsilon(0),\\
0 & q\in \Cut_{\varepsilon}(0),
\end{cases}\qquad \forall \, \varepsilon\geq 0.
\end{equation}
We remark that $|\sfG_\varepsilon|<2\pi$ by \eqref{eq:cut-can-var-eps}. We also observe that the prescribed value on the cut locus in \eqref{eq:prescribed} is arbitrary and unimportant.

Since all structures $(\H^1,\sfd_\varepsilon,\mathscr{L}^3)$ support interpolation inequalities for densities (with dimensional parameter $n=3$, see Definition \ref{def:mcptocd}), with an argument identical to step $2$ in the proof of Theorem \ref{thm:fatcptareCD}, we obtain the following.

\begin{theorem}[Uniform $\CD$ for canonical variation]\label{thm:uniformCDHeis}
Let $\beta$ be defined as in \eqref{eq:defbeta}, with
\begin{equation}
\sfs(\theta) =\theta(2-2\cos(\theta)-\theta \sin(\theta)), \qquad \theta \in [0,+\infty).
\end{equation}
Then, for all $\varepsilon\geq 0$, the gauge m.m.s.\@ $(\H^1,\sfd_\varepsilon,\mathscr{L}^3,\sfG_\varepsilon)$ satisfies the $\CD(\beta,3)$ condition.
\end{theorem}
\begin{remark}
As a matter of fact, the $\beta$ in the previous result satisfies $\beta_t(\theta)\geq t^5$ for all $t\in [0,1]$ and $\theta \in [0,2\pi]$. It follows that this $\CD(\beta,3)$ implies the $\CD(t^5,3)$ or, equivalently, the $\MCP(0,5)$. The above result can be then seen as a strengthening of the known fact, proved in \cite{RiffordCarnot}, that the canonical variation of the Heisenberg group satisfies uniformly the classical $\MCP(0,5)$. See also \cite{PL-senzamani}.
\end{remark}

\paragraph{The sub-Riemannian limit.}  As $\varepsilon\to 0$, it holds $\sfd_\varepsilon \to \sfd_0$ uniformly on compact sets, and a fortiori $(\H^1,\sfd_\varepsilon,\mathscr{L}^3)\to(\H^1,\sfd_0,\mathscr{L}^3)$ in the pmGH-sense. Furthermore, $\sfd_\varepsilon \to \sfd_0$ also in the $C^\infty$ sense on compact subsets of $\H^1 \setminus \Cut_0$ (cf.\@ \cite[Prop.\@ 2.8]{BGMR-Htype}, where this fact is proved for general canonical variations). It follows that:
\begin{equation}
\lim_{\varepsilon\to 0} \sfG_\varepsilon(q,q') = \sfG_0(q,q'),\qquad \forall\, (q,q')\notin \Cut_0.
\end{equation}
The unique optimal plan $\pi$ between $\mu_0\in \Prob_{bs}(\H^1,\sfd_0,\mathscr{L}^3)$, $\mu_1 \in \Prob_{bs}^{*} (\H^1,\sfd_0,\mathscr{L}^3)$, with disjoint supports, is concentrated on $\H^1\setminus\Cut_0$. It follows that $\sfG_\varepsilon \to \sfG_0$ $\pi$-a.e.\@ as $\varepsilon\to 0$. We also recall that $|\sfG_\varepsilon|<2\pi$ for all $\varepsilon\geq 0$. Furthermore, $\sfG_\varepsilon$ is continuous (actually smooth) out of $\Cut_\varepsilon$. Using all these facts, it follows that:	
\begin{itemize}
\item the conditions of Definition \ref{def:weakL1andregularity} are satisfied. Namely $\sfG_\varepsilon \to \sfG_0$ as $\varepsilon\to 0$ in the $L^1_{\loc}$ sense, and the regularity condition for $\sfG_0$ holds.
\item the condition of Definition \ref{def:weakL1andregularity-CD} are satisfied. Namely $\sfG_\varepsilon\to \sfG_0$ as $\varepsilon\to 0$ in the $L^1_{\loc}$ sense over plans, and the regularity condition for $\sfG_0$ over plans is verified.
\end{itemize}
In particular, Theorems \ref{thm:stabMCP} and \ref{thm:stabCD} apply. This is an explicit example of stability of the $\MCP(\beta)$ and $\CD(\beta,3)$, for the sequence of gauge metric measure spaces $(\mathbb{R}^3,\sfd_\varepsilon,\mathscr{L}^3,\sfG_\varepsilon)$, pmGH converging to the Heisenberg group $(\mathbb{R}^3,\sfd_0,\mathscr{L}^3,\sfG_0)$.

\paragraph{The adiabatic limit.} The terminology ``adiabatic limit'' can be traced back to \cite{W-anomalies}. Let $\pi:\H^1\to \R^2$ be the projection, so that $\pi$ is a Riemannian submersion from $(\H^1,\sfd_\varepsilon)$ onto the Euclidean $(\R^2,\sfd_{e})$. As $\varepsilon \to +\infty$, it holds $\sfd_\varepsilon \to \sfd_e \circ \pi$ uniformly on compact sets, and a fortiori $(\H^1,\sfd_\varepsilon,\mathscr{L}^3) \to (\R^2,\sfd_{e},\mathscr{L}^2)$ in the pmGH-sense.
One can prove that:
\begin{equation}
\lim_{\varepsilon \to +\infty} \sfG_\varepsilon(q,q') =0, \qquad \forall\, q,q'\in \H^1,
\end{equation}
and as a consequence all assumptions of Theorems \ref{thm:stabMCP} and \ref{thm:stabCD} are met by setting $\sfG_\infty =0$. Similarly as in the previous case, one can take the limit for $\varepsilon\to +\infty$, obtaining that the Euclidean space $(\R^2,\sfd_e,\mathscr{L}^2)$, equipped with the trivial gauge function $\sfG_\infty \equiv 0$, satisfies the $\CD(\beta,3)$. Notice that $\beta_t(\sfG_\infty) = \beta_t(0) = t^5$, so that in this specific case $\CD(\beta,3)$ is actually equivalent to the $\CD(t^5,3)$. In contrast with the sub-Riemannian limit, in the adiabatic limit we obtain a non-sharp result.

\subsubsection{Sasakian foliations}

A version of Theorem \ref{thm:uniformCDHeis} holds more in general for the canonical variation associated with Sasakian foliations with non-negative Tanaka-Webster curvature, and more generally for H-type foliations satisfying suitable assumptions in the sense of \cite{BGMR-Htype}. The proof makes use of more technical comparison tools, which are a sharpened version of the ones appearing in \cite{BGMR-Htype} (see also \cite{BGKT-Sasakian}). These techniques are out of the scope of this paper, and thus we do not delve into details here.

We only mention that, as a corollary of these generalizations one can show that the canonical variation $(M,\sfd_\varepsilon,\mm)$ of a $2d+1$-dimensional Sasakian foliation with non-negative Tanaka-Webster curvature satisfies the $\MCP(0,2d+3)$ for all $\varepsilon\geq 0$, obtained in \cite{PL-senzamani} with somewhat weaker curvature assumptions. See also \cite[Thm. 3.11]{BGKT-Sasakian}.

\subsection{Convergence to the tangent cone}\label{sec:ex:tangent}

Let $(M,\sfd)$ be a smooth sub-Riemannian metric space of topological dimension $n$. 
It is well-known that for any $p\in M$, the tangent cone at $p$ exists and it is isometric to a homogeneous space for a Carnot group. We describe this convergence explicitly, following \cite{Bellaiche}.

Let $X_1,\dots,X_L$, be a generating family for the sub-Riemannian structure, defined in a neighborhood of $p$. For any multi-index $I\in \{1,\dots,L\}^{\times i}$ we denote $|I| := i$ and $X_I := [X_{I_1},[\dots,[X_{I_{i-1}},X_{I_i}]]]$. Set
\begin{equation}
\distr_p^i := \spn\{X_I(p)\, :\, |I|\leq i\}, \qquad i\in \N.
\end{equation}
Let $k_i = k_i(p):=\dim \distr_p^i$. These subspaces yield a filtration
\begin{equation}\label{eq:filtration}
\{0\}=:\distr_p^0 \subset\distr_p^1\subseteq \dots \subseteq \distr_p^{s} =T_p M,
\end{equation}
where $s=s(p) \in \N$ is the step of the sub-Riemannian structure at $p$. A diffeomorphism $z:U\to \R^n$, from a neighborhood $U\subset M$ of $p$, yields local coordinates that are said:
\begin{itemize}
\item  \emph{linearly adapted} at $p$, if they are centered at $p$, i.e.\@ $z(p) =0$, and $\partial_{z_1}|_0,\dots,\partial_{z_{k_i}}|_0$ form a basis for $\distr_p^i$ in these coordinates, for all $i=1,\dots,s$. We say that the coordinate $z_i$ has \emph{weight} $w_i=j$ if $\partial_{z_i}|_0$ belongs to $\distr_p^j\setminus \distr_p^{j-1}$;
\item \emph{privileged at $p$}, if they are adapted at $p$ and $z_i(q) = O(\sfd(p,q)^{w_i})$ for all $i=1,\dots,n$.
\end{itemize}
Privileged coordinates exist in a neighborhood of any point. From now on we fix a system of privileged coordinates and we identify its domain $U\subseteq M$ with $\R^n$, and $p$ with the origin $0\in \R^n$. Similarly, vector fields defined on $U$ are identified with vector fields on $\R^n$, and the restriction of the sub-Riemannian distance $\sfd$ to $U$ is identified with a distance function on $\R^n$. We define \emph{dilations}, for $\lambda>0$:
\begin{equation}
\delta_\lambda : \R^n\to \R^n, \qquad \delta_\lambda(z_1,\dots,z_n) := (\lambda^{w_1}z_1,\dots,\lambda^{w_n}z_n).
\end{equation}
Dilations induce a concept of \emph{homogeneity} of degree $d\in \N$:
\begin{itemize}
\item for a function $f:\R^n \to \R^n$, if $f(\delta_\lambda(q)) = \lambda^d f(q)$ for all $\lambda >0$ and $q\in \R^n$;
\item for a one-form $\omega$ on $\R^n$, if $\delta_\lambda^* \omega = \lambda^d \omega$ for all $\lambda>0$;
\item for a vector field $X$ on $\R^n$, if $\delta_{\lambda*} X = \lambda^{-d} X$ for all $\lambda >0$.
\end{itemize}
The \emph{principal part} of the generating family $\mathscr{F}=\{X_1,\dots,X_L\}$ is given by
\begin{equation}\label{eq:defXhat}
\hat{X}_i:= \lim_{\varepsilon \to 0} \varepsilon \delta_{1/\varepsilon*} X_i, \qquad i=1,\dots,L.
\end{equation}
The family $\hat{\mathscr{F}}=\{\hat{X}_1,\dots,\hat{X}_L\}$ is a set of complete vector fields on $\R^n$, homogeneous of degree $-1$, with polynomial coefficients. They generate a nilpotent Lie algebra of nilpotency step $s=s(p)$, and they satisfy the bracket-generating condition at any point of $\R^n$. They define a sub-Riemannian structure on $\R^n$, called the \emph{nilpotent approximation} at $p$, and denoted with $(\R^n,\hat{\sfd})$. An essential property of $\hat{\sfd}$ is that it is homogeneous of degree $1$, that is $\hat{\sfd}(\delta_\lambda(q),\delta_\lambda(q')) = \lambda \hat{\sfd}(q,q')$ for all $q,q'\in \R^n$, and $\lambda>0$.

\subsubsection{Description of convergence to the tangent cone}\label{s:desc-convergence}

Fix $p\in M$. The following fundamental estimate holds \cite[Thm.\@ 7.32]{Bellaiche}\footnote{Our statement is a amended version of \cite[Thm.\@ 7.32]{Bellaiche}, which contains a typo.}.
\begin{theorem}\label{thm:bellaiche}
There exist $\varepsilon_p>0$ and $C_p>0$ such that for all $q,q'\in B_{\varepsilon_p}(p)$ it holds
\begin{equation}
-C_p\,  M_p(q,q')\,  \sfd(q,q')^{1/s(p)} \leq \sfd(q,q') - \hat{\sfd}(q,q') \leq C_p \, M_p(q,q') \, \hat{\sfd}(q,q')^{1/s(p)},
\end{equation}
where $M_p(q,q'):=\max\{\hat{\sfd}(0,q),\hat{\sfd}(0,q')\}$ and $s(p)$ is the step of the structure at $p$.
\end{theorem}

Recall that the tangent cone is the pGH limit for $k\to \infty$, when it exists, of $(M,k\sfd,p)$. Letting $\sfd_k := k \sfd$, the set $B_{R/k}(p)$ coincides with the ball of radius $R$ and center $p$, with respect with the rescaled metric $\sfd_k$. For $R>0$, we define approximating maps:
\begin{equation}\label{eq:explicitmaps}
f_k : B_{R/k}(p) \to \R^n,\qquad f_k(q) := \delta_k(q),
\end{equation}
where $k$ is taken so large that $B_{R/k}(p)$ is contained in the domain of privileged coordinates, thus $f_k$ is well-defined. Note that $f_k(p) = 0$ for all $k\in \N$. 

By Theorem \ref{thm:bellaiche}, for all $R>0$ and $\varepsilon>0$, there exists $k_0=k_0(R,\varepsilon)$ such that for all $k\geq k_0$ it holds:
\begin{equation}\label{eq:ballscontained}
\hat{B}_{R(1-\varepsilon)}(0) \subseteq f_k(B_{R/k}(p)) \subseteq \hat{B}_{R(1+\varepsilon)}(0),
\end{equation}
where we denoted with $\hat{B}$ the ball with respect to $\hat{\sfd}$. Furthermore, it holds
\begin{equation}
|\sfd_k(q,q') - \hat{\sfd}(f_k(q),f_k(q'))| \leq \varepsilon,\qquad \forall\, q,q'\in B_{R/k}(p).
\end{equation}

Finally, if in the given set of privileged coordinates $\mm = \varphi\,\mathscr{L}^n$ for some smooth positive function $\varphi$, we set
\begin{equation}
\mm_k:=k^{Q(p)}\mm, \qquad \text{where} \qquad Q(p):=\sum_{j=1}^n j w_j(p).
\end{equation}
Then for all $R>0$ it holds $(f_k)_\sharp(\mm_k|_{B_{R/k}}) \rightharpoonup  \varphi(0)\mathscr{L}^n|_{\hat{B}_R(0)}$ weakly as $k\to \infty$. We can assume, up to modification of the privileged coordinates, that $\varphi(0)=1$. Thus
\begin{equation}
\text{pmGH}-\lim_{k\to \infty}(M,k\sfd,k^Q\mm,p) = (\R^n,\hat{\sfd},\mathscr{L}^n,0),
\end{equation}
with approximating maps \eqref{eq:explicitmaps}. Hence $(\R^n,\hat{\sfd},\mathscr{L}^n)$ is isomorphic to the measured tangent cone at $p$ of $(M,\sfd,\mm)$. Different choices of privileged coordinates and different sequences of rescalings yield isomorphic copies of the same m.m.s.

\subsubsection{Convergence of natural gauge functions to the tangent cone}

For structures of step $\leq 2$, natural gauge functions induce a limit gauge function on the metric cone, with an explicit expression in privileged coordinates. In the following, $\widehat{\Cut}(x)\subset\R^n$ denotes the cut locus from a point $x\in \R^n$ for $(\R^n,\hat{\sfd})$, and $\widehat{\Cut}(\R^n) = \{(x,y)\mid y\in \widehat{\Cut}(x)\}\subset \R^n\times \R^n$.

\begin{theorem}[Gauge functions induced on the tangent cone]\label{thm:naturalconvergencestep2}
Let $(M,\sfd,\mm)$ be a sub-Riemannian m.m.s. Let $p\in M$, $\sfd_k = k\sfd$ and $\mm_k = k^{Q(p)}\mm$, so that $(M,\sfd_k,\mm_k,p)$ converges in the pmGH sense to the tangent cone $(\R^n,\hat{\sfd},\mathscr{L}^n,0)$, with approximating maps \eqref{eq:explicitmaps} in a set of privileged coordinates.

Let $\sfG : M\times M\to [0,\infty]$ be a natural gauge function as in Definition \ref{def:naturalgauge}. Assume that the sub-Riemannian distribution has step $s(p)\leq 2$ at $p$. Then:
\begin{enumerate}[label = (\roman*)]
\item \label{i:naturalconvergencestep2-1} the sequence $\sfG_k := \sfG$ converges in the $L^1_{\loc}$ sense of Definition \ref{def:weakL1andregularity} to a gauge function $\hat{\sfG} : \R^n\times \R^n \to [0,+\infty]$;
\item \label{i:naturalconvergencestep2-2} $\hat{\sfG}$ satisfies the regularity condition of Definition \ref{def:weakL1andregularity};
\item \label{i:naturalconvergencestep2-3} $\hat{\sfG}$ is bounded on $\R^n\times \R^n$;
\item \label{i:naturalconvergencestep2-4} $\hat{\sfG}$ is homogeneous of degree $0$ with respect to dilations, that is
\begin{equation}
\hat{\sfG}(\delta_\lambda(x),\delta_\lambda(y)) = \hat{\sfG}(x,y),\qquad \forall\, x,y\in \R^n, \quad \forall\, \lambda >0;
\end{equation}
\item \label{i:naturalconvergencestep2-5} $\hat{\sfG}$ admits the following explicit description. For $(x,y) \in \widehat{\Cut}(\R^n)$, $\hat{\sfG}(x,y)=0$. For $(x,y)\notin \widehat{\Cut}(\R^n)$, let $\hat{\lambda}^{x,y}\in T_x^*\R^n$ be the initial covector of the unique geodesic in $\Geo(\R^n,\hat{\sfd})$ joining $x$ with $y$:
\begin{equation}
\hat{\lambda}^{x,y} = \sum_{j=1}^n \hat{\lambda}_j^{x,y} \di z_j|_x.
\end{equation}
Let $g=\sum_{i,j=1}^n g_{ij}(z)\di z_i\otimes \di z_j$ be the metric tensor of the Riemannian reference $\sfd_R$, and let $f$ be the function appearing in Definition \ref{def:naturalgauge} of $\sfG$. Then there exists a constant $f(0,1)\in [0,+\infty]$ such that:
\begin{equation}
\hat{\sfG}(x,y) = f(0,1) \sqrt{\sum_{i,j \mid w_i=w_j=2} g^{ij}(0) \hat{\lambda}_i^{x,y} \hat{\lambda}_j^{x,y}}, \qquad \forall\, (x,y)\notin \widehat{\Cut}(\R^n).
\end{equation}
In other words, $\hat{\sfG}(x,y)$ is the norm of the homogeneous component of degree $2$ of $\hat{\lambda}^{x,y}$, with respect to a suitable metric on $\R^n$.
\end{enumerate}
If $\sfG:M\times M \to \RP^m_+$ is vector-valued, an analogous statement holds, with $\hat{\sfG}:\R^n\times \R^n \to \RP^m_+$ and, in item \ref{i:naturalconvergencestep2-5}, $f(0,1)\in \RP^m_+$.
\end{theorem}

\begin{proof}
We remind the following basic facts in sub-Riemannian geometry, that will be used throughout the proof (see Appendix \ref{a:SR}). Let $(M,\sfd)$ be a sub-Riemannian manifold induced by a family of smooth vector fields $\mathscr{F} = \{X_1,\dots,X_L\}$. For any $u\in L^2([0,1],\R^L)$, and point $o\in \R^n$, consider the following ODE:
\begin{equation}\label{eq:ODE-End}
\dot{\gamma}(t) = \sum_{i=1}^L u_i(t) X_i(\gamma(t)), \qquad \gamma(0) = o.
\end{equation}
Solutions $\gamma$ of \eqref{eq:ODE-End} are called horizontal trajectories, and $u$ is called a control for $\gamma$. Let $\mathcal{U}_o^{\mathscr{F}} \subset L^2([0,1],\R^L)$ be the open subset such that the solution to \eqref{eq:ODE-End} is well-defined up to time $1$. The endpoint map for the family $\mathscr{F}$, with base point $o\in \R^n$, is the map $\End_o^{\mathscr{F}} : \mathcal{U}_o^{\mathscr{F}}  \to \R^n$ which associates to $u$ the value $\gamma(1)$ of the solution of \eqref{eq:ODE-End}.

For any $q,q'\in M$, let $\gamma^{q,q'} \in \Geo(M,\sfd)$ be a geodesic from $q$ to $q'$. There exists a non-zero Lagrange multiplier $(\lambda^{q,q'}_1,\nu)\in T_{q'}^*M\times \{0,1\}$, and a control $u\in \mathcal{U}_q^{\mathscr{F}}$ with $|u(t)| = \sfd(q,q')$, such that $\gamma^{q,q'}$ satisfies \eqref{eq:ODE-End} and furthermore
\begin{equation}\label{eq:lagmultpre}
\langle \lambda^{q,q'}_1 , D_{u} \End_{q}^{\mathscr{F}} (v)\rangle = \nu (u,v)_{L^2},\qquad \forall\, v\in L^2([0,1],\R^L),
\end{equation}
where $D$ denotes the Fréchet differential. Geodesics satisfying \eqref{eq:lagmultpre} with $\nu=0$ are called \emph{abnormal}, and they correspond to controls that are critical point of the endpoint map. Geodesics satisfying \eqref{eq:lagmultpre} with $\nu=1$ are called normal. In this latter case, $\gamma$ can be seen as the projection of an integral curve $\lambda :[0,1]\to T^*M$ of a Hamiltonian flow on $T^*M$, with Hamiltonian $H:T^*M\to \R$ given by $H(\eta) = \tfrac{1}{2}\sum_{i=1}^L \langle \eta,X_i\rangle^2$. For such a lift, it holds $\lambda(1) = \lambda_1^{q,q'}$.

We now proceed with the proof assuming $\sfG= \sfD$ (cf.\@ Definition \ref{def:D}).
 
\textbf{Proof of \ref{i:naturalconvergencestep2-1} and \ref{i:naturalconvergencestep2-5}}. Fix $R>0$ and let $k$ be so large that $B_{2R/k}(p)$ is in the domain of privileged coordinates. Thus,  without loss of generality, we assume $M=\R^n$ and the structure $(M,\sfd)$ is defined by a family $\mathscr{F} = \{X_1,\dots,X_L\}$ of smooth vector fields on $\R^n$. Let $(x,y)\notin \widehat{\Cut}(\R^n)$. Let $u_k$ be the control of a geodesic $\gamma_k \in \Geo(M,\sfd)$ joining $\delta_{1/k}(x)$ with $\delta_{1/k}(y)$ such that $\|u_k\|_{L^2} = \sfd(\delta_{1/k}(x),\delta_{1/k}(y))$. For now $\gamma_k$ is not necessarily unique even though we will see that this is the case for large $k$.

Define also the family $\mathscr{F}_k := \{\tfrac{1}{k} \delta_{k*} X_1,\dots,\tfrac{1}{k} \delta_{k*} X_L\}$ of vector fields on $\R^n$, for all $k\in \N$. It is easy to show that
\begin{equation}
\delta_k (\End_{\delta_{1/k}(x)}^{\mathscr{F}}(u)) = \End_x^{\mathscr{F}_k}(k u), \qquad \forall\, u \in \mathcal{U}_{\delta_{1/k}(x)}^{\mathscr{F}}.
\end{equation}

Thus, there exist sequences of Lagrange multipliers $(\lambda^{\delta_{1/k}(x),\delta_{1/k}(y)},\nu_k)\in T_{\delta_{1/k}(y)}^*M\times \{0,1\}$, controls $u_k \in L^2([0,1],\R^L)$ with $\|u_k\|_{L^2} = \sfd(\delta_{1/k}(x),\delta_{1/k}(y))$, and corresponding geodesics $\gamma_k \in \Geo(M,\sfd)$ between $\delta_{1/k}(x)$ and $\delta_{1/k}(y)$ such that
\begin{equation}\label{eq:lagrangeK}
 k^2 \langle\delta_{1/k}^* \lambda^{\delta_{1/k},\delta_{1/k}(y)}_1, D_{ku_k} E_x^{\mathscr{F}_k}(v)\rangle  = \nu_k (k u_k,v)_{L^2},\qquad \forall\, v\in L^2([0,1],\R^L).
\end{equation}
Our goal is to take the limit in \eqref{eq:lagrangeK} for $k\to +\infty$.

Firstly, by construction $k u_k$ is a control associated with the curve $\delta_k\circ \gamma_k$, joining $x$ with $y$. Furthermore, by Theorem \ref{thm:bellaiche}, $\|k u_k\|_{L^2} = k\sfd(\delta_{1/k}(x),\delta_{1/k}(y))\to \hat{\sfd}(x,y)$. In particular $\{ku_k\}_{k\in \N}$ is weakly precompact in $L^2([0,1],\R^L)$. Let $\hat{u}\in L^2([0,1],\R^L)$ be one of its weak limits. Weak convergence of controls together with local uniform convergence of $\tfrac{1}{k}\delta_{k*}X_i\to \hat{X}_i$ for $i=1,\dots,L$, implies that the sequence of corresponding curves $\delta_k\circ \gamma_k$ has a subsequence converging to a limit one $\hat{\gamma}:[0,1]\to \R^n$, with control $\hat{u}$ with respect to the family $\hat{\mathscr{F}} = \{\hat{X}_1,\dots,\hat{X}_L\}$. In particular $\hat{\gamma}$ is the solution of
\begin{equation}
\dot{\hat{\gamma}}(t) = \sum_{i=1}^L \hat{u}_i(t)\hat{X}_i(\hat{\gamma}(t)), \qquad \hat{\gamma}(0) =x,\quad \hat{\gamma}(1) = y.
\end{equation}
By definition of $\hat{\sfd}$, and the weak semi-continuity of the norm we have
\begin{equation}
\hat{\sfd}(x,y)\leq \|\hat{u}\|_{L^2} \leq \liminf_{k\to \infty} \|k u_k\|_{L^2} = \hat{\sfd}(x,y).
\end{equation}
Thus $\hat{\gamma} \in \Geo(M,\hat{\sfd})$, with control $\hat{u}$ such that $\|\hat{u}\| = \hat{\sfd}(x,y)$. Now we use the fact that $(x,y)\notin \widehat{\Cut}(\R^n)$. In this case there is a unique such  geodesic $\hat\gamma$, and it follows that for the original sequence $\delta_k\circ\gamma_k \to \hat{\gamma}$ and $ku_k\rightharpoonup \hat{u}$, respectively.

Secondly, since $ku_k \rightharpoonup \hat{u}$, then $D_{k u_k}\End_x^{\mathscr{F}_k} \to D_{\hat{u}} \End_x^{\hat{\mathscr{F}}}$ in the operator topology (by the same argument in \cite[Prop.\@ 3.7]{Trelatvalue} or  \cite[Appendix, Thm.\@ 23]{LarryBoa}).

Lastly, $\nu_k$ must be definitely equal to $1$, otherwise $\hat{u}$ would be a critical point of $\End_x^{\hat{\mathscr{F}}}$, and thus $\hat{\gamma}$ would be an abnormal geodesic, which is not possible since $(x,y)\notin\widehat{\Cut}(\R^n)$. By the same reason, $\End^{\hat{\mathscr{F}}}_x$ is a submersion at $\hat{u}$, thus for any $V\in T_y \R^n$ there exist $\hat{v} \in L^2([0,1],\R^L)$ such that $D_{\hat{u}}\End_x^{\hat{\mathscr{F}}}(\hat{v})=V$.

We are now ready to take the limit in \eqref{eq:lagrangeK}, taking $v=\hat{v}$, we obtain that
\begin{equation}
\lim_{k\to +\infty} \langle k^2\delta_{1/k}^* \lambda^{\delta_{1/k}(x),\delta_{1/k}(y)}_1,V\rangle = (\hat{u},\hat{v}).
\end{equation}
Since $V\in T_y^*\R^n$ was arbitrary, we have proved that the sequence $k^2\delta_{1/k}^* \lambda^{\delta_{1/k}(x),\delta_{1/k}(y)}_1$ is convergent, and its limit is the Lagrange multiplier for $\hat\gamma$. We must then have
\begin{equation}\label{eq:limitforfinalcov}
\lim_{k\to\infty} k^2 \delta_{1/k}^* \lambda^{\delta_{1/k}(x),\delta_{1/k}(y)}_1= \hat{\lambda}^{x,y}_1,
\end{equation}
where $\hat{\lambda}^{x,y}_1\in T_y^*\R^n$ is the normal Lagrange multiplier of $\hat{\gamma}\in\Geo(\R^n,\hat{\sfd})$. 

We want to translate \eqref{eq:limitforfinalcov} into a statement for \emph{initial} covectors. To do it, let $H,\hat{H}:T^*\R^n\to \R$ be the Hamiltonians of $(\R^n,\sfd)$ and $(\R^n,\hat{\sfd})$, respectively:
\begin{equation}
H(\eta) = \frac{1}{2}\sum_{i=1}^L \langle \eta, X_i\rangle^2, \qquad \hat{H}(\eta) = \frac{1}{2}\sum_{i=1}^L \langle \eta , \hat{X}_i\rangle^2, \qquad \forall \eta \in T^*\R^n.
\end{equation}
We observe that $\tfrac{1}{k^2} H(\delta_{k}^*\eta) \to \hat{H}(\eta)$ as $k\to +\infty$. It follows that, for the Hamiltonian vector fields (the one encoding Hamilton's equations), it holds
\begin{equation}
\lim_{k\to +\infty} \frac{1}{k^2} (\delta_{1/k}^*)_* \vec{H} = \vec{\hat{H}},
\end{equation}
where $(\delta_{1/k}^*)_*$ is the push-forward of $\delta_{1/k}^*: T^*\R^n\to T^*\R^n$. Furthermore, using also the fact that the Hamiltonian is quadratic, namely $H(k \eta) =k^2 H(\eta)$, for flows it holds:
\begin{equation}
\lim_{k\to +\infty} k^2\delta_{1/k}^* e^{t\vec{H}}\left(\frac{1}{k^2} \delta_{k}^*\eta\right) = e^{t\vec{\hat{H}}}(\eta), \qquad \forall\, \eta\in T^*\R^n,
\end{equation}
for all $t \in \R$ (recall that $\vec{\hat{H}}$ is complete). For the initial covector $\lambda^{\delta_{1/k}(x),\delta_{1/k}(y)}$ corresponding to the Lagrange multiplier $\lambda_1^{\delta_{1/k}(x),\delta_{1/k}(y)}$, using \eqref{eq:limitforfinalcov}, we obtain
\begin{align}\label{eq:weknow}
\lim_{k\to +\infty} k^2 \delta_{1/k}^* \lambda^{\delta_{1/k}(x),\delta_{1/k}(y)} & = \lim_{k\to +\infty}
 k^2 \delta_{1/k}^* e^{-\vec{H}} \lambda^{\delta_{1/k}(x),\delta_{1/k}(y)}_1 \\
& = \lim_{k\to +\infty}\underbrace{k^2 \delta_{1/k}^* e^{-\vec{H}} \left( \frac{1}{k^2} \delta_{k}^*\right.}_{\to e^{-\vec{\hat{H}}}} \underbrace{\left.k^2 \delta_{1/k}^*\lambda^{\delta_{1/k}(x),\delta_{1/k}(y)}_1\right)}_{\to \hat{\lambda}^{x,y}_1} = \hat{\lambda}^{x,y}.
\end{align}
In other words, \eqref{eq:limitforfinalcov} holds also for the corresponding initial covectors.

Furthermore, let $\exp_q: T_q^*M\to M$ and $\widehat{\exp}_x : T_o^*\R^n \to \R^n$ be the exponential maps with base point $q\in M$ and $o\in \R^n$ for $(M,\sfd)$ and $(\R^n,\hat{\sfd})$, respectively. That is
\begin{equation}
\exp_q = \pi \circ e^{\vec{H}}|_{T_q^*M}, \qquad \widehat{\exp}_o = \pi\circ e^{\vec{\hat{H}}}|_{T_o^*\R^n}.
\end{equation}
Since $\hat{\lambda}^{x,y}$ is not a critical point for $\widehat{\exp}_x$, we claim that for all sufficiently large $k$, the covectors $\lambda^{\delta_{1/k}(x),\delta_{1/k}(y)}$, and thus the corresponding geodesic $\gamma_k\in \Geo(M,\sfd)$, are uniquely determined by the choice of the endpoints $x$ and $y$, and $\lambda^{\delta_{1/k}(x),\delta_{1/k}(y)}$ is not a critical point for $\exp_{\delta_{1/k}(x)}$. 

To prove the claim, let $F : [0,\varepsilon_0)\times T_x^*\R^n \to \R^n$ be the map
\begin{equation}
F_\varepsilon(\eta) := \begin{cases} \pi \left(\frac{1}{\varepsilon^2} \delta_{\varepsilon}^* e^{\vec{H}}\left(\varepsilon^2 \delta_{1/\varepsilon}^*\eta\right)\right) & \varepsilon>0, \\
\pi\left(e^{\vec{{\hat{H}}}}(\eta)\right) & \varepsilon=0.
\end{cases}
\end{equation}
The map  $F$ is smooth. We know that $F_0 = \widehat{\exp}_x :T_x^*\R^n \to \R^n$ is a diffeomorphism in a neighborhood of $\lambda_0:=\hat{\lambda}^{x,y}$, and that $F_0(\lambda_0) = y$. We already know that any family $\gamma_{1/\varepsilon}$ of geodesics from $\delta_\varepsilon(x)$ and $\delta_\varepsilon(y)$ are not abnormal for small $\varepsilon$. Let then $\lambda^{\delta_\varepsilon(x),\delta_\varepsilon(y)}$ be a corresponding family of initial covectors. We also know from \eqref{eq:weknow} that the sequence $\lambda_\varepsilon:= \tfrac{1}{\varepsilon^2} \delta_{\varepsilon}^* \lambda^{\delta_{\varepsilon}(x),\delta_{\varepsilon}(y)}$ is such that $\lambda_\varepsilon \to \lambda_0$, and by construction $F_\varepsilon(\lambda_\varepsilon) = y$ for all $\varepsilon$. It follows that for sufficiently small $\varepsilon$ there is a unique such $\lambda_\varepsilon$ and thus $\lambda^{\delta_{\varepsilon}(x),\delta_{\varepsilon}(y)}$ and $\gamma_{1/\varepsilon}$ are uniquely determined by $x$ and $y$. Since $\lambda_0$ is regular for $F_0$, and since $\varepsilon \mapsto d F_\varepsilon$ is continuous, then $\lambda_\varepsilon$ is a regular point of $F_\varepsilon$ for small $\varepsilon$. Unraveling the notation, this means that $\lambda^{\delta_\varepsilon(x),\delta_\varepsilon(y)}$ is a regular point for $\exp_x$ for all sufficiently small $\varepsilon$. This concludes the proof of the claim.

Summing up, so far we proved that for any $(x,y)\notin \Cut(\R^n,\hat{\sfd})$ and sufficiently large $k\in \N$ it holds:
\begin{itemize}
\item $(\delta_{1/k}(x),\delta_{1/k}(y)) \notin \Cut(M,\sfd)$;
\item the initial covectors $\lambda^{\delta_{1/k}(x),\delta_{1/k}(y)}$ of the unique geodesics $\gamma_k\in\Geo(M,\sfd)$ from $\delta_{1/k}(x)$ to $\delta_{1/k}(y)$ satisfy
\begin{equation}
\lim_{k\to +\infty} k^2\delta_{1/k}^* \lambda^{\delta_{1/k}(x),\delta_{1/k}(y)} = \hat{\lambda}^{x,y},
\end{equation}
where $\hat{\lambda}$ is the initial covector of the unique $\hat{\gamma}$ in $\Geo(\R^n,\hat{\sfd})$ joining $x$ with $y$.
\end{itemize}

Let us now discuss more precisely the action of dilations on covectors in our system of privileged coordinates. For all $q\in \R^n$, it holds:
\begin{equation}\label{eq:actionofdeltaoncovectors}
\eta = \sum_{i=1}^n \eta_i \di z^i|_{\delta_{1/k}(q)} \qquad \Rightarrow \qquad k^2\delta_{1/k}^* \eta = \sum_{i=1}^n k^{2-w_i} \eta_i  \di z^i|_{q}.
\end{equation}
Thus, the map $k^2\delta_{1/k}^*$ dilates by a factor $k^{2-j}$ the homogeneous component of degree $j$ of $\eta$. We stress the following consequences of \eqref{eq:weknow} and \eqref{eq:actionofdeltaoncovectors}:
\begin{itemize}
\item the homogeneous component of degree $1$ of $\lambda^{\delta_{1/k}(x),\delta_{1/k}(y)}$ converges to zero;
\item the homogeneous component of degree $2$ of $\lambda^{\delta_{1/k}(x),\delta_{1/k}(y)}$ converges to the homogeneous component of degree $2$ of $\hat{\lambda}^{x,y}$;
\item the homogeneous components of degree $j>2$ have growth $O(k^{j-2})$.
\end{itemize}
To conclude, recall from Proposition \ref{prop:naturalobject} that if $(\delta_{1/k}(x),\delta_{1/k}(y))\notin\Cut(M,\sfd)$, then
\begin{equation}
\sfD(\delta_{1/k}(x),\delta_{1/k}(y)) = \sqrt{\sum_{i,j=1}^n g_{ij}(\delta_{1/k}(x))\lambda^{\delta_{1/k}(x),\delta_{1/k}(y)}_i \lambda^{\delta_{1/k}(x),\delta_{1/k}(y)}_j}.
\end{equation}
Thus if the step $s(p)=2$ at $p$, we must have at all $(x,y)\notin \widehat{\Cut}(\R^n)$
\begin{equation}\label{eq:pointwiseforgauge-a.e.}
\lim_{k\to +\infty} \sfD(\delta_{1/k}(x),\delta_{1/k}(y)) = \sqrt{\sum_{w_i=w_j=2} g_{ij}(0)\hat{\lambda}^{x,y}_i \hat{\lambda}^{x,y}_j} =:\hat{\sfG}(x,y).
\end{equation}
We extend $\hat{\sfG}$ to zero at all cut points. A little care in the argument proves also that the convergence in \eqref{eq:pointwiseforgauge-a.e.} is uniform in any sufficiently small neighborhood of a point $(x_o,y_o)\notin \widehat{\Cut}(\R^n)$. Also $(x,y)\mapsto \hat{\lambda}^{x,y}$ is continuous out of the cut locus. It follows that if $(x,y)\notin \widehat{\Cut}(\R^n)$ and $x_k\to x$, $y_k\to y$, we have
\begin{equation}\label{eq:pointwiseforgauge-a.e.2}
\lim_{k\to +\infty} \sfD(\delta_{1/k}(x_k),\delta_{1/k}(y_k)) =\hat{\sfG}(x,y).
\end{equation}

We now show how \eqref{eq:pointwiseforgauge-a.e.2} implies the $L^1_{\loc}$ convergence of $\sfD$ to $\hat{\sfG}$ as in Definition \ref{def:weakL1andregularity}. Fix $\varepsilon,R>0$. Recall that by \eqref{eq:ballscontained}, $\delta_k(B_{R/k}(p))\subseteq \hat{B}_{R(1+\varepsilon)}(0)$ for large $k$. Let $q_k\in B_{R/k}(p)$ form a sequence such that $\delta_k(q_k) = x_k$ is convergent in $\R^n$. For sufficiently large $k$, then $q_k = \delta_{1/k}(x_k)$ where $x_k$ is a convergent sequence in $B_{R(1+\varepsilon)}(0)$. Then
\begin{multline}
\int_{B_{R/k}(p)} |\sfD(q_k,z) - \hat{\sfG}(\delta_k(q_k),\delta_k(z))|k^Q \varphi(z)\,\mathscr{L}^n(\di z) \\
 =
\int_{\delta_k(B_{R/k}(p))} |\sfD(\delta_{1/k}(x_k),\delta_{1/k}(z)) - \hat{\sfG}(x_k,z)|\varphi(\delta_{1/k}(z))\,\mathscr{L}^n(\di z)  \\
 \leq \int_{\hat{B}_{R(1+\varepsilon)}(0)}|\sfD(\delta_{1/k}(x_k),\delta_{1/k}(z))-\hat{\sfG}(x_k,z)|\varphi(\delta_{1/k}(z))\,\mathscr{L}^n(\di z) .\label{eq:lastintegral}
\end{multline}
Remove from the domain of integration the set $Z:=\widehat{\Cut}(\bar{x})$, given by the cut locus of $\bar{x}:=\lim x_k$ with respect to the metric $\hat{\sfd}$. It is well-known that, in the step $2$ setting, $Z$ has zero Lebesgue measure. At any $z\in \hat{B}_{R(1+\varepsilon)}(0) \setminus Z$, we have that:
\begin{equation}
\lim_{k\to +\infty}\sfD(\delta_{1/k}(x_k),\delta_{1/k}(z)) = \hat{\sfG}(\bar{x},z), \qquad \text{and} \qquad \lim_{k\to +\infty}\hat{\sfG}(x_k,z) = \hat{\sfG}(\bar{x},z).
\end{equation}
The first limit is \eqref{eq:pointwiseforgauge-a.e.2}, while the second one follows from the fact that $(x,z)\mapsto\hat{\lambda}^{x,z}$ is continuous out of the cut locus. Furthermore, the integrand of  \eqref{eq:lastintegral} is locally bounded thanks to the step $2$ assumption. In fact, $\hat{\sfG}$ is locally bounded by item \ref{i:naturalconvergencestep2-3}. On the other hand, since $s(p)\leq 2$, then the step is $\leq 2$ in a neighborhood of $p$, and thus $\sfD$ is also locally bounded close to $p$ (cf.\@ Theorem \ref{t:boundedness}). Thus, by the dominated convergence theorem, \eqref{eq:lastintegral} tends to zero as $k\to \infty$, proving the $L^1_{\loc}$ convergence of Definition \ref{def:weakL1andregularity}.

\textbf{Proof of \ref{i:naturalconvergencestep2-2}.} The map $(x,y)\mapsto \hat{\lambda}^{x,y}$ is continuous out of $\widehat{\Cut}(\R^n)$. For any $x\in \R^n$, the set $\widehat{\Cut}(x)$ has zero Lebesgue measure, and thus $\hat{\sfG}$ satisfies the regularity condition of Definition \ref{def:weakL1andregularity}.

\textbf{Proof of \ref{i:naturalconvergencestep2-3}.} The metric tangent $(\R^n,\hat{\sfd})$ has constant step equal to the one of the original structure $s(p) \leq 2$. Thus, $\hat{\sfd}$ is locally Lipschitz in charts \cite[Cor.\@ 6.2]{AAPL-transport}. For $(x,y) \in \widehat{\Cut}(\R^n)$ we have set $\hat{\sfG}(x,x)=0$. For $(x,y)\notin \widehat{\Cut}(\R^n)$, we have $\hat{\lambda}^{x,y} = - \tfrac{1}{2} \di_x \hat{\sfd}^2(\cdot,y)$. Thus $\hat{\sfG}$ is locally bounded. By item \ref{i:naturalconvergencestep2-4} it is also globally bounded.

\textbf{Proof of \ref{i:naturalconvergencestep2-4}.} Observe that dilations map $\widehat{\Cut}(\R^n)$ to itself. Furthermore, homogeneity of $\hat{\sfd}$ and the fact that $\hat{\lambda}^{x,y} = -\tfrac{1}{2} \di_x \hat{\sfd}^2(\cdot,y) \in T_x^*\R^n$ yield
\begin{equation}
\hat{\lambda}^{\delta_{1/k}(x),\delta_{1/k}(y)} = \frac{1}{k^2}\delta_{k}^* \hat{\lambda}^{x,y},\qquad \forall\, (x,y)\notin \widehat{\Cut}(\R^n).
\end{equation}
Thanks to the explicit formula for $\hat{\sfG}$ of item \ref{i:naturalconvergencestep2-5}, $\hat{\sfG}(x,y)$ depends only on the homogeneous part of degree $2$ of $\hat{\lambda}^{x,y}$, and this is invariant by \eqref{eq:actionofdeltaoncovectors} by the action of $\frac{1}{k^2}\delta_{k}^*$. In other words $\hat{\sfG}(\delta_k(x),\delta_{k}(y)) = \hat{\sfG}(x,y)$ out of the cut locus, and also at the cut locus since there we set $\hat{\sfG}\equiv 0$.

This concludes the proof in the case $\sfG = \sfD$. Assume that, instead of $\sfD$, one had chosen a general natural gauge function $\sfG: M \times M \to [0,+\infty]$, induced by a $1$-homogeneous function $f: \Omega \to [0,+\infty]$ as in Definition \ref{def:naturalgauge}. Notice that $\sfd(\delta_{1/k}(x),\delta_{1/k}(y)) \to 0$ as $k\to \infty$. Furthermore, $\sfD$ is locally bounded in the step $2$ case, and thus $f$ is also bounded and smooth on its domain. It follows that $\sfG$ converges in the $L^1_{\loc}$ sense to $f(0,\hat{\sfG}_0)$, where $\hat{\sfG}_0$ is the $L^1_{\loc}$ limit of $\sfD$ that we have described in the previous case. Using the $1$-homogeneity of $f$ we deduce that $\hat{\sfG}=f(0,1)\hat{\sfG}_0$. This concludes the proof for scalar valued gauge functions.

The proof is unchanged in the vector-valued case, with obvious modifications.
\end{proof}

We record one main consequence of Theorem \ref{thm:naturalconvergencestep2}.

\begin{theorem}[$\MCP$ on the tangent]\label{t:MCPfortangent}
Let $(M,\sfd,\mm)$ be a sub-Riemannian metric measure space equipped with a natural gauge function $\sfG$. Let $(\R^n,\hat{\sfd})$ be a tangent cone of $(M,\sfd)$ at $p$ in a system of privileged coordinates, where $n= \dim M$. Equip it with:
\begin{itemize}
\item the measure $\hat{\mm}=\lim_{k\to 0}(\delta_{k})_\sharp \mm$, proportional to the Lebesgue measure on $\R^n$;
\item the limit gauge function $\hat{\sfG}$ of Theorem \ref{thm:naturalconvergencestep2}, assuming that the step is $s(p)\leq 2$.
\end{itemize}
Let $\beta$ as in \eqref{eq:defbeta}, and assume that $\beta_t : [0,\cD)\to \R$ is locally Lipschitz. If $(M,\sfd,\sfG,\mm)$ satisfies the $\MCP(\beta)$ then also $(\R^n,\hat{\sfd},\hat{\sfG},\hat{\mm})$ satisfies the $\MCP(\beta)$.
\end{theorem}
\begin{corollary}[$\CD$ on the tangent]\label{c:CDfortangent}
In Theorem \ref{t:MCPfortangent}, assume moreover that $(M,\sfd)$ is fat. Then $(\R^n,\hat{\sfd},\hat{\sfG},\hat{\mm})$, which is an ideal Carnot group, satisfies the $\CD(\beta,n)$.
\end{corollary}
\begin{proof}[Proof of Theorem \ref{t:MCPfortangent} and Corollary \ref{c:CDfortangent}]
Using the explicit description of the convergence to the tangent cone at the beginning of Section \ref{s:desc-convergence}, we apply the stability Theorem \ref{thm:stabMCP}. To apply it, we first remark that, since $s(p)\leq 2$, then the step is $\leq 2$ in a neighborhood of $p$ so that any natural gauge function such as $\sfG$ must be locally bounded (cf. Theorem \ref{t:boundedness}). The remaining hypotheses of Theorem \ref{thm:stabMCP} are satisfied thanks to Theorem \ref{thm:naturalconvergencestep2}. This proves Theorem \ref{t:MCPfortangent}.

We now prove Corollary \ref{c:CDfortangent}. If $(M,\sfd)$ is fat, then also $(\R^n,\hat{\sfd})$ is fat. We note in passing that fat structures have constant growth vector, so that the structure is equiregular and in particular the tangent cone has the structure of a Carnot group \cite{Bellaiche}.

Furthermore, fat structures do not admit non-trivial abnormal geodesics, so they are ideal. By the results in \cite{BRInv}, ideal structures support interpolation inequalities with dimensional parameter $n$ as in Definition \ref{def:mcptocd}, and the corresponding $\nu$ appearing in Definition \ref{def:mcptocd} is concentrated on a set of geodesics that avoid the cut locus (out of which, $\hat{\sfG}$ is continuous). Notice that the parameter $N$ in the definition of $\beta$ in \eqref{eq:defbeta} must be necessarily $N\geq n$, as a consequence of Theorem \ref{thm:GeodDimEst}. 

Thus we can apply Theorem \ref{thm:mcptocd} to $(\R^n,\hat{\sfd},\hat{\sfG},\hat{\mm})$, improving the $\MCP(\beta)$ of Theorem \ref{t:MCPfortangent} to the $\CD(\beta,n)$.
\end{proof}

\begin{remark}
We can appreciate here an important difference with respect to the Riemannian case. In the Riemannian case, any system of coordinates is privileged, all coordinates have weight one, and thus by Theorem \ref{thm:naturalconvergencestep2}\ref{i:naturalconvergencestep2-5} any natural gauge function induces, in the limit, the trivial one $\hat{\sfG} =0$. Since for general distortion coefficients $\beta_t(\hat{\sfG})=\beta_t(0)=t^N$, Theorem \ref{t:MCPfortangent} can be used to prove that the metric measure tangent of Riemannian manifolds must satisfy the $\MCP(0,N)$, \emph{without using the explicit knowledge of the tangent}. However, in sub-Riemannian geometry, where weight $2$ coordinates are available, one can have non-trivial gauge functions in the limit, and thus the $\MCP(\beta)$ property on the tangent does not imply, a priori, a classical $\MCP(0,N)$.
\end{remark}

\subsection{Vector-valued gauge functions on three-dimensional structures}\label{sec:ex:vectorial}
In this section we illustrate through an example our theory in the case of vector-valued gauge functions. 

We consider here left-invariant fat sub-Riemannian structures on three-dimensional Lie groups, sometimes referred to as three-dimensional model spaces (even if with a different meaning than in Section~\ref{sec:comparison}): the Heisenberg group $\mathbb{H}^1$, the special unitary group $\SU(2)$, and the special linear group $\SL(2)$.

For a unified description, let us consider $G$ a simply connected three-dimensional Lie group endowed with a contact
sub-Rieman\-ni\-an structure whose distribution $\distr$ is generated by two left-invariant vector fields $X_{1}$ and $X_{2}$ which are orthonormal with respect to the corresponding sub-Riemannian metric. We assume the existence of a killing vector field $X_0$ transverse to $\distr$. Up to isometries and dilations, thanks to the classification in \cite{AB12,FG96}, we can assume that the following Lie bracket relations hold for some $K\in \R$:
\begin{equation}
  [X_{1},X_{2}]=X_{0},\qquad
  [X_{0},X_{1}]=K X_{2}, \qquad
  [X_{0},X_{2}]=-K X_{1}.
\end{equation}
In this description the Heisenberg group corresponds to $K=0$, the group  $\SU(2)$ for $K=1$, while $\SL(2)$ corresponds to the choice $K=-1$. Here $K$ coincides with the Tanaka-Webster curvature of the CR structure.

If we endow these sub-Riemannian structures with the Popp volume $\mm$ (see \cite{BR13}), then the corresponding volume distortion $\rho_\mm$ (cf.\@ Appendix~\ref{sec:gvd} and Section~\ref{sec:fat}) along non-trivial geodesics is always identically zero \cite{ABP19}.

Following the notation of Sections~\ref{sec:comparison}-\ref{sec:fat}, given $x,y\in G\setminus \mathrm{Cut}(G)$ and for the geodesic $\gamma:[0,1]\to G$ joining them, $\gamma$ has (reduced) Young diagram $Y$ with two columns which can be labeled as follows:
\[
Y=\ytableausetup{centertableaux}
\begin{ytableau}
b & a \\
c\\
\end{ytableau}
\]
with all sizes of the superboxes equal to $1$. The Young diagram has two levels, denoted by $I$ and $II$, the first one with length $\ell_{I}=2$, while the second one with length $\ell_{II}=1$, both with size $1$. The sub-Riemannian Ricci curvatures have been computed in \cite{AAPL-Ricci}:
\begin{equation}\label{eq:curvbounds3Dleftinv}
\mathfrak{Ric}^a_{\lambda} =\mathfrak{Ric}^c_{\lambda} = 0, \qquad
\mathfrak{Ric}^b_{\lambda} = h_{0} (\lambda)^{2}+K \sfd^{2}(x,y), 
\end{equation}
where  $\lambda = \lambda^{x,y} \in T_x^*M$ is the initial covector of the geodesic $\gamma$ and
\begin{equation}
h_0(\lambda) := \langle \lambda^{x,y},  X_0\rangle, 
\end{equation}
is its component in the direction of $X_0$. 

Notice that, out of the cut locus, by Proposition \ref{prop:naturalobject} it holds:
\begin{equation}
|h_0(\lambda^{x,y})| = \sqrt{\sfD(x,y)^{2}-\sfd(x,y)^{2}},\qquad \forall\, (x,y)\notin \Cut(G),
\end{equation}
where $\sfD$ denotes the natural gauge associated with the Riemannian extension such that $X_{1},X_{2},X_{0}$ is an orthonormal frame (cf.\@ Definition \ref{def:D}).

Therefore, on the basis of the curvature bounds \eqref{eq:curvbounds3Dleftinv}, we are led to define on $G$ the vector-valued gauge function $\sfG:G\times G\to \R^{2}_{+}$ as follows
\begin{equation} \label{eq:g1g2}
\sfG=(\sfG_{1},\sfG_{2}):=(\sqrt{\sfD^{2}-\sfd^{2}},\sfd).
\end{equation}
Since contact structures have step $2$, then $\sfG$ is locally bounded by  Theorem \ref{t:boundedness}\ref{i:boundedness3}.

 The assumptions of the main comparison Theorem \ref{thm:maincomparison} are satisfied by setting
\begin{align}
\bar{\kappa}_I & : [0,\infty) \to \R^2, &  \bar{\kappa}_I(\theta_{1},\theta_{2})  & := (\theta_{1}^{2}+K \theta_{2}^{2},0),\\
\bar{\kappa}_{II} & : [0,\infty) \to \R, & \bar{\kappa}_{II}(\theta_{1},\theta_{2}) & := 0.
\end{align}
 To make the formulation of Theorem \ref{thm:maincomparison} explicit, we display the functions occurring in its statement. For what concerns the level $I$, following the construction described in Section \ref{sec:constcurvmodels}, we have for $\kappa_1\in\R$:
\begin{equation}
\mathsf{\sfs}_{\kappa_1,0}(t) = 
\frac{2 -2 \cos \left(\sqrt{\kappa_1} t\right)-\sqrt{\kappa_1} t 
\sin \left(\sqrt{\kappa_1} t\right)}{\kappa_1^2}, 
\qquad t_{\kappa_1,0}
 = 
\begin{cases}
\tfrac{2\pi}{\sqrt{\kappa_1}} & \kappa_1>0, \\ +\infty & \kappa_1\leq 0,
\end{cases}
\end{equation}
with usual interpretation when $\kappa_1\leq 0$. Using \eqref{eq:link}, we have  
\[
\sfs_{\bar \kappa_{I}}(\theta_{1},\theta_{2})=\sfs_{\bar{\kappa}_{I}\left({\theta_1}/{|\theta|},{\theta_1}/{|\theta|} \right)}(|\theta|),
\]
and more explicitly
\begin{equation}
\sfs_{\bar \kappa_{I}}(\theta_{1},\theta_{2}) = |\theta|^2 \left(\frac{2-2\cos\left(\sqrt{\theta_1^2+K\theta_2^2}\right)-\sqrt{\theta_1^2+K\theta_2^2}\sin\left(\sqrt{\theta_1^2+K\theta_2^2}\right)}{\theta_1^2+K\theta_2^2}\right),
\end{equation}
where $|\theta| =\sqrt{\theta_1^2+\theta_2^2}$. The positivity domain of $\sfs_{\bar{\kappa}_I}$ is
\begin{equation}\label{eq:domex}
\mathrm{DOM}_{\bar \kappa_{I}}=\{(\theta_{1},\theta_{2})\in \R^2_+ \mid \theta_{1}^{2}+K \theta_{2}^{2}< 2\pi\},
\end{equation}
which is indeed an open and star-shaped set, whose shape depends on the sign of $K$, as shown in Figure~\ref{fig:dom-examples}.

The level $II$ is omitted according to the prescriptions of Theorem \ref{thm:maincomparison}, so we do not need to compute the corresponding comparison functions. 

\begin{figure}
\centering
\includegraphics[width=\textwidth]{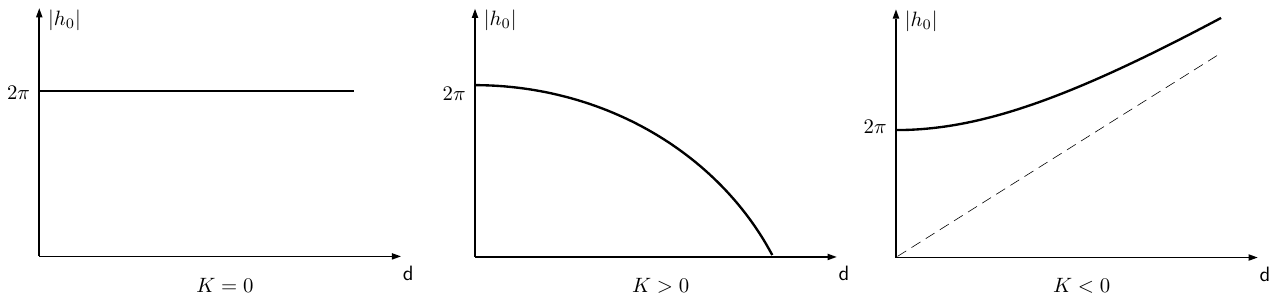}
\caption{Picture of $\DOM_{\bar\kappa}$ for $\bar{\kappa}(\theta_1,\theta_2) = \theta_1^2 + K\theta_2^2$. Here, $\theta_1 = |h_0|$ and $\theta_2=\sfd$.}\label{fig:dom-examples}
\end{figure}

Thus, Theorem \ref{thm:maincomparison} yields that for all $(x,y) \notin \Cut(G)$ we have $\sfG(x,y)\in \DOM_{\bar\kappa_{I}}$, and
\begin{equation}
\beta_{t}^{(X,\sfd,\mm)}(x,y)\geq t\frac{\sfs_{\bar \kappa_{I}}(t \sfG(x,y))}{\sfs_{\bar \kappa_{I}}(\sfG(x,y))}.
\end{equation}
The above inequality can be explicitly rewritten thanks to \eqref{eq:g1g2} in the form:
\begin{equation}
\beta_{t}^{(X,\sfd,\mm)}(x,y)\geq 
t\frac{2 -2 \cos \left(t\sqrt{h_{0}^{2}+K \sfd^{2}} \right)-t\sqrt{h_{0}^{2}+K \sfd^{2}}  
\sin \left(t\sqrt{h_{0}^{2}+K \sfd^{2}} \right) }
{2 -2 \cos \left(\sqrt{h_{0}^{2}+K \sfd^{2}} \right)-\sqrt{h_{0}^{2}+K \sfd^{2}}  
\sin \left(\sqrt{h_{0}^{2}+K \sfd^{2}} \right)},
\end{equation}
 where $h_0 = h_0(\lambda^{x,y})$ and $\sfd = \sfd(x,y)$.

\begin{remark}  One may wonder whether one could apply our framework with the \emph{scalar} gauge function $\sfG_{\mathrm{scal}}:=h_{0}^{2}+K \sfd^{2}$, instead of the vector-valued one $\sfG = (|h_0|,\sfd)$. This approach has two main drawbacks:
\begin{itemize}
\item scalar gauge functions are required to be non-negative, hence the choice of $\sfG_{\mathrm{scal}}=h_{0}^{2}+K \sfd^{2}$ as a gauge would be possible only in the case $K\geq 0$,  failing the purpose of including all three-dimensional structures on Lie groups in a unified framework;
\item   our guiding principle is that the gauge function should be  a reference function with respect to which curvature bounds can be quantified.  Gauge functions \emph{should not} contain the numerical parameters that quantify the extent of curvature bounds. The three-dimensional examples in this section nicely illustrate this idea: the only non-zero curvature is $\mathfrak{Ric}_\lambda^b\geq h_0^2 + K \sfd^2$, the non-negative \emph{functions} $|h_0|$ and $\sfd$ are the gauge to measure the extent of the curvature bound, which is then \emph{quantified} by the parameter $K\in \R$.
\end{itemize}
\end{remark}


\appendix

\section{Sub-Riemannian geometry}\label{a:SR}

In this appendix we collect basic facts and notations in sub-Riemannian geometry used in this paper. For a comprehensive introduction, we refer to \cite{nostrolibro,riffordbook,montgomerybook}.

\begin{itemize}
\item A sub-Rieman\-nian structure on a smooth, connected $n$-dimensional manifold $M$, where $n\geq 2$, is defined by a set of $m\geq 2$ global smooth vector fields $X_{1},\ldots,X_{L}$, called a \emph{generating family}.
\item We assume the \emph{bracket-generating} condition, i.e., the vector fields $X_{1},\ldots,X_{L}$ and their iterated Lie brackets at $x$ generate the tangent space $T_x M$, for all $x\in M$.
\item The (generalized) \emph{distribution} is the disjoint union $\distr = \bigsqcup_{x\in M} \distr_x$, where 
\begin{equation}
\distr_{x}:=\mathrm{span}\{X_{1}(x),\ldots,X_{L}(x)\}\subseteq T_{x}M,\qquad \forall\, x\in M.
\end{equation}
Notice that $\distr$ is a vector bundle if and only if $\dim\distr_x$ does not depend on $x$.

\item The generating family induces an inner product $g_{x}$ on $\distr_{x}$ given by:
\begin{equation}
g_{x}(v,v):=\inf\left\{\sum_{i=1}^{L}u_{i}^{2}\mid v=\sum_{i=1}^{L}u_{i}X_{i}(x)\right\},\qquad \forall\, v\in \distr_x.
\end{equation}

\item For $i\geq 1$, we define the \emph{iterated distributions} $\distr^i=\bigsqcup_{x\in M}\distr^i_x$, where:
\[
\distr^{i}_{x}:=\mathrm{span}\{[X_{i_{1}},[\ldots,[X_{i_{j-1}},X_{i_{j}}]]\mid i_{k}\in \{1,\ldots,L\},\; j\leq i\}.
\]
The \emph{step} of the distribution at $x$ is the minimal $s=s(x)$ such that $\distr^{s}_{x}=T_{x}M$.

\item A \emph{horizontal curve} is an absolutely continuous (in charts) map $\gamma : [0,1] \to M$ such that there exists $u\in L^{2}([0,1],
\R^{L})$, called \emph{control}, satisfying
\begin{equation}\label{eq:admissible}
\dot\gamma(t) =  \sum_{i=1}^L u_i(t) X_i(\gamma(t)), \qquad \mathrm{a.e. }\; t \in [0,1].
\end{equation}
The class of horizontal curves depends on the family $\mathscr{F}=\{X_1,\dots,X_L\}$ only through the $C^{\infty}(M)$-module of vector fields generated by $\mathscr{F}$.
\item We define the \emph{length} of a horizontal curve $\gamma:[0,1]\to M$ as follows:
\begin{equation}
\ell(\gamma) := \int_0^1 \sqrt{g(\dot\gamma(t),\dot\gamma(t))}\di t.
\end{equation}
The length $\ell$ is invariant by suitable reparametrizations. Every horizontal curve is the reparametrization of a suitable constant-speed one.
\item The \emph{sub-Rieman\-nian (or Carnot-Carathéodory) distance} is defined by:
\begin{equation}\label{eq:infimo}
\sfd(x,y) = \inf\{\ell(\gamma)\mid \gamma(0) = x,\, \gamma(1) = y,\, \gamma \text{ horizontal} \}.
\end{equation}
The bracket-generating condition implies that $\sfd$ is finite and continuous. If $(M,\sfd)$ is complete as metric space, then for any $x,y \in M$ the infimum in \eqref{eq:infimo} is attained.

\item
On the space of horizontal curves defined on $[0,1]$ and with fixed endpoints, the minimizers of $\ell$, parametrized with constant speed, coincide with the minimizers of the \emph{energy functional} 
\begin{equation}
J(\gamma) := \frac{1}{2}\int_0^1 g(\dot\gamma(t),\dot\gamma(t))\, \di t.
\end{equation}

\item 
The \emph{end-point map} associated with the generating family $\mathscr{F}=\{X_{1},\ldots,X_{L}\}$ and with base point $x$ is the Fréchet-smooth map $\End^{\mathscr{F}}_{x}: \mathcal{U} \to M$, which sends $u$ to $\gamma_u(1)$, where $\gamma_u$ is the solution of
\begin{equation}\label{eq:endpoints}
\dot\gamma_u(t) = \sum_{i=1}^L u_i(t) X_i(\gamma_u(t)), \qquad \gamma_u(0) = x,
\end{equation}
for every $u\in \mathcal{U} \subset L^2([0,1],\R^L)$ for which the solution $\gamma_u$ is defined on $[0,1]$.

\item \emph{Sub-Riemannian geodesics} are admissible trajectories associated with \emph{minimizing controls}, namely the ones that solve the constrained minimum problem
\begin{equation} \label{eq:Ju}
\min\{J(u) \mid  u \in \mathcal{U},\quad \End^{\mathscr{F}}_x(u) = y\},\qquad x,y\in M.
\end{equation}
Sub-Riemannian geodesics are precisely those curves $\gamma:[0,1]\to M$ such that $\sfd(\gamma_t,\gamma_s)=|t-s|\sfd(\gamma_0,\gamma_1)$.

\item If $u$ is a minimizing control with $\End^{\mathscr{F}}_x(u)=y$, then there exists a non-trivial pair $(\lambda_1,\nu) \in T_{y}^*M \times \{0,1\}$, called Lagrange multiplier, such that
\begin{equation}\label{eq:multipliers}
\lambda_1 \circ D_u \End^{\mathscr{F}}_x(v)  = \nu  (u,v)_{L^2}, \qquad \forall v\in T_u\mathcal{U} \simeq L^2([0,1],\R^L),
\end{equation}
where $\circ$ denotes the composition of linear maps, $D$ the (Fr\'echet) differential. Non-trivial means that $(\lambda_1,\nu)\neq (0,0)$.

\item The multiplier $(\lambda_1,\nu)$ and the associated curve $\gamma_u$ are called \emph{normal} if $\nu = 1$ and \emph{abnormal} if $\nu = 0$. A minimizing control $u$ may admit different multipliers so that $\gamma_{u}$ might be both normal \emph{and} abnormal. In particular we observe that $\gamma_u$ is abnormal if and only if $u$ is a critical point of $\End^{\mathscr{F}}_x$.

\item If $a:T^{*}M\to \R$ is a smooth function,  we denote by $\vec{a}$ the corresponding Hamiltonian vector field, i.e., the vector field on $T^{*}M$ satisfying $\sigma(\cdot, \vec a)=da$, where $\sigma$ is the canonical symplectic form of $T^*M$.
\item If $\gamma_u:[0,1]\to M$ is a normal geodesic, with Lagrange multiplier $\lambda_1$, then it admits a lift $\lambda :[0,1]\to T^*M$ satisfying the differential equation 
\begin{equation}
\dot{\lambda}(t) = \vec{H}(\lambda(t)),\qquad \lambda(1)=\lambda_{1},
\end{equation}
where $H : T^*M \to \R$ is the \emph{sub-Riemannian Hamiltonian}:
\begin{equation}
H(\lambda) := \frac{1}{2}\sum_{i=1}^{L} \langle \lambda,X_i\rangle^{2}, \qquad \forall\,\lambda \in T^*M,
\end{equation}
and $\vec H$ denotes the corresponding Hamiltonian vector field.

\item The \emph{exponential map} at $x \in M$ is the map $\exp_x : T_x^*M \to M$, which assigns to $\lambda_{0}\in T_x^*M$ the final point $\pi(\lambda(1))$  of the corresponding solution of 
\[
\dot{\lambda}(t) = \vec{H}(\lambda(t)),\qquad \lambda(0)=\lambda_{0}.
\]
The covector $\lambda_{0}$ is called \emph{initial covector} of the trajectory. The curve $\gamma_u(t) = \pi(\lambda(t))$, $t\in[0,1]$, has control $u_i(t)=\langle \lambda(t),X_i(\gamma_u(t))\rangle$ for $i=1,\dots,L$, and satisfies the normal Lagrange multiplier rule with Lagrange multiplier $\lambda_{1}=e^{\vec H}(\lambda_{0})$.

\item Given a normal geodesic $\gamma(t) = \exp_x(t\lambda_{0})$ with initial covector $\lambda_0 \in T_x^*M$
we say that $y=\exp_x(\bar{t}\lambda)$ is a \emph{conjugate point} to $x$ along $\gamma$ if $\bar{t}\lambda$ is a critical point for $\exp_x$. Also we say that $\gamma(s)$ and $\gamma(t)$ are \emph{conjugate} if $\gamma(t)$ is conjugate to $\gamma(s)$ along $\gamma|_{[s,t]}$.

\item A normal geodesic $\gamma:[0,1]\to M$ \emph{contains no non-trivial abnormal segments} if for every $s,s'\in [0,1]$ with $0<|s-s'|<1$, the restriction $\gamma|_{[s,s']}$ is not abnormal.
\item If a geodesic $\gamma: [0,1] \to M$ contains no non-trivial abnormal segments, then $\gamma(s)$ is not conjugate to $\gamma(s')$ for every $s,s'\in [0,1]$ with $0<|s-s'|<1$.

\item A sub-Riemannian structure is \emph{ideal} if the  corresponding metric  $\sfd$ is complete and there exist no non-trivial abnormal geodesics.

\item We say that $y\in M$ is a \emph{smooth point}, with respect to $x\in M$, if there exists a unique geodesic joining $x$ with $y$, which is not abnormal, and with non-conjugate endpoints.
\item The \emph{cut locus} $\Cut(x)$ is the complement of the set of smooth points with respect to $x$. The \emph{global cut locus} of $M$ is 
\begin{equation}
\Cut(M) := \{(x,y) \in M \times M \mid y \in \Cut(x)\}.
\end{equation}
The set of smooth points is open and dense in $M$, and the squared sub-Riemannian distance is smooth on $M\times M \setminus \Cut(M)$ \cite{agrasmoothness,RT-MorseSard}.

\item The \emph{abnormal set} $\mathrm{Abn}(x)$ is the set of points $y$ such that there exists an abnormal geodesic joining $x$ and $y$. It holds $\mathrm{Abn}(x)\subseteq \mathrm{Cut}(x)$. 
\end{itemize}

\section{Canonical curvature}\label{a:canonicalframe}

This is a self-contained account of the results in \cite{ZeLi} concerning the construction of the canonical frame and canonical curvature, adapted to our settings. This approach to the construction of curvature-type invariants and its relation with the geometry of curves in the Lagrange Grassmannian has its roots in the seminal papers \cite{AG1,AG2}. See also \cite[Appendix]{nostrolibro}. We will use basic concepts on the geometry of the Lagrange Grassmannian, for which we refer to \cite[Ch.\@ 14]{nostrolibro}.

\subsection{Curves in the Lagrange Grassmannian}\label{sec:curvesLag}
With any curve $\Lambda(\cdot)$ in a Lagrange Grassmannian $G_k(W)$ of $n$-dimensional Lagrangian subspaces of a $2n$-dimensional symplectic space $(W,\sigma)$ one can associate a curve of flags of subspaces in $W$. Denote:
\begin{equation}
\Lambda^{(i)}(t):= \spn\left\{\frac{\di^j}{\di t^j}\xi(t) \mid \xi(\cdot) \text{ is a smooth curve with $\xi(\tau)\in \Lambda(\tau)$ for all $\tau$},\; 0\leq j\leq i \right\}.
\end{equation}
Recall that the tangent space $T_\Lambda G_k(W)$ to a point $\Lambda \in G_k(W)$ is identified with the space $Q(\Lambda)$ of quadratic forms on $\Lambda$ via the symplectic form. Explicitly, this is done by identifying the tangent vector $\dot{\Lambda}$ at a point $\Lambda$ with the quadratic form $\Lambda \ni \xi \mapsto \sigma(\xi,\dot\xi_0)$, where $\tau\mapsto \xi_\tau$ is a smooth curve in $W$ such that $\xi_\tau\in \Lambda(\tau)$ for all $\tau$ and $\xi_0 = \xi$. The rank of $\Lambda(\cdot)$ at $t$ is the rank of $\dot{\Lambda}(t)$ as a quadratic form.

We define then a flag of families of subspaces
\begin{equation}
\Lambda(\tau) =:\Lambda^{(0)}(\tau) \subseteq \Lambda^{(1)}(\tau)\subseteq \dots .
\end{equation}
We make the following assumptions on the curve $\Lambda(\cdot)$:
\begin{enumerate}[(i)]
\item \label{equiregular} it is \emph{equiregular}, that is $k_i(t):=\dim\Lambda^{(i)}(t)$ does not depend on $t$, for all $i\geq 0$;
\item \label{ample} it is \emph{ample}, that is there exists $s\in \N$ such that $\Lambda^{s}(t) = W$ for all $t$;
\item \label{monotone} it is \emph{monotonically non-decreasing}: $\dot{\Lambda}(t) \leq 0$ for all $t$;
\end{enumerate}
As a consequence of \ref{equiregular}, the sequence of numbers
\begin{equation}
d_i:= \dim\Lambda^{(i+1)}-\dim\Lambda^{(i)}, \qquad i \geq 0.
\end{equation}
is non-increasing.
We assign a Young diagram to the curve $\Lambda(\cdot)$ as follows: the number of boxes in the $i$-th column of the diagram is equal to $\dim\Lambda^{(i)}-\dim\Lambda^{(i-1)}$. In particular the number of boxes in the first column coincides with the rank of the curve: $d_1 = \rank\dot\Lambda$. By \ref{ample}, the Young diagram contains $n=\dim(W)$ boxes. See Figure \ref{fig:Yd0}.

\begin{figure}
\centering
\includegraphics[scale=1]{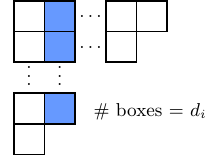}
\caption{Young diagram.}\label{fig:Yd0}
\end{figure}

\subsubsection{Reduced Young diagram}

Heuristically, a row of length $\ell$ in the Young diagram corresponds to the existence of a curve $\xi(\cdot)\in\Lambda(\cdot)$ such that $\dot{\xi}(\cdot),\dots,\xi^{(\ell)}(\cdot)$ are all independent as elements of $W$. It is reasonable then to identify subspaces that have the same behaviour with respect to these derivations. This motivates the following procedure.

For any given column of the Young diagram, we merge in a single box all boxes located in the rows with the same length.  This procedure yields a \emph{reduced Young diagram} from the original one, whose boxes are called \emph{superboxes}, and whose rows are called \emph{levels}. All superboxes of a given level have the same \emph{size} $r$, that is the number of boxes merged together in the superbox, see Figure \ref{fig:levels}.
\begin{figure}[!ht]
\centering
\includegraphics[scale=1]{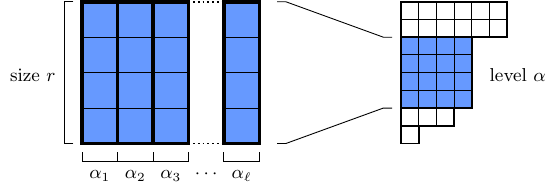}
\caption{Superboxes $\alpha_1,\dots,\alpha_\ell$ of a level $\alpha$.}\label{fig:levels}
\end{figure}

By construction, the reduced Young diagram contains $d\geq 1$ levels, and the sequence of lengths of the levels is strictly decreasing:
\begin{equation}
\ell_1 > \ell_2 > \dots > \ell_d.
\end{equation}
We call $Y$ the reduced Young diagram, which formally is identified with the set of its superboxes. Generic superboxes are denoted with the letters $a$, $b$, $\dots$. We are going to define curvature-type invariants of $\Lambda(\cdot)$ in terms of matrix-valued mappings. For this reason we need to introduce some terminology. A mapping
\begin{equation}
R: Y \times Y \to \mathrm{Mat},
\end{equation}
where $\mathrm{Mat}$ denotes the set of all matrices, is called:
\begin{itemize}
\item \emph{compatible with the Young diagram} $Y$ if to any pair $(a,b)$ of superboxes of sizes $r_{a}$ and $r_{b}$, the matrix $R(a,b)$ has size $r_a\times r_b$;
\item \emph{symmetric} if $R(b,a) = R(a,b)^*$ for any pair $(a,b)$ of superboxes.
\end{itemize}

\subsubsection{Normal mappings}

Curvature mappings will be required to satisfy algebraic conditions, formulated in terms of vanishing of some blocks $R(a,b)$, for suitable superboxes $a,b\in Y$. We refer to \cite[Def.\@ 1, Def.\@ 2]{ZeLi} for the general formulation, which we describe here only for the cases in which $Y$ has no more than two columns (and hence at most two levels). In this case, a symmetric and compatible mapping $R: Y\times Y \to \mathrm{Mat}$ is called \emph{normal} if:
\begin{itemize}
\item $R(a,b)$ is skew-symmetric, that is $R(a,b) = -R(a,b)^*$, for all pairs of distinct superboxes $a,b$ that belong to the first level of $Y$.
\end{itemize}
Of course this condition is vacuous if $Y$ has only one column (and hence only one superbox, as it happens for Jacobi curves of a Riemannian or Finsler structures), so that a symmetric compatible mapping in this latter case is just any $n\times n$ symmetric matrix.

\subsubsection{Normal moving frames}

A normal moving frame will be a smooth one-parameter family of bases of $W$, attached to a given curve $\Lambda(\cdot)$. It will be convenient to label the elements of a moving frame according to the superboxes of a reduced Young diagram $Y$. In particular, if $a$ is a superbox of $Y$ of size $r$, then
\begin{equation}
E_a(t) = (E_{a_1}(t),\dots,E_{a_r}(t)), \qquad F_a(t) = (F_{a_1}(t),\dots,F_{a_r}(t)),
\end{equation}
denote two tuples of smooth families of vector fields in $W$, which are independent at all times. The collection $\{E_a(t),F_a(t)\}_{a \in Y}$ will then be a basis for $W$ for all times. It is natural to require that the frame is Darboux, that is
\begin{equation}\label{eq:Darboux}
\sigma(E_a,E_b) = \sigma(F_a,F_b) = \sigma(E_a,F_b)-\delta_{ab} = 0,\qquad \forall\, a,b\in Y,
\end{equation}
where for tuples $V,V'$ of vectors in $W$ we denote by $\sigma(V,V')$ the matrix $\sigma(V_i,V'_j)$ and $\delta$ is a Kronecker delta, that is $\delta_{ab} = \mathbb{0}$ if $a\neq b$ and $\delta_{ab} = \mathbb{1}$ if $a=b$ (here, $\mathbb{0}$ and $\mathbb{1}$ are $r_a\times r_b$ matrices).

\begin{definition}[Normal moving frames]\label{def:normalframe}
The moving Darboux frame $\{E_a(t),F_a(t)\}_{a\in Y}$ is called a \emph{normal moving frame} of a non-decreasing curve $\Lambda(t)$ with reduced Young diagram $Y$ if for all $t$:
\begin{equation}\label{eq:attached}
\Lambda(t) = \spn\{E_a(t)\}_{a\in Y},
\end{equation}
and there exists a one-parameter family of symmetric and compatible normal mappings $R_t: Y\times Y\to \mathrm{Mat}$ such that for any superbox $a\in Y$
\begin{align}
\dot{E}_a(t) & = E_{l(a)}(t),	& \qquad a\notin\mathrm{first}(Y),\\
\dot{E}_a(t) & = -F_a(t), 			& \qquad a\in \mathrm{first}(Y),\\
\dot{F}_a(t) & = \sum_{b\in Y} R_t(a,b) \cdot E_b(t) - F_{r(a)}(t), 	& \quad a \notin\mathrm{last}(Y), \\
\dot{F}_a(t) & = \sum_{b\in Y}R_t(a,b) \cdot F_b(t), & \qquad a\in\mathrm{last}(Y),
\end{align}
where $\mathrm{first}(Y)$ and $\mathrm{last}(Y)$ denote the set of superboxes in the first (resp.\@ last) column of each level of $Y$, while $l : Y\setminus \mathrm{first}(Y) \to Y$ and $r:Y\setminus \mathrm{last}(Y)\to Y$ denote the left and right shift of superboxes on $Y$.
\end{definition}

\begin{theorem}[Existence and uniqueness of normal moving frames \cite{ZeLi}]\label{t:ZeLi}
For any ample, equiregular, monotonically non-increasing curve $\Lambda(\cdot)$ in the Lagrange Grassmannian, with reduced Young diagram $Y$, there exists a normal moving frame $\{E_a(t),F_a(t)\}_{a\in Y}$.

Furthermore $\{\tilde{E}_a(t),\tilde{F}_a(t)\}_{a\in Y}$ is another normal moving frame if and only if for any level $1\leq i\leq d$ of $Y$ with size $r_i$ there exists a constant orthogonal $r_i\times r_i$ matrix $U_i$ such that for all superboxes $a$ in the $i$-th level it holds
\begin{equation}
\tilde{E}_a(t) = U_i \cdot E_a(t),\qquad \tilde{F}_a = U_i \cdot F_a(t), \qquad \forall\, t.
\end{equation}
\end{theorem}

\subsubsection{Canonical splittings and curvatures}
Theorem \ref{t:ZeLi} yields a rich structure attached to $\Lambda(\cdot)$. Firstly, the family $\Lambda(t)$ is equipped with a \emph{canonical Euclidean structure}, given by the scalar product that makes the elements of the tuples $\{E_a(t)\}_{a\in Y}$ orthonormal. Secondly, we have the splitting
\begin{equation}
\Lambda(t) = \bigoplus_{a\in Y} V_a(t), \qquad \text{where} \qquad V_a(t) := \spn\{E_a(t)\}.
\end{equation}
Similarly one can define a Lagrangian complement $\Lambda^{\mathrm{trans}}(t)$ such that $W=\Lambda(t) \oplus \Lambda^{\mathrm{trans}}(t)$ for all times, given by
\begin{equation}
\Lambda^{\mathrm{trans}}(t) := \bigoplus_{a\in Y} H_a(t), \qquad \text{where} \qquad H_a:=\spn\{F_a(t)\}.
\end{equation}

The matrix-valued maps $R_t(a,b)$ appearing in Theorem \ref{t:ZeLi} correspond to well-defined operators. In fact, if $R_t, \tilde{R}_t: Y\times Y \to \mathrm{Map}$ are two compatible, symmetric normal mappings associated with two canonical moving frames $\{E_a(t),F_a(t)\}_{a\in Y}$ and $\{\tilde{E}_a(t),\tilde{F}_a(t)\}_{a\in Y}$, then it follows from Theorem \ref{t:ZeLi} that:
\begin{equation}
\tilde{R}_t(a,b) = U_i^*  R_t(a,b)  U_j, \qquad \forall\, t,
\end{equation}
where $a$ and $b$ belong to the $i$-th and $j$-th levels of $Y$, respectively, and $U_i,U_j$ are orthogonal constant matrices appearing in Theorem \ref{t:ZeLi}. We can then give the following formal definition.

\begin{definition}[Canonical curvatures]\label{def:cancurv}
Let $\Lambda(\cdot)$ be an ample, equiregular, monotonically non-increasing curve in the Lagrange Grassmannian, with reduced Young diagram $Y$. For any $a,b\in Y$, we define the $(a,b)$\emph{-canonical curvature map}
\begin{equation}
\mathfrak{R}^{ab}(t) : V_a(t)\times V_b(t)\to \R,
\end{equation}
as the one-parameter family of linear maps whose representative matrix correspond to the matrix $R_t(a,b)$ with respect to the bases $E_a$ and $E_b$ of $V_a$ and $V_b$. The \emph{canonical curvature map} is the one-parameter family of symmetric linear maps $\mathfrak{R}(t): \Lambda(t)\times  \Lambda(t)\to \R$ such that its restriction on $V_a(t)\times V_b(t)$ coincides with $\mathfrak{R}^{ab}(t)$ for all $a,b\in Y$. Finally, for each $a\in Y$, the \emph{canonical Ricci curvature} is the one-parameter family obtained by taking the trace on superboxes of $\mathfrak{R}$ with respect to the canonical Euclidean structure:
\begin{equation}
\mathfrak{Ric}^a(t):=\tr(\mathfrak{R}^{aa}(t)),\qquad  \forall\, a\in Y,
\end{equation}
that is the trace of the matrix $R_t(a,a)$ for any choice of normal moving frame.
\end{definition}

\subsection{Jacobi curves}\label{sec:Jaccurve}

An important case to which the previous theory applies is that of Jacobi curves $J_\lambda(\cdot)$ associated with an extremal (i.e., an integral curve of a Hamiltonian flow) that we now introduce. Let $M$ be a smooth $n$-dimensional manifold, and $H:T^*M \to \R$ be a function, generating the Hamiltonian flow $e^{t\vec{H}}$. For simplicity we assume it to be defined for all times $t\in \R$. Letting $T^*M_{\neq 0} = \{ \lambda \in T^*M \mid H(\lambda) \neq 0\}$, we require the following:
\begin{enumerate}[label=(H\arabic*)]
\item\label{h1} $H$ is smooth on $T^*M_{\neq 0}$;
\item\label{h2} the Hessian of the restriction $H_x := H|_{T_x^*M}$ at any point $\lambda \in T^*_xM_{\neq 0}$, denoted by $\di^2_\lambda H_x : T_x^*M \to \R$, is a non-negative and possibly degenerate quadratic form.
\end{enumerate}
These assumptions are verified for the Hamiltonian function of general (sub-)Riemannian and (sub-)Finsler structures, or for the Hamiltonian of a general LQ optimal control problem on $M =\mathbb{R}^n$ (see Section \ref{sec:LQ}). Notice that in the (sub-)Riemannian or in the LQ case, since $H_x$ is a quadratic form, it holds $\di^2_\lambda H_x = 2H_x$.

\begin{definition}[Jacobi curve]\label{def:Jlambdat}
Let $\lambda \in T^*M_{\neq 0}$ be the initial covector of $\lambda_t = e^{t\vec{H}}(\lambda)$. The \emph{Jacobi curve} at $\lambda$ is the following curve of Lagrangian subspaces of $T_\lambda(T^*M)$:
\begin{equation}
J_\lambda(t) := e^{-t\vec{H}}_* T_{\lambda_t}(T^*_{\gamma_t}M),
\end{equation}
where $\gamma_t = \pi(\lambda_t)$.
\end{definition}

We have the following properties for all $t,s$ where the statements make sense:
\begin{enumerate}[label=(P\arabic*)]
\item\label{P1} $J_\lambda(t+s) = e^{-t\vec{H}}_* J_{\lambda_t}(s)$;
\item\label{P2} $\dot{J}_{\lambda}(0) = -\di^2_\lambda H_x$ as quadratic forms on $J_\lambda(0) = T_\lambda(T^*_xM) \simeq T^*_xM$, where $x=\pi(\lambda)$.
\end{enumerate}
See \cite[Prop.\@ 15.2]{nostrolibro} for a proof in the (sub-)Riemannian case.

It follows from \ref{P2} that $\dot{J}_\lambda(t)\leq 0$ for all times, so that $J_\lambda(t)$ is monotonically non-increasing curve. When the curve is also ample and equiregular, we can apply Theorem \ref{t:ZeLi}. The corresponding moving frames and curvature operators can be defined in three equivalent ways. For clarity, we describe them all here.

\subsubsection{Curves in the Lagrange Grassmannian point of view} \label{sec:can-LagGrass}

If $J_\lambda(\cdot)$ is an ample and equiregular Jacobi curve, with Young diagram $Y$, Theorem \ref{t:ZeLi} yields the existence of:
\begin{itemize}
\item a normal moving frame $\{E_a(t),F_a(t)\}_{a\in Y}$ for $J_\lambda(t)$;
\item the subspaces $V_a(t) = \spn\{E_a(t)\}$ and $H_a(t) = \spn\{F_a(t)\}$ of $T_{\lambda}(T^*M)$;
\item the $(a,b)$-curvature maps $\mathfrak{R}^{ab}(t): V_a(t)\times V_b(t)\to \R$, represented by the curvature matrices $R_t(a,b)$;
\item the curvature map $\mathfrak{R}(t): J_{\lambda}(t)\to J_{\lambda}(t)$, extending all the $(a,b)$-curvature maps;
\item the Ricci curvatures $\mathfrak{Ric}^a(t)= \tr \mathfrak{R}^{aa}(t)$.
\end{itemize}

\subsubsection{Canonical frame along the extremal point of view}\label{sec:can-along-extremal}
Notice that $e^{t\vec{H}}_*$ maps $T_\lambda(T^*M)$ diffeomorphically onto $T_{\lambda_t}(T^*M)$, sending in particular $J_{\lambda}(t)$ to the vertical space $V_{\lambda_t} = \ker \di_{\lambda_t}\pi$. Therefore, a normal moving frame can be identified with a family of Darboux bases along $\lambda_t$, by setting
\begin{equation}\label{eq:canframedef}
E_a|_{\lambda_t}:= e^{t\vec{H}}_* E_a(t), \qquad F_a|_{\lambda_t}:= e^{t\vec{H}}_* F_a(t), \qquad \forall\, a \in Y.
\end{equation}
Definition \ref{def:normalframe} for a normal moving frame can be reformulated in an equivalent way in terms of the frame \eqref{eq:canframedef}. Firstly, equation \eqref{eq:attached} becomes the condition
\begin{equation}
V_{\lambda_t}:=\ker \di_{\lambda_t} \pi = \spn \{E_a|_{\lambda_t}\}_{a\in Y}.
\end{equation}
Secondly, the structural equations in Definition \ref{def:normalframe} are expressed in terms of \eqref{eq:canframedef} replacing the time-derivative with the derivative in the direction of $\vec{H}$.

With this approach, to any extremal $\lambda_t = e^{t\vec{H}}(\lambda)$ associated with an ample and equiregular Jacobi curve with Young diagram $Y$, Theorem \ref{t:ZeLi} yields the existence of:
\begin{itemize}
\item a normal moving frame $\{E_a|_{\lambda_t},F_a|_{\lambda_t}\}_{a\in Y}$ along $T_{\lambda_t}(T^*M)$, also called \emph{canonical frame};
\item the subspaces $V_{a|\lambda_t} = \spn\{E_a|_{\lambda_t}\}$ and $H_{a|\lambda_t} = \spn\{F_a|_{\lambda_t}\}$ of $T_{\lambda_t}(T^*M)$;
\item the $(a,b)$-canonical curvature maps $\mathfrak{R}_{\lambda_t}^{ab} : V_{a|\lambda_t}\times V_{b|\lambda_t}\to \R$, represented by the curvature matrices $R_t(a,b)$;
\item the canonical curvature map $\mathfrak{R}_{\lambda_t} : V_{\lambda_t} \times V_{\lambda_t} \to \R$, extending all the $(a,b)$-canonical curvature maps;
\item the Ricci curvatures $\mathfrak{Ric}_{\lambda_t}^{a}=\tr\mathfrak{R}_{\lambda_t}^{aa}$.
\end{itemize}

This notation is consistent, in the sense that if $\lambda_s$, for some $s\in \R$, is used as initial covector to produce the extremal $(\lambda_s)_t = e^{t\vec{H}}(\lambda_s)$, then it holds
\begin{equation}
\mathfrak{R}_{(\lambda_s)_t} = \mathfrak{R}_{\lambda_{t+s}}, \qquad \forall\, t,s\in \R.
\end{equation}
This property is a consequence of \ref{P1}. It follows that $\mathfrak{R}_{\lambda_t}$ is the restriction to $\lambda_t$ of a well-defined operator-valued map $\lambda \mapsto \mathfrak{R}_{\lambda}$, also called canonical curvature map, defined at all those points $\lambda$ associated with an ample and equiregular Jacobi curve.

\subsubsection{Smooth families of operators point of view}\label{sec:can-operators}
Assuming that there exists a neighborhood $U\subseteq T^*M$ of covectors with ample and equiregular Jacobi curve with the same reduced Young diagram $Y$, the maps $\lambda \mapsto \mathfrak{R}_{\lambda}^{ab}$ are smooth on $U$\footnote{This is a consequence of the proofs \cite{ZeLi}, as there one only uses inversion of matrices, solutions of ODES, all of them depending smoothly on the initial covector $\lambda$.}. This is the point of view adopted in Section \ref{sec:fat}, in order to study the singularity of $\mathfrak{R}_{\lambda}$ as $\lambda \to \partial U$ for fat sub-Riemannian structures (in that case $U= \{ \lambda \in T^*M \mid H(\lambda)=0\}$).

\subsection{Geodesic volume derivative} \label{sec:gvd}

Let $(\distr,g)$ be a sub-Riemannian structure on a smooth manifold $M$, equipped with a smooth measure $\mm$. Let $U\subseteq T^*M$ be an open subset where any $\lambda \in U$ is associated with an ample and equiregular Jacobi curve, so that the canonical frame is well-defined.

\begin{definition}[Geodesic volume derivative] The \emph{geodesic volume derivative} with respect to $\mm$ is the smooth function $\rho_{\mm}:U \to \R$, defined by
\begin{equation}
\rho_{\mm,\lambda} := \left.\frac{d}{dt}\right|_{t=0}\log\mm\left(\bigwedge_{a\in Y} \pi_* F_a|_{\lambda_t}\right), \qquad \forall\, \lambda \in U,
\end{equation}
for any choice of canonical frame along $\lambda_t=e^{t\vec{H}}(\lambda)$.
\end{definition}
The definition is well-posed, that is it does not depend on the choice of canonical frame, thanks to the uniqueness part of Theorem \ref{t:ZeLi}.

\subsection{The Riemannian case}
In the Riemannian case the Jacobi curve at any $\lambda \in T^*M$ is ample and equiregular with the same reduced Young diagram made of one superbox (so we can omit it from the notation). It is proven in \cite[Lemma 15]{BR-comparison} that:
\begin{equation}
\mathfrak{R}_{\lambda} = \mathrm{R}(\cdot,\lambda,\lambda,\cdot), \qquad \mathfrak{Ric}_{\lambda} = \mathrm{Ric}(\lambda,\lambda),
\end{equation}
where we identify $T_\lambda(T^*_xM) = T^*_x M \simeq T_x M$ via the Riemannian structure, while $\mathrm{R}$ and $\mathrm{Ric}$ denote the Riemann and Ricci curvature tensors, respectively.

Concerning the geodesic volume derivative, one can also prove that, in the Riemannian case, the canonical frame projects onto an orthonormal frame. If $M$ is equipped with a smooth reference volume $\mm = e^{-V}\mathrm{vol}$, where $V:M \to \R$ is a smooth function and $\mathrm{vol}$ is the Riemannian density, then the geodesic volume derivative reduces to
\begin{equation}
\rho_{\mm,\lambda} := \left.\frac{d}{dt}\right|_{t=0}\log e^{-V(\gamma(t))}=-g(\nabla_{\gamma_0} V, \dot \gamma_0),
\end{equation}
where $\gamma_t = \pi(\lambda_t)$. 		

\bibliographystyle{abbrv}
\bibliography{biblio-unification}

\begin{thebibliography}{10}

\bibitem{AB12}
A.~Agrachev and D.~Barilari.
\newblock Sub-{R}iemannian structures on 3{D} {L}ie groups.
\newblock {\em J. Dyn. Control Syst.}, 18(1):21--44, 2012.

\bibitem{nostrolibro}
A.~Agrachev, D.~Barilari, and U.~Boscain.
\newblock {\em A comprehensive introduction to sub-{R}iemannian geometry},
  volume 181 of {\em Cambridge Studies in Advanced Mathematics}.
\newblock Cambridge University Press, Cambridge, 2020.

\bibitem{ABR-curvature}
A.~Agrachev, D.~Barilari, and L.~Rizzi.
\newblock Curvature: a variational approach.
\newblock {\em Mem. Amer. Math. Soc.}, 256(1225):v+142, 2018.

\bibitem{AAPL-transport}
A.~Agrachev and P.~Lee.
\newblock Optimal transportation under nonholonomic constraints.
\newblock {\em Trans. Amer. Math. Soc.}, 361(11):6019--6047, 2009.

\bibitem{AAPL-Ricci}
A.~Agrachev and P.~W.~Y. Lee.
\newblock Generalized {R}icci curvature bounds for three dimensional contact
  subriemannian manifolds.
\newblock {\em Math. Ann.}, 360(1-2):209--253, 2014.

\bibitem{ARS-LQ}
A.~Agrachev, L.~Rizzi, and P.~Silveira.
\newblock On conjugate times of {LQ} optimal control problems.
\newblock {\em J. Dyn. Control Syst.}, 21(4):625--641, 2015.

\bibitem{Agrachevbook}
A.~Agrachev and Y.~L. Sachkov.
\newblock {\em Control theory from the geometric viewpoint}, volume~87 of {\em
  Encyclopaedia of Mathematical Sciences}.
\newblock Springer-Verlag, Berlin, 2004.
\newblock Control Theory and Optimization, II.

\bibitem{AG1}
A.~A. Agrachev.
\newblock Feedback-invariant optimal control theory and differential geometry.
  {II}. {J}acobi curves for singular extremals.
\newblock {\em J. Dynam. Control Systems}, 4(4):583--604, 1998.

\bibitem{agrasmoothness}
A.~A. Agrach\"{e}v.
\newblock Any sub-{R}iemannian metric has points of smoothness.
\newblock {\em Dokl. Akad. Nauk}, 424(3):295--298, 2009.

\bibitem{ABP19}
A.~A. Agrachev, D.~Barilari, and E.~Paoli.
\newblock Volume geodesic distortion and {R}icci curvature for {H}amiltonian
  dynamics.
\newblock {\em Ann. Inst. Fourier (Grenoble)}, 69(3):1187--1228, 2019.

\bibitem{AG2}
A.~A. Agrachev and R.~V. Gamkrelidze.
\newblock Feedback-invariant optimal control theory and differential geometry.
  {I}. {R}egular extremals.
\newblock {\em J. Dynam. Control Systems}, 3(3):343--389, 1997.

\bibitem{AS-minimalityVSsubanalyticity}
A.~A. Agrachev and A.~V. Sarychev.
\newblock Sub-{R}iemannian metrics: minimality of abnormal geodesics versus
  subanalyticity.
\newblock {\em ESAIM Control Optim. Calc. Var.}, 4:377--403, 1999.

\bibitem{ACM-Sobolev}
L.~Ambrosio, M.~Colombo, and S.~Di~Marino.
\newblock Sobolev spaces in metric measure spaces: reflexivity and lower
  semicontinuity of slope.
\newblock In {\em Variational methods for evolving objects}, volume~67 of {\em
  Adv. Stud. Pure Math.}, pages 1--58. Math. Soc. Japan, [Tokyo], 2015.

\bibitem{AMSmemo}
L.~Ambrosio, A.~Mondino, and G.~Savar\'{e}.
\newblock Nonlinear diffusion equations and curvature conditions in metric
  measure spaces.
\newblock {\em Mem. Amer. Math. Soc.}, 262(1270):v+121, 2019.

\bibitem{AmbStef}
L.~Ambrosio and G.~Stefani.
\newblock Heat and entropy flows in {C}arnot groups.
\newblock {\em Rev. Mat. Iberoam.}, 36(1):257--290, 2020.

\bibitem{AT-measure}
L.~Ambrosio and P.~Tilli.
\newblock {\em Topics on analysis in metric spaces}, volume~25 of {\em Oxford
  Lecture Series in Mathematics and its Applications}.
\newblock Oxford University Press, Oxford, 2004.

\bibitem{BS10}
K.~Bacher and K.-T. Sturm.
\newblock Localization and tensorization properties of the curvature-dimension
  condition for metric measure spaces.
\newblock {\em J. Funct. Anal.}, 259(1):28--56, 2010.

\bibitem{BR-realanalMCP}
Z.~Badreddine and L.~Rifford.
\newblock Measure contraction properties for two-step analytic sub-{R}iemannian
  structures and {L}ipschitz {C}arnot groups.
\newblock {\em Ann. Inst. Fourier (Grenoble)}, 70(6):2303--2330, 2020.

\bibitem{BKS}
Z.~M. Balogh, A.~Krist\'{a}ly, and K.~Sipos.
\newblock Geometric inequalities on {H}eisenberg groups.
\newblock {\em Calc. Var. Partial Differential Equations}, 57(2):Paper No. 61,
  41, 2018.

\bibitem{BKS2}
Z.~M. Balogh, A.~Krist\'{a}ly, and K.~Sipos.
\newblock Jacobian determinant inequality on corank 1 {C}arnot groups with
  applications.
\newblock {\em J. Funct. Anal.}, 277(12):108293, 36, 2019.

\bibitem{BR13}
D.~Barilari and L.~Rizzi.
\newblock A formula for {P}opp's volume in sub-{R}iemannian geometry.
\newblock {\em Anal. Geom. Metr. Spaces}, 1:42--57, 2013.

\bibitem{BR-comparison}
D.~Barilari and L.~Rizzi.
\newblock Comparison theorems for conjugate points in sub-{R}iemannian
  geometry.
\newblock {\em ESAIM Control Optim. Calc. Var.}, 22(2):439--472, 2016.

\bibitem{BR-Connection}
D.~Barilari and L.~Rizzi.
\newblock On {J}acobi fields and a canonical connection in sub-{R}iemannian
  geometry.
\newblock {\em Arch. Math. (Brno)}, 53(2):77--92, 2017.

\bibitem{BRInv}
D.~Barilari and L.~Rizzi.
\newblock Sub-{R}iemannian interpolation inequalities.
\newblock {\em Invent. Math.}, 215(3):977--1038, 2019.

\bibitem{BRMathAnn}
D.~Barilari and L.~Rizzi.
\newblock Bakry-\'{E}mery curvature and model spaces in sub-{R}iemannian
  geometry.
\newblock {\em Math. Ann.}, 377(1-2):435--482, 2020.

\bibitem{BaudoinBook}
F.~Baudoin.
\newblock Sub-{L}aplacians and hypoelliptic operators on totally geodesic
  {R}iemannian foliations.
\newblock In {\em Geometry, analysis and dynamics on sub-{R}iemannian
  manifolds. {V}ol. 1}, EMS Ser. Lect. Math., pages 259--321. Eur. Math. Soc.,
  Z\"{u}rich, 2016.

\bibitem{BG17}
F.~Baudoin and N.~Garofalo.
\newblock Curvature-dimension inequalities and {R}icci lower bounds for
  sub-{R}iemannian manifolds with transverse symmetries.
\newblock {\em J. Eur. Math. Soc. (JEMS)}, 19(1):151--219, 2017.

\bibitem{BGKT-Sasakian}
F.~Baudoin, E.~Grong, K.~Kuwada, and A.~Thalmaier.
\newblock Sub-{L}aplacian comparison theorems on totally geodesic {R}iemannian
  foliations.
\newblock {\em Calc. Var. Partial Differential Equations}, 58(4):Paper No. 130,
  38, 2019.

\bibitem{BGMR-Htype}
F.~Baudoin, E.~Grong, G.~Molino, and L.~Rizzi.
\newblock Comparison theorems on {H}-type sub-riemannian manifolds.
\newblock 2019.

\bibitem{BGMR-Htype2}
F.~Baudoin, E.~Grong, L.~Rizzi, and G.~Vega-Molino.
\newblock H-type foliations.
\newblock {\em Differential Geom. Appl.}, 85:Paper No. 101952, 2022.

\bibitem{Bellaiche}
A.~Bella\"{\i}che.
\newblock The tangent space in sub-{R}iemannian geometry.
\newblock In {\em Sub-{R}iemannian geometry}, volume 144 of {\em Progr. Math.},
  pages 1--78. Birkh\"{a}user, Basel, 1996.

\bibitem{Besse}
A.~L. Besse.
\newblock {\em Einstein manifolds}, volume~10 of {\em Ergebnisse der Mathematik
  und ihrer Grenzgebiete (3) [Results in Mathematics and Related Areas (3)]}.
\newblock Springer-Verlag, Berlin, 1987.

\bibitem{LarryBoa}
F.~Boarotto and A.~Lerario.
\newblock Homotopy properties of horizontal path spaces and a theorem of
  {S}erre in subriemannian geometry.
\newblock {\em Comm. Anal. Geom.}, 25(2):269--301, 2017.

\bibitem{Bogachev-Weak}
V.~I. Bogachev.
\newblock {\em Weak convergence of measures}, volume 234 of {\em Mathematical
  Surveys and Monographs}.
\newblock American Mathematical Society, Providence, RI, 2018.

\bibitem{CavaMil}
F.~Cavalletti and E.~Milman.
\newblock The globalization theorem for the curvature-dimension condition.
\newblock {\em Invent. Math.}, 226(1):1--137, 2021.

\bibitem{CavaMondCCM}
F.~Cavalletti and A.~Mondino.
\newblock Optimal maps in essentially non-branching spaces.
\newblock {\em Commun. Contemp. Math.}, 19(6):1750007, 27, 2017.

\bibitem{CavaMond}
F.~Cavalletti and A.~Mondino.
\newblock Sharp and rigid isoperimetric inequalities in metric-measure spaces
  with lower {R}icci curvature bounds.
\newblock {\em Invent. Math.}, 208(3):803--849, 2017.

\bibitem{CavaMondAPDE}
F.~Cavalletti and A.~Mondino.
\newblock New formulas for the {L}aplacian of distance functions and
  applications.
\newblock {\em Anal. PDE}, 13(7):2091--2147, 2020.

\bibitem{Chav06}
I.~Chavel.
\newblock {\em Riemannian geometry}, volume~98 of {\em Cambridge Studies in
  Advanced Mathematics}.
\newblock Cambridge University Press, Cambridge, second edition, 2006.
\newblock A modern introduction.

\bibitem{CheegerColdingJDG1}
J.~Cheeger and T.~H. Colding.
\newblock On the structure of spaces with {R}icci curvature bounded below. {I}.
\newblock {\em J. Differential Geom.}, 46(3):406--480, 1997.

\bibitem{CJT-Genericity}
Y.~Chitour, F.~Jean, and E.~Tr\'{e}lat.
\newblock Genericity results for singular curves.
\newblock {\em J. Differential Geom.}, 73(1):45--73, 2006.

\bibitem{CEMS}
D.~Cordero-Erausquin, R.~J. McCann, and M.~Schmuckenschl\"{a}ger.
\newblock A {R}iemannian interpolation inequality \`a la {B}orell, {B}rascamp
  and {L}ieb.
\newblock {\em Invent. Math.}, 146(2):219--257, 2001.

\bibitem{Coron}
J.-M. Coron.
\newblock {\em Control and nonlinearity}, volume 136 of {\em Mathematical
  Surveys and Monographs}.
\newblock American Mathematical Society, Providence, RI, 2007.

\bibitem{EKS}
M.~Erbar, K.~Kuwada, and K.-T. Sturm.
\newblock On the equivalence of the entropic curvature-dimension condition and
  {B}ochner's inequality on metric measure spaces.
\newblock {\em Invent. Math.}, 201(3):993--1071, 2015.

\bibitem{FG96}
E.~Falbel and C.~Gorodski.
\newblock Sub-{R}iemannian homogeneous spaces in dimensions {$3$} and {$4$}.
\newblock {\em Geom. Dedicata}, 62(3):227--252, 1996.

\bibitem{FR-mass}
A.~Figalli and L.~Rifford.
\newblock Mass transportation on sub-{R}iemannian manifolds.
\newblock {\em Geom. Funct. Anal.}, 20(1):124--159, 2010.

\bibitem{Fukaya}
K.~Fukaya.
\newblock Collapsing of {R}iemannian manifolds and eigenvalues of {L}aplace
  operator.
\newblock {\em Invent. Math.}, 87(3):517--547, 1987.

\bibitem{GKMS-quotients}
F.~Galaz-Garc\'{\i}a, M.~Kell, A.~Mondino, and G.~Sosa.
\newblock On quotients of spaces with {R}icci curvature bounded below.
\newblock {\em J. Funct. Anal.}, 275(6):1368--1446, 2018.

\bibitem{GMS}
N.~Gigli, A.~Mondino, and G.~Savar\'{e}.
\newblock Convergence of pointed non-compact metric measure spaces and
  stability of {R}icci curvature bounds and heat flows.
\newblock {\em Proc. Lond. Math. Soc. (3)}, 111(5):1071--1129, 2015.

\bibitem{Gr}
M.~Gromov.
\newblock {\em Metric structures for {R}iemannian and non-{R}iemannian spaces}.
\newblock Modern Birkh\"{a}user Classics. Birkh\"{a}user Boston, Inc., Boston,
  MA, english edition, 2007.

\bibitem{SobmetPoincare}
P.~Haj\l~asz and P.~Koskela.
\newblock Sobolev met {P}oincar\'{e}.
\newblock {\em Mem. Amer. Math. Soc.}, 145(688):x+101, 2000.

\bibitem{HughenPhD}
W.~K. Hughen.
\newblock {\em The sub-{R}iemannian geometry of three-manifolds}.
\newblock ProQuest LLC, Ann Arbor, MI, 1995.
\newblock Thesis (Ph.D.)--Duke University.

\bibitem{Jeanbook}
F.~Jean.
\newblock {\em Control of nonholonomic systems: from sub-{R}iemannian geometry
  to motion planning}.
\newblock SpringerBriefs in Mathematics. Springer, Cham, 2014.

\bibitem{NJ09}
N.~Juillet.
\newblock Geometric inequalities and generalized {R}icci bounds in the
  {H}eisenberg group.
\newblock {\em Int. Math. Res. Not. IMRN}, (13):2347--2373, 2009.

\bibitem{NJ21}
N.~Juillet.
\newblock Sub-{R}iemannian structures do not satisfy {R}iemannian
  {B}runn-{M}inkowski inequalities.
\newblock {\em Rev. Mat. Iberoam.}, 37(1):177--188, 2021.

\bibitem{Jurdjevicbook}
V.~Jurdjevic.
\newblock {\em Geometric control theory}, volume~52 of {\em Cambridge Studies
  in Advanced Mathematics}.
\newblock Cambridge University Press, Cambridge, 1997.

\bibitem{Kaplan}
A.~Kaplan.
\newblock Fundamental solutions for a class of hypoelliptic {PDE} generated by
  composition of quadratic forms.
\newblock {\em Trans. Amer. Math. Soc.}, 258(1):147--153, 1980.

\bibitem{Klartag}
B.~Klartag.
\newblock Needle decompositions in {R}iemannian geometry.
\newblock {\em Mem. Amer. Math. Soc.}, 249(1180):v+77, 2017.

\bibitem{KuwaeShyoiaTAMS208}
K.~Kuwae and T.~Shioya.
\newblock Variational convergence over metric spaces.
\newblock {\em Trans. Amer. Math. Soc.}, 360(1):35--75, 2008.

\bibitem{SardProp}
E.~Le~Donne, R.~Montgomery, A.~Ottazzi, P.~Pansu, and D.~Vittone.
\newblock Sard property for the endpoint map on some {C}arnot groups.
\newblock {\em Ann. Inst. H. Poincar\'{e} Anal. Non Lin\'{e}aire},
  33(6):1639--1666, 2016.

\bibitem{PL-senzamani}
P.~W.~Y. Lee.
\newblock On measure contraction property without {R}icci curvature lower
  bound.
\newblock {\em Potential Anal.}, 44(1):27--41, 2016.

\bibitem{LeeLiZel-Sasakian}
P.~W.~Y. Lee, C.~Li, and I.~Zelenko.
\newblock Ricci curvature type lower bounds for sub-{R}iemannian structures on
  {S}asakian manifolds.
\newblock {\em Discrete Contin. Dyn. Syst.}, 36(1):303--321, 2016.

\bibitem{ZeLi2}
C.~Li and I.~Zelenko.
\newblock Jacobi equations and comparison theorems for corank 1
  sub-{R}iemannian structures with symmetries.
\newblock {\em J. Geom. Phys.}, 61(4):781--807, 2011.

\bibitem{LVJFA}
J.~Lott and C.~Villani.
\newblock Weak curvature conditions and functional inequalities.
\newblock {\em J. Funct. Anal.}, 245(1):311--333, 2007.

\bibitem{lottvillani:metric}
J.~Lott and C.~Villani.
\newblock Ricci curvature for metric-measure spaces via optimal transport.
\newblock {\em Ann. of Math. (2)}, 169(3):903--991, 2009.

\bibitem{MagnaRossi22}
M.~Magnabosco and T.~Rossi.
\newblock Almost-{R}iemannian manifolds do not satisfy the curvature-dimension
  condition.
\newblock {\em Calc. Var. Partial Differential Equations}, 62(4):Paper No. 123,
  27, 2023.

\bibitem{McC97}
R.~J. McCann.
\newblock A convexity principle for interacting gases.
\newblock {\em Adv. Math.}, 128(1):153--179, 1997.

\bibitem{MR-branching}
T.~Mietton and L.~Rizzi.
\newblock Branching geodesics in sub-{R}iemannian geometry.
\newblock {\em Geom. Funct. Anal.}, 30(4):1139--1151, 2020.

\bibitem{MilmanSR}
E.~Milman.
\newblock The quasi curvature-dimension condition with applications to
  sub-{R}iemannian manifolds.
\newblock {\em Comm. Pure Appl. Math.}, 74(12):2628--2674, 2021.

\bibitem{Miranda}
M.~Miranda, Jr.
\newblock Functions of bounded variation on ``good'' metric spaces.
\newblock {\em J. Math. Pures Appl. (9)}, 82(8):975--1004, 2003.

\bibitem{montgomerybook}
R.~Montgomery.
\newblock {\em A tour of subriemannian geometries, their geodesics and
  applications}, volume~91 of {\em Mathematical Surveys and Monographs}.
\newblock American Mathematical Society, Providence, RI, 2002.

\bibitem{OhtaMCP}
S.-i. Ohta.
\newblock On the measure contraction property of metric measure spaces.
\newblock {\em Comment. Math. Helv.}, 82(4):805--828, 2007.

\bibitem{Ohta2009}
S.-i. Ohta.
\newblock Finsler interpolation inequalities.
\newblock {\em Calc. Var. Partial Differential Equations}, 36(2):211--249,
  2009.

\bibitem{OttoVillani}
F.~Otto and C.~Villani.
\newblock Generalization of an inequality by {T}alagrand and links with the
  logarithmic {S}obolev inequality.
\newblock {\em J. Funct. Anal.}, 173(2):361--400, 2000.

\bibitem{RiffordCarnot}
L.~Rifford.
\newblock Ricci curvatures in {C}arnot groups.
\newblock {\em Math. Control Relat. Fields}, 3(4):467--487, 2013.

\bibitem{riffordbook}
L.~Rifford.
\newblock {\em Sub-{R}iemannian geometry and optimal transport}.
\newblock SpringerBriefs in Mathematics. Springer, Cham, 2014.

\bibitem{RT-MorseSard}
L.~Rifford and E.~Tr\'{e}lat.
\newblock Morse-{S}ard type results in sub-{R}iemannian geometry.
\newblock {\em Math. Ann.}, 332(1):145--159, 2005.

\bibitem{R-MCP}
L.~Rizzi.
\newblock Measure contraction properties of {C}arnot groups.
\newblock {\em Calc. Var. Partial Differential Equations}, 55(3):Art. 60, 20,
  2016.

\bibitem{RS-3Sas}
L.~Rizzi and P.~Silveira.
\newblock Sub-{R}iemannian {R}icci curvatures and universal diameter bounds for
  3-{S}asakian manifolds.
\newblock {\em J. Inst. Math. Jussieu}, 18(4):783--827, 2019.

\bibitem{RS-Failure}
L.~Rizzi and G.~Stefani.
\newblock Failure of curvature-dimension conditions on sub-{R}iemannian
  manifolds via tangent isometries.
\newblock {\em J. Funct. Anal.}, 285(9):Paper No. 110099, 31, 2023.

\bibitem{Stefani-Heat}
G.~Stefani.
\newblock Generalized {B}akry-\'{E}mery curvature condition and equivalent
  entropic inequalities in groups.
\newblock {\em J. Geom. Anal.}, 32(4):Paper No. 136, 98, 2022.

\bibitem{Strichartz}
R.~S. Strichartz.
\newblock Sub-{R}iemannian geometry.
\newblock {\em J. Differential Geom.}, 24(2):221--263, 1986.

\bibitem{sturm:I}
K.-T. Sturm.
\newblock On the geometry of metric measure spaces. {I}.
\newblock {\em Acta Math.}, 196(1):65--131, 2006.

\bibitem{sturm:II}
K.-T. Sturm.
\newblock On the geometry of metric measure spaces. {II}.
\newblock {\em Acta Math.}, 196(1):133--177, 2006.

\bibitem{Trelatvalue}
E.~Tr\'{e}lat.
\newblock Some properties of the value function and its level sets for affine
  control systems with quadratic cost.
\newblock {\em J. Dynam. Control Systems}, 6(4):511--541, 2000.

\bibitem{Vil}
C.~Villani.
\newblock {\em Optimal transport}, volume 338 of {\em Grundlehren der
  Mathematischen Wissenschaften [Fundamental Principles of Mathematical
  Sciences]}.
\newblock Springer-Verlag, Berlin, 2009.
\newblock Old and new.

\bibitem{CVNB}
C.~Villani.
\newblock In\'{e}galit\'{e}s isop\'{e}rim\'{e}triques dans les espaces
  m\'{e}triques mesur\'{e}s [d'apr\`es {F}. {C}avalletti \& {A}. {M}ondino].
\newblock Number 407, pages Exp. No. 1127, 213--265. 2019.
\newblock S\'{e}minaire Bourbaki. Vol. 2016/2017. Expos\'{e}s 1120--1135.

\bibitem{VRN}
M.-K. von Renesse.
\newblock On local {P}oincar\'{e} via transportation.
\newblock {\em Math. Z.}, 259(1):21--31, 2008.

\bibitem{vRSt}
M.-K. von Renesse and K.-T. Sturm.
\newblock Transport inequalities, gradient estimates, entropy, and {R}icci
  curvature.
\newblock {\em Comm. Pure Appl. Math.}, 58(7):923--940, 2005.

\bibitem{W-anomalies}
E.~Witten.
\newblock Global gravitational anomalies.
\newblock {\em Comm. Math. Phys.}, 100(2):197--229, 1985.

\bibitem{ZeLi}
I.~Zelenko and C.~Li.
\newblock Differential geometry of curves in {L}agrange {G}rassmannians with
  given {Y}oung diagram.
\newblock {\em Differential Geom. Appl.}, 27(6):723--742, 2009.

\end{thebibliography}

\end{document}